\documentclass{memo-l}
\usepackage{amssymb}
\usepackage{hyperref}
\usepackage[alphabetic]{amsrefs}
\usepackage{graphicx, color}
\usepackage{enumerate}
\usepackage{xspace}
\newtheorem{theorem}{Theorem}[chapter]
\newtheorem{Thm}[theorem]{Theorem}
\newtheorem{lemma}[theorem]{Lemma}
\newtheorem{Lem}[theorem]{Lemma}
\newtheorem{corollary}[theorem]{Corollary}

\newtheorem{proposition}[theorem]{Proposition}
\newtheorem{Prop}[theorem]{Proposition}
\newtheorem{Conj}[theorem]{Conjecture}
\newtheorem{conjecture}[theorem]{Conjecture}

\newtheorem{question}[theorem]{Question}

\newtheorem*{claim}{Claim}
\newtheorem{claim-swm}{Claim}

\newtheorem{case-wm}{Case}
\newtheorem{case-He}{Case}
\newtheorem{case-sc}{Case}
\newtheorem{case-ppm}{Case}
\newtheorem{case-fan}{Case}
\newtheorem{case-mpm}{Case}
\newtheorem{case-Wa}{Case}
\newtheorem{case-CG}{Case}
\newtheorem{case-A5}{Case}
\newtheorem{case-A34}{Case}
\newtheorem{subcase-A34-2}{Subcase}[case-A34]
\newtheorem{case-ar}{Case}

\theoremstyle{definition}
\newtheorem{definition}[theorem]{Definition}
\newtheorem{example}[theorem]{Example}
\newtheorem{Rem}[theorem]{Remark}
\newtheorem{remark}[theorem]{Remark}

\newcommand{\ta}{\ensuremath{\widetilde{a}}\xspace}
\newcommand{\tb}{\ensuremath{\widetilde{b}}\xspace}
\newcommand{\tc}{\ensuremath{\widetilde{c}}\xspace}
\newcommand{\td}{\ensuremath{\widetilde{d}}\xspace}
\newcommand{\te}{\ensuremath{\widetilde{e}}\xspace}

\newcommand{\tu}{\ensuremath{\widetilde{u}}\xspace}
\newcommand{\tv}{\ensuremath{\widetilde{v}}\xspace}
\newcommand{\tw}{\ensuremath{\widetilde{w}}\xspace}
\newcommand{\tx}{\ensuremath{\widetilde{x}}\xspace}
\newcommand{\ty}{\ensuremath{\widetilde{y}}\xspace}
\newcommand{\tz}{\ensuremath{\widetilde{z}}\xspace}
\newcommand{\ts}{\ensuremath{\widetilde{s}}\xspace}
\newcommand{\tB}{\ensuremath{\widetilde{B}}\xspace}
\newcommand{\tS}{\ensuremath{\widetilde{S}}\xspace}
\newcommand{\tG}{\ensuremath{\widetilde{G}}\xspace}
\newcommand{\tX}{\ensuremath{\widetilde{X}}\xspace}
\newcommand{\tK}{\ensuremath{\widetilde{K}}\xspace}

\newcommand{\mr}{\mathrm}
\newcommand{\tr}{_{\scalebox{0.53}{$\triangle$}}}
\newcommand{\trsq}{_{\scalebox{0.53}{$\triangle\square$}}}

\newcommand{\sq}{_{\scalebox{0.53}{$\square$}}}

\newcommand{\ar}{\mathop{\rm Area}\xspace}
\newcommand{\lgate}{{\langle\langle}}
\newcommand{\rgate}{{\rangle\rangle}}

\numberwithin{section}{chapter}
\numberwithin{equation}{chapter}
\numberwithin{figure}{chapter}

\newcommand{\RR}{\mathbb R}
\newcommand{\ZZ}{\mathbb Z}

\makeindex

\begin{document}

\title[Weakly Modular Graphs]{Weakly Modular Graphs \\ and Nonpositive Curvature}

\begin{abstract}
This article investigates structural, geometrical, and topological
characterizations and properties of weakly modular graphs and of cell complexes derived from them.
The unifying themes of  our investigation are various ``nonpositive curvature" and ``local-to-global'' properties and characterizations
of weakly modular graphs  and their subclasses.
Weakly modular graphs  have been introduced
as a far-reaching common generalization of median graphs (and more generally,
of modular and orientable modular graphs), Helly graphs, bridged graphs, and dual polar graphs occurring under different disguises
($1$--skeletons, collinearity graphs, covering graphs, domains, etc.) in several seemingly-unrelated fields of mathematics:
\begin{itemize}
\item Metric graph theory
\item Geometric group theory
\item Incidence geometries and buildings
\item Theoretical computer science and combinatorial optimization
\end{itemize}
We give a local-to-global characterization of weakly modular graphs and their subclasses in terms of simple connectedness of associated
triangle-square complexes and specific local combinatorial
conditions. In particular, we revisit characterizations of dual polar
graphs by Cameron and by Brouwer-Cohen. We also show that (disk-)Helly
graphs are precisely the clique-Helly graphs with simply connected
clique complexes.  With $l_1$--embeddable weakly modular and sweakly modular
graphs we associate high-dimensional cell complexes, having several
strong topological and geometrical properties (contractibility and the
CAT(0) property). Their cells have a specific structure: they are
basis polyhedra of even $\triangle$--matroids in the first case and
orthoscheme complexes of gated dual polar subgraphs in the second
case.  We resolve some open problems concerning subclasses of weakly
modular graphs: we prove a Brady-McCammond conjecture about CAT(0)
metric on the orthoscheme complexes of modular lattices; we answer
Chastand's question about prime graphs for pre-median graphs.  We also
explore negative curvature for weakly modular graphs.

\end{abstract}

\author[J.\ Chalopin]{J\' er\'emie Chalopin}
\address{Laboratoire d'Informatique Fondamentale, Aix-Marseille Universit\'e and CNRS, Facult\'e des Sciences de Luminy, F-13288 Marseille Cedex 9, France}
\email{jeremie.chalopin@lif.univ-mrs.fr}
\thanks{J.C. was partially supported by ANR project MACARON (\textsc{anr-13-js02-0002}).}

\author[V.\ Chepoi]{Victor Chepoi}
\address{Laboratoire d'Informatique Fondamentale, Aix-Marseille Universit\'e and CNRS, Facult\'e des Sciences de Luminy, F-13288 Marseille Cedex 9, France}
\email{victor.chepoi@lif.univ-mrs.fr}
\thanks{V.C.\ was partially supported by ANR projects TEOMATRO (\textsc{anr-10-blan-0207}) and GGAA (\textsc{anr-10-blan-0116}).}

\author[H.\ Hirai]{Hiroshi Hirai}
\address{Department of Mathematical Informatics,
Graduate School of Information Science and Technology,
The University of Tokyo, Tokyo, 113-8656, Japan}
\email{hirai@mist.i.u-tokyo.ac.jp}
\thanks{H.H.\ was partially supported by JSPS KAKENHI Grant Numbers 25280004, 26330023, 17K00029.}

\author[D.\ Osajda]{Damian Osajda}
\address{Instytut Matematyczny,
Uniwersytet Wroc\l awski,
pl.\ Grunwaldzki 2/4,
50--384 Wroc{\l}aw, Poland}
\address{Institute of Mathematics, Polish Academy of Sciences,
	\'Sniadeckich 8, 00-656 War\-sza\-wa, Poland}
\email{dosaj@math.uni.wroc.pl}
\thanks{D.O.\ was partially supported by Narodowe Centrum Nauki, grants no.\ UMO-2012/06/A/ST1/00259 and UMO-2015/\-18/\-M/\-ST1/\-00050.}



%\date{\today}

\maketitle

\tableofcontents

%--------------------------- Sections -------------------------------------------

\mainmatter
\chapter{Introduction}
\label{s:intro}

\section{Avant-propos} This article investigates structural, geometrical, and topological
characterizations and properties of weakly modular graphs and of cell complexes derived from them.
{\it Weakly modular graphs} are defined as graphs satisfying the following two simple distance conditions (for every $k > 0$):
\begin{itemize}
\item \emph{Triangle condition} (TC): For any vertex $u$ and
any two adjacent vertices $v,w$ at distance $k$ to $u$,
there exists a common neighbor $x$ of $v,w$ at distance $k-1$ to $u$.
\item \emph{Quadrangle condition} (QC):
For any vertices $u,z$ at distance $k$ and any two neighbors $v,w$ of $z$ at distance
$k-1$ to $u$, there exists a common neighbor $x$ of $v,w$ at distance $k-2$ from $u$.
\end{itemize}
Weakly modular graphs have been introduced in \cite{Ch_metric} and \cite{BaCh_helly}  as
a far-reaching common generalization of graphs studied in {\it metric graph theory}:
examples are median graphs,  modular graphs,
Helly graphs, and bridged graphs.

Our investigation is motivated by metric graph theory as well as by
recent unexpected appearances of weakly modular graphs
in %several 
other fields of mathematics:
\begin{itemize}
\item Geometric group theory
\item Incidence geometries and buildings
\item Theoretical computer science and combinatorial optimization
\end{itemize}
The unifying themes of  our investigation
are various ``nonpositive-curvature" and ``local-to-global'' properties and characterizations
of weakly modular graphs  and their subclasses.

\section{Motivation}
Let us mention some motivating examples and predecessors of this article (for some undefined terms,
see Chapter~\ref{s:prel}).
%

%% \medskip\noindent
%%     {\bf Metric graph theory.}
\subsection*{Metric graph theory}
The main subject of metric graph theory is the investigation and
structural characterization of graph classes whose standard graph-metric
satisfies the main metric and convexity properties of
classical metric geometries  like ${\mathbb R}^n$
endowed with $l_2$, $l_1$, or $l_{\infty}$--metric, hyperbolic spaces,
Boolean spaces, or trees. Among such properties one can mention
convexity of balls or of neighborhoods of convex sets, Helly property for balls, isometric and
low-distortion embeddings into  classical host spaces, retractions, various four-point conditions,
uniqueness or existence of medians, etc.; for a survey of this theory, see \cite{BaCh}.

The main classes of graphs occurring in metric graph theory are median graphs, modular graphs, Helly graphs,
bridged graphs, $\delta$--hyperbolic graphs, $l_1$--graphs, and isometric subgraphs of hypercubes (partial cubes).
{\it Helly graphs} are the graphs in which the balls satisfy the Helly property.
The {\it bridged graphs} are the graphs in which the neighborhoods of convex sets are convex (one of the basic properties of
Euclidean convexity).
Finally, {\it median graphs} can be viewed as graphs
satisfying a basic property of trees and Boolean cubes (and more generally, covering graphs of distributive lattices):
each triplet of vertices admits a unique
median, i.e., a vertex lying simultaneously on shortest paths between any pair of vertices from the triplet.
They also can be viewed as the discrete median algebras or as covering graphs of median semilattices. Similarly to median graphs,
{\it modular graphs}  are the graphs in which each triplet of vertices admits a (not necessarily unique)
median.

Later, it turned out that these graphs give rise to important
cubical and simplicial complexes. The cubical complexes associated with median graphs
(called {\it median complexes}) endowed
with any of the intrinsic $l_1,l_2,$ or $l_{\infty}$--metrics constitute
geodesic metric spaces with strong properties. For example, median complexes
equipped with the $l_1$--metric are  median metric spaces  (i.e.,
every triplet of points has a unique median) and are isometric subspaces of
$l_1$--spaces \cite{vdV}. If a median complex carries  the intrinsic
$l_{\infty}$--metric instead, then the resulting metric space is
injective \cites{MaTa, vdV2}. Finally, if we impose the intrinsic
$l_2$--metric,  we obtain a metric space with global
non-positive curvature in the sense of Gromov \cite{Gr}.
In fact, median graphs are exactly the $1$--skeletons of CAT(0) cubical
complexes \cite{Ch_CAT}. CAT(0) cubical complexes were
characterized by Gromov \cite{Gr} in a local-to-global combinatorial
way as the simply connected cubical complexes
in which the links of vertices ($0$--cubes) are simplicial flag complexes.

Analogously to median graphs, bridged
graphs have been characterized in \cite{Ch_CAT} as the $1$--skeletons of simply connected simplicial flag
complexes in which the links of vertices ($0$--simplices) do not contain induced $4$-- and $5$--cycles.
Those simplicial complexes were rediscovered in \cites{Hag,JS} and  dubbed {\it systolic complexes}.
Systolic complexes satisfy many global properties of CAT(0) spaces (contractibility, fixed point property)
and were suggested in \cite{JS} as a variant of simplicial complexes of combinatorial nonpositive curvature.

Another interesting subclass of weakly modular graphs forbidding
generalizing median graphs is that of {\it weakly median graphs}:
those are the weakly modular graphs in which the quasi-medians of triplets are unique
(see Chapter~\ref{s:prel} for the definition of quasi-medians). Weakly median graphs can also be characterized as weakly
modular graphs not containing four forbidden induced subgraphs.
Similarly to the way the median graphs are built from cubes (which are Cartesian products of $K_2$),
it was shown in \cite{BaCh_weak} that weakly median graphs are obtained from Cartesian products of
some simple graphs (suboctahedra, $5$--wheels, and systolic plane triangulations), via successive gated amalgamations.

Generalizing the proof of the decomposition
theorem of \cite{BaCh_weak}, Chastand \cites{Cha1,Cha2} presented a general framework of fiber-complemented graphs allowing to
establish many  general properties, previously
proved only for particular classes of graphs.  An important subclass
of  fiber-complemented and weakly modular graphs is the class of pre-median graphs \cites{Cha1,Cha2}, defined by forbidding only two of the
four forbidden for weakly median graphs. However the structure of prime pre-median graphs was not clear and it was formulated in  \cite{Cha1}*{p.\ 121}
as an open problem to characterize all prime pre-median graphs. Bridged, weakly bridged, and Helly graphs are prime graphs, moreover bridged and weakly bridged
graphs are prime pre-median graphs (Weakly bridged graphs (and their complexes) defined in \cite{O-cnpc} and investigated in \cites{ChOs,O-cnpc} are the weakly modular
graphs with convex balls).  Bucolic graphs introduced in \cite{BCC+} via local conditions are the weakly modular
graphs whose primes are weakly bridged graphs. Triangle-square complexes of bucolic graphs were characterized in \cite{BCC+} in a local-to-global manner, like in the case of CAT(0) cubical complexes
or systolic complexes. Moreover, the prism complexes of bucolic graphs, called in \cite{BCC+} bucolic complexes have been proved to be contractible cell
complexes satisfying other  nonpositive-curvature-like properties.

Finally, let us mention a research on {\it basis graphs of matroids},  close in spirit to the above-mentioned work
(for definition, see Subsection~\ref{s:bamat}). Maurer \cite{Mau} characterized  graphs which can be represented as
basis graphs of matroids using two local conditions and a global metric condition (which is close to but weaker than the quadrangle
condition); this result has  been extended in \cite{Che_bas} to even
$\triangle$--matroids.  In particular, all basis graphs (of matroids and even $\triangle$--matroids) are meshed. Maurer \cite{Mau} conjectured
that the global metric condition  can be replaced by simple connectivity of the triangle-square complex of the graph and this conjecture
was confirmed in \cite{ChChOs_matroid}.

%% \medskip\noindent
%%     {\bf Geometric group theory.}
\subsection*{Geometric group theory}
Geometric group theory investigates  groups by exploring connections between them and topological and geometric properties of spaces
on which these groups act. Another important principle of this area is to consider groups themselves as geometric objects; see e.g., \cites{BrHa,Davis,Gr} as rich sources of ideas, examples, and references.
Weakly modular graphs and associated complexes appear in many contexts in this field, and play often an important role in understanding properties of corresponding groups.

%% Every hyperbolic group acts geometrically on a finitely dimensional contractible flag simplicial complex
%% associated to a Helly graph, that is on a Helly complex; see Section~\ref{s:biautomatic} below. Moreover, since thickenings  of median graphs are Helly graphs, also CAT(0)
%% cubical groups act geometrically on Helly complexes.

As we noticed already, median graphs are precisely the $1$--skeletons of {CAT(0) cube complexes}. The latter ones were introduced into
geometric group theory by Gromov~\cite{Gr}, although they had appeared naturally in many situations before.
For example, products of trees constitute one of the simplest source of such complexes, and lattices in the isometry groups of
products of trees were thoroughly investigated; see e.g., \cite{Mozes}. The other natural sources of CAT(0) cube complexes are
spaces associated naturally with right-angled Coxeter groups and right-angled Artin groups (RAAG's); cf.\ \cite{Davis}.
Sageev~\cite{Sag} showed that groups endowed with a kind of ``walls" admit geometric (that is, isometric, proper, and cocompact)
actions on CAT(0) cube complexes. This triggered the development of the so-called ``cubulations" of various groups. Such a procedure
relies on finding a nice action of a given group on \emph{space with walls}; see \cite{HaPau}.
Subsequently many classical groups were cubulated, that is they were shown to act geometrically on CAT(0) cubical complexes; see e.g.,
\cites{W-sc,W-qch}. Cubulated groups posses many nice properties; see e.g., \cite{ChDrHag}.
Using the so-called \emph{special cube complexes} \cite{HagWise}, by cubulating fundamental groups of $3$--manifolds, the
old standing Virtual Haken Conjecture was solved \cites{W-qch,Agol}.

{\it Buildings} were introduced by Tits~\cite{Ti}. They are certain simplicial complexes which are very symmetric, so that their automorphisms groups are rich and provide
many interesting examples; see e.g.\ \cites{BuildingBook,Ti}. Some buildings posses a natural structure of CAT(0) cubical complexes
(e.g., products of trees or buildings associated to right-angled Coxeter groups; see \cite{Davis}). Some other ones give rise to (weakly) systolic complexes (see e.g., \cites{JS,O-cnpc}),
or to other weakly modular graphs (see Section~\ref{subsec:building}).

Systolic complexes, being the clique complexes of bridged graphs, were introduced into geometric group theory by Januszkiewicz-\'Swi{\c a}tkowski \cites{JS} and Haglund \cite{Hag}.
Systolic and weakly systolic complexes \cites{O-cnpc,ChOs} lead
to numerous constructions of ``highly dimensional" (hyperbolic) groups
\cites{JS,O-chg}, with often unexpected properties. Moreover, finitely presented classical \cite{LySch} and graphical small cancellation groups
act geometrically on systolic complexes \cite{Wise-manu}.

{\it Orthoscheme complexes} for some lattices (see Chapter~\ref{s:ortho}) appear in \cite{BradyMcCammond} in relations to braid groups. Many other instances of weakly modular graphs %or graphs
%closely related to them 
appear in other places in geometric group theory and related fields. As the closest in spirit examples let us mention here
the studies of the Kakimizu graphs \cite{PrzSch} and of $1$--skeletons of some triangulations of $3$--manifolds \cite{O-cnct}.

%% \medskip\noindent
%%     {\bf Incidence geometries and buildings.}
\subsection*{Incidence geometries and buildings}
The main subject of incidence geometry  is the characterization and classification of
geometries defined by simple axioms on points and lines. The characterization of some such geometries
involve only metric and combinatorial properties of their collinearity graphs, and their
subspaces  can be recovered from these graphs. Projective and polar spaces are the most fundamental types of
incidence geometries; for definition, see Subsection~\ref{prel:dps} and for a full account of their theory, see
\cites{Sh,Foundation}.  Complemented modular lattices (whose covering graphs are modular) are equal to
subspace lattices of  projective spaces. Other weakly modular graphs arise from polar spaces:
it was recently observed in \cite{BaCh} that {\it dual polar graphs},
which are the collinearity graphs of incidence geometries dual to polar spaces,  are weakly modular.
This is a simple consequence of Cameron's characterization~\cite{Ca} of dual polar graphs.
Notice also that there is a local-to-global characterization
of dual polar spaces, due to Brouwer and Cohen \cite{BroCo}.

Projective and polar spaces also give rise to particular buildings.
Indeed, the order complexes of the subspaces of projective and polar spaces are spherical buildings
of special type. Spherical buildings are CAT$(1)$ spaces. By using this fact, Haettel, Kielak, and Schwer~\cite{HKS}
showed that the orthoscheme complexes of  complemented modular lattices are  CAT$(0)$ spaces.
Nonpositive curvature property of buildings of affine type (Euclidean buildings)
is well known: they are CAT$(0)$ spaces \cite{Davis}.
%

%% \medskip\noindent
%%     {\bf Theoretical computer science and combinatorial optimization.}
\subsection*{Theoretical computer science and combinatorial optimization}
Except geometric group theory, median graphs and their algebraic analogs --
median algebras and median semilattices -- naturally
arise in
some areas of theoretical computer science.
Median semilattices can be viewed as
domains \cites{ArOwSu,BaCo} of coherent event structures
\cite{WiNi}, a classical by now model on concurrency in computer science. For a presentation of
CAT(0) cube complexes from this perspective, see \cite{ChHa}; this geometric point of view
allowed already to disprove a conjecture on event structures, see \cite{Chepev}.
On the other hand, median--stable subsets of Boolean algebras
characterize the solution sets of certain Boolean expressions, namely of $2$--SAT instances \cite{Scha}.

A remarkable appearance of (weakly) modular graphs
has occurred in classifying computational complexity of the $0$--extension problem.
The {\it $0$--extension problem} is a version of facility location problems on graphs,
originated from operations research and known under the name of {\it multifacility location problem}.
This problem generalizes some basic combinatorial optimization problems,
such as minimum cut problem, and
has numerous applications in machine learning and computer vision.
%Chepoi \cite{Ch_facility} showed that the $0$--extension problem on median graphs is polynomially solvable.
%As it turned out later, this algorithmic result is
%much deeper and intrinsically connected
%with modularity and orientability of the underlying graph.
%In a sequence of several papers, Karzanov \cite{Kar98a, Kar04} and Hirai \cite{HH12}
%presented a complete dichotomy characterization of
%the tractability of the $0$--extension problem:
%if $G$ is an orientable modular graph,
%then the $0$--extension problem on $G$ is polynomially solvable, otherwise it is NP-hard.
In mid 90's, Chepoi \cite{Ch_facility} and Karzanov \cite{Kar98a} independently
found that the computational complexity of the $0$--extension problem
is much deeper and intrinsically connected with the (weak) modularity of the underlying graph.
Both works showed that the 0-extension problem on median graphs is polynomially  solvable.
The latter work, furthermore, showed the polynomial time solvability for the 0-extension problem
on a related subclass of modular graphs, and showed the NP-hardness for a graph that is not orientable modular.
Following a subsequent development \cite{Kar04},
Hirai \cite{HH12} finally proved the polynomial time solvability of
the $0$--extension problem on orientable modular graphs, and
thus a complete dichotomy characterization of the tractability of the 0-extension problem
was established: if $G$ is an orientable modular graph,
then the $0$--extension problem on $G$ is polynomially solvable, otherwise it is NP-hard.

Here a modular graph is called {\it orientable} if its edges can be directed in a such a way
that the opposite edges of each square have the same direction
(median graphs are  orientable by directing the edges away from a fixed base-point,
i.e., by viewing them as covering graphs of median semilattices). The covering graph of
a modular lattice is a fundamental example of an orientable modular graph.

Hirai \cite{HH12} formulated the $0$--extension problem as an optimization on
a certain simplicial complex
associated with an orientable modular graph $G$. Unexpectedly,
it turned out that this simplicial complex can be viewed as a gluing of
the orthoscheme complexes of complemented modular lattices corresponding to some
subgraphs of $G$.

\section{Main results}
In most of the above cases, the importance of the classes of graphs is often related to their nonpositive-curvature-like
properties. In this paper, we explore such properties for the class of weakly modular graphs and --- when it goes to some stronger
results --- for certain large subclasses  (pre-median and prime pre-median graphs, Helly graphs,
dual polar graphs, and so-called sweakly modular graphs). The graphs from those subclasses are or can be defined as weakly modular graphs
satisfying some local conditions (namely, by forbidding some small induced or isometric subgraphs). With an arbitrary
weakly modular graph $G$ we associate  its clique complex $X(G)$, its triangle complex $X\tr(G)$, its square complex $X\sq(G)$, and its triangle-square complex $X\trsq(G)$ (with  the graphs from some above-mentioned
subclasses we can also associate high-dimensional cell complexes, see below).

First (Chapter~\ref{s:plotoglo}) we present a local-to-global characterization of the
triangle-square complexes  $X\trsq(G)$ of all weakly modular graphs: they are exactly the simply connected triangle-square complexes whose graphs are locally weakly
modular (i.e., their balls of radius $3$ are weakly modular).  This result
may be viewed as an analogue of Cartan-Hadamard theorem for globally convex and globally nonpositively curved spaces, and
has several consequences for particular classes of weakly modular graphs (Helly, modular,
dual polar, etc). In particular, we show that Helly graphs are exactly the clique-Helly graphs $G$ whose clique  complexes $X(G)$ are simply connected
(or  contractible). Analogously,  modular graphs are the graphs which are locally modular and whose square complexes $X\sq(G)$ are simply connected. On the other hand,
an example shows that similar local-to-global characterizations do not longer hold for meshed graphs.

Next (Chapter~\ref{sec:premedian}) we turn to {\it pre-median graphs}, which are the weakly modular graphs not containing two induced subgraphs: $K_{2,3}$ and $W^-_4$. In the same (topological)
vein --- but using different techniques --- we prove that prime pre-median graphs are the pre-median graphs $G$ whose triangle complexes $X\tr(G)$ are simply
connected. This answers the question by Chastand mentioned above.
We prove that several known instances of polyhedral graphs
are prime pre-median graphs: Schl\"afli and Gosset graphs, hyperoctahedra and their subgraphs, Johnson graphs and half-cubes, as well as
all weakly modular graphs which are basis graphs of matroids and even $\triangle$--matroids. We believe  that
{\it thick} (i.e., each pair of vertices at distance $2$ is included in a square)  isometric weakly modular subgraphs of all pre-median graphs $G$ define
contractible cell complexes, but we were able to confirm this only in the case of $L_1$--embeddable pre-median graphs: in this case,
these contractible complexes $C(G)$ have Euclidean cells which arise from matroidal and subhyperoctahedral
subgraphs of $G$.
The construction of $C(G)$ generalizes constructions of
the above-mentioned nonpositively curved complexes: indeed,
if $G$ is a median graph, then $C(G)$ coincides with the median complex of $G$,
if $G$ is a (weakly) bridged graph, then $C(G)$ coincides
with the (weakly) systolic complex of $G$, and more generally,
if $G$ is a bucolic graph, then $C(G)$ coincides with the bucolic complex of $G$.

Then (Chapter~\ref{s:dupol}) we revisit Cameron's characterization of dual polar graphs and significantly simplify it. Namely, we show that in fact dual polar graphs are exactly
the thick weakly modular graphs not containing induced $K^-_4$ and isometric $K^-_{3,3}$ (i.e., $K_4$ and $K_{3,3}$ minus one edge). Using this new characterization
of dual polarity  and the local-to-global characterization of weak modularity, we provide a completely different approach to the
local-to-global characterization of dual polar graphs of Brouwer and Cohen.

The {\it swm-graphs} (sweakly modular graphs) represent a
natural extension of dual polar graphs, because they are defined as
the weakly modular graphs not containing induced $K^-_4$ and isometric
$K^-_{3,3}$. On the other hand, they naturally generalize the
orientable modular graphs (and strongly modular graphs of \cite{BVV})
\footnote{According to the  urbain dictionary http://fr.urbandictionary.com/, {\it sweak}
is a  hybrid of  {\it super and weak}. This is because {\it sweakly modular graphs} inherit the properties of strongly modular
graphs and weakly modular graphs.}.
We establish (Chapter~\ref{sec:swm}) that particular
swm-graphs also arise in a natural way from Euclidean buildings of
type $C_n$.  However, it was not a priori clear how swm-graphs are
related to dual polar and to orientable modular graphs. We elucidate
those two connections, which both turn out to be crucial.  Namely, we
show that all thick gated subgraphs (which we call ``Boolean-gated'')
of an swm-graph $G$ are dual polar graphs.  In the case of median
graphs, the Boolean-gated subgraphs are all cubes. Extending the fact
that each median graph is a gluing of its cubes, we obtain that each
swm-graph is a gluing of its gated dual polar subgraphs. Analogously
to (and generalizing) the well-known result that a thickening of a
median graph is Helly, we show that the thickening $G^{\Delta}$ of an
swm-graph $G$ is also a Helly graph, and thus its clique complex is
contractible (two vertices are adjacent in $G^{\Delta}$ if and only if
their gated hull in $G$ is a dual polar subgraph).
Analogously to normal cube paths in CAT(0) cube complexes of Niblo and
Reeves~\cite{NiRe}, we define normal Boolean-gated paths in
swm-graphs. Extending the biautomaticity of groups acting
geometrically on CAT(0) cube complexes~\cite{NiRe}, we prove a similar
result for groups acting geometrically on swm-graphs.
Denoting by ${\mathcal B}(G)$ the set of all Boolean-gated subgraphs and considering it as a (graded)
poset with respect to the reverse inclusion, we can associate to $G$ two objects:
the covering graph $G^*$ of ${\mathcal B}(G)$ and
the orthoscheme complex $K(G) $ of ${\mathcal B}(G)$. We call $G^*$ the {\it barycentric graph} of $G$ and show that $G^*$ is an orientable modular graph and that
$G$ is isometrically embeddable in $G^*$ (the edges of $G^*$ have length $1/2$).
The sequence of barycentric graphs $G,G^*,(G^*)^*,\ldots$ of swm-graphs
{\em converges} to $K(G)$ (in some sense),
where each orthoscheme complex $K(G^{*i})$
is isometric to $K(G)$ and $K(G^{*i})$ is a simplicial subdivision of $K(G^{*i-1})$.
As a simple bi-product of those results and the result of \cite{HH12}, we immediately obtain a factor $2$ approximation
to the $0$--extension problem on swm-graphs.

The orthoscheme complex $K(G)$ of an swm-graph $G$ has many important
structural properties analogous to the metric properties of median
complexes (CAT(0) cube complexes) mentioned above.  We show that
$K(G)$ is contractible (Chapter~\ref{sec:complex}). Moreover, if $G$
is locally finite and $K(G)$ is endowed with the intrinsic
$l_{\infty}$--metric, then $K(G)$ is injective and $G^{\Delta}$ is an
isometric subspace. If $K(G)$ is endowed with the intrinsic
$l_1$--metric, then $K(G)$ is a strongly modular metric space (modular
space without isometrically embedded copies of $K^-_{3,3}$) and $G$ is
an isometric subspace. Finally, we conjecture that if $K(G)$ is
endowed with the $l_2$--metric, then $K(G)$ is a CAT$(0)$ space.  We
were not able to settle this conjecture in its full generality, but we
prove it in several important cases (Chapter~\ref{s:ortho} and
Chapter~\ref{sec:complex}). Our main result here is the proof that the
orthoscheme complexes of modular lattices are CAT$(0)$, thus answering
a question by Brady and McCammond~\cite{BradyMcCammond}. We also prove
that the orthoscheme complexes of median semilattices are CAT$(0)$ and
we show that to prove the conjecture for orthoscheme complexes of all
swm-graphs it suffices to prove it for orthoscheme complexes of
modular semilattices.

We conclude the paper with some further results about general weakly modular graphs (Chapter~\ref{s:metric}). We prove that weakly modular graphs (and more generally, meshed graphs) satisfy the quadratic isoperimetric inequality. Similarly to the fact that hyperbolicity of median graphs is controlled by the size of isometrically embedded square grids, we show that the hyperbolicity of weakly modular graphs is controlled by the sizes of metric triangles and square grids. Answering a question of \cite{Ch_dism} we prove that
any weakly modular graph admits a distance-preserving ordering of its vertices and that such an ordering can be found using breadth-first-search. (Notice that for many classes of weakly
modular graphs, like bridged, weakly bridged and Helly graphs, as well as for Kakimizu graphs, breadth-first-search or its stronger or weaker versions provide a dismantling order,
which imply that the clique complexes of the respective graphs are contractible.)
We finish by proposing a notion of a ``weakly modular" complex. We show that in some special cases our general definition results
in well-known complexes. We were not able to prove that the complex has ``nice" properties in general, and we believe it is worth further studies.

\section*{Acknowledgements}
% We are grateful to the anonymous referee for a careful reading of all parts of the paper and numerous useful comments.
 We are very grateful to the anonymous referee for the careful reading of all parts of the paper, for all his efforts and
 time spent, and for his numerous corrections and improvements.

\chapter{Preliminaries}
\label{s:prel}

\section{Basic notions}
\label{s:bapre}

\subsection{Graphs}
\label{s:bagraphs}
A \emph{graph} $G=(V,E)$ consists of a set of vertices $V:=V(G)$ and a set of edges $E:=E(G)\subseteq V\times V$. All graphs considered in this paper are
undirected, connected, contain no multiple edges, neither loops, and are not
necessarily finite or locally finite. That is, they are \emph{one-dimensional simplicial complexes}.
For two distinct vertices $v,w\in V$ we write $v\sim w$ (respectively, $v\nsim w$) when there is an (respectively, there is no) edge connecting
$v$ with $w$, that is, when $vw:=\{v,w \} \in E$.
(Using ``$\sim$" we specify the underlying graph if it is not clear.)
For vertices $v,w_1,\ldots,w_k$, we write $v\sim w_1,\ldots,w_k$ (respectively, $v\nsim w_1,\ldots,w_k$) or $v\sim A$ (respectively,
$v\nsim A$) when $v\sim w_i$ (respectively, $v\nsim w_i$), for each $i=1,\ldots, k$, where $A=\{ w_1,\ldots,w_k\}$.
As maps between graphs $G=(V,E)$ and $G'=(V',E')$ we always consider \emph{simplicial maps}, that is functions of the form
$f\colon V\to V'$ such that if $v\sim w$ in $G$ then $f(v)=f(w)$ or $f(v)\sim f(w)$ in $G'$.
A $(u,w)$--path $(v_0=u,v_1,\ldots,v_k=w)$ of \emph{length} $k$ is a sequence of vertices with $v_i\sim v_{i+1}$. If $k=2,$
then we call $P$ a {\it 2-path} of $G$. If $x_i\ne x_j$ for $|i-j|\ge 2$, then $P$ is
called a {\it simple $(a,b)$--path}.
A $k$--cycle $(v_0,v_1,\ldots,v_{k-1})$ is a path $(v_0,v_1,\ldots,v_{k-1},v_0)$.
For a subset
$A\subseteq V,$ the subgraph of $G=(V,E)$  {\it induced by} $A$
is the graph $G(A)=(A,E')$ such that $uv\in E'$ if and only if $uv\in E$.
$G(A)$ is also called a {\it full subgraph} of $G$. We will say
that a graph $H$ is {\it not an induced subgraph} of $G$ if $H$ is not isomorphic
to any induced subgraph $G(A)$ of $G$. A \emph{square} $uvwz$ (respectively, \emph{triangle} $uvw$) is an induced $4$--cycle $(u,v,w,z)$ (respectively, $3$--cycle $(u,v,w)$).

The {\it  distance}
$d(u,v)=d_G(u,v)$ between two vertices $u$ and $v$ of a graph $G$ is the
length of a shortest $(u,v)$--path.  For a vertex $v$ of $G$ and an integer $r\ge 1$, we will denote  by $B_r(v,G)$ (or by $B_r(v)$)
the \emph{ball} in $G$
(and the subgraph induced by this ball)  of radius $r$ centered at  $v$, i.e.,
$B_r(v,G)=\{ x\in V: d(v,x)\le r\}.$ More generally, the $r$--{\it ball  around a set} $A\subseteq V$
is the set (or the subgraph induced by) $B_r(A,G)=\{ v\in V: d(v,A)\le r\},$ where $d(v,A)=\mbox{min} \{ d(v,x): x\in A\}$.
As usual, $N(v)=B_1(v,G)\setminus\{ v\}$ denotes the set of neighbors of a vertex
$v$ in $G$. The \emph{link} of $v \in V(G)$ is the subgraph of $G$ induced by $N(v)$. A graph $G=(V,E)$
is {\it isometrically embeddable} into a graph $H=(W,F)$
if there exists a mapping $\varphi : V\rightarrow W$ such that $d_H(\varphi (u),\varphi
(v))=d_G(u,v)$ for all vertices $u,v\in V$.

The {\em wheel} $W_k$ is a graph obtained by connecting a single
vertex -- the {\em central vertex} $c$ -- to all vertices of the
$k$--cycle $(x_1,x_2, \ldots, x_k)$; the {\it almost wheel}
$W_k^{-}$ is the graph obtained from $W_k$ by deleting a spoke (i.e.,
an edge between the central vertex $c$ and a vertex $x_i$ of the
$k$--cycle). Analogously $K^-_4$ and $K^-_{3,3}$ are the graphs obtained from $K_4$ and $K_{3,3}$ by removing one edge.

An $n$--{\it octahedron} $K_{n\times 2}$ (or, a {\it hyperoctahedron}, for short) is the complete graph $K_{2n}$
on $2n$ vertices minus a perfect matching. Any induced subgraph of $K_{n\times 2}$ is called a {\it subhyperoctahedron}.
A {\it hypercube} $H(X)$ is a graph having the finite subsets of $X$ as vertices
and two  such sets $A,B$ are adjacent in $H(X)$
if and only if $|A\triangle B|=1$. A {\it half-cube} $\frac{1}{2}H(X)$ has the vertices of a hypercube $H(X)$ corresponding to finite subsets of $X$ of even cardinality
as vertices and two  such vertices are adjacent in $\frac{1}{2}H(X)$  if and only if their  distance
in $H(X)$ is 2 (analogously one can define a half-cube on finite subsets of odd cardinality). If $X$ is finite and has size $n$, then the half-cube on $X$ is denoted by $\frac{1}{2}H_n$.  For a positive integer $k$, the {\it Johnson graph}  $J(X,k)$ has the subsets of $X$ of size $k$ as vertices and two such vertices are adjacent in $J(X,k)$
if and only if their  distance in $H(X)$ is $2$.  All Johnson graphs $J(X,k)$ with even $k$ are isometric subgraphs of the half-cube $\frac{1}{2}H(X)$.  If $X$ is finite and $|X|=n$, then the hypercube, the half-cube, and the Johnson graphs are usually denoted by $H_n, \frac{1}{2}H_n,$ and $J(n,k),$
respectively. Finite hypercubes,  half-cubes, Johnson graphs, and hyperoctahedra are distance-regular graphs \cite{BrCoNeu}.

A {\it retraction}
$\varphi$ of a graph $G$ is an idempotent nonexpansive mapping of $G$ into
itself, that is, $\varphi^2=\varphi:V(G)\rightarrow V(G)$ with $d(\varphi
(x),\varphi (y))\le d(x,y)$ for all $x,y\in W$ (equivalently, a retraction is a
simplicial idempotent map $\varphi: G\rightarrow G$). The subgraph of $G$
induced by the image of $G$ under $\varphi$ is referred to as a {\it retract} of $G$.

The {\it interval}
$I(u,v)$ between $u$ and $v$ consists of all vertices on shortest
$(u,v)$--paths, that is, of all vertices (metrically) {\it between} $u$
and $v$: $I(u,v)=\{ x\in V: d(u,x)+d(x,v)=d(u,v)\}.$ If $d(u,v)=2,$ then
$I(u,v)$ is called a {\it 2-interval}. A $2$--interval $I(u,v)$ of a graph
$G$ is called {\it thick} if $I(u,v)$ contains a square (induced
$4$--cycle). Obviously each thick $2$--interval $I(u,v)$ has a square containing the vertices $u$
and $v$. A graph $G$ is called {\it thick} (or {\it 2-thick}) if all its 2-intervals are thick.
A graph $G$ is called {\it thin} (or {\it 2-thin}) if neither of its 2-intervals is thick.
Equivalently, a graph is thin if and only if it does not contains squares.

An induced
subgraph of $G$ (or the corresponding vertex set $A$) is called {\it  convex}
if it includes the interval of $G$ between any pair of its
vertices. The smallest convex subgraph containing a given subgraph $S$
is called the {\it convex hull} of $S$ and is denoted by conv$(S)$. An induced
subgraph of $G$ (or the corresponding vertex set $A$) is called {\it
locally  convex} (or {\it 2-convex}) if it includes the interval of $G$
between any pair of its vertices at distance two having a common neighbor
in $A$.  An induced subgraph $H$ (or the corresponding vertex set of $H$)
of a graph $G$
is {\it gated} \cite{DrSch} if for every vertex $x$ outside $H$ there
exists a vertex $x'$ in $H$ (the {\it  gate} of $x$)
such that  $x'\in I(x,y)$ for any $y$ of $H$. Gated sets are convex and
the intersection of two gated sets is
gated. By Zorn's lemma there exists a smallest gated subgraph
$\lgate S \rgate$ containing a given subgraph $S$, called the
{\it gated hull} of $S$. A graph $G$ is a {\it gated amalgam} of two
graphs $G_1$ and $G_2$ if $G_1$ and $G_2$ are (isomorphic to) two intersecting
gated subgraphs of $G$ whose union is all of $G.$

Let $G_{i}$, $i \in \Lambda$ be an arbitrary family of graphs. The
\emph{Cartesian product} $\prod_{i \in \Lambda} G_{i}$ is a graph
whose vertices are all functions $x: i \mapsto x_{i}$, $x_{i} \in
V(G_{i})$.  Two vertices $x,y$ are adjacent if there exists an index
$j \in \Lambda$ such that $x_{j} y_{j} \in E(G_{j})$ and $x_{i} =
y_{i}$ for all $i \neq j$. Note that a Cartesian product of infinitely
many nontrivial graphs is disconnected. Therefore, in this case the
connected components of the Cartesian product are called {\it weak
  Cartesian products}.  Given an arbitrary family of graphs $G_{i}$,
$i \in \Lambda$, the \emph{strong product} $\boxtimes_{i \in \Lambda}
G_{i}$ is a graph whose vertices are all functions $x: i \mapsto
x_{i}$, $x_{i} \in V(G_{i})$.  Two distinct vertices $x,y$ are
adjacent if for every index $j \in \Lambda$, either $x_j = y_j$ or
$x_{j} y_{j} \in E(G_{j})$.

\subsection{Complexes}
\label{s:bacom}
All complexes considered in this paper are CW complexes. Following \cite{Hat}*{Chapter 0}, we call them simply \emph{cell complexes} or
just \emph{complexes}.
If all cells are simplices and the nonempty intersections of two cells is their common face,
then $X$ is called a {\em simplicial complex}.
For a cell complex $X$, by $X^{(k)}$ we denote its \emph{$k$--skeleton}. All cell complexes considered in this paper will have graphs
(that is, one-dimensional simplicial complexes)
as their $1$--skeleta. Therefore, we use the notation $G(X):=X^{(1)}$.
As morphisms between cell complexes we always consider \emph{cellular maps}, that is, maps sending $k$--skeleton into the $k$--skeleton.

For a graph $G$, we define its \emph{triangle} (respectively, \emph{square})
complex $X\tr$ (respectively, $X\sq$) as a two-dimensional cell complex with $1$--skeleton $G$, and such that the two-cells
are (solid) triangles (respectively, squares) whose boundaries are identified by isomorphisms with (graph) triangles (respectively, squares)
in $G$. A \emph{triangle-square complex} $X\trsq(G)$ is defined analogously, as the union of $X\tr$ and $X\sq$ sharing common
$1$--skeleton $G$. A triangle-square complex is {\it flag}
if it coincides with the triangle-square complex of its 1-skeleton.

The \emph{star} of a vertex $v$ in a complex $X$, denoted $\mr{St}(v,X)$, is the subcomplex spanned by all cells containing $v$.

An {\em abstract simplicial complex} $\Delta$ on a set $V$
is a set of nonempty subsets of $V$ such that each member of $\Delta$, called a {\em simplex},
is a finite set, and any nonempty subset of a simplex is also a simplex.
A simplicial complex $X$ naturally gives rise
to an abstract simplicial complex $\Delta$ on the set of vertices ($0$--dimensional cells) of $X$
by: $U \in \Delta$ if and only if there is a simplex in $X$ having $U$ as its vertices.
Combinatorial and topological structures of $X$ are completely recovered from $\Delta$.
Hence we sometimes identify simplicial complexes and abstract simplicial complexes.

The {\it clique complex}
of a graph $G$ is the abstract simplicial complex $X(G)$ having the cliques (i.e.,
complete subgraphs) of $G$ as simplices. A simplicial complex $X$ is a {\it flag simplicial complex} if $X$ is the clique
complex of its $1$--skeleton.

Let $X$ be a cell complex and $C$ be a cycle in the $1$--skeleton of $X$. Then a cell complex $D$ is called a  {\it singular disk diagram}
(or Van Kampen diagram) for $C$ if the $1$--skeleton of $D$ is a plane graph whose inner faces are exactly the $2$--cells of $D$ and
there exists   a cellular map $\varphi:D\rightarrow X$  such that $\varphi|_{\partial D}=C$ (for more details see \cite{LySch}*{Chapter V}).
According to Van Kampen's lemma (\cite{LySch}, pp.\ 150--151), for every
cycle $C$ of a simply connected simplicial complex, one can construct a
singular disk diagram. A singular disk diagram with no cut vertices (i.e.,
its $1$--skeleton is $2$--connected) is
called a {\it disk diagram.} A {\it minimal (singular) disk} for $C$ is a
(singular) disk diagram ${D}$ for $C$ with a minimum number of $2$--faces.
This number is called the {\it (combinatorial) area} of $C$ and is
denoted Area$(C).$ If $X$ is a simply connected triangle-square complex, then
for each cycle $C$ all inner faces in a singular disk diagram $D$ of $C$ are
triangles or squares.

\subsection{Lattices}
\label{s:balat}
We continue with some standard definitions concerning partially ordered sets and lattices.
Let  $({\mathcal L}, \preceq)$ be a partially ordered set.
We set $x\prec y$ if $x\preceq y$ and $x\ne y$.
The maximum element, if it exists, is denoted by $1$,  and the minimum element, if it exists,
is denoted by $0$. An element $q$ {\it covers}
element $p$ if $p\prec q$ and $p\preceq x\preceq q$ implies that either $x=p$ or $x=q$. The underlying undirected graph of the Hasse diagram of
$\mathcal L$ is called the {\it covering graph} of $\mathcal L$. If $\mathcal L$ has a minimum element 0,
then an {\it atom} is an element that covers 0.
Given two elements $p,q$ of $\mathcal L$ for $p \preceq q$,
the {\it interval} $[p,q]$ is the set  $\{ x\in {\mathcal L}: p\preceq x\preceq q\}$.
A {\it chain} is a totally ordered subset $x_0\prec x_1\prec \cdots \prec x_k$,
and the {\it length} of this chain is $k$.
The {\it length} $r[x,y]$ of an interval $[x,y]$ is defined as
the maximum length of a chain from $x$ to $y$.
If $\mathcal L$ has the minimum element 0, then the {\it rank} (or  the {\it  height}) $r(x)$ of an element $x\in {\mathcal L}$ is defined as
$r(x)=r[0,x]$.
A {\it grade function} is a function $f: \mathcal{L} \to \ZZ$
such that $f(y) = f(x) +1$ holds provided $y$ covers $x$.
If a grade function $f$ exists, then $\mathcal{L}$ is called {\em graded}.
In this case, for any two elements
$x\prec y$, all maximal chains between $x$ and $y$
have the same length $r[x,y] = f(y) - f(x)$.
In addition, if $\mathcal L$ has the minimum element $0$,
then the grade function $f$ and the rank function $r$
are the same up to a constant, i.e.,  $f= r + f(0)$.

A pair $x,y$ of elements is said to be {\em upper-bounded}
if there is a common upper bound, and is said to be
{\em lower-bounded}  if  there is a common lower bound.
For $x,y\in {\mathcal L}$, the least common upper bound, if it exists is denoted by $x\vee y$, and the greatest common lower bound, if it exists,
is denoted by $x\wedge y$. The elements $x\vee y$ and $x\wedge y$ are called, respectively, the {\it join} and {\it meet} of $x$ and $y$.
$\mathcal L$ is said to be a {\it lattice} if both $x\vee y$ and $x\wedge y$ exist for any $x,y \in {\mathcal L}$, and is said to be
a {\it (meet-)semilattice} if $x\wedge y$ exists for any $x,y\in {\mathcal L}.$
Note that if $\mathcal L$ is a semilattice, then any upper-bounded pair $x,y$ has the join.
For a set $X \subseteq \mathcal{L}$,
let $\bigvee X = \bigvee_{x \in X} x$ and $\bigwedge X = \bigwedge_{x \in X} x$ denote
the least common upper bound and the greatest common lower bound of $X$, which are also called the join and meet of $X$, respectively,

The (principal) {\em ideal} $(p)^{\downarrow}$ of an element $p$
is the subset $\{ q \in \mathcal{L}: q \preceq p \}$.
Dually, the (principal) {\em filter} $(p)^{\uparrow}$ of an element $p$
is the subset $\{ q \in \mathcal{L}: p \preceq q \}$.

A lattice $\mathcal L$ is called {\it modular}
if $x\vee (y\wedge z)=(x\vee y)\wedge z$  for any $x,y,z\in{\mathcal L}$ with $x\preceq z$. Equivalently,
a lattice $\mathcal L$ is a modular lattice if and only if its rank function  satisfies the {\it modular equality}
$r(x)+r(y)=r(x\wedge y)+r(x\vee y)$ for all $x,y\in {\mathcal L}$ \cite{Birkhoff}*{Chapter III, Corollary 1}.
A lattice is said to be {\em complemented}
if every element $p$ has another element $q$ with property $p \wedge q = 0$
and $p \vee q = 1$; such an element is called a {\it complement} of $p$.
It is known \cite{Birkhoff} that a modular lattice of a finite rank is complemented
if and only if the maximum element is the join of atoms.
More generally, a  lattice is said to be {\em relatively complemented}
if every interval $[x,y]$ is a complemented lattice.
In the case of a modular lattice with finite rank,
the complementarity and the relative complementarity are equivalent.

The {\em order complex} of a poset $\mathcal{L}$ is
an abstract simplicial complex $\Delta(\mathcal{L})$ consisting of all chains of finite lengths of $\mathcal{L}$.
In the case when $\mathcal{L}$ has $0$ or $1$ or both,
the order complex of $\mathcal{L} \setminus \{0,1\}$ is often useful,
and is called the {\em reduced order complex} of $\mathcal{L}$.

\subsection{Incidence geometries}
\label{s:bageo}

We recall some basic notions and facts about incidence geometries; see
\cites{Foundation, Sh}.
A {\em point-line geometry} is a triple $\Pi = (P, L;R)$ of sets $P, L$
and a relation $R \subseteq P \times L$ between $P$ and $L$.
Elements of $P$ are called  {\em points}, and  elements of $L$ are called  {\em lines}.
If $(p,\ell) \in R$, then we say
that the point $p$ {\it lies} on the line $\ell$  or that the line $\ell$ {\it contains} the point $p$.
If two points $p,q$ lie on a common line $\ell$, then we say that $p$ and $q$ are {\it collinear}.
The {\em collinearity graph} $G:=G(\Pi)$ of $\Pi$
is the graph whose vertex set is the set $P$ of points so that $p,q \in P$ define an edge if and only if
$p$ and $q$ are collinear. A set $S \subseteq P$ of points is called a {\em subspace} of $\Pi$ if for every line $\ell$ either
$|\ell\cap S|\le 1$ or $\ell\subseteq S$. Any subspace can be regarded as a point-line geometry $\Pi_S$ consisting of
the points $S$ and the lines of $\Pi$ generated by $S$ together with the relation obtained by restricting $R$.
The intersection of any collection of subspaces is a subspace, thus for any subset
$X$ of $P$ there exists the smallest subspace containing $X$. A subspace $S$ is called a {\it singular subspace} if
any two points of $S$ are collinear, i.e., the subgraph of $G(\Pi)$ induced by $S$ is a clique. A subspace $S$ of $\Pi$
is called a {\it convex subspace} (respectively, a {\it gated subspace}) if $S$ induces a convex (respectively, gated) subgraph of $G(\Pi).$
A point $y$ of a subspace $S$ (in particular,  of a line $\ell$) of a point-line geometry $\Pi$ is a {\it nearest point} to a point $x$,
if $y$ is a closest to $x$ point of $S$ with respect to the graph-metric of $G$, i.e., $d_G(x,y)=\min \{ d_G(x,y'): y'\in S\}$.

\subsection{Matroids}
\label{s:bamat}

A {\it matroid} on a finite set $I$ is a collection
$\mathcal B$ of subsets of $I,$ called {\it bases,}  which satisfy the
following exchange property: for all $A,B\in {\mathcal B}$ and $a\in A\setminus B$
there exists $b\in B\setminus A$ such that $A\setminus
\{ a\}\cup \{ b\}\in {\mathcal B}$ (the base $A\setminus \{ a\} \cup \{ b\}$ is obtained
from the base $A$ by an elementary exchange). It is
well-known that all the bases of a matroid have the same
cardinality. The {\it basis graph} $G=G({\mathcal B})$ of a matroid $\mathcal B$ is the
graph whose vertices are the bases of $\mathcal B$ and edges are the
pairs  $A,B$ of bases differing by a single exchange (i.e., $\vert
A\triangle B\vert=2,$ where the {\it symmetric difference} of two
sets $A$ and $B$ is written and defined by $A\triangle B=(A\setminus
B)\cup (B\setminus A)$).

A {\it $\triangle$--matroid} is a collection $\mathcal B$
of subsets of a finite set $I,$  called {\it bases} (not
necessarily equicardinal) satisfying the  symmetric exchange
property: for any  $A,B\in {\mathcal B}$ and $a\in A\triangle B,$ there
exists $b\in B\triangle A$ such that $A\triangle\{ a,b\}\in {\mathcal
B}.$ A $\triangle$--matroid whose bases all have the same cardinality
modulo 2 is called an {\it even $\triangle$--matroid}. The {\it basis
graph} $G=G({\mathcal B})$ of an even $\triangle$--matroid $\mathcal B$ is the
graph whose vertices are the bases of $\mathcal B$ and edges are the
pairs $A,B$ of bases differing by a single exchange, i.e., $\vert
A\triangle B\vert=2$ (the basis graphs of arbitrary collections
of subsets of even size of $I$ can be defined in a similar way).

\subsection{Group actions}
\label{s:bagroupactions}

For a set $X$ and a group $\Gamma$, a \emph{$\Gamma$--action on $X$} is a group homomorphism $\Gamma \to \mr{Aut}(X)$.
If $X$ is equipped with an additional structure then $\mr{Aut}(X)$ refers to the automorphisms group of this structure.
We say then that \emph{$\Gamma$ acts on $X$ by automorphisms}, and $x\mapsto gx$ denotes the automorphism being the image of $g$.
In the current paper $X$ will be a graph or a complex, and thus $\mr{Aut}(X)$ will denote graph automorphisms or cellular
automorphisms.
For a group $\Gamma$ acting by automorphisms on a complex $X$, the \emph{minimal displacement} for the action is the number
$\min \{ d(v,gv)\; : \; g\in \Gamma \setminus \{ 1 \}, \; v\in X^{(0)} \}$.

\subsection{CAT(0) spaces and Gromov hyperbolicity} Let $(X,d)$ be a metric space.
A {\it geodesic segment} joining two
points $x$ and $y$ from $X$ is a map $\rho$ from the segment $[a,b]$
of ${\mathbb R}^1$ of length $|a-b|=d(x,y)$ to $X$ such that
$\rho(a)=x, \rho(b)=y,$ and $d(\rho(s),\rho(t))=|s-t|$ for all $s,t\in
[a,b].$ A metric space $(X,d)$ is {\it geodesic} if every pair of
points in $X$ can be joined by a geodesic segment. Every (combinatorial)
graph $G=(V,E)$ equipped with its standard
distance $d:=d_G$ can be transformed into a geodesic (network-like)
space $(X_G,d)$ by replacing every edge $e=uv$ by a segment
$\gamma_{uv}=[u,v]$ of length 1; the segments may intersect only at
common ends.  Then $(V,d_G)$ is isometrically embedded in a natural
way in $(X_G,d)$.

A {\it geodesic triangle} $\Delta (x_1,x_2,x_3)$ in a geodesic
metric space $(X,d)$ consists of three points in $X$ (the vertices
of $\Delta$) and a geodesic  between each pair of vertices (the
edges of $\Delta$). A {\it comparison triangle} for $\Delta
(x_1,x_2,x_3)$ is a triangle $\Delta (x'_1,x'_2,x'_3)$ in the
Euclidean plane  ${\mathbb E}^2$ such that $d_{{\mathbb
E}^2}(x'_i,x'_j)=d(x_i,x_j)$ for $i,j\in \{ 1,2,3\}.$ A geodesic
metric space $(X,d)$ is defined to be a {\it CAT(0) space}
\cite{Gr} if all geodesic triangles $\Delta (x_1,x_2,x_3)$ of $X$
satisfy the comparison axiom of Cartan--Alexandrov--Toponogov:

\medskip%\noindent
{\it If $y$ is a point on the side of $\Delta(x_1,x_2,x_3)$ with
vertices $x_1$ and $x_2$ and $y'$ is the unique point on the line
segment $[x'_1,x'_2]$ of the comparison triangle
$\Delta(x'_1,x'_2,x'_3)$ such that $d_{{\mathbb E}^2}(x'_i,y')=
d(x_i,y)$ for $i=1,2,$ then $d(x_3,y)\le d_{{\mathbb
E}^2}(x'_3,y').$}

\medskip%\noindent
This simple axiom turned out to be very powerful, because CAT(0)
spaces can be characterized  in several different  natural ways (for
a full account of this theory consult the book \cite{BrHa}).  CAT(0) is
also equivalent to the convexity of the function $f\colon[0,1]\rightarrow X$
given by $f(t)=d(\alpha (t),\beta (t)),$ for any geodesics $\alpha$
and $\beta$ (which is further equivalent to the convexity of the
neighborhoods of convex sets). This implies that CAT(0) spaces are
contractible. Any two points of a CAT(0) space   can be joined
by a unique geodesic.

A metric space $(X,d)$ is $\delta$--{\it hyperbolic}
\cites{BrHa,Gr} if for any four points $u,v,x,y$ of
$X$, the two larger of the three distance sums $d(u,v)+d(x,y)$,
$d(u,x)+d(v,y)$, $d(u,y)+d(v,x)$ differ by at most $2\delta \geq 0$. A
graph $G = (V,E)$ is $\delta$--\emph{hyperbolic} if $(V,d_G)$ is
$\delta$--{hyperbolic}.  In case of geodesic metric spaces and graphs,
$\delta$--hyperbolicity can be defined in several other equivalent
ways. Here we recall some of them, which we will use in our proofs.
For geodesic metric spaces and graphs, $\delta$--hyperbolicity can be
defined (up to a constant factor) as spaces in which all geodesic triangles
are $\delta$--slim. Recall that  a geodesic triangle $\Delta(x,y,z)$ is called
$\delta$--{\it slim} if for
any point $u$ on the side $[x,y]$ the distance from $u$ to $[x,z]\cup
[z,y]$ is at most $\delta$. Equivalently, $\delta$--hyperbolicity can be defined
via the linear isoperimetric inequality: all cycles in a  $\delta$--hyperbolic
graph or  geodesic
metric space admit a disk diagram of linear area and vice-versa all graphs or geodesic
metric spaces  in which all  cycles   admit disk diagrams of linear area are hyperbolic.

\section{Further notions}
\label{s:fupre}

\subsection{Weakly modular graphs}
\label{s:wmgra}

\begin{definition}[Weak modularity] \cites{BaCh_helly,Ch_metric} A graph $G$ is \emph{weakly modular with respect to a vertex $u$}
 if its distance function $d$ satisfies the
following triangle and quadrangle conditions
(see Figure~\ref{fig-conditions}):
\begin{itemize}
\item
\emph{Triangle condition} TC($u$):  for any two vertices $v,w$ with
$1=d(v,w)<d(u,v)=d(u,w)$ there exists a common neighbor $x$ of $v$
and $w$ such that $d(u,x)=d(u,v)-1.$
\item
\emph{Quadrangle condition} QC($u$): for any three vertices $v,w,z$ with
$d(v,z)=d(w,z)=1$ and  $2=d(v,w)\leq d(u,v)=d(u,w)=d(u,z)-1,$ there
exists a common neighbor $x$ of $v$ and $w$ such that
$d(u,x)=d(u,v)-1.$
\end{itemize}
A graph $G$ is called \emph{weakly modular}
if $G$ is weakly modular with respect to any vertex $u$.
\end{definition}

\begin{figure}[ht]
\begin{center}
\includegraphics{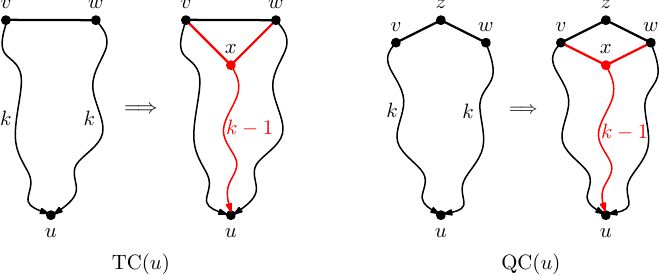}
\end{center}
\caption{Triangle and quadrangle conditions}\label{fig-conditions}
\end{figure}

Vertices $v_1,v_2,v_3$ of a graph $G$ form a {\it metric triangle}
$v_1v_2v_3$ if the intervals $I(v_1,v_2), I(v_2,v_3),$ and
$I(v_3,v_1)$ pairwise intersect only in the common end-vertices, i.e.,
$I(v_i, v_j) \cap I(v_i,v_k) = \{v_i\}$ for any $1 \leq i, j, k \leq 3$.
If $d(v_1,v_2)=d(v_2,v_3)=d(v_3,v_1)=k,$ then this metric triangle is
called {\it equilateral} of {\it size} $k.$ A metric triangle
$v_1v_2v_3$ of $G$ is a {\it quasi-median} of the triplet $x,y,z$
if the following metric equalities are satisfied:
$$\begin{array}{l}
d(x,y)=d(x,v_1)+d(v_1,v_2)+d(v_2,y),\\
d(y,z)=d(y,v_2)+d(v_2,v_3)+d(v_3,z),\\
d(z,x)=d(z,v_3)+d(v_3,v_1)+d(v_1,x).\\
\end{array}$$
If $v_1$, $v_2$, and $v_3$ are the same vertex $v$,
or equivalently, if the size of $v_1v_2v_3$ is zero,
then this vertex $v$ is called a {\em median} of $x,y,z$.
A median may not exist and may not be unique.
On the other hand, a quasi-median of every $x,y,z$ always exists:
first select any vertex $v_1$ from $I(x,y)\cap I(x,z)$ at maximal
distance to $x,$ then select a vertex $v_2$ from $I(y,v_1)\cap
I(y,z)$ at maximal distance to $y,$ and finally select any vertex
$v_3$ from $I(z,v_1)\cap I(z,v_2)$ at maximal distance to $z.$

In the sequel, we will use the following properties and characterizations of weakly modular graphs:

\begin{lemma} \cite{Ch_metric} \label{lem-weakly-modular} A graph $G$
  is weakly modular if and only if for any metric triangle $v_1v_2v_3$
  of $G$ and any two vertices $x,y\in I(v_2,v_3),$ the equality
  $d(v_1,x)=d(v_1,y)$ holds. In particular, all metric triangles of weakly modular graphs are equilateral.
\end{lemma}

\begin{lemma} \cite{Ch_metric} \label{lem-weakly-modular_gated_hull} Let $G$ be a weakly modular graph
and $H$ is a connected subgraph of $G$. Then
\begin{itemize}
\item[(i)] $H$ is convex if and only if $H$ is locally convex;
\item[(ii)] $H$ is gated if and only if for any two distinct vertices $u,v$ of $H$,
any common neighbor of $u,v$ belongs to $H$.
\end{itemize}
\end{lemma}

By Lemma \ref{lem-weakly-modular_gated_hull},
the gated hull $\lgate S\rgate$ of any set $S$ inducing a connected subgraph
of a finite weakly modular graph $G$ can be constructed by the following procedure:

%% \begin{algorithm}
%% \caption{Gated-Hull($S$)\label{algo-gatedhull}}
%% $U \leftarrow S$\;
%% \While{there exist $u, v \in U$ and $w \notin U$ such that $w \sim u,v$}
%%       {$U \leftarrow U \cup \{w\}$}
%% \Return $U$ \;
%% \end{algorithm}

\medskip

\begin{itemize}
\item[] {\textbf{GATED-HULL}}($S$)

\item[] $U \leftarrow S$;
\item[] \textbf{while} there exist $u, v \in U$ and $w \notin U$ such that $w \sim u,v$ \textbf{do}
\item[] \qquad \quad $U \leftarrow U \cup \{w\}$
\item[] \textbf{return} $U$.
\end{itemize}

\medskip

We can extend the procedure GATED-HULL to arbitrary weakly modular graphs $G$ in the following way.
Let $\triangleleft$ be a well-order on $V(G)$ and let $S$ be any subset of vertices inducing a
connected subgraph of $G$.   We define a subgraph $K$ of
$G$ by (possibly transfinite) induction as follows. Set $H_0:=G(S)$. Given an ordinal $\alpha$, assume that for
every $\beta < \alpha$, we have defined $H_\beta$, and let
$H_{<\alpha}$ be the subgraph induced by $\bigcup_{\beta < \alpha}V(H_\beta)$.  Let
$$X=\{v \in V(G)\setminus V(H_{< \alpha}):
\mbox{ there exist } x,y \in V(H_{<\alpha}) \mbox{ such that } v \sim x,y\}.$$
If $X$ is nonempty, then let $v$ be the least element of $(X,\triangleleft)$
and define $H_{\alpha}$ to be the subgraph of $G$ induced by $V(H_{< \alpha}\cup\{v\})$.
If $X$ is empty, then set $K:=H_{< \alpha}$.

\begin{lemma} \label{p:gate3}
$K$ is the gated hull of $S$ in $G$.
\end{lemma}

\begin{proof}
Let $A$ be the gated hull of $S$. First we prove that all vertices
of $K$ belong to $A$. Suppose by way of contradiction that
$K\setminus A \neq \emptyset.$ From all vertices in $K \setminus A$
we choose $v$ with smallest $\alpha$, such that $v \notin H_{<\alpha}$,
$v \in V(H_\alpha)$, i.e., all vertices from $H_{<\alpha}$ are contained
in $A$. Since $v \in V(H_\alpha)$, it has at least two neighbors in $H_{< \alpha}$
and thus in $A$. Therefore, there is no gate of $v$ in $A$, a contradiction.

On the other hand, since $G$ is weakly modular, by Lemma \ref{lem-weakly-modular}, $K$ is gated if
and only if $K$ is connected and for every $x,y\in V(K)$ at distance at most $2$, any common
neighbor $v$ of $x$ and $y$ also belongs to $K$.
All these are obviously true by the definition of $K$.
\end{proof}

The following property of gated sets in arbitrary metric spaces is fundamental. Recall that a
family of subsets $\mathcal F$ of a set $X$ satisfies the {\it (finite) Helly property} if for any (finite)
subfamily ${\mathcal F}'$ of $\mathcal F$, the intersection $\bigcap {\mathcal F}'=\bigcap \{ F: F\in {\mathcal F}'\}$
is nonempty if and only if $F\cap F'\ne \emptyset$ for any pair $F,F'\in {\mathcal F}'$. The finite Helly property
for gated sets is well-known, see, for example, \cite{vdV}*{Proposition 5.12 (2)}.

\begin{lemma}\label{lem:gated_Helly}
The family of gated sets of any metric space $(X,d)$ has the finite Helly property.
\end{lemma}

\begin{definition}[Local weak modularity]
\label{d:loc-weak-mod}
A graph $G$ is \emph{locally weakly modular with respect to a vertex $u$}
 if its distance function $d$ satisfies the
following local triangle and quadrangle conditions
(see Figure~\ref{fig-conditions}):
\begin{itemize}
\item
\emph{Local triangle condition} LTC($u$):  for any two adjacent vertices $v,w$
 such that $d(u,v)=d(u,w)=2$ there exists a common neighbor $x$ of $u$,  $v$
and $w$.
\item
\emph{Local quadrangle condition} LQC($u$): for any three vertices
$v,w,z$ such that $z \sim v, w$ and $d(v,w)=d(u,v)=d(u,w)=d(u,z)-1=
2$, there exists a common neighbor $x$ of $u$, $v$ and $w$.
\end{itemize}
A graph $G$ is \emph{locally weakly modular} if $G$ is locally weakly modular with respect to any vertex $u$.
\end{definition}

\subsection{Classes of weakly modular graphs}\label{classes}
We continue with the definition and basic properties of several known or new classes of weakly modular graphs. These classes
are defined either by forbidden isometric or induced subgraphs or by restricting the size of the metric
triangles of $G$.

%% \medskip\noindent
%% {\it 1. Median, modular, orientable modular, strongly modular graphs, and swm-graphs.}

\subsubsection{Median, modular, orientable modular, strongly modular graphs, and swm-graphs}
A graph $G$  is called {\it median} if  $|I(u,v)\cap I(v,w)\cap I(w,v)|=1$ for every triplet $u,v,w$ of vertices, i.e., every triplet of vertices has a unique median.
Median graphs
can be characterized in several different ways and that they play an important role in geometric group theory.
Here we present only the following characterizations and properties of median graphs,
which are related with the subject of our paper:
\begin{theorem} \label{median_survey} For a graph $G$, the following conditions are equivalent:
\begin{itemize}
\item[(i)] $G$ is a median graph;
\item[(ii)] \cite{Ba_retracts} $G$ is a retract of a hypercube;
\item[(iii)] $G$ is an isometric subgraph of a hypercube in which the half-spaces have convex boundaries;
\item[(iv)] \cite{Ch_CAT} $G$ is the 1-skeleton of a CAT(0) cube complex $X$;
\item[(v)] \cite{Gr} $G$ is the 1-skeleton of a simply connected cube complex $X$
satisfying the cube condition.
\end{itemize}
\end{theorem}
For other properties and characterizations of median graphs, see the survey \cite{BaCh}; for some other results on CAT(0) cube complexes, see the  paper \cite{Sag}.

A graph $G$ is called {\it modular} if $I(u,v)\cap I(v,w)\cap I(w,v)\ne \emptyset$ for every triplet $u,v,w$ of vertices, i.e., every triplet of vertices admits
a (not necessarily unique) median. Clearly a median graph is modular.
In view of Lemma \ref{lem-weakly-modular}, modular graphs are weakly modular.
Moreover, summarizing some well-known results from \cite{BVV} and \cite{Ch_metric},
modular graphs can be characterized in the following way:
\begin{lemma} \label{lem-modular} For a graph $G$, the following conditions are equivalent:
\begin{itemize}
\item[(i)] $G$ is modular;
\item[(ii)] all metric triangles of $G$ have size 0;
\item[(iii)] $G$ is a bipartite weakly modular graph;
\item[(iv)] $G$ is a triangle-free graph satisfying the quadrangle condition.
\end{itemize}
\end{lemma}

The term ``modular" comes from a connection to modular lattices.
Indeed, it is known that a lattice is modular if and only if
its covering graph is modular.
There is a semilattice analogue of a modular lattice satisfying this property.
A meet-semilattice ${\mathcal L}$ is called {\it modular}
if for any $x\in {\mathcal L}$ the principal ideal $(x)^{\downarrow} = [0,x]$ is
a modular lattice and  $x \vee y \vee z$ exists provided each of the joins $x\vee y,x\vee z,$ and $y\vee z$ exist.
This notion is due to Bandelt, van de Vel, and Verheul~\cite{BVV}.
\begin{Thm}[{\cite{BVV}*{Theorem 5.4}}]\label{thm:BVV1}
A discrete semilattice is modular if and only if its covering graph is modular.
\end{Thm}
Here a semilattice is said to be {\em discrete} if every element has  finite rank.
In a poset, the covering relation naturally
induces an orientation of its covering graph, which we call the {\em Hasse orientation}.
One can see that
the Hasse orientation of a modular (semi)lattice has the following property:
for each square $x_1x_2x_3x_4$, we have
$x_i \rightarrow x_{i+1}$ if and only if $x_{i+3} \rightarrow x_{i+2}$,
where $\rightarrow$ denotes the direction of the orientation and the indices
are considered cyclically. See Figure~\ref{fig:adm_orientation}.
\begin{figure}[ht]
\begin{center}
\includegraphics[scale=0.3]{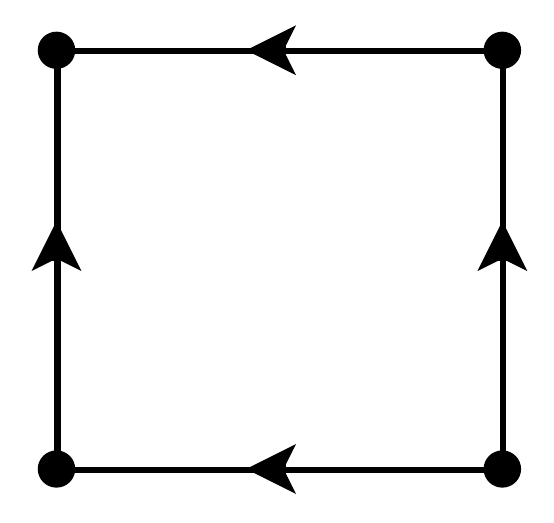}
\end{center}
\caption{Admissible orientation}\label{fig:adm_orientation}
\end{figure}

Such an orientation of a graph is called {\it admissible}.
A modular graph is {\it orientable}~\cite{Kar98a}
if it has an admissible orientation.
In particular, the covering graph of a modular (semi)lattice is
orientable modular with respect to the Hasse orientation.
Orientable modular graphs occur in connection with the so-called $0$--{\it extension problem}
\cite{Kar98a} and characterize the polynomial instances of this problem \cite{HH12}.
Notice that median graphs are orientable.
On the other hand, orientable modular graphs do not contain
$K^-_{3,3}$ as an isometric subgraph (but this does not characterize orientable modular graphs).
Another related class is the class of modular graphs not containing  isometric $Q_3$ (the 3-cube), $K^-_{3,3}$, and $K_{3,3}$, investigated in \cites{Ch_CAT,HH_folder,Kar98a}
in relation with so-called {\it folder complexes}, a class of 2-dimensional CAT(0) complexes with right triangles as cells.

In this paper, we will further investigate orientable modular graphs
and their following generalization:
a graph $G$ is called {\it strongly modular} if $G$ is a modular graph not containing $K^-_{3,3}$ as an isometric subgraph. Strongly modular graphs appeared in~\cite{BVV}.
A strongly modular graph can be characterized by
the modular-lattice structure of intervals in the following way.
For a vertex $b$, the {\em base point order} on $V(G)$ is
a partial order $\preceq_b$ defined by: $x \preceq_b y$ if $x \in I(b,y)$.
Regard $I(p,q)$ as a poset with respect to a base point order $\preceq_p$.
\begin{Thm}[{\cite{BVV}*{Theorem 4.7}}]\label{thm:BVV2}
A bipartite graph is strongly modular if and only if every interval is a modular lattice.
\end{Thm}
Notice that the modular-lattice structure of intervals
directly implies the quadrangle condition.

We will consider a nonbipartite generalization of strongly modular graphs,
called {\em sweakly modular graphs} or {\em swm-graphs},
which are defined as weakly modular graphs without induced $K_{4}^-$ and isometric $K_{3,3}^-$.
We will extend Theorem \ref{thm:BVV2} and many other important properties previously
known for median graphs to all swm-graphs.

%% \medskip\noindent
%%     {\it 2. Pseudo-modular,  Helly graphs, and quasi-median graphs.}
\subsubsection{Pseudo-modular,  Helly graphs, and quasi-median graphs}
    A graph $G$ is called {\it pseudo-modular} if any three pairwise
    intersecting balls of $G$ have a nonempty intersection
    \cite{BaMu_pmg}. This condition easily implies both the triangle
    and quadrangle conditions, and thus pseudo-modular graphs are
    weakly modular. In fact, pseudo-modular graphs are quite specific
    weakly modular graphs: from the definition also follows that all
    metric triangles of pseudo-modular graphs have size 0 or 1.
    Pseudo-modular graphs can be also characterized by a single metric
    condition similar to (but stronger than) both triangle and
    quadrangle conditions:

\begin{proposition} \cite{BaMu_pmg} \label{pseudo-modular}
A graph $G$ is pseudo-modular if and only if for any three vertices $u,v,w$ such that
$1\le d(u,w)\le 2$ and $d(v,u)=d(v,w)=k\ge 2,$ there exists
a vertex  $x\sim u,w$ and $d(v,x)=k-1$.
\end{proposition}

An important subclass of pseudo-modular graphs is constituted by Helly
graphs.  A graph $G$ is a {\it (finitely) Helly graph} if the family
of balls of $G$ has the (finite) Helly property, that is, every
(finite) collection of pairwise intersecting balls of $G$ has a
nonempty intersection.  Helly graphs are the discrete analogues of
hyperconvex spaces: namely, the requirement that radii of balls are
from the nonnegative reals is modified by replacing the reals by the
integers.  In perfect analogy with hyperconvexity, there is a close
relationship between Helly graphs and absolute retracts. A graph is an
{\it absolute retract} exactly when it is a retract of any larger
graph into which it embeds isometrically.  Then absolute retracts and
Helly graphs are the same \cite{HeRi}.  In particular, for any graph
$G$ there exists a smallest Helly graph comprising $G$ as an isometric
subgraph. A graph $G$ is a {\it (finitely) $1$--Helly graph} if the
family of unit balls (i.e., balls of radius $1$) of $G$ has the (finite)
Helly property.  A {\it (finitely) clique-Helly graph} is a graph in
which the collection of maximal cliques has the (finite) Helly
property.  A vertex $x$ of a graph $G$ is {\it dominated} by another
vertex $y$ if the unit ball $B_1(y)$ includes $B_1(x).$ A graph $G$ is
{\it dismantlable} if its vertices can be well-ordered $\prec$ so
that, for each $v$ there is a neighbor $w$ of $v$ with $w\prec v$
which dominates $v$ in the subgraph of $G$ induced by the vertices
$u\preceq v$.  The following theorem summarizes some of the
characterizations of finite Helly graphs:

\begin{theorem} \label{Helly} For a finite  graph $G$, the following
statements  are equivalent:

\begin{itemize}

\item[(i)] $G$ is a Helly graph;

\item[(ii)]  \cite{HeRi} $G$ is a retract of a strong product of
paths;

\item[(iii)] \cite{BaPr} $G$ is a dismantlable clique-Helly graph;

\item[(iv)] \cite{BP-absolute} $G$ is a weakly modular $1$--Helly graph.

\end{itemize}
\end{theorem}

For arbitrary graphs, the following compactness result for Helly property has been
proved by Polat and Pouzet:

\begin{proposition} \label{Helly_Polat} \cite{Po_helly} A graph $G$ not containing infinite cliques
is Helly if and only if $G$ is finitely Helly.
\end{proposition}

We also recall the following simple characterization of clique-Helly graphs:

\begin{proposition} \label{clique_Helly_triangle} \cites{Dragan,Szwarc} A graph $G$ is clique-Helly if and only
if for any triangle $T$ of $G$ the set $T^*$ of all vertices of $G$ adjacent with at least two vertices of $T$ contains a universal vertex, i.e., a vertex
adjacent to all remaining vertices of $T^*$.
\end{proposition}

From Proposition \ref{clique_Helly_triangle} and Theorem \ref{Helly} it immediately follows that finite clique-Helly and Helly
graphs can be recognized in polynomial time.

\medskip
In analogy with the fixed-cube property of median graphs,  every
automorphism of a Helly graph has a fixed clique.

The strong product  is the $l_{\infty}$ version of the Cartesian
product. Thus, when we turn all $k$--cubes of the Cartesian product
of $k$ paths into simplices, then we have the corresponding strong
product of $k$ paths. More generally, a similar operator
transforms median graphs into Helly graphs: let $G^{\Delta}$ be
the graph having the same vertex set as $G,$ where two vertices
are adjacent if and only if they belong to a common cube of $G$; $G^{\Delta}$ is called the
{\it thickening} of $G$.

\begin{proposition}
\label{p:medHel}
\cite{BavdV3}
If $G$ is a median graph, then $G^{\Delta}$ is a finitely Helly graph.
\end{proposition}

Finally, {\it quasi-median graphs} \cite{BaMuWi} are the $K^-_4$ and $K_{2,3}$--free weakly modular
graphs; equivalently, they are exactly the retracts of Hamming graphs (weak Cartesian products
of complete graphs).

%% \medskip\noindent
%%     {\it 3. Bridged and weakly bridged graphs.}
\subsubsection{Bridged and weakly bridged graphs}
    Bridged and weakly bridged graphs constitute other important
    subclasses of weakly modular graphs. A graph $G$ is called {\it
      bridged} \cites{FaJa,SoCh} if it does not contain any isometric
    cycle of length greater than $3$.  Alternatively, a graph $G$ is
    bridged if and only if the balls $B_r(A,G)=\{ v\in V: d(v,A)\le
    r\}$ around convex sets $A$ of $G$ are convex.  Bridged graphs are
    exactly weakly modular graphs that do not contain induced $4$--
    and $5$--cycles (and therefore do not contain $4$-- and
    $5$--wheels) \cite{Ch_metric}. A graph $G$ is {\it weakly bridged}
    \cites{ChOs,O-cnpc} if $G$ is a weakly modular graph with convex
    balls $B_r(x,G).$ Equivalently, weakly bridged graphs are exactly
    the weakly modular graphs without induced $4$--cycles $C_4$
    \cite{ChOs}.

%% \medskip\noindent
    %%     {\it 4. Pre-median graphs.}
\subsubsection{Pre-median graphs}
    A graph $G$ is called {\it pre-median} \cites{Cha1,Cha2} ({\it
      pm-graph}, for short) if $G$ is a weakly modular graph without
    induced $K_{2,3}$ and $W^-_4$. Here are the main properties of
    pre-median graphs (for definitions, see below):

\begin{theorem} \cites{Cha1,Cha2} \label{premedian_Chastand} Let $G$ be a pre-median graph. Then:
\begin{itemize}
\item[(i)] $G$ is elementary if and only if $G$ is prime;
\item[(ii)] $G$ is fiber-complemented;
\item[(iii)] $G$ is isometrically embeddable in a weak Cartesian product of its primes;
\item[(iv)] if each prime subgraph of $G$ is a moorable graph, then $G$ is the retract of a weak Cartesian product of its primes;
\item[(v)] if $G$ is finite, then $G$ can be obtained by gated amalgams from Cartesian products of
its prime subgraphs.
\end{itemize}
\end{theorem}

A graph $G$ is said to be {\it elementary} if the only proper gated
subgraphs of $G$ are singletons.  A graph with at least two vertices
is said to be \emph{prime} if it is neither a Cartesian product nor a
gated amalgam of smaller graphs. The prime gated subgraphs of a graph
$G$ are called the {\it primes} of $G$.  A {\it prime pre-median
  graph} (a {\it ppm-graph} for short) is a pre-median graph which is
a prime graph.

A gated subset $S$ of a graph $G$ gives rise to a partition $F_{a}$
$(a\in S)$ of the vertex-set of $G;$ viz., the {\em fiber} $F_{a}$ of
$a$ relative to $S$ consists of all vertices $x$ (including $a$
itself) having $a$ as their gate in $S$.  According to Chastand
\cites{Cha1,Cha2}, a graph $G$ is called {\it fiber-complemented} if
for any gated set $S$ all fibers $F_{a}$ $(a\in S)$ are gated sets of
$G$.

A map $f: V(G) \to V(G)$ is a \emph{mooring} of a graph $G$ onto $u$ if the following holds:
\begin{enumerate}
\item $f(u) = u$ and for every $v \neq u$, $f(v)\sim v$ and $d(f(v),u)
  = d(v,u)-1$.
\item for every edge $vw$ of $G$, $f(v)$ and $f(w)$ coincide or are
  adjacent.
\end{enumerate}
A graph $G$ is \emph{moorable} if, for every $u \in V(G)$, there
exists a mooring of $G$ onto $u$.

Known sub-classes of pre-median graphs are the median, quasi-median, weakly median, and bucolic graphs. They
can be characterized via their primes in the following way:

\begin{theorem} \label{premedian_primes} Let $G$ be a pre-median graph. Then:
\begin{itemize}
\item[(i)] \cite{Is} $G$ is median if and only if all its primes are $K_2$;
\item[(ii)] \cite{BaMuWi} $G$ is quasi-median if and only if all its primes are cliques;
\item[(iii)] \cite{BaCh_weak} $G$ is weakly median if and only if all its primes are subgraphs of hyperoctahedra,
the 5-wheel $W_5$, or planar $K_4$--free bridged graphs;
\item[(iv)] \cite{BCC+} $G$ is bucolic if and only if all its primes are weakly bridged graphs.
\end{itemize}
\end{theorem}

Analogously to the definition of median graphs, {\it weakly median graphs} are exactly the weakly modular graphs in which
all triplets have a unique quasi-median.
In particular, quasi-median graphs are the $K_4^-$--free weakly median graphs.
It was shown
in \cite{BaCh_weak} that all weakly median graphs are $L_1$--graphs (which are defined below).
{\it Bucolic graphs} are exactly the weakly modular graphs
without induced $K_{2,3}$, $W_4$, and $W_4^-$.
It was shown in \cite{BCC+} that bucolic graphs give raise
(i.e., are the 1-skeletons)
to contractible prism complexes (called {\it bucolic complexes}), which satisfy many properties of spaces of non-positive
curvature. In particular, they can be characterized in a local-to-global manner similar to CAT(0) cube complexes
(Theorem \ref{median_survey}(v)).

Chastand \cite{Cha1}*{p.\ 121} formulated as an open problem the
question of a characterization of all prime pre-median graphs. We will
answer this question in Chapter \ref{sec:premedian}.  For every
$L_1$--weakly modular graph $G$ admitting a scale-embedding into an
hypercube, we also show how to construct a contractible topological
space $C(G)$, which is a union of Euclidean polyhedra, such that the
union of 1-skeleta of cells of $C(G)$ coincides with $G$.

We will consider two classes of weakly modular graphs (pre-median graphs and swm-graphs)  which are
defined as weakly modular graphs not containing some subgraphs from the list ${\mathcal L}=\{ K^-_4,K_{2,3},K^-_{3,3},W^-_4\}$ as
isometric subgraphs ($W^-_4$ and $K_{2,3}$ are forbidden in pre-median graphs, and $K^-_4$ and $K^-_{3,3}$ are forbidden in swm-graphs).
The following result shows that the class of weakly modular graphs and its two subclasses are closed by
basic operations:

\begin{proposition} \label{prod-retracts} The classes  of weakly modular graphs, pre-median graphs, and swm-graphs are closed
under taking (weak) Cartesian products, gated amalgams, and retracts.
\end{proposition}

\begin{proof} The proof that Cartesian products and gated amalgams of weakly modular graphs are weakly modular is standard. Now, let $G$ be a retract of a weakly modular
graph $H$ and let $\varphi: V(H)\rightarrow V(G)$ be a retraction map. As a retract, $G$ is an isometric subgraph of $H$.
Let $u,v,w$ be three vertices of $G$ such that $v\sim w$ and $d_G(u,v)=d_G(u,w)$. Let $x'$ be a common neighbor of $v,w$ in $H$ having distance $k-1$ to $u$. Let $x=\varphi(x').$ Since the map
$\varphi$ is non-expansive, $\varphi(u)=u, \varphi(v)=v,$ and $\varphi(w)=w,$ and $G$ is an isometric subgraph of $H$, we conclude that $d_G(x,u)=k-1,d_G(x,v)=d_G(x,w)=1,$ thus $G$ satisfies the triangle condition (TC). The quadrangle condition (QC) can be verified in an analogous way. This shows that $G$ is weakly modular.

Notice that all four graphs from $\mathcal L$ are {\it irreducible graphs} 
sensu \cite{GrWi}, i.e., in any isometric embedding into the Cartesian
product of graphs, they appear as isometric subgraphs of a factor.
This means that $K_{2,3}$ and $W^-_4$ cannot arise in Cartesian
products of pre-median graphs and $K^-_4$ and $K_{3,3}^-$ cannot arise
in the Cartesian products of swm-graphs, establishing that those two
classes of graphs are closed under Cartesian products. If $G$ is a
retract of $H$, then $G$ is an isometric subgraph of $H$. Thus, if $H$
does not contain one or several subgraphs of $\mathcal L$ as isometric
subgraphs, then $G$ cannot contain them either. Therefore, if $H$ is
pre-median or sweakly modular, then $G$ is also pre-median or sweakly
modular.

Suppose now that $G$ is the gated amalgam of two weakly modular graphs $H$ and $H'$
along a common gated subgraph $H_0$. Suppose that $H$ and $H'$ do not contain a graph $L$ of
$\mathcal L$ as an isometric subgraph. Suppose to the contrary that $G$ contains $L$ as an
isometric subgraph. Then necessarily $L$ contains a vertex $x$ in $H\setminus H'$ and a vertex
$y\in H'\setminus H$. Then $L_0:=L\cap H_0$ must be an $(x,y)$--separator of $L$. Necessarily, $L_0$
contains at least two vertices. Then one can easily see that such a separator $L_0$ in each of the
four graphs of $\mathcal L$ leads to a contradiction with the fact that
$H_0$ is a gated subgraph of $H$ and $H'$. This shows that pre-median and swm-graphs
are also closed by taking gated amalgams.
\end{proof}

\subsection{Polar and dual polar spaces} \label{prel:dps}

For a point-line geometry $\Pi = (P, L; R)$, consider the following conditions:
\begin{itemize}
\item[(P0)] For every two points of $P$, there is a unique line in $L$ containing them.
\item[(P1)] (Veblen-Young axiom) For any four lines $\ell_1,\ell_2,\ell_3,\ell_4$ of $\mathcal L$ such that $\ell_1$ and $\ell_2$ intersect in a point $p$ and each of the lines
 $\ell_3,\ell_4$ intersects each of the lines $\ell_1,\ell_2$ in a point different from $p,$ the lines $\ell_3$ and $\ell_4$ themselves intersect at a point.
\item[(P2)] Every line contains at least three points.
\item[(P2$'$)] Every line contains at least two points.
\end{itemize}

A {\em  projective space} (respectively, {\em generalized projective space}) is a point-line geometry $\Pi = ({\mathcal P},{\mathcal L}; R)$
satisfying (P0), (P1), and (P2) (respectively, (P2$'$)). The {\em dimension} of a subspace $S$
of a projective space $\Pi$
is the length of a maximal chain of subspaces from $\emptyset$ to $S$ minus $1$.
Let ${\mathcal S}(\Pi) \subseteq 2^{P}$ denote the set of all subspaces of
$\Pi$.
We regard ${\mathcal S}(\Pi)$ as a poset with respect to the inclusion order $\subseteq$.

\begin{theorem}[\cite{Birkhoff}]\label{thm:Birkhoff}\ 
  
\begin{itemize}
\item[{\rm (1)}] For a generalized projective space $\Pi$ of dimension $n-1$,
the subspace poset ${\mathcal S}(\Pi)$ is a complemented modular lattice of rank $n$ with $\wedge = \cap$.
\item[{\rm (2)}]
For a complemented modular lattice ${\mathcal M}$ of rank $n$,
let $P$ and $L$ be the sets of  rank $1$  and rank $2$ elements of ${\mathcal M}$, respectively, and let $R \subseteq P \times L$ be a relation defined
as $(a,b) \in R$ if $a \preceq b$. Then  $\Pi = (P, L; R)$ is a generalized projective space of dimension $n-1$
 with ${\mathcal S}(\Pi) = {\mathcal M}$.
 \end{itemize}
\end{theorem}

For a point-line geometry $\Pi = (P, L; R)$, consider the following conditions:
\begin{itemize}
\item[(Q1)] For a point $p$ and a line $\ell$ not containing $p$,
either exactly one point on $\ell$ is collinear with $p$, or
all points on $\ell$ are collinear with $p$.
\item[(Q2)] Every line contains at least three points.
\item[(Q2$'$)] Every line contains at least two points.
\item[(Q3)] For every point $p$ there exists a point $q$ such that $p$ and $q$ are not collinear.
\end{itemize}
A {\em  polar space} (respectively,  {\em generalized polar space})
is a point-line geometry $\Pi = (P, L; R)$
satisfying (Q1) and (Q2) (respectively,  (Q2$'$)). In addition, if (Q3) is satisfied, then
$\Pi$ is said to be {\em nondegenerate}. The {\it rank} of a polar space $\Pi$  is the length $n$
of maximal chains of subspaces (ordered by inclusion).
\begin{Thm}[{\cite{Foundation}*{Theorem 2.18}}]\label{thm:Tits}
Let  $\Pi = (P, L; R)$ be a nondegenerate (generalized) polar space of rank $n$.
\begin{itemize}
\item[{\rm (P1)}] Any maximal proper subspace $S$ together with the subspaces it contains
is a (generalized) projective space of dimension $n-1$.
\item[{\rm (P2)}] The intersection of subspaces is a subspace.
\item[{\rm (P3)}] For a maximal subspace $U$ and a point $p \in P \setminus U$
there exists a unique maximal subspace $W$ such that $W$ contains $p$ and the dimension of $U \cap W$ is $n - 2$.
Furthermore, the points of $U$ collinear with $p$ are exactly the points of $W \cap U$.
\item[{\rm (P4)}] There exist two disjoint subspaces of dimension $n-1$.
\end{itemize}
\end{Thm}

Tits \cite{Ti}*{p.102} defines a polar space as a geometry
satisfying (P1) to (P4); in his definition, $\Pi_S$ is supposed to be  a ``generalized" projective space in (P1).
Polar spaces represent one of the fundamental types of incidence
geometries. Polar spaces of rank at least 3 with
thick lines (i.e., all lines contain at least three points) have
been classified in a seminal work by Tits \cite{Ti}: they can be
constructed from sesquilinear or pseudoquadratic forms on vector
spaces; cf. \cites{Foundation,Sh}. Polar spaces and spherical buildings 
of type $C$ constitute the same objects by a theorem of Tits \cite{Ti}; see the discussion on buildings in Section~\ref{subsec:building}
and the formulation of Tits' result in Theorem~\ref{thm:Tits_polar}.

A (generalized) polar space $\Pi = (P, L; R)$ gives rise to
another point-line geometry $\Pi^* = (P^*, L^*; R^*)$.
The point set $P^*$ is the set of all $n$--dimensional subspaces of $\Pi$,
and the line set $L^*$ is the set of all $(n-1)$--dimensional subspaces of $\Pi$, where
the relation $R^* \subseteq P^* \times L^*$ is defined as
$(W,U) \in R^*$ if $W \supseteq U$. A {\em dual polar space} is a point-line geometry $\Pi^*$
obtained from some (generalized) polar space $\Pi$ in this way. A {\em dual polar graph} $G$
is the collinearity graph of a dual polar space $\Pi^*$. A characterization of dual polar graphs was given by
Cameron \cite{Ca}; we will present and use this characterization in Chapter \ref{s:dupol}. As noticed
in \cite{BaCh}, from this characterization immediately follows that dual polar graphs are weakly modular.
Moreover the subspace poset of a polar space is a modular semilattice,
and its covering graph is orientable modular.
Here a modular semilattice is called {\em complemented} if each principal ideal is a complemented modular lattice.
\begin{Lem}\label{lem:polar=>modular_semilattice}
For a polar space $\Pi$, the subspace poset ${\mathcal S}(\Pi)$
is a complemented modular semilattice,
and its covering graph is orientable modular.
\end{Lem}
\begin{proof}
By (P2), ${\mathcal S}(\Pi)$ is a semilattice with $\wedge = \cap$.
By (P1) and Theorem~\ref{thm:Birkhoff},
every lower ideal is a complemented modular lattice.
Therefore every element is the join of atoms.
It suffices to show that every pairwise bounded set of atoms $u_1, u_2, u_3,\ldots,u_k$
has the join $u_1 \vee u_2 \vee u_3 \vee \cdots \vee u_k$.
Take a maximal subspace $X$ containing  as many as possible of $u_1,u_2, u_3,\ldots, u_k$.
If $X$ does not contain some $u_i$,
then by (P3) there exists a maximal subspace $Y$ such that
$Y$ contains $u_i$ and all atoms of
$X$ collinear to $u_i$ (having joins with $u_i$).
This is a contradiction to the maximality of $X$.
The latter part follows from Theorem~\ref{thm:BVV1}.
\end{proof}

We will see that dual polar graphs constitute a natural class of swm-graphs.
A dual polar graph can be completely recovered from the original polar space and its subspace poset.
By Lemma \ref{lem:polar=>modular_semilattice}, the covering graph of this poset is an orientable modular
graph. We will show that a similar relation holds for  arbitrary swm-graphs. 

\subsection{Other related graph classes}\label{prel:other} We continue with some classes of graphs
related to weakly modular graphs.

%% \medskip\noindent
%%     {\it 1. Meshed graphs.}
\subsubsection{Meshed graphs}
    A graph $G=(V,E)$ is called {\it meshed} \cite{BaCh} if for any
    three vertices $u,v,w$ with $d(v,w)=2,$ there exists a common
    neighbor $x$ of $v$ and $w$ such that $2d(u,x)\le d(u,v)+d(u,w).$
    Meshed graphs are thus characterized by some (weak) convexity
    property of the radius functions $d(\cdot,u)$ for $u\in V.$ This
    condition ensures that all balls centered at cliques in a meshed
    graph $G$ induce isometric subgraphs of $G$. All basis graphs of
    matroids and even $\triangle$--matroids are meshed
    \cite{Che_bas}. Also it is well known that all weakly modular
    graphs are meshed; for self-completeness, we provide its simple
    proof.

\begin{lemma} \label{wm-meshed}
Any weakly modular graph $G$ is meshed.
\end{lemma}

\begin{proof} Let $u,v,w$ be  three vertices of $G$ with $d(v,w)=2$.
If $|d(u,v)-d(u,w)|=2$, then any common neighbor of $v$ and $w$ can be taken as $x$. Now suppose that
$d(u,v)=d(u,w)=k$. If there exists a common neighbor $x$ of $v$ and $w$ with $d(u,x)\le k$, then we are done.
So, suppose that $v$ and $w$ have a common neighbor $y$ with $d(u,y)=k+1$. By (QC) we will find a common neighbor
$x$ of $v$ and $w$ with $d(u,x)=k-1,$ and we are done. Finally, let $d(u,v)=k-1$ and $d(u,w)=k$. Suppose that there exists
a common neighbor $y$ of $v,w$ with $d(u,y)=k$, otherwise we are done. By (TC) there exists $z\sim y,w$ with $d(u,z)=k-1$.
If $z\sim v$, then we are done. Otherwise, by (QC) there exists $u'\sim v,z$ with $d(u,u')=k-2$. Hence $d(w,u')=d(w,v)=2$. By
(TC) there exists $x\sim v,u',w$. Since $x$ is adjacent to $u'$, necessarily $d(u,x)\le k-1$, and $x$ is as required.
\end{proof}

Meshed graphs do not satisfy the quadrangle condition (QC), however they satisfy (TC):

\begin{lemma} \label{TC-meshed} Any meshed graph $G$ satisfies the triangle condition (TC).
\end{lemma}

\begin{proof}  Let $u,v,w$ be  three vertices of $G$ with $d(v,w)=1$ and $d(u,v)=d(u,w)=k$. We proceed by induction on $k$.
Let $z$ be any neighbor of $v$ with $d(u,z)=k-1$. If $z\sim w$, then we are done. Otherwise, since $G$ is meshed, there exists a common
neighbor $y$ of $z$ and $w$ with $d(u,y)=k-1$. By induction assumption, we can find a vertex $u'\sim z,y$ with $d(u,u')=k-2$. Then $d(w,u')=2$.
By meshedness of $G$ applied to the triplet $w,v,u'$, we can find a vertex $x\sim w,v,u'$. Then necessarily $d(u,x)=k-1$, and we are done.
\end{proof}

%% \medskip\noindent
%%     {\it 2. Basis graphs of matroids and $\triangle$--matroids.}

\subsubsection{Basis graphs of matroids and $\triangle$--matroids}
    Basis graphs of matroids and even $\triangle$--matroids (see
    Subsection \ref{s:bamat}) constitute a subclass of meshed
    graphs. Basis graphs of matroids of rank $k$ on a set of size $n$
    are isometric subgraphs of the Johnson graph $J(n,k)$ and basis
    graphs of even $\triangle$--matroids on a set of size $n$ are
    isometric subgraphs of the half-cube $\frac{1}{2}H_n$. Johnson
    graphs $J(n,k)$ are the basis graphs of uniform matroids of rank
    $k$, analogously the half-cubes $\frac{1}{2}H_n$ are the basis
    graphs of uniform even $\triangle$--matroids.

We recall now the  characterization of basis graphs of matroids and $\triangle$--matroids.  For this purpose,
we introduce the following {\it positioning} and {\it 2-interval conditions}:
\begin{itemize}
\item
{\it Positioning condition} (PC): for each vertex $u$ and each
square $v_1v_2v_3v_4$ of $G$ the equality $d(u,v_1)+d(u,v_3)=d(u,v_2)+d(u,v_4)$
holds.
\item
{\it 2-Interval condition} (IC$m$):  each 2-interval $I(u,v)$ is an induced subgraph of the $m$--hyperoctahedron $K_{m\times 2}$.
\end{itemize}

It is known that basis graphs of matroids and even $\triangle$--matroid graphs satisfy
(PC) \cites{Che_bas,Mau}. It is well-known \cite{Mau}
that the 2-intervals of basis graphs of matroids are either squares, or pyramids, or 3-octahedra, thus basis graphs of matroids
are thick and satisfy the 2-interval condition (IC3); analogously, the 2-intervals of even $\triangle$--matroids are thick and satisfy  (IC4).
Notice also that (PC) together with (IC3) or (IC4) imply the meshedness of basis graphs.
Maurer \cite{Mau}  presented a full characterization of finite graphs which
are basis graphs of matroids. Recently, answering a question of \cite{Mau}, this characterization
was refined in \cite{ChChOs_matroid}.  Extending Maurer's result, a characterization of basis graphs of even
$\triangle$--matroids was given in \cite{Che_bas}.   These characterizations
can be formulated in the following way:

\begin{theorem} \label{Th_Mau}  \cites{ChChOs_matroid,Che_bas,Mau}
A finite graph $G$ is the basis graph of a
matroid if and only if $G$ is a connected thick graph
satisfying  (IC3) and (PC). A finite graph $G$ is a basis
graph of an even $\triangle$--matroid if and only if $G$ is a connected thick graph
satisfying (IC4), (PC),  and the links of vertices of $G$ do not contain
induced $W_5$ and $W_6$.
\end{theorem}

%% \medskip\noindent
%% {\it 3. $L_1$--graphs.}

\subsubsection{$L_1$--graphs}
A graph $G=(V,E)$ is an $l_1$--{\it graph} if it admits an isometric
embedding into some finite-dimensional space ${\mathbb R}^n$ endowed with the $l_1$--metric \cite{DeLa}. 
A {\it cut} (alias a {\it split} or a {\it bipartition}) of $G$ is a pair $\{ A,B\}$ such that $A\cup B=V$ 
and $A\cap B=\emptyset$.  A cut $\{ A,B\}$ {\it separates} two vertices $u$ and $v$ if $u\in A,v\in B$ or $u\in B,v\in A$. 
Then, equivalently, a graph $G$ is an $l_1$-graph if and only if there exists a collection ${\mathcal C}=\{ \{ A_i,B_i\}: i\in I\}$ 
of cuts and positive numbers $\lambda_i, i\in I$ such that any pair of vertices $u,v$ of $G$ is separated 
by a finite number of cuts and 
$d(u,v)=\sum \{\lambda_i: \{ A_i,B_i\} \mbox{ separates } u \mbox{ and } v\}$ \cite{DeLa}.  

For an integer $\lambda>0$,
a {\it scale $\lambda$ isometric embedding} of a graph $G$ into a graph $H$ is a mapping
$\varphi : V\rightarrow W$ such that $d_H(\varphi (u),\varphi(v))=\lambda\cdot d_G(u,v)$ for all vertices $u,v\in V$.
It is well known \cite{DeLa} that a finite graph $G$ is an $l_1$--graph if and only if for some positive integer $\lambda$,  $G$ admits a scale $\lambda$
embedding into a hypercube. This is no longer true in the case of infinite graphs (even if they are locally finite): the infinite binary tree can be
isometrically embedded in a  hypercube but cannot be isometrically embedded in a finite-dimensional $l_1$--space.  Scale embeddable into hypercubes
graphs are not necessarily locally finite: an infinite  clique or star are examples.

All scale embeddable into hypercubes  graphs and all $l_1$--graphs are $L_1$--graphs: a graph $G$ is an $L_1$--{\it graph} if it admits an 
isometric embedding into an $L_1$--space. 
(we refer to \cite{DeLa}*{Chapter 3} for the  definition of $L_1$--spaces). Not every (even locally finite) $L_1$--graph is scale embeddable
into a hypercube:

\begin{example}
Let $G$ be a graph in the form of a one-way infinite path of blocks (2-connected components) $B_1,B_2,\ldots,$ where each $B_n, n=1,2,\ldots,$ is
isomorphic to the $n$--dimensional hyperoctahedron $K_{n\times 2}$ and the two articulation points of $B_n$ $(n\ge 2)$ are antipodal vertices of $K_{n\times 2}$.
Each $B_n$ admits a scale $\lambda_n$ embedding into a hypercube, but $\lambda_n$ is increasing
with $n$; see \cite{DeLa}*{Subsection 7.4}. Hence $G$ is not scale $\lambda$ embeddable in a hypercube for any positive integer $\lambda$. On the other hand,
$G$ is an $L_1$--graph: each  $B_n$ admits an $l_1$--embedding
in which all the coefficients of cuts are $\frac{1}{\lambda_n}$. We can extend the cuts of $B_n$ to the cuts of $G$ by assigning all vertices of $B_1,\ldots,B_{n-1}$ to the half of the cut
containing the articulation point between $B_{n-1}$ and $B_n$ and assigning all vertices of $B_{n+1},B_{n+2},\ldots$ to the half of the cut containing the articulation point between $B_{n}$ and $B_{n+1}$
(we used here the fact that the halves of each cut involved in an $l_1$--embedding are convex; see \cite{DeLa}*{Lemma 4.2.8}). This provides us with an $L_1$--embedding of $G$ with coefficients of cuts converging to 0.
\end{example}

We conclude this subsection with some properties of $L_1$--graphs scale embeddable into hypercubes. In several subsequent results we will use the following fundamental result by Shpectorov \cite{Shpec}:

\begin{theorem} \label{th:shpectorov} \cite{Shpec} If a  graph $G$ is scale embeddable into a hypercube,  then $G$ is isometrically embeddable
into a weak Cartesian product of finite hyperoctahedra and half-cubes.
\end{theorem}

In \cite{Shpec}, Theorem \ref{th:shpectorov} was proven only for finite graphs (and provides a characterization of finite $l_1$--graphs), however the finiteness of $G$ is not necessary, but only that $G$ admits a scale isometric embedding into a hypercube.

We will also use the following property:

\begin{lemma} \label{lemma:L1_finete_hull} If $G$ admits a scale embedding $\varphi$ into a hypercube $H(X)$, then the convex hull in $G$ of any finite set $Y$ of vertices
is finite.
\end{lemma}

\begin{proof} The convex hull in $H(X)$ of the (finite) set $\{ \varphi(x): x\in Y\}$ is a finite dimensional cube $H'$ of $H(X)$ defined by all
finite subsets $A$ of $X$ such that $\bigcap \{ \varphi(x): x\in Y\}\subseteq A\subseteq \bigcup \{ \varphi(x): x\in Y\}$. Since the image in $H(X)$ of the
convex hull of $Y$ in $G$ is contained in $H'$ (in fact, it coincides with the intersection of $H'$ with the vertex set of $G$), this convex hull is finite.
\end{proof}

\subsection{Complexes}
\label{s:fucom}

As morphisms between cell complexes we consider all \emph{cellular maps}, i.e.,
maps sending (linearly) cells to cells. An \emph{isomorphism} is a bijective
cellular map being a linear isomorphism (isometry) on each cell. A
\emph{covering (map)} of a cell complex $X$ is a cellular surjection $p\colon
\widetilde{X} \to X$ such that $p|_{\mbox{St}(\tv,\widetilde{X})}\colon \mbox{St}(\tv,\widetilde{X})\to \mbox{St}(p(\tv),X)$ is
an isomorphism for every vertex $\tv$ in $\tX$; compare \cite{Hat}*{Section 1.3}.
The space $\widetilde{X}$ is then called a \emph{covering space}.
A \emph{universal cover} of $X$ is a simply connected covering space $\widetilde{X}$. It
is unique up to isomorphism. In particular, if $X$ is simply connected, then
its universal cover is $X$ itself. (Note that $X$ is connected iff $G(X)=X^{(1)}$ is connected, and $X$ is
simply connected (i.e., every continuous map $S^1\to X$ is null-homotopic) iff
$X^{(2)}$ is so.) A group $F$ \emph{acts by automorphisms} on a cell complex $X$ if there is a
homomorphism $F\to \mbox{Aut}(X)$ called an \emph{action of $F$}. The action is
\emph{geometric} (or \emph{$F$ acts geometrically}) if it is proper (i.e., cells
stabilizers are finite) and cocompact (i.e., the quotient $X/F$ is compact). In the current paper we usually consider geometric
actions on graphs, viewed as one-dimensional complexes.

The weak modularity implies
the simple connectivity of the triangle-square complex.
The following lemma was proved in \cite{BCC+} in
the case of polyhedral cell complexes, that is, for CW complexes
in which cells intersect along sub-cells. Nevertheless, the proof presented there can be also applied in our setting
without any modifications.
\begin{lemma}[{\cite{BCC+}*{Lemma 5.5}}] \label{simplyconnected}
Let $G$ be a weakly modular graph with respect to a vertex $u$.
Then the triangle-square complex $X\trsq(G)$ of $G$ is simply connected.
\end{lemma}
The local weak modularity (Definition~\ref{d:loc-weak-mod}) implies the following.
\begin{lemma}
\label{l:no456}
Let $G$ be a locally weakly modular graph. Then every cycle of length $\leq 6$ in $X\trsq(G)$ is homotopically trivial.
\end{lemma}
\begin{proof}
By the definition of $X\trsq(G)$, cycles of length $3$ and $4$ are null-homotopic. Let $C=(v_1,v_2,\ldots,v_5)$ be a $5$--cycle.
If $C$ is not simple, or $v_1\sim v_3$, or $v_1\sim v_4$ then $C$ decomposes into $l$--cycles, with $l\leq 4$, and is thus null-homotopic.
For $d(v_1,v_3)=d(v_1,v_4)=2$, by the local triangle condition, there is $w\sim v_1,v_3,v_4$. Then $C$ is homotopically trivial,
since the cycles $(v_1,v_2,v_3,w)$, $(v_1,v_5,v_4,w)$, and $(v_3,v_4,w)$ are null-homotopic.
The similar, straightforward proof for $6$--cycles is left to the reader.
\end{proof}

Note that the $7$--cycle is a locally weakly modular graph.
The \emph{injectivity radius} of a complex $X$ is the the length of the shortest homotopically nontrivial loop in $X^{(1)}$.
Recall, see e.g.\ \cite{Hat}*{Chapter 1.3}, that the injectivity radius is the minimal displacement for the action
of $\pi_1(X)$ on the universal cover $\tX$ of $X$, by deck transformations.

\begin{lemma}
\label{l:noincov}
Let $\tG$ be the $1$--skeleton of a cover of the triangle-square complex $X\trsq(G)$ of a graph $G$.
Let $L$ be $K_{2,3}$, or $W_4^-$, or $K^-_4$.
Then  $\tG$ contains an induced $L$ if and only if $G$ contains an induced $L$.
\end{lemma}
\begin{proof}
Let $p\colon \tX\trsq(G)\to X\trsq(G)$ be the covering map. We denote by $p$ its restriction $p|_{\tG}$
to the $1$--skeleton as well.
Suppose that $\tK$ is an induced subgraph of $\tG$ isomorphic to $L$. Since $p$ is a covering, the graph
induced by images of vertices of $\tK$ is isomorphic to $L$ as well. Conversely, let $K$ be an induced subgraph of $G$ isomorphic
to $L$. Pick a vertex $v\in K$. Since the triangle-square complex of $K$ is simply connected, there exists a unique lift $\tK$ of $K$
in $\tG$ containing $v$. Then the map $p|_{\tK}\colon \tK \to K$ is an isomorphism. It is easy to check that $\tK$ is an induced subgraph.
\end{proof}

\begin{lemma}
\label{l:noiscov}
Let $\tG$ be the $1$--skeleton of a cover of the triangle-square complex $X\trsq (G)$ of a locally weakly modular graph $G$.
If $\tG$ contains an isometric $K_{3,3}^-$ then $G$ contains an isometric $K_{3,3}^-$.
\end{lemma}
\begin{proof}
Let $p\colon \tG \to G$ be the restriction of the covering map. Let $\tK$ be an isometric subgraph of $\tG$
isomorphic to $K_{3,3}^-$.
As in the proof of Lemma~\ref{l:noincov}, the graph $K$ induced by $p(\tK)$ is isomorphic to $K_{3,3}^-$.
We have to check, that it is isometric.
If not then there exists a vertex $v\sim p(a),p(b)$, where $a,b$ are the vertices at distance $3$ in $\tK$.
Since, by Lemma~\ref{l:no456}, the triangle-square complex of the graph induced by $K\cup \{ v \}$
is simply connected, there exists its preimage $\tK'$ in $\tG$ containing $\tK$. This implies that $\tK$ is not
isometric, a contradiction.
\end{proof}

Note that the above lemma may fail in the non locally weakly modular case: The universal cover of the
triangle-square complex of $K_{3,3}^-$ plus
a vertex adjacent to the vertices at distance $3$, contains an isometric $K_{3,3}^-$.

%\subsubsection{Automatic and biautomatic groups}

\chapter{Local-to-Global Characterization}
\label{s:plotoglo}

In this chapter, we present local-to-global characterizations of the
triangle-square complexes of weakly modular graphs.  In this and
subsequent chapters, these results will be specified for some
subclasses of weakly modular graphs.  In particular, in this chapter,
we give a local-to-global characterization of the clique complexes of
the Helly graphs.  Earlier, similar characterizations were given for
median graphs (i.e., 1-skeletons of CAT(0) cube complexes
\cites{Gr,Sag}) \cite{Ch_CAT}, bridged graphs (i.e., 1-skeletons of
systolic complexes \cites{Hag,JS}) \cite{Ch_CAT}, and, in the most
general form, for bucolic graphs (i.e., 1-skeletons of bucolic prism
complexes) \cite{BCC+}.

\section{Main results}

Recall that a graph $G$ is called locally weakly modular if for any vertex of $G$ it satisfies the local triangle and quadrangle conditions (defined in
Subsection \ref{s:wmgra}).  A graph is called \emph{locally modular} if it is locally weakly modular and does not contain triangles.

Here is the first  main result of this chapter and one of the main results of our paper:

\begin{theorem}
\label{t:lotoglo2} Let $G$ be a locally weakly modular graph, and let $\tG$ be the $1$--skeleton of the universal
cover $\widetilde X:=\widetilde{X}\trsq(G)$ of the triangle-square complex $X:=X\trsq(G)$ of $G$. Then $\tG$ is weakly modular. In particular, a graph $G$ is
a weakly modular graph if and only if $G$ is a locally weakly modular graph whose triangle-square complex $X\trsq(G)$ is simply connected.
\end{theorem}

This theorem may be viewed as an analogue of Cartan-Hadamard theorem for globally convex and globally nonpositively curved spaces; see \cite{BrHa}*{Theorem 4.1}.
The proof of Theorem \ref{t:lotoglo2} closely follows the proof of \cite{BCC+}*{Theorem 1.1 (i)$\Rightarrow$(ii)}.
Theorem \ref{t:lotoglo2} implies analogous characterizations for
modular graphs,  which we present next (for other similar consequences, see the next three chapters of the paper).

\begin{corollary}
\label{t:lotoglo_mod}
A graph is modular if and only if it is locally modular and its square complex is simply connected.
\end{corollary}

%\begin{theorem}
%\label{t:lotogloHell}
%A graph $G$ is a Helly graph if and only if its
% clique complex is simply connected and its balls of radius $1$ satisfy the
% Helly property ($1$--Helly property).
%\end{theorem}

Let us also formulate the following version of Theorem~\ref{t:lotoglo2} and Corollary~\ref{t:lotoglo_mod}.
It is in the vein of the ``local" characterization for dual polar spaces of Brouwer-Cohen \cite{BroCo}, and will be of use in
Chapter~\ref{s:dupol}.

\begin{theorem}
\label{t:lotoglo_G}
Every locally weakly modular (respectively, locally modular) graph arises as a quotient of a weakly modular
(respectively, modular) graph by an automorphisms group action
with the minimal displacement at least $7$.
\end{theorem}

The second main result of this chapter concerns the relationship
between clique-Helly, $1$--Helly, and Helly (i.e., disk-Helly)
graphs. As we noticed already above, finite dismantlable clique-Helly
graphs are exactly the Helly graphs \cite{BP-absolute}. In relation
with this result, the following question was raised in \cite{Prisner}*{p.\ 205} (conjectured in relation with clique-convergence of graphs), \cite{LaPiVi}
(in a different but equivalent form), and independently by the second
author of this paper.

\begin{question} \label{Helly_question} Is it true that a finite graph $G$ is Helly  if and only if $G$ is
clique-Helly  and its clique complex $X(G)$ is simply connected?
\end{question}

The following theorem answers this question in the affirmative for arbitrary, not necessarily
finite or locally finite graphs:

\begin{theorem}
\label{t:lotogloHell_bis}
Let $G$ be a (finitely) clique-Helly graph and let $\tG$ be the
$1$--skeleton of the universal cover $\widetilde
X:=\widetilde{X}\tr(G)$ of the triangle complex $X:=X\tr(G)$ of
$G$. Then $\tG$ is a (finitely) Helly graph.  In particular, $G$ is a
(finitely) Helly graph if and only if $G$ is (finitely) clique-Helly
and its triangle complex (and thus its clique complex) is simply
connected.
\end{theorem}

\begin{Rem}
  The formulation of Theorem \ref{t:lotogloHell_bis} (and of the
  propositions in Section \ref{helly:proof}) comprises  two
  statements. The first one is: if $G$ is clique-Helly (for
  arbitrary families of cliques), then the 1-skeleton of its universal
  cover is Helly (for arbitrary families of balls).  The second
  interpretation is: if $G$ is finitely clique-Helly, then the
  1-skeleton of its universal cover is finitely Helly.
\end{Rem}

\begin{Rem} If a family of sets $\mathcal F$ satisfies the Helly property then
any family ${\mathcal F}'$ whose members are obtained as intersections of sets of
$\mathcal F$ also satisfies the Helly property. In particular, since each (maximal)
clique $C$ of a graph is the intersection of unit balls centered at the vertices of $C$,
the Helly property for balls or unit balls of $G$ implies that $G$ is clique-Helly.
Theorem \ref{t:lotogloHell_bis} proves that under simple connectivity the converse implication holds as well.
\end{Rem}

Additionally, we obtain the following characterizations of Helly graphs:

%%%JEREMIE: There is a problem in the finite Helly case to prove the
%%%dismantlability !!!!
\begin{theorem}
\label{t:lotogloHell}
For a graph $G$, the following conditions are equivalent:
\begin{itemize}
\item[(i)] $G$ is Helly;
\item[(ii)] $G$ is $1$--Helly and weakly modular;
\item[(iii)] $G$ is clique-Helly and dismantlable;
\item[(iv)] $G$ is clique-Helly with a simply connected clique complex.
\end{itemize}
Moreover, if the clique complex $X(G)$ of $G$ is finite-dimensional, then the conditions (i)-(iv) are equivalent to
\begin{itemize}
\item[(v)] $G$ is clique-Helly with a contractible clique complex.
\end{itemize}
\end{theorem}

The following algorithmic result is also an immediate consequence of Theorems \ref{t:lotoglo2} and
\ref{t:lotogloHell_bis}:

\begin{corollary}
\label{t:simple_connectivity_algo}
It can be decided in polynomial time if  a finite triangle-square flag complex $X$ having a locally weakly modular or
clique-Helly 1-skeleton is simply connected.
\end{corollary}

Indeed, according to Theorems \ref{t:lotoglo2} and \ref{t:lotogloHell_bis} it suffices to check that the 1-skeleton of $G$ is weakly modular or Helly. (As we noticed already, a polynomial
algorithm for testing if a finite graph is Helly or clique-Helly follows from Theorem \ref{Helly} and Proposition \ref{clique_Helly_triangle}). An alternative proof of this corollary follows from the
proofs of Theorems \ref{t:lotoglo2} and \ref{t:lotogloHell_bis}. In fact, in both cases it leads to a more efficient $O(nm)$ time, where $n$ is the number of
0-simplices and $m$ is the number of 1-simplices of $X$ (notice that  deciding simple connectivity of a finite simplicial (or triangle) complex is an
undecidable problem \cite{Hak}).

%% \begin{question} \label{Helly_question_new} To what extent the characterizations of
%% Theorem \ref{t:lotogloHell} extend to non locally finite graphs and finitely (or non finitely) clique-Helly
%% or Helly properties?
%% \end{question}

\section{Weakly modular graphs: proofs}
%of Theorems \ref{t:lotoglo2} and \ref{t:lotoglo_G}}

In this section, we provide the proofs of all announced above results about weakly modular graphs.
Lemma \ref{simplyconnected} implies that triangle-square complexes of weakly modular graphs and the square complexes of modular graphs are simply connected, establishing one direction
of the proof of Theorem \ref{t:lotoglo2} and Corollary \ref{t:lotoglo_mod}. In the remaining part of this section we will prove the converse direction.

Let $G$ be a locally weakly modular graph, and let $\tG$ be the $1$--skeleton of the universal cover $\widetilde X:=\widetilde{X}\trsq(G)$ of the triangle-square
complex $X:=X\trsq(G)$ of $G$. To prove that $\tG$ is a weakly modular graph,
we will construct the universal cover $\widetilde{{X}}$ of ${X}$ as an increasing
union $\bigcup_{i\ge 1} \widetilde{{X}}_i$ of triangle-square
complexes. The complexes $\widetilde{{X}}_i$ will be in fact spanned by
concentric combinatorial balls $\widetilde{B}_i$ in $\widetilde{{X}}$.
The covering map $f$ is then the union $\bigcup_{i\ge 1} f_i,$ where
$f_i: \widetilde{{X}}_i\rightarrow {X}$ is a locally injective
cellular map such that $f_i|_{\widetilde{{X}}_j}=f_j$, for every $j\le i$. We
denote by $\widetilde{G}_i=G(\widetilde{{X}}_i)$ the underlying graph of
$\widetilde{{X}}_i$. We denote by $\tS_i$ the set of vertices
$\tB_i\setminus \tB_{i-1}$.

Pick any vertex $v$ of ${X}$ as the base-point. Define $\widetilde{B}_0=\{
\widetilde{v}\}:=\{ v\}, \widetilde{B}_1:=B_1(v,G)$.
Let $\widetilde{{X}}_1$  be the triangle-square complex of $B_1(v,G)$,
and let $f_1\colon \tX_1 \to X$ be the cellular map induced by Id$_{B_1(v,G)}$.
Assume that, for $i\geq 1$,  we have constructed the vertex sets
$\widetilde{B}_1,\ldots,\widetilde{B}_i,$ and we have defined
the triangle-square complexes $\widetilde{{X}}_1\subseteq \cdots\subseteq \widetilde{{X}}_i$
(for any $1\le j<k\le i$ we have an identification map $\widetilde{{X}}_j\rightarrow \widetilde{{X}}_{k}$)
and the corresponding cellular maps $f_1,\ldots,f_i$ from
$\widetilde{{X}}_1,\ldots,\widetilde{{X}}_i,$ respectively,  to ${X}$ so
that the graph $\tG_i=G(\widetilde{X}_i)$ and the complex $\widetilde{X}_i$
satisfy the following conditions:

\begin{enumerate}[(A{$_i$})]
\item[(P$_i$)]
$B_j(\tv,\tG_i)=\widetilde{B}_j$ for any $j\le i$;
\item[(Q$_i$)]
$\widetilde{G}_i$ is weakly modular with respect to $\widetilde{v}$ (i.e.,
$\widetilde{G}_i$ satisfies the conditions TC($\widetilde{v}$) and
QC($\widetilde{v}$));
\item[(R$_i$)]
for any $\widetilde{u}\in \widetilde{B}_{i-1},$ $f_i$
defines an isomorphism between the subgraph of $\tG_i$ induced by
$B_1(\widetilde{u},\tG_i)$ and the subgraph of $G$ induced by
$B_1(f_i(\widetilde{u}),G)$;
\item[(S$_i$)]
for any $\widetilde{w},\widetilde{w}'\in \widetilde{B}_{i-1}$ such that the
vertices $w=f_i(\widetilde{w}),w'=f_i(\widetilde{w}')$ belong to a square
$ww'uu'$ of ${X}$, there exist $\widetilde{u},\widetilde{u}'\in
\widetilde{B}_i$ such that $f_i(\widetilde{u})=u, f_i(\widetilde{u}')=u'$ and
$\widetilde{w}\widetilde{w}'\widetilde{u}\widetilde{u}'$ is a square of
$\widetilde{{G}}_i$ (by (R$_i$),
such $\tilde u, \tilde u'$ are unique in $B_1(\tilde w, \tilde G_i)$ and  $B_1(\tilde w', \tilde G_i),$ respectively).
\item[(T$_i$)]
 for any $\widetilde{w}\in \widetilde{S}_i:=\widetilde{B}_i\setminus
  \widetilde{B}_{i-1},$ $f_i$
  defines an isomorphism between the subgraphs of $\tG_i$ and of $G$ induced by, respectively,
$B_1(\widetilde{w},\tG_i)$ and
  $f_i(B_1(\widetilde{w},\tG_i))$.
\end{enumerate}

 It can be easily checked that
$\widetilde{B}_1,\widetilde{G}_1,\widetilde{X}_1$ and $f_1$ satisfy the
conditions (P$_1$),(Q$_1$),(R$_1$),(S$_1$), and
(T$_1$). Now we construct the set $\widetilde{B}_{i+1},$ the graph
$\widetilde{G}_{i+1}$ having $\widetilde{B}_{i+1}$ as the vertex-set, the
triangle-square complex $\widetilde{{X}}_{i+1}$ having $\tG_{i+1}$ as its
1-skeleton,  and the map $f_{i+1}: \widetilde{{X}}_{i+1}\rightarrow {X}.$
Let
 $$Z=\{ (\widetilde{w},z): \widetilde{w} \in \widetilde{S}_i \mbox{ and } z\in
B_1(f_i(\widetilde{w}),G)\setminus f_i(B_1(\widetilde{w},\tG_i))\}.$$
On $Z$ we define a binary relation $\equiv$ by setting $(\widetilde{w},z)\equiv
(\widetilde{w}',z')$ if and only if $z=z'$  and one of the following two
conditions is satisfied:

\begin{itemize}
\item[(Z1)] $\widetilde{w}$ and $\widetilde{w}'$ are the same or adjacent in
$\widetilde{G}_i$;%%  and $z\in B_1(f_i(\widetilde{w}),G)\cap
%% B_1(f_i(\widetilde{w}'), G)$;
\item[(Z2)] there exists $\widetilde{u}\in \widetilde{B}_{i-1}$ adjacent in
$\widetilde{G}_i$ to $\widetilde{w}$ and $\widetilde{w}'$ and such that
$f_i(\widetilde{u})f_i(\widetilde{w})zf_i(\widetilde{w}')$ is a square in $G$.
\end{itemize}

In what follows, the above relation will be used in the inductive step to construct $\tG_{i+1}$, $\tX_{i+1}$, $f_{i+1}$ and all related objects.
\medskip

First, however, we show that the relation $\equiv$ defined above is an equivalence relation.
The set of vertices of the graph $\tG_{i+1}$ will be then defined as the union of the set of vertices of the previously
constructed graph $\tG_{i}$ and the set of equivalence classes of $\equiv$.
In the remaining part of the proof, for a vertex $\widetilde{w}\in \widetilde{B}_i$, we denote by $w$ its image $f_i(\widetilde{w})$ in $X$ under $f_i$.

\begin{lemma} \label{equiv} The relation $\equiv$ is an equivalence relation on
$Z$.
\end{lemma}

\begin{proof}
Since the binary relation $\equiv$ is reflexive and symmetric, it
suffices to show that $\equiv$ is transitive. Let
$(\widetilde{w},z)\equiv (\widetilde{w}',z')$ and
$(\widetilde{w}',z')\equiv (\widetilde{w}'',z'')$. We will prove that
$(\widetilde{w},z)\equiv (\widetilde{w}'',z'').$ By the definition of
$\equiv,$ we conclude that $z=z'=z''$ and that $z\in B_1(w,G)\cap
B_1(w',G)\cap B_1(w'',G).$

If $\tw = \tw''$ or $\tw \sim \tw''$ (in $\widetilde{G}_i$) then, by the definition of $\equiv$,
$(\tw,z)\equiv
(\tw'',z)$ and we are done. If $\tw \nsim \tw'' \neq \tw$ and if there exists
$\tu \in \tB_{i-1}$ such that $\tu \sim \tw, \tw''$ then, by (R$_i$)
applied to $\tu$, we obtain that $u \sim w, w''$ and $w \nsim
w''$. Since $(\tw,z), (\tw'',z) \in Z$, we have $z\sim w,w''$. Moreover,
if $z\sim u$ then, by (R$_i$) applied to $u$, there exists $\tz \in
\tB_i$, such that $\tz \sim \tu, \tw, \tw''$ and $f_i(\tz)=z$.  Thus
$(\tw,z), (\tw',z) \notin Z$, which is a contradiction. Consequently,
if $\tw \nsim \tw''$ and if there exists $\tu \in \tB_{i-1}$ such that
$\tu \sim \tw, \tw''$, then $uwzw''$ is an induced
square in $G$,
and by condition (Z2), $(\widetilde{w},z)\equiv (\widetilde{w}'',z'')$. Therefore, in the rest of the
proof, we will make the following assumptions and show that they lead to a
contradiction.
\begin{enumerate}
\item[($A_1$)] $\tw \nsim \tw''$;
\item[($A_2$)] there is no $\tu \in \tB_{i-1}$ such that $\tu \sim
\tw, \tw''$.
\end{enumerate}

\begin{claim}
For any ordered pair $(\tw,z) \in Z$ the following properties hold:
\begin{enumerate}
\item[($A_3$)] there is no neighbor $\tz \in
  \tB_{i-1}$ of $\tw$ such that $f_i(\tz)=z$;
\item [($A_4$)] there is no neighbor  $\tu \in
  \tB_{i-1}$ of $\tw$ such that $u \sim z$;
\item [($A_5$)] there are no $\tx, \ty \in \tB_{i-1}$ such that
  $\tx \sim \tw, \ty$ and   $y \sim z$.
\end{enumerate}
\end{claim}

\begin{proof}
If $\tw$ has a neighbor $\tz \in \tB_{i-1}$ such that $f_i(\tz)=z$,
then $(\tw,z) \notin Z$,  a contradiction. This establishes $(A_3)$.

If $\tw$ has a neighbor $\tu \in \tB_{i-1}$ such that $u \sim z$, then
by (R$_i$) applied to $\tu$, there exists $\tz \in \tB_i$ such that
$\tz \sim \tu,\tw$ and $f_i(\tz)=z$. Thus $(\tw,z) \notin Z$, a contradiction, establishing $(A_4)$.

If there exist $\tx, \ty \in \tB_{i-1}$ such that
$\tx \sim \tw, \ty$ and $y \sim z$ then, by $(A_4)$, $yxwz$ is an
induced square in $G$. From (S$_i$) applied to $\ty,\tx$, there
exists $\tz \in \tB_i$ such that $\tz \sim \ty,\tw$ and $f_i(\tz)=z$. Thus
$(\tw,z)
\notin Z$, a contradiction. Therefore $(A_5)$ holds as well, and the claim is established.
\end{proof}

\begin{claim}
There are vertices $\tu,\tu'\in \tB_{i-1}$ and $\tx \in \tB_{i-2}$ with the following properties:
$\tu \sim \tw, \tw',\tx$; $\tu' \sim \tw',\tw'',\tx$; $\tu \nsim \tw''$; $\tu'\nsim \tw$.
\end{claim}
\begin{proof}
If $\tw \sim \tw'$ then there is $\tu \in \tB_{i-1}$ with $\tu \sim \tw,\tw'$,
by the triangle condition TC($\tv$) from (Q$_i$). If $\tw \nsim \tw'$ then the existence of such $\tu$ follows from the definition
of $\equiv$.
Similarly for $\tu'$. By ($A_2$) we have $\tu\neq\tu'$, $\tu \nsim \tw''$, and $\tu' \nsim \tw$.
If $\tu \sim \tu'$ then, by the triangle condition TC($\tv$), there exists a vertex $\tx \in \tB_{i-2}$ adjacent to both
$\tu$ and $\tu'$. If $\tu \nsim \tu'$, then the existence of $\tx$ follows from the quadrangle condition QC($\tv$).
\end{proof}
From ($A_4)\&(A_5$), $x \neq z$ and $x \nsim z$.  By (R$_i$) applied
to $\tu$ or $\tu'$, we have $d(w,x)=d(w'',x)=2$.  We claim that
$d(x,z)=3$. If not, i.e., if $d(x,z)=2$, then by the local triangle
condition, there exists $s\sim x,w,z$.  By (R$_i$) there is $\ts \in
\tB_{i-1}$ with $\ts \sim \tx$ and $f_i(\ts)=s$. By (S$_i$) if $s\nsim u$, or by
(R$_i$) otherwise, we have $\ts \sim \tw$, which contradicts
($A_4$). Thus $d(x,z)=3$.  Hence, by the local quadrangle condition, there
is $y\sim x,w,w''$.  By (R$_i$) applied to $\tu$ if $u \sim y$, and by
(S$_i$) applied to the square $xuwy$ otherwise, there exists $\ty \sim
\tx, \tw$ with $f_i(\ty)=y$. By (R$_i$) applied to $\tu'$ if $u' \sim y$, and by (S$_i$)
applied to the square $xu'w''y$ otherwise, we have $\ty \sim
\tw''$. Consequently, $\ty \sim \tx,\tw,\tw''$, contradicting $(A_2)$.
This finishes the proof of the lemma.
\end{proof}

Let $\widetilde{S}_{i+1}$ denote the set of equivalence classes of $\equiv$, i.e.,
$\widetilde{S}_{i+1}=Z/_{\equiv}$. For an ordered pair $(\widetilde{w},z)\in Z$,
we will denote by $[\widetilde{w},z]$ the equivalence class of $\equiv$
containing  $(\widetilde{w},z).$ Set $\widetilde{B}_{i+1}:=\widetilde{B}_i\cup
\widetilde{S}_{i+1}.$
Let $\widetilde{G}_{i+1}$ be the graph having  $\widetilde{B}_{i+1}$ as the
vertex set, in which two vertices $\widetilde{a},\widetilde{b}$ are adjacent if
and
only if one of the following conditions holds:
\begin{itemize}
\item[(1)] $\widetilde{a},\widetilde{b}\in \widetilde{B}_i$ and
$\widetilde{a}\widetilde{b}$ is an edge of $\widetilde{G}_i$,
\item[(2)] $\widetilde{a}\in \widetilde{B}_i$,  $\widetilde{b}\in
\widetilde{S}_{i+1}$ and $\widetilde{b}=[\widetilde{a},z]$,
\item[(3)] $\widetilde{a},\widetilde{b}\in \widetilde{S}_{i+1},$
$\widetilde{a}=[\widetilde{w},z]$,
$\widetilde{b}=[\widetilde{w},z']$ for a vertex $\widetilde{w}\in \tB_i,$ and
$z\sim z'$ in the graph $G$.
\end{itemize}

Finally, we define the map $f_{i+1}: \widetilde{B}_{i+1}\rightarrow V({X})$ in
the
following way: if $\widetilde{a}\in \widetilde{B}_i$, then
$f_{i+1}(\widetilde{a})=f_i(\widetilde{a}),$  otherwise, if $\widetilde{a}\in
\widetilde{S}_{i+1}$ and
$\widetilde{a}=[\widetilde{w},z],$ then $f_{i+1}(\widetilde{a})=z.$ Notice that
$f_{i+1}$ is well-defined because all ordered pairs representing $\widetilde{a}$ have one and the same vertex $z$ in the
second argument. Following our convention, in the sequel, all vertices of $\widetilde{B}_{i+1}$ will
be denoted with a tilde and their images in $G$ under $f_{i+1}$ will be denoted
without tilde, e.g. if $\widetilde{w}\in \widetilde{B}_{i+1},$
then $f_{i+1}(\widetilde{w})=w$.
\medskip

Now we check our inductive assumptions, verifying the
properties (P$_{i+1}$), (Q$_{i+1}$),(R$_{i+1}$), (S$_{i+1}$), and  (T$_{i+1}$) for $\tG_{i+1}$ and $f_{i+1}$
defined above. In particular, it allows us to define the corresponding
complex $\tX_{i+1}$.
\medskip

The following four lemmata together with their proofs are the same as, respectively,
Lemmata 5.7, 5.8, 5.9 \& 5.10 in \cite{BCC+} (since weak modularity is our main topic,
we reproduce the proof of Lemma 5.8 that the graph $\widetilde{G}_{i+1}$
is weakly modular with respect to  $\widetilde v$).

\begin{lemma} \label{Pi+1}  $\tG_{i+1}$ satisfies the property $(P_{i+1})$,
i.e.,
$B_j(\widetilde{v},\tG_{i+1})=\widetilde{B}_j$ for any $j\le i+1.$
\end{lemma}

\begin{lemma} \label{Qi+1}  $\widetilde{G}_{i+1}$ satisfies the property
$(Q_{i+1}),$ i.e., the graph $\widetilde{G}_{i+1}$
is weakly modular with respect to the base-point $\widetilde v$.
\end{lemma}

\begin{proof} First we show that  $\widetilde{G}_{i+1}$ satisfies the triangle
condition TC($\widetilde{v}$).
Pick two adjacent vertices $\widetilde{x},\widetilde{y}$ having in
$\widetilde{G}_{i+1}$ the same distance to $\widetilde{v}.$
Since by Lemma \ref{Pi+1}, $\widetilde{G}_{i+1}$ satisfies the property
(P$_{i+1}$) and the graph $\tG_i$ is weakly modular
with respect to $\widetilde v$, we can suppose that
$\widetilde{x},\widetilde{y}\in \tS_{i+1}.$ From the definition of the edges
of $\tG_{i+1}$, there exist two ordered pairs $(\widetilde{w},z),(\widetilde{w},z')\in
Z$ such that $\widetilde{w}\in \tB_i,$ $z$ is
adjacent to $z'$ in $G,$ and
$\widetilde{x}=[\widetilde{w},z],\widetilde{y}=[\widetilde{w},z'].$ Since
$\widetilde{w}$ is
adjacent in $\tG_{i+1}$ to both $\widetilde{x}$ and $\widetilde{y},$ the
triangle condition TC($\widetilde{v}$) is established.

Now we show that $\widetilde{G}_{i+1}$ satisfies the quadrangle condition
QC($\widetilde{v}$). Since  the properties
(P$_{i+1}$) and (Q$_i$) hold, it suffices to consider a vertex $\widetilde{x}\in
\tS_{i+1}$ having two nonadjacent
neighbors $\widetilde{w},\widetilde{w}'$ in $\widetilde{S}_i.$ By the definition of
$\tG_{i+1},$ there exists
a vertex $z$ of $G$ and ordered pairs $(\widetilde{w},z),(\widetilde{w}',z)\in Z$
such that $\widetilde{x}=[\widetilde{w},z]$
and $\widetilde{x}=[\widetilde{w}',z].$ Hence $(\widetilde{w},z)\equiv
(\widetilde{w}',z).$ Since $\widetilde{w}$ and $\widetilde{w}'$
are not adjacent, by condition (Z2) in the definition of $\equiv$ there exists
$\widetilde{u}\in \tB_{i-1}$ adjacent
to $\widetilde{w}$ and $\widetilde{w}'$, whence
$\widetilde{x},\widetilde{w},\widetilde{w}'$ satisfy QC($\widetilde{v}$).
\end{proof}

\begin{lemma}\label{lem-homomorphism}
For any edge $\ta\tb$ of $\tG_{i+1}$, $ab$ is an edge of $G$ (in
  particular $a \neq b$).
\end{lemma}

\begin{lemma}\label{lem-locally-surjective}
If $\ta \in \tB_i$ and if $b \sim a$ in $G,$ then there exists
a vertex $\tb$ of $\tG_{i+1}$  adjacent to $\ta$ such that $f_{i+1}(\tb) = b$.
\end{lemma}

We prove now that $f_{i+1}$ is locally injective.

\begin{lemma}\label{lem-locally-injective}
If $\ta \in \tB_{i+1}$ and  $\tb, \tc$ are distinct neighbors of
$\ta$ in $\tG_{i+1}$, then $b \neq c$.
\end{lemma}

\begin{proof}
First, note that if $\tb \sim \tc$ then the assertion holds by
Lemma~\ref{lem-homomorphism}; thus further we assume that $\tb
\nsim \tc$. If $\ta, \tb, \tc \in \tB_i$, the lemma holds by (R$_i$)
or (T$_i$) applied to $\ta$.
Suppose first that $\ta \in \tB_i$. If $\tb, \tc \in \tS_{i+1}$, then
$\tb = [\ta,b]$ and $\tc = [\ta,c],$ and thus $b\neq c$. If $\tb \in
\tB_i$ and $\tc = [\ta,c] \in \tS_{i+1}$, then $(\ta,b) \notin Z$, and
thus $c \neq b$. Therefore further we consider $\ta\in \tS_{i+1}$.

If $\tb, \tc \in \tB_i$ then $\ta =
[\tb,a] = [\tc,a]$. Since $(\tb,a) \equiv (\tc,a)$ and since $\tb
\nsim \tc$, there exists $\tu \in \tB_{i-1}$ such that $\tu \sim \tb,
\tc$ and $abuc$ is an induced square of $G$. This
implies that $b \neq c$.

If $\ta, \tb, \tc \in \tS_{i+1}$, then there exist
  $\tw, \tw' \in \tB_i$ such that $\tb = [\tw,b]$, $\tc =[\tw',c],$
  and $\ta= [\tw,a] = [\tw',a]$. Suppose that $b=c$; note that this
  implies that $\tw \neq \tw'$.  Since $(\tw,a) \equiv (\tw',a)$
  either $\tw\sim \tw'$ or there exists $\tu\in \tB_{i-1}$ such that
  $uwaw'$ is a square.  If $\tw \sim \tw'$ then $\tb=\tc$ ---
  contradiction. Thus $\tw \nsim \tw'$. Let $\tu$ be as above.  If
  $uwbw'$ is not a square then $b\sim u$ and, by (R$_i$) applied to
  $\tu$, we have that $(\tw,b)\notin Z$ --- contradiction.  Thus
  $uwbw'$ is a square and hence $\tb=[\tw,b]=[\tw',c]=\tc$ ---
  contradiction.

If $\ta, \tb \in \tS_{i+1}$ and $\tc \in \tS_i$, then there exists
$\tw \in \tS_i$ such that $\tb = [\tw,b]$ and $\ta = [\tw,a] =
[\tc,a]$. If $\tw \sim \tc$, then $(\tw,c) \notin Z$, and thus
$(\tw,c) \neq (\tw,b)$, i.e., $b \neq c$. If $\tw \nsim \tc$, since
$[\tw,a] = [\tc,a]$, there exists $\tu \in \tS_{i-1}$ such that
$\tu\sim \tw,\tc$ and such that $acuw$ is an induced square of
$G$. Since $\tw$ and $\tc$ are not adjacent, by (R$_i$) applied
to $\tu$, $w$ and
$c$ are not adjacent as well. Since $w \sim b$, this implies that $b \neq c$.
\end{proof}

\begin{lemma}\label{lem-triangles}
If $\ta \sim \tb, \tc$ in $\tG_{i+1}$, then $\tb \sim \tc$ if and only
if $b \sim c$.
\end{lemma}

\begin{proof}
If $\tb \sim \tc$, then  $b \sim c$ by Lemma~\ref{lem-homomorphism}.
Conversely, suppose that $b \sim c$ in $G$.
If $\ta, \tb, \tc \in
\tB_i$ then  $\tb \sim \tc$ by conditions (R$_i$) and (T$_i$).
Therefore, further we assume that at least one of the
vertices $\ta, \tb, \tc$ does not belong to $\tB_i$.

First, suppose that $\ta \in \tB_i$. If $\tb, \tc \in \tS_{i+1}$ then $\tb =
[\ta,b]$ and $\tc = [\ta,c]$. Since $b\sim c$, by the construction of $\tG_{i+1}$, we
have $\tb \sim \tc$ in $\tG_{i+1}$. Suppose now that $\tb = [\ta,b]
\in S_{i+1}$ and $\tc \in \tB_i$. If there exists $\tb' \sim \tc$ in
$\tG_i$ such that $f_i(\tb') = b$ then, by (T$_i$) applied to $\tc$, we have
$\ta \sim \tb'$ and $(\ta,b) \notin Z$, which is a contradiction. Thus
$(\tc,b) \in Z$ and, since $\tc \sim \ta$, we have $[\tc,b]=[\ta,b]=\tb$, and
consequently, $\tc \sim \tb$. Therefore, further we consider $\ta \in \tS_{i+1}$.

If $\tb, \tc \in \tB_i$ and $\ta \in \tS_{i+1}$, then
$\ta =[\tb,a] = [\tc,a]$ and either $\tb \sim \tc$, or there exists
$\tu \in \tS_{i-1}$ such that $\tu \sim \tb, \tc$ and $ubac$ is an
induced square in $G$, which is impossible because $b \sim c$.

If $\ta, \tb \in \tS_{i+1}$ and $\tc \in \tB_i$, then there exists
$\tw \in \tB_i$ such that $\tb = [\tw,b]$ and $\ta = [\tw,a] =
[\tc,a]$.
If $(\tc,b)\notin Z$ then there is $\tb'\in \tB_i$ with $\tc\sim\tb'$ and $f_i(\tb')=b$.
Thus $\ta \sim \tb'$ by a previous case. This however contradicts the local injectivity of $f_{i+1}$ (Lemma~\ref{lem-locally-injective}).
Therefore $(\tc,b)\in Z$. Since $(\tc,a)\equiv (\tw,a)$, either $\tw \sim \tc$ or there is $\tu \in \tB_{i-1}$ adjacent to
$\tw,\tc$ and with $uwac$ being a square.
If $\tw \sim \tc$ then $\tb = [\tc,b]$ and thus $\tb \sim \tc$. If $\tw \nsim \tc$ then let $\tu$ be as above.
If $uwbc$ is not a square, then $b\sim u$ and, by (R$_i$) applied to $\tu$, we conclude that $(\tw,b)\notin Z$ --- contradiction.
Thus $uwbc$ is a square, and hence $\tb=[\tc,b]$.
Consequently, $\tc \sim \tb$.

If $\ta, \tb, \tc \in \tS_{i+1}$ then there exist $\tw, \tw' \in
\tB_i$ such that $\tb = [\tw,b]$, $\tc = [\tw',c]$ and $\ta = [\tw,a]
= [\tw',a]$. If $\tw \sim \tc$ (respectively, $\tw' \sim \tb$), then we are in
a previous case, replacing $\ta$ by $\tw$ (respectively, $\tw'$) and
consequently $\tb \sim \tc$. Suppose now that $\tw \nsim \tc$ and
$\tw' \nsim \tb$. From a previous case applied to $\ta,\tb \in
\tS_{i+1}$ (respectively, $\ta,\tc \in \tS_{i+1}$) and $\tw' \in
\tB_i$ (respectively, $\tw \in \tB_i$), it follows that $w \nsim c$
and $w' \nsim b$.  By the triangle condition TC($\tv$) when $\tw \sim
\tw'$, or by the definition of $\tS_{i+1}$ when $\tw \nsim \tw'$,
there exists a vertex $\tu \in \tB_{i-1}$ adjacent to both
$\tw,\tw'$. By a previous case, $u \nsim a,b,c$.  By local weak
modularity, there is a vertex $y\sim b,c,u$. By (R$_i$), there is a
vertex $\ty \sim \tu$ with $f_i(\ty)=y$. If $\ty \in \tB_{i-1}$ then,
by (S$_i$) if $y \nsim w$, or by (R$_i$) otherwise, there exists $\tb'
\in \tB_i$ such that $\tb' \sim \ty, \tw$ and $f_i(\tb')=b$. This
however contradicts the fact that $(\tw,b)\in Z$.  Therefore $\ty \in
\tS_i$.
\begin{claim}
$(\ty,b)\in Z$.
\end{claim}
\begin{proof}
Suppose this is not the case. Then there is $\tb'\in \tB_i$ adjacent
to $\ty$ and such that $f_i(\tb')=b$. By (T$_i$) and since $b\nsim u$,
we have $\tb' \nsim \tu$.  If $\tb' \in \tS_{i-1}$ then, by (Q$_i$),
in $\tS_{i-2}$ there is a vertex $\tx \sim \tb',\tu$. By (R$_i$)
applied to $\tu$, we have $x \neq w$ and $x \nsim w$.  Consequently, $uxbw$ is
a square, and by (S$_i$) applied to $\tu$ and $\tx$, we have $\tb' \sim \tw$,
contradicting the fact that $(\tw,b) \in Z$. Therefore $\tb' \in
\tS_i$. By (Q$_i$), in $\tS_{i-1}$ there is a vertex $\tx \sim
\tb',\ty$. By (T$_i$), we have $x\sim b,y$.  Observe that $\tx \neq \tu$. If
$\tx \sim \tu$ then $x\sim u$ and, by (S$_i$) applied to the square
$uxbw$ if $w \nsim x$, or by (R$_i$) if $w \sim x$, we have that $\tb'
\sim \tw$. This, however, contradicts the fact that $(\tw,b)\in Z$.
Therefore $\tx \nsim \tu$. By (Q$_i$), there is $\tz \in \tS_{i-2}$
adjacent to $\tx,\tu$.  By (R$_i$), we have $z\nsim b,w$. Hence, by
the local triangle condition, there is $s\sim b,w,z$.  By (R$_i$) there
is $\ts \in \tB_{i-1}$ adjacent to $\tz$, with $f_i(\ts)=s$.  By
(R$_i$) or by (S$_i$) (depending whether $s\sim x$ and/or $s\sim u$),
we obtain that $\ts \sim \tb', \tw$. By (R$_i$), it follows that $\tw
\sim \tb'$. This, however, contradicts the fact that $(\tw,b)\in
Z$. Hence the claim holds.
\end{proof}
Analogously, we have that $(\ty,c) \in Z$. It follows that $\tb = [\tw,b]=[\ty,b]\sim [\ty,c]=[\tw',c]=\tc$.
\end{proof}

We can now prove that the image under $f_{i+1}$ of a triangle (respectively, a
square) is a triangle (respectively, a square).  This will allow us to extend
the map $f_{i+1}$ to a cellular map $f_{i+1}\colon \tX_{i+1} \to X$.
The following two lemmata together with their proofs correspond to
Lemmata 5.14 \& 5.15 from \cite{BCC+}.

\begin{lemma} \label{cellular}
If $\widetilde{a}\widetilde{b}\widetilde{c}$ is a triangle in $\tG_{i+1}$, then
$abc$ is a triangle in $G$. If
$\widetilde{a}\widetilde{b}\widetilde{c}\widetilde{d}$ is a square in
$\tG_{i+1}$, then $abcd$ is a square in $G$.
\end{lemma}

\begin{lemma}\label{Ri+1}\label{Ti+1}
$f_{i+1}$ satisfies the conditions $(R_{i+1})$ and $(T_{i+1}).$
\end{lemma}

\begin{lemma} \label{Si+1} For any adjacent vertices $\widetilde{w},\widetilde{w}'\in
  \widetilde{B}_{i}$ such that the vertices
  $w=f_{i+1}(\widetilde{w}),w'=f_{i+1}(\widetilde{w}')$ belong to a
  square $ww'u'u$ of ${X}$, there exist
  $\widetilde{u},\widetilde{u}'\in \widetilde{B}_{i+1}$ such that
  $f_{i+1}(\widetilde{u})=u, f_{i+1}(\widetilde{u}')=u'$ and
  $\widetilde{w}\widetilde{w}'\widetilde{u}'\widetilde{u}$ is a square
  of $\widetilde{{X}}_{i+1}$, that is, $\widetilde{{X}}_{i+1}$
  satisfies the property $(S_{i+1}).$
\end{lemma}

\begin{proof}
By Lemma~\ref{Ri+1} applied to $\tw$ and $\tw'$, we know that in
$\tG_{i+1}$ there exists a unique $\tu$ (respectively, a unique $\tu'$) such
that $\tu\sim\tw$ (respectively, $\tu'\sim\tw'$) and $f_{i+1}(\tu)=u$
(respectively, $f_{i+1}(\tu)=u'$).

Note that if $\tw, \tw' \in \tB_{i-1}$, the lemma holds by
condition (S$_i$). Let us assume further that $\tw \in \tS_i$.

%% \medskip
%% \noindent{\textbf{Case 1.} $\tw' \in \tS_{i-1}$.}

\begin{case-wm}
  $\tw' \in \tS_{i-1}$.
\end{case-wm}

If $\tu' \in \tB_{i-1}$ then, by (S$_i$) applied to $\tw'$ and $\tu'$, we
conclude that $\tw\tw'\tu'\tu$ is a square in $\tG_{i+1}$.
If $\tu' \in \tS_i$ and $\tu \in \tS_{i-1}$, then Lemma~\ref{Ri+1}
applied to $\tw$, implies that $\tu$ is not adjacent to $\tw'$. Thus,
by the quadrangle condition QC($\tv$), there exists $\tx \in \tS_{i-2}$
such that $\tx\sim\tu, \tw'$.
By (R$_i$), we have
$x\nsim u'$. Thus, by (S$_i$)
applied to the square $xw'u'u$, we obtain that $\tu \sim \tu'$.
Hence $\tw\tw'\tu'\tu$ is a square in $\tG_{i+1}$.

Suppose now that $\tu', \tu \in \tS_i$. By TC($\tv$), there exists
$\tx \in \tS_{i-1}$ different from $\tw'$, such that $\tx\sim\tu,\tw$.
If $x\sim w'$ then by (R$_i$), $\tx \sim \tw'$, and by (R$_i$) when
$x\sim u'$, or by (S$_i$) otherwise, we obtain that $\tu \sim \tu'$,
and hence $\tw\tw'\tu'\tu$ is a square.  If $x\nsim w'$, then $\tilde x \nsim \tilde w'$ and
by (Q$_i$) there is $\ty \in \tS_{i-2}$ adjacent to $\tx,\tw'$. By
(R$_i$), we have $y\nsim u,u'$. By the local triangle condition, there
exists $s\sim u,u',y$. By (R$_i$) there is $\ts \sim \ty$ with
$f_i(\ts)=s$.  By (S$_i$) when $s\nsim x$, or by (R$_i$) otherwise, we
have that $\ts \sim \tu$.  Similarly $\ts \sim \tu'$.  Thus, by
(R$_i$), we have $\tu \sim \tu'$ and $\tw\tw'\tu'\tu$ is a square in
$\tG_{i+1}$.

Finally, we consider the case when $\tu' \in \tS_i$ and $\tu \in \tS_{i+1}$.
If $(\tu',u) \in Z$ then, by definition, $\tu \sim \tu'$ and we are done.
If not then there is $\tu'' \sim  \tu'$ contained in $\tB_i$ and such that
$f_i(\tu'')=u$. Then we are in one of the preceding cases (after exchanging:
$\tw$ by $\tu'$, $\tu'$ by $\tw$, $\tu$ by $\tu''$), and $\tu'' \sim \tw$, contradicting Lemma~\ref{lem-locally-injective}.

%% \medskip
%% \noindent{\textbf{Case 2.} $\tw' \in \tS_{i}$.}

%% \medskip

\begin{case-wm}
  $\tw' \in \tS_{i}$. 
\end{case-wm}

If $\tu \in \tS_{i-1}$, exchanging the role of $\tw'$ and $\tu$, we
are in the previous case and thus there exists $\tu'' \sim \tw',\tu$
such that $f_{i+1}(\tu'') = u'$. By Lemma~\ref{Ri+1}, we obtain that
$\tu'=\tu''$ and we are done. For the same reasons, if $\tu' \in
\tS_{i-1}$, applying Case 1 with $\tw'$ in the role of $\tw$ and
$\tu'$ in the role of $\tw'$, we are done.

If $\tu \in \tS_{i}$ then, by TC($\tv$), there exists $\tx
\in \tB_{i-1}$ such that $\tx\sim \tw,\tu$.
If $x\sim u'$ then, by (R$_{i+1}$), we have $\tu\sim \tu'$ and we
are done. If $x \sim w'$ and $x \nsim u'$, then by Case 1 for the
square $xw'u'u$, we obtain $\tu \sim \tu'$.

Suppose now that $x\nsim u',w'$. By the local triangle condition,
there is $y$ in $G$ such that $y \sim u',w',x$. Observe that $y\neq
w,u$. By (R$_{i+1}$), there exists $\ty \sim \tx$ such that $f_i(\ty)
= y$. If $y\sim w$ then, by (R$_{i+1}$) applied to $\tw$, we have $\ty \sim
\tw'$. If $y\nsim w$ then, by Case 1 applied to the square $xww'y$,
we have $\ty \sim \tw'$.  By (R$_{i+1}$) applied to $\tw'$, we obtain $\ty\sim \tu'$.
If $y \sim u$, by (R$_{i+1}$) applied to $\tx$ and $\ty$, we have $\ty \sim
\tu$ and $\tu \sim \tu'$. If $y \nsim u$, by Case 1 applied to the
square $xyu'u$, we conclude that $\tu \sim \tu'$.

Suppose now that $\tw$ has no neighbor in $\tB_i$ mapped to $u$ and
that $\tw'$ has no neighbor in $\tB_i$ mapped to $u'$. Thus, there
exist $[\tw,u]$ and $[\tw',u']$ in $\tS_{i+1}$. By TC($\tv$), there
exists $\tx \in \tS_{i-1}$ such that $\tx \sim \tw, \tw'$.  By
Lemma~\ref{lem-triangles}, $x \nsim u, u'$. By the local triangle
condition, there exists $y \sim u,u',x$. By (R$_i$) applied to $\tx$,
there exists $\ty$ in $\tB_i$ such that $\ty \sim \tx$ and $\ty \sim
\tw$ (respectively, $\ty \sim \tw'$) if and only if $y \sim w$ (respectively, $y \sim
w'$). By Lemma~\ref{lem-triangles} if $\tw \sim \ty$ and by Case 1
otherwise, we have $\ty \sim \tu$. Similarly, one shows that $\tu' \sim
\ty$. Consequently, $\ty \in \tS_i$, $\tu = [\ty,u] = [\tw,u]$ and
$\tu' = [\ty,u'] = [\tw',u']$. By the definition of $\tG_{i+1}$, we
get that $\tu \sim \tu'$ and, consequently, $\tw\tw'\tu'\tu$ is a
square.
\end{proof}

%% \noindent
%% {\bf 
\subsection*{The universal cover $\tX$}
Let $\widetilde{X}_v$ denote the triangle-square complex obtained as the
directed union
$\bigcup_{i\ge 0} \widetilde{X}_i$, with a vertex $v$ of $X$ as the
base-point. Denote by $\tG_v$ the
$1$--skeleton of $\widetilde{X}_v.$ Since each $\tG_i$
is weakly modular with respect to $\tv,$ the graph $\tG_v$ is also weakly
modular with respect to $\tv$. Therefore, the complex $\widetilde{X}_v$ is
simply connected, by virtue of Lemma \ref{simplyconnected}. Let
$f=\bigcup_{i\ge 0}f_i$ be the map from $\widetilde{X}_v$ to $X$.

\begin{lemma} \label{covering_map}
For any $\tw \in \widetilde{X}_v$, the map
$f|_{{\rm St}(\tw,\widetilde{X}_v)}$ is an isomorphism onto
${\rm St}(w,X)$, where $w=f(\tw)$. Consequently,
$f:~\!\!\widetilde{{X}}_v\rightarrow~\!\!X$ is a covering
map.
\end{lemma}
\begin{proof}
Note that, since $\widetilde{X}_v$ is a flag complex, a vertex $\tx$
of $\widetilde{X}_v$ belongs to $\mbox{St}(\tw,\widetilde{X}_v)$ if and only if either $\tx \in
B_1(\tw,\tG_v)$ or $\tx$ has two non-adjacent neighbors in
$B_1(\tw,\tG_v)$.

Consider a vertex $\tw$ of $\tX_v$. Let $i$ be the distance between
$\tv$ and $\tw$ in $\tG_v$ and consider the set $\tB_{i+2}$.  Then the
vertex set of $\mbox{St}(\tw,\widetilde{X}_v)$ is included in
$\tB_{i+2}$. From ($R_{i+2}$) we know that $f$ is an isomorphism
between the graphs induced by $B_1(\tw,\tG_v)$ and $B_1(w,G).$

For any vertex $x$ in  $\mbox{St}(w,X)\setminus B_1(w,G)$ there exists an
induced square $wuxu'$ in $G$. From ($R_{i+2}$), there exist $\tu,
\tu' \sim \tw$ in $\tG_v$ such that $\tu \nsim \tu'$. From ($S_{i+2}$) applied
to $\tw,\tu$ and since $\tw$ has a unique neighbor $\tu'$ mapped to
$u'$, (by R$_{i+2}$) there exists a vertex $\tx$ in  $\tG_v$ such that $f(\tx) = x$, $\tx \sim
\tu, \tu'$ and $\tx \nsim \tw$. Consequently, $f$ is a surjection
from $V(\mbox{St}(\tw,\widetilde{X}_v))$ to $V(\mbox{St}(w,X))$.

Suppose, by way of contradiction, that there exist two distinct vertices
$\tu, \tu'$ of $\mbox{St}(\tw,\widetilde{X}_v)$ such that $f(\tu) =
f(\tu') = u$. If $\tu,\tu' \sim \tw$ then, by condition ($R_{i+1}$) applied
to $\tw$, we obtain a contradiction.  Suppose now that $\tu \sim \tw$ and
$\tu' \nsim \tw$ and let $\tz \sim \tw, \tu'$. This implies that $w,
u, z$ are pairwise adjacent in $G$. Since $f$ is an isomorphism
between the graphs induced by $B_1(\tw,\tG_v)$ and $B_1(w,G)$, we
conclude that $\tz \sim \tu$. But then $f$ is not locally injective
around $\tz$, contradicting the condition (R$_{i+2}$).  Suppose now
that $\tu, \tu' \nsim \tw$. Let $\ta \sim \tu, \tw$ and $\ta' \sim
\tu',\tw$. If $\ta' = \ta$ then, by (R$_{i+2}$) applied to $\ta$, we obtain a
contradiction. If $\ta \sim \ta'$ then, by (R$_{i+2}$) applied to $\ta'$, we have
a contradiction.
Hence $\ta \nsim \ta' \neq \ta$.
Consequently, by
(R$_{i+1}$) applied to $\tw$, we have $a\nsim a'$, and $awa'u$ is a square. By
(S$_{i+2}$) applied to $\tw$ and $\ta'$, we obtain $\ta \sim \tu'$; again, we
get a contradiction.

Therefore $f$ is a bijection between the vertex sets of
$\mbox{St}(\tw,\widetilde{X}_v)$ and $\mbox{St}(w,X)$. Since
$\tX_v$ is a flag complex, by (R$_{i+2}$), $\ta\sim \tb$ in
$\mbox{St}(\tw,\widetilde{X}_v)$ if and only if $a \sim b$ in
$\mbox{St}(w,X)$.  Applying (R$_{i+2}$) to $w$ and taking into account that $X$ and
$\widetilde{X}_v$ are flag complexes, we infer that $\ta\tb\tw$ is a triangle in
$\mbox{St}(\tw,\widetilde{X}_v)$ if and only if $abw$ is a triangle
in $\mbox{St}(w,X)$. For the same reasons (($R_{i+2}$) and flagness of $X$),
if $\ta\tb\tc\tw$ is a square in
$\mbox{St}(\tw,\widetilde{X})$, then $abcw$ is a square in
$\mbox{St}(w,X)$. Conversely, by the conditions ($R_{i+2}$) and
($S_{i+2}$) and flagness of $\widetilde{X}_v$, we conclude that if
$abcw$ is a square in $\mbox{St}(w,X)$, then $\ta\tb\tc\tw$ is a
square in $\mbox{St}(\tw,\widetilde{X}_v)$. Consequently, for any
$\tw \in \widetilde{X}_v$, $f$ defines an isomorphism between
$\mbox{St}(\tw,\widetilde{X}_v)$ and $\mbox{St}(w,X)$, and thus
$f$ is a covering map.
\end{proof}

This finishes the proof of Theorem \ref{t:lotoglo2} and Corollary \ref{t:lotoglo_mod}, since $\tX_v$ is a simply connected, and thus it is the universal covering
space of $X$. It is then unique, i.e., not depending on $v$ and thus $\tG(=\tG_v)$  is weakly modular with respect to every  vertex $v$.
\medskip

We continue with the proof of Theorem \ref{t:lotoglo_G}. Let $G$ be a locally weakly modular graph. Then the $1$--skeleton $\tX\trsq(G)^{(1)}$ of the universal cover of its triangle-square
complex $X\trsq(G)$ is weakly modular, by Theorem~\ref{t:lotoglo2}. It follows that $X\trsq(G)$ is a quotient of $\tX\trsq(G)$
by a group action --- the action of the fundamental group $\pi_1(X\trsq)$ on $\tX\trsq$; cf.\ e.g.\ \cite{Hat}*{Chapter 1.3}.
By Lemma~\ref{l:no456}, the injectivity radius of $X\trsq$ --- that is, the minimal displacement for the $\pi_1(X\trsq)$--action --- is at least $7$. Thus $G=X\trsq(G)^{(1)}$ is the quotient of the $\pi_1(X\trsq)$--action on $\tX\trsq(G)^{(1)}$. The same proof works for the locally modular case.
\medskip

\section{Helly and clique-Helly graphs: proofs}\label{helly:proof}
% of Theorems \ref{t:lotogloHell_bis} and \ref{t:lotogloHell}}

The proofs of Theorems \ref{t:lotogloHell_bis} and \ref{t:lotogloHell}
reside on several propositions which we prove first. In Proposition
\ref{prop-1-Helly-clique-Helly} we show that under some local
conditions (local weak modularity and the $(C_4,W_4)$-condition
defined below), the clique-Helly property implies the $1$--Helly property. In Proposition
\ref{prop-1-Helly-Helly} we prove that for weakly modular graphs, the
$1$--Helly property implies the Helly property, thus extending a result
by Bandelt and Pesch \cite{BP-absolute} for finite graphs. For
establishing these two results, in Proposition
\ref{prop-clh-domination} we extend to arbitrary graphs (and under
weaker assumptions) a domination property established in
\cite{BP-absolute} for finite graphs. %Namely, we prove that if $u$ and
%$v$ are two vertices of a weakly modular clique-Helly graph satisfying
%the $(C_4,W_4)$-condition, if $d(u,v)=k+1$ and $W\subseteq B_1(v)\cap B_{k+1}(u),$
%then there exists a vertex $x\in B_1(v)\cap B_k(u)$ such that $W\subseteq B_1(x).$
Then, in Proposition \ref{t:universal-cover-Helly}, we adapt our methods of
proof of Theorem \ref{t:lotoglo2} to show that if $G$ is a
clique-Helly graph and $\tG$ is the $1$--skeleton of the universal
cover $\widetilde X:=\widetilde{X}\tr(G)$ of the triangle complex
$X:=X\tr(G)$ of $G$, then $\tG$ is weakly modular, clique-Helly, and
satisfies the $(C_4,W_4)$-condition. Unfortunately, even if the main
steps of both proofs are similar, the proofs of these steps are
different. This is because the local conditions in both results are
different: local weak modularity in Theorem \ref{t:lotoglo2} and the
clique-Helly property in Proposition
\ref{t:universal-cover-Helly}. Therefore, the definitions of the
equivalence relation $\equiv$ in both proofs are different.

\medskip

We will say that a graph $G$ satisfies the $(C_4,W_4)$-{\it condition} if every square of $G$ ``lives'' in a $W_4$, i.e., for
every square $abcd$ of $G$, there exists a vertex $x \sim a,b,c,d$.

In this section, we will use transfinite induction. Given
a set $X$ and a well-order $\prec$ on $X$, for any $x \in X$, we will
denote the set $\{x'\in X: x' \prec x\}$ by $X_{\prec x}$ and the set
$\{x'\in X: x' \preceq x\}$ by $X_{\preceq x}$.

In the following proofs, by ``clique'' we will mean a complete
subgraph and when we apply the Helly property to pairwise intersecting
cliques $C_i, i\in I$, in fact we apply it to a collection of maximal
cliques $C'_i, i\in I$, extending them in order to derive a vertex $x$
either belonging to all $C_i$ or adjacent to all vertices of all
$C_i$, i.e., $x\in B_1(v)$ for all $v\in \bigcup_{i\in I} C_i$.

\begin{proposition} \label{prop-clh-domination} Let $G$ be a (finitely) clique-Helly weakly modular graph
satisfying the $(C_4,W_4)$-condition.  Then for any two vertices $u,v$
of $G$ with $d(u,v)=k+1 \ge 1$ and for any (finite) subset $W
\subseteq B_1(v) \cap B_{k+1}(u)$, there exists a vertex $y \in B_1(v)
\cap B_k(u)$ that is adjacent to all vertices of $W \setminus \{y\}$.

Moreover, if $G$ is a (finitely) clique-Helly locally weakly modular
graph satisfying the $(C_4,W_4)$-condition, then the property holds
for any vertices $u, v$ at distance  $2$.
\end{proposition}

\begin{proof}  We first state a lemma that will be useful in the proof:

\begin{lemma}\label{lem-clique-Helly-carre-triangle} Let $G$ be a finitely clique-Helly locally weakly
modular graph satisfying the $(C_4,W_4)$-condition. Then for any  vertices $t,v,w,y$ such that $v \sim y,w$,
such that $t \sim y$,  and such that $d(v,t) = d(w,t) = d(w,y) = 2$, there exists $z \sim  t,v,w,y$.
\end{lemma}

\begin{proof}  By the local triangle condition, there exists $x \sim v,w,t$.
If  $x\sim y$, we are done. Otherwise, by the $(C_4,W_4)$-condition applied to the square $vxty$, there exists $u \sim v,x,t,y$.
If  $u \sim w$, we are done. Otherwise, consider a maximal clique  containing the triangle $vwx$, a maximal clique containing the
triangle $tux$, and a maximal clique containing the triangle  $uvy$. Note that these three cliques pairwise intersect in $u$, $v$, or $x$.
Since $G$ is clique-Helly, there exists a vertex $z$  in the intersection of these cliques. Consequently, $z \sim  t,v,w,y$.
\end{proof}

We prove Proposition \ref{prop-clh-domination} by induction on $k=d(u,v)-1$. If $k=0$, then set $y:= u$.
Assume now that the result  holds for $k-1$.  Consider a well-order $\prec$ on the vertices of  $W$.

\begin{lemma}\label{lem-clh-dom-lim}
  Let $W'$ be a non-empty subset of $W$ and assume that there exists a
  function $x: W' \to B_1(v) \cap B_k(u)$ such that $x(w') \in
  B_1(w'') \cap B_1(x(w''))$ for every $w''\in W' \cap W_{\preceq w'
  }$. Then there exists $y \in B_1(v) \cap B_k(u)$ such that for every
  $w' \in W'$, $y \in B_1(w') \cap B_1(x(w'))$.
  %%
  %% Let $W'$ be a non-empty subset of $W$ and assume that for every $w'
  %% \in W'$ there exists a vertex $x(w') \in B_1(v) \cap B_k(u)$ such
  %% that $x(w') \in B_1(w'') \cap B_1(x(w''))$ for any $w''\in W' \cap
  %% W_{\preceq w' }$. Then there exists $y \in B_1(v) \cap B_k(u)$ such
  %% that for every $w' \in W'$, $y \in B_1(w') \cap B_1(x(w'))$.
\end{lemma}

\begin{proof} If there exists $w' \in W'$ such that $W' \subseteq W_{\preceq  w'}$, we can set $y:= x(w')$.
Note that if $W'$ (or $W$) is finite, we are necessarily in this case.
Otherwise, for every $w' \in W'$, let $K(w'):= \{v,w'\} \cup \{x(w'')
: w'' \in W' \mbox{ and } w' \preceq w'' \}$. By our assumptions, for
any $w' \in W'$, $K(w')$ is a clique, and for any $w', w'' \in W'$
such that $w' \prec w''$, the cliques $K(w')$ and $K(w'')$ intersect
in $x(w'')$. Let $w'_0$ be the least element of $W'$. Since $d(u,x(w_0')) \leq  k$, by our induction hypothesis on $k$, there
exists $t \in B_1(x(w'_0)) \cap B_{k-1}(u)$ such that $t \sim z$ for any other neighbor $z$ of $x(w'_0)$ that
is at distance at most $k$ from $u$. Since $K(w'_0) \setminus \{w'_0,v\} = \{x(w'): w' \in W'\} \subseteq B_1(x(w'_0)) \cap B_k(u)$,
we have that $K(w_0) \setminus \{w_0,v\} \cup \{t\}$ is a clique of $G$ that intersects $K(w')$ on $x(w')$ for every $w' \in W'$.
Consequently, all these cliques pairwise intersect, and since $G$ is clique-Helly, there exists a vertex
$y \in B_1(v) \cap B_1(t) \cap \bigcap_{w' \in W'} B_1(w') \cap \bigcap_{w' \in W'} B_1(x(w'))$. Since
$y \in B_1(v) \cap B_1(t)$, necessarily  $d(u,y)=k$. This concludes the proof of the lemma.
\end{proof}

\begin{lemma}\label{lem-clh-dom-w}
  There exists a function $x: W' \to B_1(v)\cap B_k(u)$ such that
  $x(w) \in B_1(w')\cap B_1(x(w'))$ for any $w'\in W_{\preceq w }$.
%%
%%   For any $w \in W$, there exists $x(w) \in B_1(v) \cap B_k(u)$ such that  $x(w) \in B_1(w')\cap B_1(x(w'))$
%% for any $w'\in W_{\preceq w }$.
\end{lemma}

\begin{proof} We define the function $x$ by (possibly transfinite) induction.
  Let $w_0$ be the least element of $(W,\prec)$. If $d(w_0,u) = k$,
  then set $x(w_0):= w_0$. If $d(w_0,u)=k+1$, by TC($u$) (or LTC($u$)
  when $k = 1$) applied to the edge $vw_0$ and $u$, there exists $x_0
  \in B_k(u)$ such that $x_0\sim v,w_0$; we set $x(w_0) := x_0$.
%%  In both cases, the lemma holds for $w_0$ and $x(w_0)$ is defined.

  Given a vertex $w$, assume that for every $w' \in W_{\prec w}$, the
  vertex $x(w')$ has been defined and that $x(w') \in B_1(w'')\cap
  B_1(x(w''))$ for any $w''\in W_{\preceq w'}$. Applying
  Lemma~\ref{lem-clh-dom-lim} with $W' = W_{\prec w}$, there exists a
  vertex $y \in B_k(u) \cap B_1(v)$ such that for every $w'\in
  W_{\prec w}$, $y \in B_1(w') \cap B_1(x(w'))$.

  If $y \in B_1(w)$, then let $x(w):= y$ and then $x(w) \in
  B_1(w')\cap B_1(x(w'))$ for any $w'\in W_{\preceq w }$.  Suppose now
  that $y \notin B_1(w)$. If $d(u,w) = k$, by QC($u$), there exists $t
  \in B_{k-1}(u)$ such that $t \sim y, w$ (Note that $t = u$ if $k =
  1$).  Consequently, by the $(C_4,W_4)$-condition, there exists
  $z\sim t,v,w,y$. If $d(u,w) = k+1$, by TC($u$) (or LTC($u$) if
  $k=1$), there exists $s \sim v,w$ such that $d(u,w) = k$.  By
  TC($u$) if $s \sim y$ or by QC($u$) otherwise, there exists $t \in
  B_{k-1}(u)$ such that $t \sim y,s$ (Note that $t = u$ if $k =
  1$). Consequently, $d(t,w)=2$. By
  Lemma~\ref{lem-clique-Helly-carre-triangle}, there exists $z\sim
  t,v,w,y$.

Consider the cliques $K(w)= \{v,w,z\}$ and $K(t) = \{t,y,z\}$. For any
$w'\in W_{\prec w}$, consider the clique $K(w')=\{v,w',x(w'),y\}$.
Note that all these cliques pairwise intersect in $v$, $y$, or
$z$. Moreover, note that if $W$ is finite, we only have a finite
number of cliques.  Since $G$ is clique-Helly (or finitely
clique-Helly when $W$ is finite), there exists a vertex $x$ common to
maximal cliques extending all these cliques. Let $x(w) := x$ and note
that $x(w) \in B_1(v) \cap B_1(w) \cap B_1(t) \cap \bigcap_{w' \prec
  w} B_1(w') \cap \bigcap_{w' \prec w} B_1(x(w'))$.  Since $x(v) \in
B_1(v) \cap B_1(t)$, necessarily $d(u,x(w))=k$. This ends the proof of
the lemma.
\end{proof}

From Lemma~\ref{lem-clh-dom-w}, for any $w \in W$, there exists $x(w)
\in B_1(v) \cap B_k(u)$ such that for any $w' \in W_{\preceq w}$, we have $x(w) \in
B_1(w')$ and $x(w) \in B_1(x(w'))$. Therefore, we can apply
Lemma~\ref{lem-clh-dom-lim} with $W'= W$. By this lemma, there exists
a vertex $y \in B_1(v) \cap B_k(u)$ such that for any $w \in W$, $y
\in B_1(w)$.  This finishes the proof of
Proposition~\ref{prop-clh-domination}.
\end{proof}

\begin{proposition} \label{prop-1-Helly-Helly} A weakly modular
  graph $G$ is (finitely) Helly if and only if it is (finitely)
  $1$--Helly.
\end{proposition}

\begin{proof}
Note that any (finitely) Helly graph is trivially (finitely)
$1$--Helly. Assume now that $G$ is (finitely) $1$--Helly and weakly
modular. First we will show that $G$ is pseudo-modular.  By
Proposition \ref{pseudo-modular} it suffices to check that if $1\le
d(u,w)\le 2$ and $d(v,u)=d(v,w)=k\ge 2,$ then there exists a vertex
$x\sim u,w$ with $d(v,x)=k-1$. If $d(u,w)=1$, then as $x$ we can take
any vertex provided by (TC) for vertices $u,w,v$. So, let
$d(u,w)=2$. Pick an arbitrary common neighbor $y$ of $u,w$. If
$d(v,y)=k-1,$ then we are done. If $d(v,y)=k+1$, then applying (QC) we
will find a common neighbor $x$ of $u,w$ with $d(v,x)=k-1$, and we are
done. Hence $d(v,y)=k$. Applying the triangle condition twice, we will
find vertices $z'\sim u,y$ and $z''\sim y,w$ with
$d(v,z')=d(v,z'')=k-1$. We can suppose that $z'\ne z''$, otherwise we
can take $z'=z''$ as $x$. By applying (QC) if $z'\nsim z''$ or (TC) if
$z'\sim z''$ we will find a vertex $z\sim z',z''$ with
$d(v,z)=k-2$. The unit balls centered at $u,w,$ and $z$ pairwise
intersect. By the $1$--Helly property, they have a common vertex
$x$. Since $u,w,z$ are pairwise non-adjacent, $x$ is a common neighbor
of $u,w,$ and $z$. Then $d(v,x)=k-1$, and $x$ is the required
vertex. This shows that $G$ is pseudo-modular, and therefore any
collection of three pairwise intersecting balls of $G$ has a nonempty
intersection.

Now, we will prove that $G$ is (finitely) Helly. A collection
${\mathcal S}=\{ S_i: i\in I\}$ of balls is called {\it
  Helly-critical} \cite{Po_helly} if the sets of $\mathcal S$ pairwise
intersect but their intersection is empty. Now, following the proof of
assertion (a) of Proposition 3.1.2 of \cite{Po_helly} we will show
that if $G$ is not (finitely) Helly, then $G$ has a (finite)
Helly-critical collection of balls in which one ball has radius 1. Let
$r$ be the smallest positive integer such that there exists a (finite)
Helly-critical family ${\mathcal B}=\{ B_{r_i}(v_i): i\in I\}$ with
$r_{i_0}=r$ for some $i_0\in I$.  Assume $r>1$.  Consider the new
family of balls ${\mathcal B}'=\{ B_{r'_i}(v_i): i\in I\}$ with the
same set of centers and $r'_i=r_i+1$ if $i\ne i_0$ and
$r'_{i_0}=r_{i_0}-1$. Clearly, the balls of ${\mathcal B}'$ pairwise
intersect and ${\mathcal B}'$ is finite if ${\mathcal B}$ is. By the
minimality choice of $r$, ${\mathcal B}'$ is not Helly-critical, i.e.,
there exists a vertex $v$ common to all balls $B_{r'_i}(v_i)$ of
${\mathcal B}'$. Then the family of balls ${\mathcal B}''$ obtained by
adding $B_1(v)$ to the family $\mathcal B$ consists of pairwise
intersecting balls but has an empty intersection because $\mathcal B$
is Helly-critical.  Hence ${\mathcal B}''$ is a (finite)
Helly-critical family of balls of $G$ containing a ball of unit
radius.

Now, suppose that $G$ is (finitely) $1$--Helly and weakly modular but
is not (finitely) Helly. From the previous assertion we conclude that
$G$ contains a (finite) Helly-critical family of balls ${\mathcal
  B}=\{ B_{r_i}(v_i): i\in I\}$ with $r_{i_0}=1$ for some $i_0\in I$.
Since $G$ is pseudo-modular, for any $i,j\in I$ the balls
$B_{r_i}(v_i),B_{r_j}(v_j),$ and $B_1(v_{i_0})$ have a nonempty
intersection $S_{ij}$. For any $i,j\in I$, let $w_{ij}$ be an
arbitrary vertex from $S_{ij}$, and for any $i \in I$, let $W_i =
\{w_{ij}: j \in I\setminus \{ i\}\}$. Notice that if
${\mathcal B}$ is finite, then $W_i$ is also finite.

For any $i \in I$, we define a vertex $v_i'$ as follows depending on
the distance $d(v_i,v_{i_0})$. Note that since $B_{r_i}(v_i) \cap
B_1(v_{i_0})\ne\emptyset$, we have $d(v_i,v_{i_0}) \leq r_i+1$.
If $d(v_i,v_{i_0}) \leq r_i - 1$, let $v_i' = v_{i_0}$.
If $d(v_i,v_{i_0}) = r_i$, for any $w_{ij} \in W_i$, $d(v_i,w_{ij})
\leq d(v_i,v_{i_0})$ and $w_{ij} \in B_1(v_{i_0})$.  Consequently, by
Proposition~\ref{prop-clh-domination}, there exists $v_i' \in
B_1(v_{i_0}) \cap B_{r_i-1}(v_i)$ such that $v_i' \in \bigcap_{w_{ij}
  \in W_i} B_1(w_{ij})$.
If $d(v_i,v_{i_0}) = r_i + 1$, for any $w_{ij} \in W_i$,
$d(v_i,w_{ij}) = d(v_i,v_{i_0})-1$ and $w_{ij} \in
B_1(v_{i_0})$. Consequently, by Proposition~\ref{prop-clh-domination}, there
exists $x_i \in B_1(v_{i_0}) \cap B_{r_i}(v_i)$ such that $x_i \in
\bigcap_{w_{ij} \in W_i} B_1(w_{ij})$. Note that for any $w_{ij} \in
W_i$, $d(v_i,w_{ij}) = d(v_i,x_i)$. Consequently, by
Proposition~\ref{prop-clh-domination}, there exists $v_i' \in B_1(x_i)
\cap B_{r_i-1}(v_i)$ such that $v_i' \in \bigcap_{w_{ij} \in W_i}
B_1(w_{ij})$.
Note that for any $i \in I$, $B_1(v_i') \subseteq
B_{r_i}(v_i)$ and for any $w_{ij} \in W_i$, $w_{ij} \in B_1(v'_1) \cap
B_1(v_{i_0})$.

Consider the collection of unit balls ${\mathcal B}'=\{ B_{1}(v'_i):
i\in I\}$. Note that ${\mathcal B}'$ is finite if ${\mathcal B}$ is
finite.  Since any two balls $B_{1}(v'_i), B_{1}(v'_j)$ of ${\mathcal
  B}'$ intersect in $w_{ij}$ and since $G$ is (finitely) $1$--Helly,
there exists a vertex $x \in \bigcap_{i \in I} B_1(v_i') \subseteq
\bigcap_{i \in I} B_{r_i}(v_i)$. Thus, we have found a common vertex
of the balls of $\mathcal B$, contrary to the assumption that
$\mathcal B$ is Helly-critical.
\end{proof}

%%%%%FIN VICTOR

%%%%% DEBUT CLIQUE-HELLY \implies 1-HELLY

\begin{proposition}\label{prop-1-Helly-clique-Helly}
A graph $G$ is (finitely) $1$--Helly if and only if $G$ is (finitely)
clique-Helly, locally weakly modular graph, and it satisfies the
$(C_4,W_4)$-condition.
\end{proposition}

\begin{proof}
  If $G$ is a $1$--Helly graph, then the $1$--Helly property for unit
  balls implies that $G$ satisfies the local triangle and the local
  quadrangle conditions; hence $G$ is locally weakly modular.
  Moreover, any $1$--Helly graph is trivially a clique-Helly
  graph. Finally, for any square $abcd$ of $G$, the $1$-balls centered
  at $a$, $b$, $c$, and $d$ pairwise intersect and consequently, there
  exists a vertex $e \sim a,b,c,d$.

  In the following, we show the reverse implication.  Consider a
  (finite) set $S$ of vertices of $G$ such that the 1-balls of the
  family ${\mathcal B}=\{ B_{1}(v): v \in S \}$ pairwise intersect.
  Note that for any $v, v' \in S$, we have $d(v,v') \leq 2$. Denote by
  $H$ the subgraph of $G$ induced by the set $\bigcap_{v \in S}
  B_2(v)$.
%% By Proposition~\ref{prop-clique-Helly-clique-Helly}, $H$ is a
%% clique-Helly graph.
  We claim that there exists a vertex $x$ of $H$ belonging to
  $\bigcap_{v \in S} B_1(v)$.  Consider a well-order $\prec$ on the
  vertices of $S$.

  \begin{lemma}\label{lem-1H-clh-lim}
    Let $S'$ be a non-empty subset of $S$ and assume that there exists
    a function $y:S' \to V(H)$ such that $y(v') \in B_1(v'')\cap
    B_1(y(v''))$ for every $v'' \in S_{\preceq v'} \cap S'$.  Then
    there exists a vertex $x \in V(H)$ such that $x \in \bigcap_{v'
      \in S'} B_1(v') \cap \bigcap_{v' \in S'} B_1(y(v'))$.
    %% Let $S'$ be a non-empty subset of $S$ and assume that for every
    %% $v' \in S'$ there exists a vertex $y(v') \in V(H)$ such that for
    %% every $v'' \in S_{\preceq v'} \cap S'$, $y(v') \in B_1(v'')\cap
    %% B_1(y(v''))$.  Then there exists a vertex $x \in V(H)$ such that
    %% $x \in \bigcap_{v' \in S'} B_1(v') \cap \bigcap_{v' \in S'}
    %% B_1(y(v'))$.
  \end{lemma}

  \begin{proof}
    If there exists $v' \in S'$ such that $S' \subseteq S_{\preceq
      v'}$, we can set $x:= y(v')$ and we are done. Note that if $S'$
    (or $S$) is finite, we are necessarily in this case.

    Otherwise, let $Y' = \{y(v') : v' \in S'\}$. Note that by our
    assumptions, $Y'$ is a clique.  For any $v' \in S'$, consider the
    clique $K(v') = \{y(v''): v'' \in S \mbox{ and } v' \preceq v''\}
    \cup \{v'\}$. Note that for any $v' \in S'$, $K(v')$ is a
    clique. Moreover, for any $v', v'' \in S'$ with $v' \prec v''$, we
    have $K(v') \cap K(v'') \neq \emptyset$ since $y(v'') \in K(v')
    \cap K(v'')$.

    For any $v \in S \setminus S'$, if $d(v,y') \leq 1$ for some $y'
    \in Y'$, let $t(v)= y'$. Otherwise, note that $d(y',v) \leq 2$
    for all $y' \in Y'$ since $Y' \subseteq V(H)$. Applying
    Proposition~\ref{prop-clh-domination} for the vertices $v$ and
    some $y'_0 \in Y'$ with $W=Y'$, we know that there exists $t(v) \sim v$
    such that $t(v) \sim y'$ for all $y' \in Y'$. In both cases,
    consider the clique $K(v) = Y' \cup \{t(v)\}$.

    Since $G$ is a clique-Helly graph, there exists a vertex $x$ such
    that $x \in \bigcap_{v' \in S'} B_1(v') \cap \bigcap_{v' \in S'}
    B_1(y(v')) \cap \bigcap_{v \in S \setminus S'} B_1(t(v))$. Since
    for any $v \in S \setminus S'$, $t(v) \in B_1(v)$, necessarily, $x
    \in V(H) = \bigcap_{v \in S} B_2(v)$. Consequently, $x$ is a
    vertex of $H$ that lies in $\bigcap_{v' \in S'} B_1(v') \cap
    \bigcap_{v' \in S'} B_1(y(v'))$.
  \end{proof}

  \begin{lemma}\label{lem-1H-clh-v}
    There exists a function $y: S \to V(H)$ such that for every $v'
    \in S_{\preceq v}$, $y(v) \in B_1(v')\cap B_1(y(v'))$.
    %% For any vertex $v\in S$, there exists a vertex $y(v) \in V(H)$
    %% such that for every $v' \in S_{\preceq v}$, $y(v) \in B_1(v')\cap
    %% B_1(y(v'))$.
  \end{lemma}

  \begin{proof}
    We define the function $y$ by (possibly transfinite) induction.
    Let $v_0$ be the least element of $(S,\prec)$ and set $y(v_0):=
    v_0$.

    Given a vertex $v$, assume that for every $v' \in S_{\prec v}$,
    $y(v') \in V(H)$ has been defined and that for every $v'' \in
    S_{\preceq v'}$, $y(v') \in B_1(v'')\cap B_1(y(v''))$.  Note that
    for every $v'$, $d(v,y(v')) \leq 2$ since $y(v') \in V(H)$.  By
    Lemma~\ref{lem-1H-clh-lim} applied with $S' = S_{\prec v}$, there
    exists a vertex $x \in V(H)$ such that $x \in \bigcap_{v' \in
      S_{\prec v}} B_1(y(v')) \cap \bigcap_{v' \in S_{\prec v}}
    B_1(v')$. Note that since $x \in V(H)$, $d(v,x) \leq 2$. If $v \in
    B_1(x)$, let $y(v) = x$ and we are done.

    Suppose now that $d(v,x) = 2$. Note that for every $v'\in S_{\prec v}$,
    $d(v,v') \leq d(v,x) = 2$, $d(v,y(v')) \leq d(v,x) = 2$, and $x \in
    B_1(v') \cap B_1(y(v'))$. Moreover, for any $v'' \in S$, we have
    $d(v'',x) \leq 2$. If $d(v'',x) \leq 1$, let $t(v'') =
    x$. Otherwise, let $s$ be a common neighbor of $v''$ and $x$. If
    $d(v,s) \leq 2$, let $t(v'') = s$. If $d(v,s) = 3$, by LQC($v$),
    there exists $s' \sim v'',x,v$; in this case, let $t(v'' ) =
    s'$. In any case, $t(v'')\in B_1(v'')$ and $t(v'') \in B_2(v) \cap
    B_1(x)$. Consider the set $W = \{v' : v' \in S_{\prec v}\} \cup
    \{y(v') : v' \in S_{\prec v}\} \cup \{t(v'') : v'' \in S\}$. Note
    that if $S$ is finite, $W$ is necessarily also finite.  Consequently,
    by Proposition~\ref{prop-clh-domination} applied to the vertices
    $v$ and $x$ and to the set $W$, there exists $t \sim v,x$ such
    that $t \sim v$ and $t \in \bigcap_{v' \in S_{\prec v}} B_1(y(v')) \cap
    \bigcap_{v' \in S_{\prec v}} B_1(v') \cap \bigcap_{v'' \in S}
    B_1(t(v'')) $. Since for any $v'' \in S$, $B_1(t(v'')) \subseteq
    B_2(v'')$, we conclude  that $t \in V(H)$. Consequently, we can set
    $y(v) := t$ and we are done.
%%
    %% We claim that $t \in V(H)$. Indeed, for any $v' \preceq v$, $t \in
    %% B_1(v')$ and thus $d(t,v') \leq 2$. Consider now a vertex $v''$
    %% such that $v \prec v''$. Since $x \in V(H)$, $d(x,v'') \leq 2$. If
    %% $d(x,v'') \leq 1$, then $d(t,v'') \leq 2$. If $d(x,v'') = 2$, let
    %% $q$ be a common neighbor of $x$ and $v''$. If $d(v,q) \leq d(v,x)
    %% = 2$, then $t\sim q$ and thus $d(t,v'') \leq 2$. If $d(v,q) = 3$,
    %% by QC($v$), there exists $r \sim v,v'',x$. Since $d(v,r) \leq
    %% d(v,x)$, $d(r,t) \leq 1$ and consequently, $d(t,v'') \leq
    %% 2$. Therefore, for any $v' \in S$, $t \in B_2(v')$ and
    %% consequently, $t \in V(H)$. Thus, we can set $y(v) := t$ to
    %% satisfy the conclusion of the lemma for $v$.
  \end{proof}

  From Lemma~\ref{lem-1H-clh-v}, for any $v \in S$, there exists $y(v)
  \in V(H)$ such that $y(v) \in \bigcap_{v' \in S_{\preceq v}}
  B_1(y(v')) \cap \bigcap_{v' \in S_{\preceq v}} B_1(v')$.  Therefore,
  we can apply Lemma~\ref{lem-1H-clh-lim} with $S' = S$. By this
  lemma, there exists a vertex $x \in \bigcap_{v \in S} B_1(y(v)) \cap
  \bigcap_{v \in S} B_1(v)$. Consequently, $x$ belongs to $\bigcap_{v
    \in S} B_1(v)$. This shows that $G$ is $1$--Helly and this ends
  the proof of Proposition~\ref{prop-1-Helly-clique-Helly}.
\end{proof}

\begin{proposition} \label{t:universal-cover-Helly} Let $G$ be a
  (finitely) clique-Helly graph and let $\tG$ be the $1$--skeleton of
  the universal cover $\widetilde X:=\widetilde{X}\tr(G)$ of the
  triangle complex $X:=X\tr(G)$ of $G$. Then $\tG$ is weakly modular,
  (finitely) clique-Helly, and satisfies the $(C_4,W_4)$-condition.
\end{proposition}

\begin{proof}
Let $G$ be a finitely clique-Helly graph, and let $\tG$ be the
$1$--skeleton of the universal cover $\widetilde
X:=\widetilde{X}\tr(G)$ of the triangle complex $X:=X\tr(G)$ of
$G$. To prove that $\tG$ is a weakly modular graph satisfying the
$(C_4,W_4)$-condition, we will construct the universal cover
$\widetilde{{X}}$ of ${X}$ as an increasing union $\bigcup_{i\ge 1}
\widetilde{{X}}_i$ of triangle complexes. The complexes
$\widetilde{{X}}_i$ are in fact spanned by concentric combinatorial
balls $\widetilde{B}_i$ in $\widetilde{{X}}$.  The covering map $f$ is
then the union $\bigcup_{i\ge 1} f_i,$ where $f_i:
\widetilde{{X}}_i\rightarrow {X}$ is a locally injective simplicial
map such that $f_i|_{\widetilde{{X}}_j}=f_j$, for every $j\le i$. We
denote by $\widetilde{G}_i=G(\widetilde{{X}}_i)$ the underlying graph
of $\widetilde{{X}}_i$. We denote by $\tS_i$ the set of vertices
$\tB_i\setminus \tB_{i-1}$.

Pick any vertex $v$ of ${X}$ as the base-point. Define
$\widetilde{B}_0=\{ \widetilde{v}\}:=\{ v\},
\widetilde{B}_1:=B_1(v,G)$.  Let $\widetilde{{X}}_1$ be the triangle
complex of $B_1(v,G)$ and let $f_1\colon \tX_1 \to X$ be the
simplicial map induced by Id$_{B_1(v,G)}$.  Assume that, for $i\geq
1$, we have constructed the vertex sets
$\widetilde{B}_1,\ldots,\widetilde{B}_i,$ and we have defined the
triangle complexes $\widetilde{{X}}_1\subseteq \cdots\subseteq
\widetilde{{X}}_i$ (for any $1\le j<k\le i$ we have an identification
map $\widetilde{{X}}_j\rightarrow \widetilde{{X}}_{k}$) and the
corresponding simplicial maps $f_1,\ldots,f_i$ from
$\widetilde{{X}}_1,\ldots,\widetilde{{X}}_i,$ respectively, to ${X}$
so that the graph $\tG_i=G(\widetilde{X}_i)$ and the complex
$\widetilde{X}_i$ satisfy the following conditions:

\begin{enumerate}[(A{$_i$})]
\item[(P$_i$)]
$B_j(\tv,\tG_i)=\widetilde{B}_j$ for any $j\le i$;
\item[(Q$_i$)]
$\widetilde{G}_i$ is weakly modular with respect to $\widetilde{v}$;
%(i.e.,$\widetilde{G}_i$ satisfies the conditions TC($\widetilde{v}$) and
%QC($\widetilde{v}$));
%
\item[(R$_i$)]
for any $\widetilde{u}\in \widetilde{B}_{i-1},$ $f_i$
defines an isomorphism between the subgraph of $\tG_i$ induced by
$B_1(\widetilde{u},\tG_i)$ and the subgraph of $G$ induced by
$B_1(f_i(\widetilde{u}),G)$;
%
%\item[(S$_i$)] for any $j \leq i$ and for any square
%  $\tw_1\tw_2\tw_3\tw_4$ such that $\tw_1 \in \tS_j$, $\tw_2, \tw_4
%  \in \tS_{j-1}$ and $\tw_3 \in \tS_{j-2}$, there exists $\tu \in
%  \tS_{j-1}$ such that $\tu \sim \tw_1,\tw_2,\tw_3,\tw_4$.
%
\item[(S$_i$)] for any $j \leq i$ and for any $\tw,\tw'\in \tS_{j-1}$
  which are not adjacent and have a common neighbor $\tz$ in $\tS_j$,
  there exist $\tu\in \tS_{j-2}$ and $\tu'\in \tS_{j-1}$ such that
  $\tz\tw\tu\tw'$ is a square and $\tu' \sim \tz, \tw, \tu, \tw'$.
\item[(T$_i$)] for any $\widetilde{w}\in
  \widetilde{S}_i:=\widetilde{B}_i\setminus \widetilde{B}_{i-1},$
  $f_i$ defines an isomorphism between the subgraphs of $\tG_i$ and of
  $G$ induced by, respectively, $B_1(\widetilde{w},\tG_i)$ and
  $f_i(B_1(\widetilde{w},\tG_i))$.
%\item[(U$_i$)] the triangle complex $\tX_i$ is simply connected.
\end{enumerate}

Note that the main difference with the proof of
Theorem~\ref{t:lotoglo2} (the case of general weakly modular graphs)
is that we strengthen the quadrangle condition QC($\tv$) (condition
(S$_i$)). This ensures that the triangle complex of $\tG$ is simply
connected and $\tG$ satisfies the $(C_4,W_4)$-condition.

It can be easily checked that $\widetilde{B}_1, \widetilde{G}_1,
\widetilde{X}_1$ and $f_1$ satisfy the conditions (P$_1$), (Q$_1$),
(R$_1$), (S$_1$), and (T$_1$). Now we construct the set
$\widetilde{B}_{i+1},$ the graph $\widetilde{G}_{i+1}$ %having
%$\widetilde{B}_{i+1}$ as the vertex-set,
the triangle complex
$\widetilde{{X}}_{i+1}$, % having $\tG_{i+1}$ as its 1-skeleton,
and the map $f_{i+1}: \widetilde{{X}}_{i+1}\rightarrow {X}.$ Let
 $$Z=\{ (\widetilde{w},z): \widetilde{w} \in \widetilde{S}_i \mbox{ and } z\in
B_1(f_i(\widetilde{w}),G)\setminus f_i(B_1(\widetilde{w},\tG_i))\}.$$
On $Z$ we define a binary relation $\equiv$ by setting $(\widetilde{w},z)\equiv
(\widetilde{w}',z')$ if and only if $z=z'$  and one of the following two
conditions is satisfied:

\begin{itemize}
\item[(Z1)] $\widetilde{w}$ and $\widetilde{w}'$ are the same or adjacent in
$\widetilde{G}_i$;%%  and $z\in B_1(f_i(\widetilde{w}),G)\cap
%% B_1(f_i(\widetilde{w}'), G);$
\item[(Z2)] there exist $\widetilde{u}\in \widetilde{B}_{i-1}$ and
  $\widetilde{u}'\in \widetilde{B}_{i}$ such that $\tu \sim \tw,\tw'$,
  that $\tu' \sim \tu, \tw, \tw'$, that
  $f_i(\widetilde{u})f_i(\widetilde{w})zf_i(\widetilde{w}')$ is a
  square in $G$, and that $f_i(\tu') \sim
  f_i(\tw),f_i(\tu),f_i(\tw'),z$ in $G$.
\end{itemize}

In what follows, the above relation will be used in the inductive step to
construct $\tG_{i+1}$, $\tX_{i+1}$, $f_{i+1}$ and all related objects.
\medskip

First, however, we show that the relation $\equiv$ defined above is an
equivalence relation.  The set of vertices of the graph $\tG_{i+1}$
will be then defined as the union of the set of vertices of the
previously constructed graph $\tG_{i}$ and the set of equivalence
classes of $\equiv$.  In the remaining part of the proof, for a vertex
$\widetilde{w}\in \widetilde{B}_i$, we denote by $w$ its image $f_i(\widetilde{w})$
in $X$ under $f_i$.

\begin{Rem}\label{rem-local-inj}
     For any $(\tw,z) \in Z$ and for any neighbor $\tu$ of $\tw$
     in $\tB_i$, $u \neq z$.

     Moreover, for any neighbor $\tu$ of $\tw$ in $\tB_{i-1}$, if $u
     \sim z$, then by (R$_i$) applied to $\tu$ there exists $\tz \in \tB_i$
     such that $\tz \sim \tu$ and $\tz \sim \tw$, a
     contradiction. Similarly, for any neighbor $\tu$ of $\tw$ in
     $\tS_{i}$, if $u \sim z$ and $(\tu,z) \notin Z$, by (T$_i$)
     applied to $\tu$, there exists $\tz \in \tB_i$ such that $\tz
     \sim \tu$ and $\tz \sim \tw$, a contradiction.

     Consequently, for any neighbor $\tu \in \tB_{i-1}$ of $\tw$, we have $u
     \nsim z$. Furthermore,  for any neighbor $\tu \in \tS_{i}$, either $u \nsim
     z$ or $(\tu,z)\in Z$.
\end{Rem}

\begin{lemma} \label{equiv-clh} The relation $\equiv$ is an equivalence relation on
$Z$.
\end{lemma}

\begin{proof}
Since the binary relation $\equiv$ is reflexive and symmetric, it
suffices to show that $\equiv$ is transitive. Let $(\tw,z)\equiv
(\tw',z')$ and $(\tw',z')\equiv
(\tw'',z'')$. We will prove that $(\widetilde{w},z)\equiv
(\widetilde{w}'',z'').$ By the definition of $\equiv,$ we conclude
that $z=z'=z''$ and that $z\in B_1(w,G)\cap B_1(w',G)\cap B_1(w'',G)$.

If $\tw = \tw''$ or $\tw \sim \tw''$ (in $\widetilde{G}_i$) then, by
the definition of $\equiv$, $(\tw,z)\equiv (\tw'',z)$ and we are done.
Therefore, further we assume that  $\tw \neq \tw'' \nsim \tw$.
In the following, we distinguish three cases: either $\tw' \sim \tw$,
$\tw''$, or $\tw'$ is adjacent to only one of $\tw, \tw''$, or $\tw'
\nsim \tw, \tw''$.

%% \medskip
%% \noindent\textbf{Case 1.} $\tw \sim \tw'$ and $\tw' \sim \tw''$.
%% \smallskip

\begin{case-He}
  $\tw \sim \tw'$ and $\tw' \sim \tw''$.
\end{case-He}

By (Q$_i$), there exist $\tx,\tx' \in \tS_{i-1}$ such that $\tx \sim
\tw, \tw'$ and $\tx' \sim \tw', \tw''$. By Remark~\ref{rem-local-inj},
$x \neq z$, $x' \neq z$, and $z \nsim x,x'$.

Note that if $\tx \sim \tw''$, then by (T$_i$) applied to $\tw'$, $x
\sim w,w',w''$ and $w \nsim w''$. Moreover, by
Remark~\ref{rem-local-inj}, $x \nsim z$. Consequently, $xwzw''$ is a
square in $G$ and $w' \sim w,x,z,w''$. By (Z2) with $\tu = \tx$ and
$\tu' = \tw'$, $(\tw,z) \equiv (\tw',z)$ and we are done. Similarly,
if $\tx' \sim \tw$, we have that $(\tw,z) \equiv (\tw',z)$ by (Z2).

Suppose now that $\tx \nsim \tw''$ and $\tx' \nsim \tw$. Note that by
(T$_i$) applied to $\tw'$ and since $(\tw,z) \in Z$, all vertices in
$\{w,w',w'',x,x',z\}$ are distinct and $w' \sim w,w'',x,x'$; $x \sim
w$, and $x' \sim w''$. Moreover, $\tx \sim \tx'$ if and only if $x \sim
x'$.

%VC: change to the new (S_i)
We claim that we can assume that $\tx \sim \tx'$. Suppose that $\tx
\nsim \tx'$.  By (S$_i$), there exist $\ty\in \tS_{i-2}$ and $\ty' \in
\tS_{i-1}$ such that $\tx\tw'\tx'\ty$ is a square and $\ty' \sim
\ty,\tx,\tw',\tx'$.  By (T$_i$) applied to $\tw'$ and (R$_i$) applied
to $\ty'$, we have $y' \sim x, x',w'$ and $y \sim x, x',y'$. Note that
by Remark~\ref{rem-local-inj}, $y \neq z$ and $y' \neq z$.
Consequently, by (R$_i$) applied to $\tx$ and $\tx'$, all vertices in
$\{w,w',w'',x,x',y,y',z\}$ are distinct.
In $G$, consider a maximal clique containing the triangle $w'w''x'$, a
maximal clique containing the triangle $w'xy'$, and a maximal clique
containing the triangle $x'yy'$. Note that these three cliques
pairwise intersect in $w'$, $x'$, or $y'$.  Consequently, since $G$ is
finitely clique-Helly, there exists $x'' \sim x,w',w'',y$.
By (R$_i$) applied to $\tx$ and (T$_i$) applied to
$\tw'$, there exists $\tx'' \sim \tx,\tw',\tw'',\ty$. Since $\tx''
\sim \tw', \ty$, we have that $\tx'' \in \tS_{i-1}$ and we can replace
$\tx'$ by $\tx''$ and assume that $\tx \sim \tx'$.

Note that by (R$_i$) applied to $\tw'$, we have $x \sim x'$.  By
(Q$_i$), there exists $\ty \in \tS_{i-2}$ such that $\ty\sim
\tx,\tx'$. By Remark~\ref{rem-local-inj}, $y \neq z$ and by (R$_i$)
applied to $\tx$ and $\tx'$, all vertices in $\{w,w',w'',x,x',y,z\}$ are
distinct and $y \sim x,x'$.
In $G$, consider a maximal clique containing the triangle $ww'x$, a
maximal clique containing the triangle $w'w''x'$, and a maximal clique
containing the triangle $xx'y$. Note that these three cliques pairwise
intersect in $w'$, $x$, or $x'$. Since $G$ is
finitely clique-Helly, there exists $u \sim w,w',w'',x,x',y$.
By (R$_i$) applied to $\ty$, $\tx$, and $\tx'$, there exists $\tu \sim
\tw,\tw',\tw'',\ty$ such that $f(\tu) = u$. Since $\tu \sim \ty,
\tw'$, necessarily $\tu \in \tS_{i-1}$. By Remark~\ref{rem-local-inj},
$u \neq z$ and $u \nsim z$. Consequently, $uwzw''$ is a square in $G$
and $w' \sim u,w,z,w''$. Therefore, by (Z2), $(\tw,z) \equiv
(\tw'',z)$, and we are done.

%% \medskip
%% \noindent\textbf{Case 2.} $\tw' \sim \tw''$ and there exist $\tu' \in
%% \tS_i$, $\tu \in \tS_{i-1}$ such that $\tu \sim \tw, \tw'$; $\tu' \sim
%% \tw, \tw', \tu$ and $uwzw'$ is a square in $G$ with $u' \sim
%% u,w,z,w'$.
%% \smallskip

\begin{case-He}
  $\tw' \sim \tw''$ and there exist $\tu' \in
  \tS_i$, $\tu \in \tS_{i-1}$ such that $\tu \sim \tw, \tw'$; $\tu' \sim
  \tw, \tw', \tu$ and $uwzw'$ is a square in $G$ with $u' \sim
  u,w,z,w'$.
\end{case-He}

Since $\tw'' \nsim \tw$, we have $\tu' \neq \tw''$, and by (T$_i)$
applied to $\tu'$ and $\tw'$, all vertices in $\{w,w',w'',u,u',z\}$
are distinct.
In $G$, consider a maximal clique containing the triangle $u'wz$, a
maximal clique containing the triangle $w'w''z$, and a maximal clique
containing the triangle $uu'w'$. These three cliques pairwise
intersect in $u'$, $w'$, or $z$.
Since $G$ is finitely clique-Helly, there exists $y \sim z,w,w', w'', u,u'$. By
(R$_i$) applied to $\tu$, and (T$_i$) applied to $\tw'$, there exists
$\ty \in \tB_i$ such that $\ty \sim \tu,\tu',\tw,\tw',\tw''$. Since $y
\sim w,z$, we know by Remark~\ref{rem-local-inj} that $\ty \in \tS_i$
and $(y,z) \in Z$. Consequently, we can replace $\tw'$ by $\ty$ and we
are in Case 1.

\begin{case-He}
  There exist $\tu_1' \in \tS_i$, $\tu_1 \in \tS_{i-1}$ such that
  $\tu_1 \sim \tw, \tw'$, $\tu_1' \sim \tw, \tw', \tu_1$ and $u_1wzw'$
  is a square in $G$ with $u_1' \sim u_1,w,z,w'$, and there exist
  $\tu_2' \in \tS_i$, $\tu_2 \in \tS_{i-1}$ such that $\tu_2 \sim
  \tw', \tw''$, $\tu_2' \sim \tw', \tw'', \tu_2$ and $u_2w'zw''$ is a
  square in $G$ with $u_2' \sim u_2,w',z,w''$.
\end{case-He}

%% \medskip
%% \noindent\textbf{Case 3.} There exist $\tu_1' \in \tS_i$, $\tu_1 \in
%% \tS_{i-1}$ such that $\tu_1 \sim \tw, \tw'$, $\tu_1' \sim \tw, \tw',
%% \tu_1$ and $u_1wzw'$ is a square in $G$ with $u_1' \sim u_1,w,z,w'$,
%% and there exist $\tu_2' \in \tS_i$, $\tu_2 \in \tS_{i-1}$ such that
%% $\tu_2 \sim \tw', \tw''$, $\tu_2' \sim \tw', \tw'', \tu_2$ and $u_2w'zw''$
%% is a square in $G$ with $u_2' \sim u_2,w',z,w''$.
%% \smallskip

If $\tw \sim \tu'_2$ (respectively, $\tw'' \sim \tu_1'$), then we can replace
$\tw'$ by $\tu_2'$ (respectively,  $\tu_1'$) and we are in Case~1. Applying
Case~2 where we replace $\tw''$ by $\tu_2'$, we know that there exist
$\tu_3' \in \tS_i$ and $\tu_3 \in \tS_{i-1}$ such that $\tu_3 \sim \tw,
\tu_2'$; $\tu_3' \sim \tw, \tu_2', \tu_3$ and $u_3wzu_2'$ is a square
in $G$ with $u_3' \sim u_3,w,z,u_2'$. Consequently, we can replace
$\tw'$ by $\tu_2'$ and we are in Case~2.
\end{proof}

Let $\widetilde{S}_{i+1}$ denote the set of equivalence classes of
$\equiv$, i.e., $\widetilde{S}_{i+1}=Z/_{\equiv}$. For an ordered pair
$(\widetilde{w},z)\in Z$, we will denote by $[\widetilde{w},z]$ the
equivalence class of $\equiv$ containing $(\widetilde{w},z)$. Set
$\widetilde{B}_{i+1}:=\widetilde{B}_i\cup \widetilde{S}_{i+1}$. Let
$\widetilde{G}_{i+1}$ be the graph having $\widetilde{B}_{i+1}$ as the
vertex set, in which two vertices $\widetilde{a},\widetilde{b}$ are
adjacent if and only if one of the following conditions holds:
\begin{itemize}
\item[(1)] $\widetilde{a},\widetilde{b}\in \widetilde{B}_i$ and
$\widetilde{a}\widetilde{b}$ is an edge of $\widetilde{G}_i$,
\item[(2)] $\widetilde{a}\in \widetilde{B}_i$,  $\widetilde{b}\in
\widetilde{S}_{i+1}$ and $\widetilde{b}=[\widetilde{a},z]$,
\item[(3)] $\widetilde{a},\widetilde{b}\in \widetilde{S}_{i+1},$
$\widetilde{a}=[\widetilde{w},z]$,
$\widetilde{b}=[\widetilde{w},z']$ for a vertex $\widetilde{w}\in \tB_i,$ and
$z\sim z'$ in the graph $G$.
\end{itemize}

Let $\tX_{i+1}=X\tr(\tG_{i+1})$ be the triangle complex of
$\tG_{i+1}$. Finally, we define the map $f_{i+1}:
\widetilde{B}_{i+1}\rightarrow V({X})$ in the following way: if
$\widetilde{a}\in \widetilde{B}_i$, then
$f_{i+1}(\widetilde{a})=f_i(\widetilde{a}),$ otherwise, if
$\widetilde{a}\in \widetilde{S}_{i+1}$ and
$\widetilde{a}=[\widetilde{w},z],$ then
$f_{i+1}(\widetilde{a})=z$. As in the proof of Theorem~\ref{t:lotoglo2},
all vertices of $\widetilde{B}_{i+1}$ will be denoted with a tilde and
their images in $G$ under $f_{i+1}$ will be denoted without tilde.
%e.g. if $\widetilde{w}\in \widetilde{B}_{i+1},$ then
%$f_{i+1}(\widetilde{w})=w$.
\medskip

Now we check our inductive assumptions, verifying the properties
(P$_{i+1}$),(Q$_{i+1}$),(R$_{i+1}$), (S$_{i+1}$), and (T$_{i+1}$) for
$\tG_{i+1}$ and $f_{i+1}$ defined above.
\medskip

The following four lemmata together with their proofs are the same as,
respectively, Lemmata 5.7, 5.8, 5.9, and 5.10 in \cite{BCC+}, and
Lemmata~\ref{Pi+1}, \ref{Qi+1}, \ref{lem-homomorphism}, and
\ref{lem-locally-surjective}.

\begin{lemma} \label{Pi+1-clH}  $\tG_{i+1}$ satisfies the property $(P_{i+1})$,
i.e.,
$B_j(\tv,\tG_{i+1})=\widetilde{B}_j$ for any $j\le i+1.$
\end{lemma}

\begin{lemma} \label{Qi+1-clH}  $\widetilde{G}_{i+1}$ satisfies the property
$(Q_{i+1}),$ i.e., the graph $\widetilde{G}_{i+1}$
is weakly modular with respect to the base-point $\widetilde v$.
\end{lemma}

\begin{lemma}\label{lem-homomorphism-clH}
For any edge $\ta\tb$ of $\tG_{i+1}$, $ab$ is an edge of $G$ (in
  particular $a \neq b$).
\end{lemma}

\begin{lemma}\label{lem-locally-surjective-clH}
If $\ta \in \tB_i$ and if $b \sim a$ in $G$, then there exists a
vertex $\tb$ of $\tG_{i+1}$ adjacent to $\ta$ such that $f_{i+1}(\tb)
= b$.
\end{lemma}

We prove now that $f_{i+1}$ is locally injective. The proof is almost
the same as the proof of Lemma \ref{lem-locally-injective}.

\begin{lemma}\label{lem-locally-injective-clH}
If $\ta \in \tB_{i+1}$ and  $\tb, \tc$ are distinct neighbors of
$\ta$ in $\tG_{i+1}$, then $b \neq c$.
\end{lemma}

\begin{proof}
  First, note that if $\tb \sim \tc$ then the assertion holds by
  Lemma~\ref{lem-homomorphism-clH}; thus further we assume that $\tb \nsim
  \tc$. If $\ta, \tb, \tc \in \tB_i$, the lemma holds by (R$_i$) or
  (T$_i$) applied to $\ta$.  Suppose first that $\ta \in \tB_i$. If
  $\tb, \tc \in \tS_{i+1}$, then $\tb = [\ta,b]$ and $\tc = [\ta,c],$
  and thus $b\neq c$. If $\tb \in \tB_i$ and $\tc = [\ta,c] \in
  \tS_{i+1}$, then $(\ta,b) \notin Z$, and thus $c \neq b$. Therefore,
  further we consider $\ta\in \tS_{i+1}$.

  If $\tb, \tc \in \tB_i$ then $\ta = [\tb,a] = [\tc,a]$. Since
  $(\tb,a) \equiv (\tc,a)$ and since $\tb \nsim \tc$, there exists
  $\tu \in \tB_{i-1}$ such that $\tu \sim \tb, \tc$ and $abuc$ is an
  induced square of $G$. This implies that $b \neq c$.
%% JC: the only difference with the previous proof is in the
%% following paragraph.

  If $\ta, \tb, \tc \in \tS_{i+1}$, then there exist $\tw, \tw' \in
  \tB_i$ such that $\tb = [\tw,b]$, $\tc =[\tw',c],$ and $\ta= [\tw,a]
  = [\tw',a]$. Suppose by way of contradiction that $b=c$.  Note that
  this implies that $\tw \neq \tw'$. If $\tw \sim \tw'$ then
  $\tb=\tc$, a contradiction; consequently, $\tw \nsim \tw'$.  Since
  $(\tw,a) \equiv (\tw',a)$ there exists $\tu\in \tB_{i-1}$ and $\tu'
  \in \tS_i$ such that $\tu \sim \tw,\tw'$, $\tu'\sim \tu,\tw, \tw'$,
  $uwaw'$ is a square in $G$ and $u'\sim u,w,a,w'$.  If $uwbw'$ is not
  a square then $b\sim u$ and, by (R$_i$) applied to $\tu$, we have
  that $(\tw,b)\notin Z$, a contradiction.  Thus $uwbw'$ is a square.
  Consider a maximal clique containing the triangle $awb$, a maximal
  clique containing the triangle $au'w'$, and a maximal clique
  containing the triangle $uu'w$. These three cliques
  pairwise intersect in $a$, $w$, or $u'$. Since $G$ is finitely clique-Helly,
  there exists $u'' \sim b, u, w, w'$.
  By (R$_i$) applied to $\tu \in \tB_{i-1}$, there exists $\tu''$ such
  that $\tu'' \sim \tw, \tw', \tu$ and $f_{i+1}(\tu'') = u''$. By
  Remark~\ref{rem-local-inj}, $\tu''\in \tS_i$ and thus, replacing
  $\tw$ by $\tu''$, we get that $\tb=[\tw,b]=[\tw',b]=\tc$, a
  contradiction.

  If $\ta, \tb \in \tS_{i+1}$ and $\tc \in \tS_i$, then there exists
  $\tw \in \tS_i$ such that $\tb = [\tw,b]$ and $\ta = [\tw,a] =
  [\tc,a]$. If $\tw \sim \tc$, then $(\tw,c) \notin Z$, and thus
  $(\tw,c) \neq (\tw,b)$, i.e., $b \neq c$. If $\tw \nsim \tc$, since
  $[\tw,a] = [\tc,a]$, there exists $\tu \in \tS_{i-1}$ such that
  $\tu\sim \tw,\tc$ and such that $acuw$ is an induced square of
  $G$. Consequently, the vertices $w$ and $c$ are not adjacent; since
  $w \sim b$, this implies that $b \neq c$.
\end{proof}

The following lemma is the counterpart of
Lemma~\ref{lem-triangles}.

\begin{lemma}\label{lem-triangles-clH}
If $\ta \sim \tb, \tc$ in $\tG_{i+1}$, then $\tb \sim \tc$ if and only
if $b \sim c$.
\end{lemma}

\begin{proof}
If $\tb \sim \tc$, then $b \sim c$ by Lemma~\ref{lem-homomorphism-clH}.
Conversely, suppose that $b \sim c$ in $G$.  If $\ta, \tb, \tc \in
\tB_i$ then $\tb \sim \tc$ by conditions (R$_i$) and (T$_i$).
Therefore, further we assume that at least one of the vertices $\ta,
\tb, \tc$ does not belong to $\tB_i$.

First, suppose that $\ta \in \tB_i$. If $\tb, \tc \in \tS_{i+1}$ then
$\tb = [\ta,b]$ and $\tc = [\ta,c]$. Since $b\sim c$, by the
construction of $\tG_{i+1}$, we have $\tb \sim \tc$ in
$\tG_{i+1}$. Suppose now that $\tb = [\ta,b] \in S_{i+1}$ and $\tc \in
\tB_i$. If there exists $\tb' \sim \tc$ in $\tG_i$ such that
$f_i(\tb') = b$ then, by (T$_i$) applied to $\tc$, we have $\ta \sim
\tb'$ and $(\ta,b) \notin Z$, which is a contradiction. Thus $(\tc,b)
\in Z$ and, since $\tc \sim \ta$, we have $[\tc,b]=[\ta,b]=\tb$, and
consequently, $\tc \sim \tb$. Therefore, further we consider $\ta \in
\tS_{i+1}$.

If $\tb, \tc \in \tB_i$ and $\ta \in \tS_{i+1}$, then $\ta =[\tb,a] =
[\tc,a]$ and either $\tb \sim \tc$, or there exists $\tu \in
\tS_{i-1}$ such that $\tu \sim \tb, \tc$ and $ubac$ is an induced
square in $G$, which is impossible because $b \sim c$.

%%JC: I just modified the following paragraph, but we can also use
%%this change in the previous proof (to have a unique proof}.
If $\ta, \tb \in \tS_{i+1}$ and $\tc \in \tB_i$, then by
Lemma~\ref{lem-locally-surjective-clH}, there exists $\tb' \sim \tc$
such that $f_{i+1}(\tb') = b$. By a previous case applied to $\tc,
\ta, \tb'$, we have that $\ta \sim \tb'$. By Lemma,
~\ref{lem-locally-injective-clH} applied to $\ta$, we obtain $\tb = \tb'$ and we
are done.

%%JC: here, there are major differences with the loc. weakly modular
%%case.
If $\ta, \tb, \tc \in \tS_{i+1}$ then there exist $\tw, \tw' \in
\tB_i$ such that $\tb = [\tw,b]$, $\tc = [\tw',c]$, and $\ta = [\tw,a]
= [\tw',a]$. If $\tw \sim \tc$ (respectively, $\tw' \sim \tb$), then
we are in a previous case, replacing $\ta$ by $\tw$ (respectively,
$\tw'$) and consequently $\tb \sim \tc$. Suppose now that $\tw \nsim
\tc$ and $\tw' \nsim \tb$. From a previous case applied to $\ta,\tb
\in \tS_{i+1}$ (respectively, $\ta,\tc \in \tS_{i+1}$) and $\tw' \in
\tB_i$ (respectively, $\tw \in \tB_i$), it follows that $w' \nsim b$
(respectively $w \nsim c$).

We claim that we can assume that $\tw \sim \tw'$. Suppose that $\tw
\nsim \tw'$. Since $(\tw,a) = (\tw',a)$, there exists $\tu \in
\tB_{i-1}$ and $\tu' \in \tS_i$ such that $\tu \sim \tw, \tw'$, $\tu'
\sim \tu, \tw, \tw'$, $uwaw'$ is an induced square in $G$ and $u' \sim
u,w,a,w'$.  Since $c \nsim w$ and $b \nsim w'$, $u, u' \notin
\{b,c\}$.
  Consider a maximal clique containing the triangle $au'w$, a maximal
  clique containing the triangle $acw'$, and a maximal clique
  containing the triangle $uu'w'$. These three cliques
  pairwise intersect in $a$, $u'$, or $w'$. Since $G$ is finitely clique-Helly,
  there exists $u'' \sim a, c, w, u$.
By (R$_i$) applied to $\tu$, there exists $\tu'' \sim \tu, \tw$. By a
previous case applied to $\tw, \tu'' \in \tB_i$ and $\ta \in
\tS_{i+1}$, we have $\tu''\sim \ta$. By a previous case applied to
$\tu'' \in \tB_i$ and $\ta, \tc \in \tS_{i+1}$, we have $\tu''\sim
\tc$. Consequently, we can replace $\tw'$ by $\tu''$ and assume that
$\tw \sim \tw'$.

By the triangle condition TC($\tv$), there exist $\tu \in \tB_{i-1}$
adjacent to both $\tw,\tw'$. By Lemma~\ref{lem-locally-injective-clH},
$u \notin \{a,b,c\}$ and by Lemma~\ref{lem-homomorphism-clH}, $u \sim
w,w'$.
Consider a maximal clique containing the triangle $abw$, a maximal
clique containing the triangle $acw'$, and a maximal clique containing
the triangle $uww'$. Note that these three cliques pairwise intersect
in $a$, $w$, or $w'$. Since $G$ is finitely clique-Helly, there exists $u''
\sim a, b, c, w, w', u$.
By (R$_i$) applied to $\tu$, there exists $\tu'' \sim \tu, \tw,
\tw'$. Since $\tu \in \tB_{i-1}$ and from Remark~\ref{rem-local-inj},
$\tu'' \in \tS_i$.  By a previous case applied to $\tw, \tu'' \in
\tB_i$ (respectively, $\tw', \tu''$) and $\tb \in \tS_{i+1}$
(respectively, $\tc$), we obtain $\tu'' \sim \tb$ (respectively,
$\tu''\sim \tc$). Consequently, $\tb = [\tu'',b]$ and $\tc =
   [\tu'',c]$ and thus $\tb \sim \tc$.
\end{proof}

From the previous lemma, the image under $f_{i+1}$ of a triangle is a
triangle.  This allows us to extend the map $f_{i+1}$ to a simplicial
map $f_{i+1}\colon \tX_{i+1} \to X$.

\begin{lemma} \label{cellular-clH}
If $\widetilde{a}\widetilde{b}\widetilde{c}$ is a triangle in
$\tG_{i+1}$, then $abc$ is a triangle in $G$.
\end{lemma}

From Lemmata~\ref{lem-homomorphism-clH},
\ref{lem-locally-surjective-clH}, \ref{lem-locally-injective-clH}, and
\ref{lem-triangles-clH}, we immediatly get the following lemma.

\begin{lemma}\label{Ri+1-clH}\label{Ti+1-clH}
$f_{i+1}$ satisfies the conditions $(R_{i+1})$ and $(T_{i+1}).$
\end{lemma}

We now show that condition (S$_{i+1}$) also holds.
\begin{lemma} \label{Si+1-clH} For any $j \leq i $ and for any $\tw,\tw'\in \tS_{i}$
which are not adjacent and have a common neighbor  $\tz$ in $\tS_{i+1}$, there exist
$\tu\in \tS_{i-1}$ and $\tu'\in \tS_{i}$ such that $\tz\tw\tu\tw'$ is
  a square and $\tu' \sim \tz, \tw, \tu, \tw'$.
\end{lemma}

\begin{proof}
Since $\tz = [\tw,z] = [\tw',z]$, and since $\tw
\nsim \tw'$, by the definition of $\equiv$, there exist $\tu\in
\tS_{i-1}$ and $\tu'\in \tS_{i}$ such that $\tz\tw\tu\tw'$ is a square
and $\tu' \sim \tz, \tw, \tu, \tw'$.
\end{proof}

%% We now show that condition (U$_{i+1}$) also holds.
%% \begin{lemma}
%% The simplicial complex $\tX_{i+1}$ is simply connected.
%% \end{lemma}

%% \begin{proof} \label{Ui+1-clH}
%%   Consider any cycle $c$ in $\tX_{i+1}$.  If all vertices belong to
%%   $\tB_i$, the result holds by (U$_i$).

%% \end{proof}

%% \noindent
%% {\bf
\subsection*{The universal cover $\tX$}  Let $\widetilde{X}_v$ denote the
triangle complex obtained as the directed union $\bigcup_{i\ge 0}
\widetilde{X}_i$, with a vertex $v$ of $X$ as the base-point. Denote
by $\tG_v$ the $1$--skeleton of $\widetilde{X}_v.$ Since each $\tG_i$
is weakly modular with respect to $\tv$, the graph $\tG_v$ is also
weakly modular with respect to $\tv$. Let $f=\bigcup_{i\ge 0}f_i$ be
the map from $\widetilde{X}_v$ to $X$.

In the next lemma, we show that the triangle complex of $\tG_v$ is
simply connected. The proof is similar to the proof of
Lemma~\ref{simplyconnected} (Lemma 5.5 in \cite{BCC+}).

\begin{lemma} \label{simplyconnected-clH}
The simplicial complex $\tX_v$ is simply connected.
\end{lemma}

\begin{proof}
 By contradiction, let $A$ be the set of cycles in $\tG_v$, which are
 not freely homotopic to $\tv$, and assume that $A$ is non-empty.  For
 a cycle $C\in A$, let $r(C)$ denote the maximal distance $d(\tw,\tv)$
 of a vertex $\tw\in C$ to the basepoint $\tv$.  Clearly $r(C)\geq 2$
 for any cycle $C\in A$ (otherwise $C$ would be null-homotopic). Let
 $B\subseteq A$ be the set of cycles $C$ with minimal $r(C)$ among
 cycles in $A$. Let $r:= r(C)$ for some $C\in B$. Let $D\subseteq B$
 be the set of cycles having minimal number $e$ of edges in $\tS_r$,
 i.e., with both endpoints at distance $r$ from $\tv$. Further, let
 $E\subseteq D$ be the set of cycles with the minimal number $m$ of
 vertices at distance $r$ from $\tv$.

 Consider a cycle $C=(\tw_1,\tw_2,...,\tw_k,\tw_1)\in E$.  We can
 assume without loss of generality that $d(\tw_2,\tv)=r$.  We
 distinguish two cases, depending on whether $\tw_2$ has a neighbor at
 distance $r$ from $\tv$ in $C$ or not.

 %% \medskip\noindent{\textbf{Case 1.}} $d(\tw_1,\tv)=r$ or
 %%   $d(\tw_3,\tv)=r$.
 
 \begin{case-sc}
   $d(\tw_1,\tv)=r$ or $d(\tw_3,\tv)=r$.
 \end{case-sc}

  Assume without loss of generality that $d(\tw_1,\tv)=r$. Then, by
  (Q$_r$), there exists a vertex $\tw\sim \tw_1,\tw_2$ with
  $d(\tw,\tv)=r-1$.  Observe that the cycle
  $C'=(\tw_1,\tw,\tw_2,\ldots,\tw_k,\tw_1)$ belongs to $B$ -- in
  ${\tX_v}$ it is freely homotopic to $C$ by a homotopy going through
  the triangle $\tw\tw_1\tw_2$. The number of edges of $C'$ lying on
  the $r$-sphere around $\tv$ is less than $e$ (we removed the edge
  $\tw_1\tw_2$). This contradicts the choice of the number $e$.

  %%   \medskip\noindent{\textbf{Case 2.}} $d(\tw_1,\tv)=d(\tw_3,\tv)=r-1$.
  \begin{case-sc}
    $d(\tw_1,\tv)=d(\tw_3,\tv)=r-1$.
  \end{case-sc}

  By (S$_r$), there exists a vertex $\tw\sim \tw_1,\tw_2,\tw_3$ with
  $d(\tw,\tv)=r-1$.  Again, the cycle
  $C'=(\tw_1,\tw,\tw_3,...,\tw_k,\tw_1)$ is freely homotopic to $C$
  (via the triangles $\tw_1\tw\tw_2$ and $\tw_2\tw\tw_3$). Thus $C'$
  belongs to $D$ and the number of its vertices at distance $r$ from
  $\tv$ is equal to $m-1$. This contradicts the choice of the number
  $m$.
\medskip

In both cases above we get a contradiction. It follows that the set $A$
is empty and hence the lemma is proved.
\end{proof}

\begin{lemma} \label{covering_map-clH}
For any $\tw \in \widetilde{X}_v$, the map
$f|_{{\rm St}(\tw,\widetilde{X}_v)}$ is an isomorphism onto
${\rm St}(w,X)$, where $w=f(\tw)$. Consequently,
$f:~\!\!\widetilde{{X}}_v\rightarrow~\!\!X$ is a covering
map.
\end{lemma}

\begin{proof}
  Consider a vertex $\tw$ of $\tX_v$. Let $i$ be the distance between
  $\tv$ and $\tw$ in $\tG_v$ and consider the set $\tB_{i+2}$.  Then,
  from ($R_{i+2}$) we know that $f$ is an isomorphism between the
  graphs induced by $B_1(\tw,\tG_v)$ and $B_1(w,G)$, and consequently
  since $\tX_v$ and $X$ are flag complexes,
  $f:~\!\!\widetilde{{X}}_v\rightarrow~\!\!X$ is a covering map.
\end{proof}

Consequently, since $\tX_v$ is simply connected, it is the universal
covering space of $X$. It is then unique, i.e., not depending on $v$
and thus $\tG(=\tG_{v})$ is weakly modular with respect to every
vertex $\tv$ of $\tG$.

\begin{lemma}\label{lem-tG-clH}
  The graph $\tG$ is (finitely) clique-Helly and it satisfies the
  $(C_4,W_4)$-condition.
  %% Moreover, if $G$ is clique-Helly, then $\tG$ is clique-Helly.
\end{lemma}

\begin{proof}
  We first show that $\tG$ is (finitely) clique-Helly.  Consider any
  (finite) set $S$ of pairwise intersecting maximal cliques of $\tG$.
  Note that since $f$ is a covering map, for each $K \in S$, $f(K)$ is
  a maximal clique in $G$. Let $K_0 \in S$ and for each $K \in S$, let
  $\tz(K)$ be any vertex in $K_0 \cap K$. Note that for any $K \in S$,
  $\tz(K_0) \sim \tz(K)$ and $f(\tz(K_0)) \sim f(\tz(K))$.

  Since $G$ is (finitely) clique Helly, there exists $y \in \bigcap_{K
    \in S} f(K)$.  In particular, for each $K \in S$, $f(\tz(K)) \sim
  y$.  Since $f$ is a covering map from $\tX$ to $X$,
  $f(N_{\tG}(\tz(K_0)))$ is isomorphic to $N_G(f(\tz(K_0)))$ and thus
  there exists $\ty$ in $\tG$ such that $\ty \sim \tz(K)$ for all $K
  \in S$. Moreover, for each $K \in S$, since $y \in f(K)$ and since
  $f(N_{\tG}(\tz(K)))$ is isomorphic to $N_G(f(\tz(K)))$, necessarily
  $\ty \in K$.

  \medskip

  We now show that $\tG$ satisfies the $(C_4,W_4)$-condition. For any
  square $\ta\tb\tc\td$ of $\tG$, let $a = f(\ta), b = f(\tb), c =
  f(\tc)$, and $d = f(\td)$. Since $f$ is a covering map, $abcd$ is a
  square in $G$. In $\tX = \tX_{a}$, note that $\tc \in \tS_2$ and
  that $\tb, \td \in \tS_1$. Consequently, by (S$_2$), there exists
  $\te \in \tS_1$ and $\ta' \in \tS_0$ such that $\te \sim
  \ta',\tb,\tc,\td$. Since $\tS_0=\{\ta\}$, we have $\ta = \ta'$ and
  we are done.
\end{proof}

Summarizing, $\tX$ is simply connected, $\tG = \tX^{(1)}$ is weakly
modular, (finitely) clique-Helly, and it satisfies the
$(C_4,W_4)$-condition. This ends the proof of
Proposition~\ref{t:universal-cover-Helly}.
\end{proof}

We can now prove Theorem~\ref{t:lotogloHell_bis}. Given a (finitely)
clique-Helly graph $G$, by Proposition~\ref{t:universal-cover-Helly},
the $1$-skeleton $\tG$ of the universal cover of the triangle complex
of $G$ is weakly modular graph, (finitely) clique-Helly, and satisfies
the $(C_4,W_4)$-condition. Therefore, by
Proposition~\ref{prop-1-Helly-clique-Helly}, $\tG$ is (finitely)
$1$--Helly. Consequently, by Proposition~\ref{prop-1-Helly-Helly}, $G$
is (finitely) Helly. This finishes  the proof of Theorem
\ref{t:lotogloHell_bis}.

We conclude this section with the proof of Theorem
\ref{t:lotogloHell}.
%% (i) => (ii)
The finite Helly property easily implies that
Helly graphs are weakly modular, thus (i)$\Rightarrow$(ii).
%% (ii) => (iii)
If $G$ is a weakly modular $1$--Helly graph, then the fact that $G$ is
dismantlable follows from Proposition~\ref{prop-clh-domination}.  By
this proposition, any vertex $v$ at distance $k+1$ from a basepoint
$u$ is dominated in the ball $B_{k+1}(u)$ by a neighbor $x$ with
$d(u,x)=k$. This implies that any breadth-first-search order of the
vertices of $G$ starting with $u$ provides a dominating order of $G$
(By~\cite{Po-infbrid}, any graph admits a BFS ordering that is a
well-order). Since any $1$--Helly graph trivially is clique-Helly,
this establishes (ii)$\Rightarrow$ (iii) (notice that the equivalence
of (ii) and (iii) for finite graphs was proved in \cite{BP-absolute}).
%%(iii) => (iv)
The proof of the implication (iii)$\Rightarrow$(iv) is similar to the
proof of Lemma~\ref{simplyconnected-clH}. Let $\prec$ be a dismantling
order of $V(G)$ and let $v_0$ be the least element of this well-order.
We will show that any cycle is freely homotopic to $v_0$. Suppose this
is not the case. Among all cycles that are not freely homotopic to
$v_0$, consider a cycle $C=(w_1,w_2,\ldots,w_k)$ that minimizes
$\alpha(C) = \max_\prec \{w_i : w_i \in V(C)\}$. Note that any cycle
has a finite number of vertices and thus $\alpha(C)$ is
well-defined. Moreover, since $\prec$ is a well-order, there exists a
cycle $C$ that minimizes $\alpha(C)$. Among all such cycles $C$,
consider a shorter cycle $C$. If $w_i = w_j$ for some $1 \leq i < j
\leq k$, either $C_1=(w_i, \ldots, w_{j-1})$ or $C_2=(w_j,w_{j+1},
\ldots, w_k, w_1, \ldots, w_{i-1})$ is non contractible and since
$\alpha(C_1) \preceq \alpha(C)$ and $\alpha(C_2) \preceq \alpha(C)$,
it contradicts our choice of $C$; consequently, we can assume that $C$
is a simple cycle. Since $C$ is not freely homotopic to $v_0$,
necessarily $k \geq 3$ and $v_0 \prec \alpha(C)$. Without loss of
generality, assume that $\alpha(C) = w_2$. Note that there exists
$w_2' \prec w_2$ such that $w'_2 \in B_1(w_1)\cap B_1(w_2) \cap
B_1(w_3)$ and consider the cycle $C'=(w_1,w_2',w_3,\ldots,w_k)$. Note
that $C$ is freely homotopic to $C'$ via the triangles $w_1w_2w_2'$
and $w_2w_2'w_3$. Moreover, since $\alpha(C') \prec \alpha(C)$, from
our choice of $C$, $C'$ is freely homotopic to $v_0$. Consequently,
$C$ is freely homotopic to $v_0$, a contradiction. This establishes
that (iii)$\Rightarrow$(iv).  The implication (iv)$\Rightarrow$(i) is
the content of Theorem \ref{t:lotogloHell_bis}. Notice also that the
implication (v)$\Rightarrow$(iv) holds for all graphs

%%(ii) => (v)
To conclude the proof, let $G$ be a graph with a finite-dimensional
clique complex $X(G)$. We will
show that (ii)$\Rightarrow$(v) holds in this case. Let $v_0$ be a basepoint of
$G$. For any integer $i\ge 0$ let $X_i$ denote the subcomplex of
$X(G)$ spanned by the vertices of $G$ in the ball $B_i(v_0)$. We
define a map $f_i: X_{i+1}\rightarrow X_i$ as follows.  First, we set
$f_i(v):=v$ for any vertex $v\in B_i(v_0)$. By Proposition
\ref{prop-clh-domination}, for any vertex $v\in B_{i+1}(v_0)\setminus
B_i(v_0)$ there exists a vertex $x\in B_1(v)\cap B_i(v_0)$ such that
$B_1(v)\cap B_{i+1}(v_0)\subseteq B_1(x)$.  We set $f_i(v):=x$.  We
assert that $f_i$ can be extended to a simplicial retraction map from
$X_{i+1}$ to $X_i$. Let $\sigma$ be a simplex of $X_{i+1}$. By the
definition of $f_i(v)$ for $v\in \sigma$ we conclude that
$\bigcup_{v\in \sigma} f_i(v)$ is a simplex $\sigma'$ of $X_i$ (in fact, $\sigma\cup \sigma'$
is a simplex of $X_{i+1}$), thus
we can extend $f_i$ by setting $f_i(\sigma)=\sigma'$. Since $X(G)$ is
the directed union $\bigcup_{i\geq 0} X_i$, $X(G)$ is contractible by
Whitehead's theorem.  This finishes the proof of
Theorem~\ref{t:lotogloHell}.
%%% Quesiton for Damian: Do we use Whitehead theorem in a correct way ????

\begin{Rem} Theorem \ref{t:lotogloHell} implies the Proposition \ref{Helly_Polat} of Polat and Pouzet \cite{Po_helly}. Indeed,
if $G$ is a finitely Helly graph not containing infinite cliques, then $G$ is weakly modular and satisfies the
$(C_4,W_4)$-condition (by applying the Helly property for triplets and quadruplets of balls). Therefore the clique complex of $G$ is simply connected.
Since $G$ is finitely Helly, obviously $G$ is  finitely clique-Helly. By compactness since all cliques of $G$ are finite, $G$ is clique-Helly, thus $G$ is Helly by (iv).
\end{Rem}

\section{A note on meshed graphs}
\label{s:locmesh}
The notion of a meshed graph (defined in Subsection~\ref{prel:other}) seems to be a natural
generalization of a weakly modular graph. In Chapter~\ref{s:metric} we show that meshed graphs have some features
typical for nonpositive curvature. However, in contrast with Theorem~\ref{t:lotoglo2}, there is no local-to-global characterization
for meshed graphs as we show below. This is a strong restriction for considering this natural class of graphs as further analogues for
nonpositive curvature. Some counterexamples for the local-to-global behavior are depicted in Figure~\ref{fig-lomeshnomesh}, and further examples (for
larger radii) can be constructed analogously.

\begin{figure}[ht]
\begin{center}
\scalebox{0.5}{\includegraphics{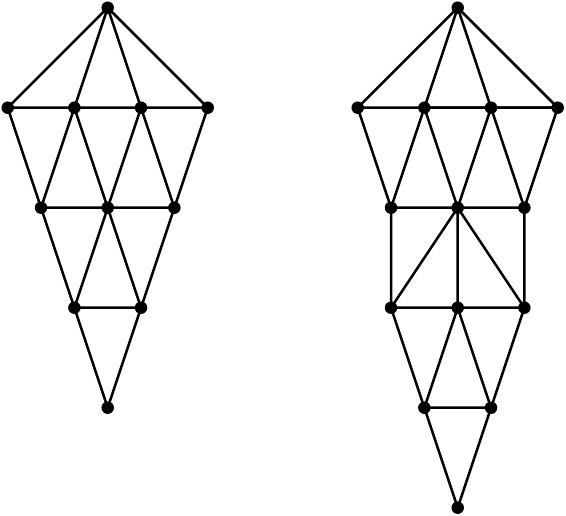}}
\end{center}
\caption{Balls of radii $1$ (on the left) or up to $2$ (on the right) in those graphs are meshed, but the whole graphs are not meshed.}\label{fig-lomeshnomesh}
\end{figure}

\chapter{Pre-Median Graphs}\label{sec:premedian}
In this chapter, we study the structure of pre-median graphs (pm-graphs) introduced and investigated
by Chastand \cites{Cha1,Cha2} (for definition, see
Subsection \ref{classes}). We characterize the prime pre-median graphs (ppm-graphs), thus answering a question by M. Chastand.
We also present several examples of thick pre-median graphs. Some of them are related to basis graphs of matroids,
even $\triangle$--matroids, and, more generally, to $l_1$--graphs and hypermetric graphs. We believe  that
thick pre-median (or convex) subgraphs of all pre-median graphs $G$ define  contractible cell complexes on $G$.
We confirm this  in the case
of $L_1$--embeddable weakly modular graphs (which are all pre-median): in this case, these contractible complexes
have cells derived from octahedral and matroidal thick subgraphs of $G$.

\section{Main results}
The following two results are immediate corollaries of Theorems~\ref{t:lotoglo2} and Lemma~\ref{l:noincov}.

\begin{theorem}
\label{t:lotoglo_prem}
A graph $G$ is pre-median if and only if $G$  is locally weakly modular, does not contain induced $K_{2,3}$ and $W^-_4$, and
its triangle-square complex is simply connected.
\end{theorem}

\begin{theorem}
\label{t:lotoglo_prem2}
Let $G$ be a locally weakly modular graph not containing induced $K_{2,3}$ and $W^-_4$, and let $\tG$ be the $1$--skeleton of the universal cover
$\widetilde{X}\trsq(G)$ of the triangle-square complex $X\trsq(G)$ of $G$. Then $\tG$ is pre-median.
\end{theorem}

\begin{figure}[ht]\label{W_4M_4}
\begin{center}
\includegraphics[scale=0.5]{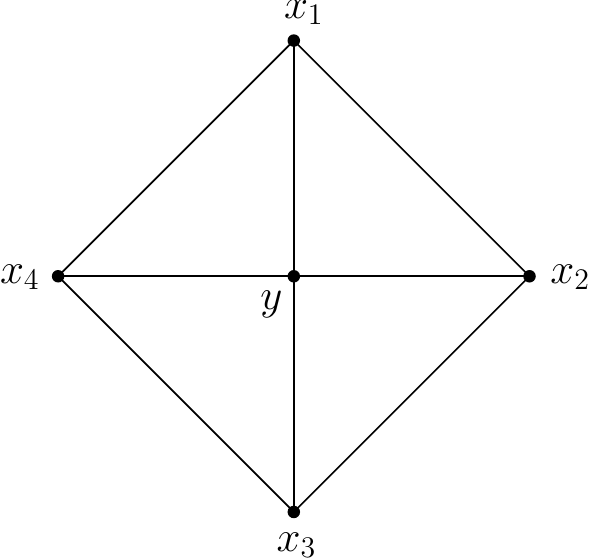}\qquad\qquad\includegraphics[scale=0.5]{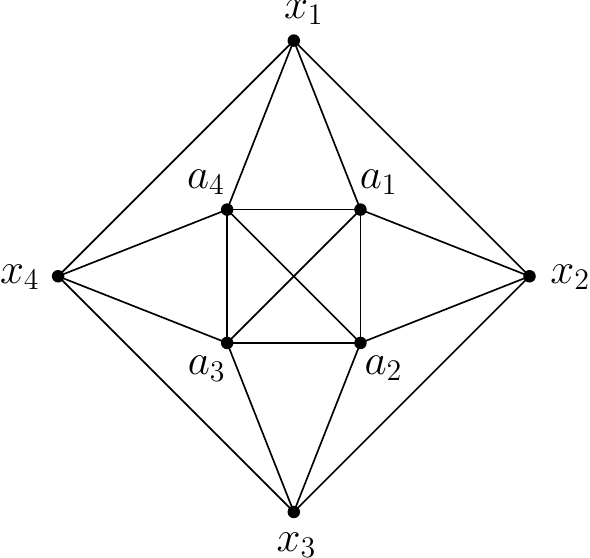}
\end{center}
\caption{A $W_4$ (left) and a $M_4$ (right)}\label{fig-M4-W4}
\end{figure}

We continue with the  characterization of ppm-graphs. Let $W_k^-=(c,x_1,\ldots,x_k)$ denote the almost $k$--wheel in which $C=(x_1,\ldots,x_k,x_1)$
is an induced $k$--cycle and $c$
is adjacent to all vertices of $C$ except $x_1$. By $M_4$ we denote the graph
consisting of an induced $4$--cycle $(x_1,x_2,x_3,x_4)$ and four
pairwise adjacent vertices $a_1,a_2,a_3,a_4$ such that $a_1\sim
x_1,x_2; a_2\sim x_2,x_3; a_3\sim x_3,x_4; a_4\sim x_4,x_1$ and
$a_1\nsim x_3,x_4; a_2\nsim x_1,x_4; a_3\nsim x_1,x_2; a_4\nsim
x_2,x_3$ (see Figure~\ref{fig-M4-W4}, right).

\begin{theorem} \label{prime_premedian} For a graph $G$, the following conditions are equivalent:
\begin{itemize}
\item[(i)] $G$ is a prime pre-median graph;
\item[(ii)] $G$ is a 2-connected pre-median graph and each square of $G$ is included in an induced $W_4$ or  $M_4$;
\item[(iii)]  $G$ is a 2-connected pre-median graph and its triangle complex $X\tr(G)$ (and hence its clique complex $X(G)$)
is simply connected;
\item[(iv)] $G$ is a 2-connected locally weakly modular graph not containing induced $K_{2,3},W^-_4$, and its triangle
complex $X\tr(G)$ is simply connected.
\end{itemize}
\end{theorem}

In the sequel, with some abuse of notation (and in view of Theorem
\ref{prime_premedian}(ii)), we will call a graph $G$ {\it prime
  pre-median} if $G$ is a pre-median graph in which each square is
included in an induced $W_4$ or $M_4$ (i.e., we do not require
2-connectedness). Recall that a graph $G$ satisfies the 2-interval condition
(IC$m$) if each 2-interval $I(u,v)$ is an induced subgraph of the
$m$--hyperoctahedron.
A {\it propeller} is the graph $K_5-K_3$, i.e.,
the graph consisting of three triangles glued along a common edge. Recall that a graph $G$ is thick if each $2$--interval of $G$ contains an induced square.
The following result
summarizes several new examples of
pm- and ppm-graphs (recall that thick graphs are 2-connected):

\begin{proposition} \label{premedian-examples}\ 
  
  \begin{itemize}
  \item[(1)] Hyperoctahedra, half-cubes, Johnson graphs, the Schl\"afli graph, and the Gosset graphs are
thick ppm-graphs.
Their retracts are (not necessarily 2-connected) ppm-graphs and the retracts of their weak Cartesian products are pm-graphs.

\item[(2)] Thick $L_1$--pm-graphs are weak Cartesian products of thick $L_1$--ppm-graphs. Finite thick $L_1$--ppm-graphs are either
subgraphs of hyperoctahedra or  basis graphs of even $\triangle$--matroids (and thus isometric subgraphs of half-cubes). The
latter graphs are exactly the finite ppm-graphs satisfying the condition (IC4), not containing propellers, and the links
of vertices of which do not contain induced $W_5$.

\item[(3)] Prime $L_1$--pm-graphs are 2-connected subhyperoctahedra and isometric ppm-graphs of half-cubes.
  \end{itemize}
\end{proposition}

Let now $G$ be an $L_1$--embeddable weakly modular graph admitting a scale embedding into a hypercube $H(X)$.
Since $K_{2,3}$ and $W^-_4$ are not $L_1$--graphs, $G$ is a pre-median graph.
Let $\varphi$ be a (scale) isometric embedding of $G$ into a hypercube $H(X)$ on a set $X$ and suppose that $X$ is countable.
Let ${\mathcal C}(G)$ be the set of all finite thick isometric  weakly modular subgraphs of $G$. By Proposition \ref{premedian-examples}(2), each subgraph $H$ of ${\mathcal C}(G)$ is
a weak Cartesian product of  basis graphs of even $\triangle$--matroids or thick subgraphs of hyperoctahedra. By this and the results of \cites{Che_bas,GeGoMcPhSe} (see Proposition \ref{basis-polyhedron} below),
$H$ is the 1-skeleton of a  polyhedron $[H]$ of ${\mathbb R}^X$, which is the convex hull of
the characteristic vectors $\varphi(v)$ of the vertices $v$ of $H$. Let $C(G)=\bigcup \{ [H]: H\in {\mathcal C}(G)\}$ be the subspace of ${\mathbb R}^X$, which is the union of all such
polyhedra (the formal definition of $C(G)$ is given in Section \ref{top-space}).
This construction of $C(G)$ generalizes known constructions of
several contractible nonpositively curved complexes.
Examples are median complexes for median graphs, weakly systolic complexes of weakly systolic graphs, and more generally, bucolic complexes for bucolic graphs.

\begin{theorem} \label{contractible} Let $G$ be an $L_1$--weakly modular graph admitting a scale embedding into a hypercube $H(X)$ with
countable $X$. Then $C(G)$ is a contractible subspace of ${\mathbb R}^X$ and the union of 1--skeleta  of cells of $C(G)$ coincides with $G$.
\end{theorem}

\section{Prime pre-median graphs}\label{ppm}
In this section, we prove Theorem \ref{prime_premedian}. The equivalence (iii)$\Leftrightarrow$(iv) follows from Theorem \ref{t:lotoglo_prem}.
First we prove that if the clique complex of a $2$--connected pre-median graph $G$ is simply connected, then $G$ is prime, thus
establishing the implication (iii)$\Rightarrow$(i).
Since a pre-median graph is prime if and only if it is elementary by Theorem~\ref{premedian_Chastand}(i), it suffices to show that $G$ is elementary. This is an immediate consequence
of the following more general result:

\begin{lemma}\label{lem-simpl-connected-elmentary} If $G$ is a $2$--connected graph whose clique complex is simply
connected, then $G$ is elementary.
\end{lemma}

\begin{proof} To prove this lemma, we will use minimal disk diagrams defined in Subsection \ref{s:bacom}. Let $G=(V,E)$ be a 2-connected graph
such that its clique complex $X=X(G)$ is simply connected. Suppose, by way of contradiction, that
$G$ is not elementary: then $G$ contains a proper gated subset $A$ containing at least two vertices. Let $u$ be a vertex of $A$
having a neighbor $w\in V\setminus A$. Since $A$ is connected and contains at least two vertices, $u$ is adjacent to a vertex $v$ of $A$.
Then $w\nsim v$, otherwise we obtain a contradiction with the assumption that $A$ is gated. Since $G$ is 2-connected, there exists a $(v,w)$--path
$P=(v=x_1,x_2\ldots,x_{k-1},x_k=w)$ of $G$ avoiding $u$. Let $C=(v,x_2\ldots,x_{k-1},w,u)$ be the simple cycle of $G$  obtained by concatenating $P$ with the 2-path
$(v,u,w)$. Let $(D,\varphi)$ be a minimal disk diagram for $C$ (such a diagram exists because the clique complex of $G$ is simply connected). Then $D$ is a 2-connected
plane triangulation  having the following properties  \cite{Ch_CAT}: (1) $\varphi$ bijectively maps
$\partial {D}$ to $C$ and (2) the image of a $2$--simplex of $D$ under
$\varphi$ is a $2$--simplex, and two adjacent $2$--simplices of $D$ have distinct images under $\varphi.$
For simplicity,
we will identify the vertices of $\partial D$ with the corresponding vertices of $C$. We will label the vertices of
$D$ with $a$ and $b$ in the following way: a vertex $x$ of $D$ gets label $a$ if and only if $\varphi(x)\in A$, otherwise, if $\varphi(x)\in V\setminus A,$ then the label
of $x$ is $b$. Consider the vertices of $D$ adjacent to $u$. Since $D$ is a planar 2-connected triangulation, these vertices induce in $D$ a
simple $(v,w)$--path $R$. Each vertex of $R$ is labeled $a$ or $b$. Since $v$ is labeled $a$ and $w$ is labeled $b$, necessarily $R$ contains
two adjacent vertices $v'$ and $w'$, such that $v'$ is labeled $a$ and $w'$ is labeled $b$. Hence $u=\varphi(u),\varphi(v')\in A$ and
$\varphi(w')\in V\setminus A$. By (2), the image of the triangle $uv'w'$ of $D$ under $\varphi$ is a 2-simplex of $X$, yielding a
contradiction that $A$ is gated.
\end{proof}

We continue by showing that (ii)$\Rightarrow$(iii). Let $G$ be a pre-median $2$--connected graph such that each square of $G$ is included in a $W_4$ or $M_4$. Since $G$ is
a weakly modular graph, its triangle-square complex is simply connected. Since the triangle complexes of $W_4$ and $M_4$ are both simply connected, by condition (ii),
each induced $4$--cycle of $G$ is null-homotopic in the triangle complex $X\tr(G)$ of $G$.
Therefore, $X\tr(G)$ is simply connected, whence the clique complex $X(G)$ of $G$ is simply
connected as well.

The remaining part is dedicated to the proof of the implication (i)$\Rightarrow$(ii). In the first part of the proof
we establish some local properties of pre-median graphs. In the second part, we adapt  the proof of Theorem 2 of \cite{BCC+} for
primes of bucolic graphs to our more general setting.

\begin{lemma}\label{lem-prem-W5}
In a pre-median graph $G$, for every $W_5^-=(c,x_1,x_2,x_3,x_4,x_5)$, where $c\nsim x_1$, either there exists $y\sim c,x_1,x_2,x_5$,
or there exist $a,b$ such that $a \sim x_1,x_2,x_3,x_4$, $b \sim
a,x_1,x_3,x_4,x_5$, $a \nsim x_5$, and $b \nsim x_2$.
\end{lemma}

\begin{proof}
Suppose there does not exist $y\sim c,x_1,x_2,x_5$.  By TC($x_4$)
applied to the edge $x_1x_2$, there exists $a \sim x_1,x_2,x_4$. Since
$G$ does not contain $W_4^-$ and there is no $y\sim
c,x_1,x_2,x_5$, we have $a \nsim c,x_5$. To avoid a forbidden $W_4^-$ induced by $a,c,x_2,x_3,x_4,$ necessarily
$a \sim x_3$. Similarly, one can show that there exists $b \sim
x_1,x_3,x_4,x_5$ and $b \nsim c,x_2$. If $a \nsim b$, then $a,b,x_1,x_3,x_4$ will induce a forbidden
$W_4^-$. Hence $a\sim b$ and we are done.
\end{proof}

\begin{lemma}\label{lem-prem-W6}
In a pre-median graph $G$, for every $W_6^-=(c,x_1,x_2,x_3,x_4,
x_5,x_6)$ with $c\nsim x_1$,  there exists $y\sim c,x_1,x_2,x_6$.
\end{lemma}

\begin{proof}
We distinguish two cases, depending whether $d(x_4,x_1)=2$ or
$d(x_4,x_1)=3$. If $d(x_4,x_1) = 2$, by TC($x_1$) applied to the edge
$x_4c$, there exists $y \sim x_1, x_4, c$. Since $y$ and the square
$x_1x_2cx_6$ cannot induce a $K_{2,3}$ or a $W_4^-$, it holds $y\sim
c,x_1,x_2,x_6$ and we are done. Suppose now that $d(x_4,x_1) = 3$; by
QC($x_1$) applied to $x_3,x_5\in I(x_4,x_1)$, there exists $y \sim
x_1,x_3,x_5$. Since $c \sim x_3,x_4,x_5$, there is a square $yx_3x_4x_5$, and
$G$ does not contain an induced $W_4^-$, we have $c\sim y$. Since $y$ and the
square $x_1x_2cx_6$ cannot induce a $K_{2,3}$ or a $W_4^-$, necessarily
$y \sim x_2,x_6$. Since $y\sim x_1,c,$  we are done.
\end{proof}

\begin{lemma}\label{lem-prem-Wk}
In a pre-median graph $G$, for every $W_k^-=(c,x_1,x_2,\ldots,
x_k)$ with $c\nsim x_1$,  the induced $4$--cycle $C=(x_1,x_2,c,x_k)$ of $W_k^-$ is included in an induced  $W_4$ or in an $M_4$, i.e., either
there exists $y\sim C$ or there exist four pairwise adjacent vertices $a_1,a_2,a_3,a_4$ such that $a_1 \sim x_1,x_2$; $a_2\sim x_2,c$;
$a_3\sim c,x_k$; $a_4\sim x_1,x_k$; $a_1 \nsim x_k,c$; $a_2 \nsim x_1,x_k$; $a_3 \nsim x_1,x_2$, and $a_4 \nsim x_2,c$.
In particular, the $4$--cycle $(c,x_2,x_1,x_k)$ is null-homotopic in the clique complex of $G$.
\end{lemma}

\begin{proof}
We prove the assertion of the lemma by induction on $k$. If $k=5,$ the assertion holds by Lemma~\ref{lem-prem-W5}:
either there exists $y\sim C$ or we set
$a_1=a;a_2=x_3;a_3=x_4$ and $a_4=b$. If $k=6,$ the assertion follows from Lemma~\ref{lem-prem-W6}.
Suppose now that the lemma holds for $k$, and consider a $W_{k+1}^-$
$(c,x_1,x_2,\ldots ,x_{k+1})$ with $c\nsim x_1$. By TC($x_k$) applied to the edge
$x_1x_2$, there exists $u \sim x_1,x_2,x_k$. Note that since $G$ does
not contain $W_4^-$, $u \sim c$ if and only if $u \sim x_{k+1}$. Hence, if $u
\sim c$ or $u \sim x_{k+1}$, then the lemma holds by setting $y=u$.

Suppose now that $u \nsim c, x_{k+1}$. Then $ux_2cx_k$ is a square. Since $x_k\nsim x_3$, $x_{k-1} \nsim x_2$, and $G$ does not contain
$W_4^-$, necessarily we have $u \nsim x_3,x_{k-1}$. If there exists
an index $4\le i\le k-2$ such that $x_i \sim u$, then the vertices $u,c,x_2,x_i,x_k$
induce a $K_{2,3}$, a contradiction. Consequently, $(c,u,x_2,x_3,\ldots,x_k)$
is a $W_k^-$. By induction assumption, either there exists $z \sim u,x_2,c,x_k$ or
there exist $b_1,b_2,b_3,b_4$ such that $b_1 \sim u,x_2$; $b_2\sim
x_2,c,b_1$; $b_3\sim c,x_k,b_1,b_2$; $b_4\sim u,x_k,b_1,b_2,b_3$; $b_1
\nsim x_k,c$; $b_2 \nsim u,x_k$; $b_3 \nsim u,x_2$; and $b_4 \nsim
x_2,c$. We consider the two cases separately and in both cases, we are
going to show that there is a common neighbor of $x_1,x_2,c,x_{k+1}$
or that $x_1,x_2,c,x_{k+1}$ belong to a $W_5^-$ or to a $W_6^-$.
%% \medskip

%% \noindent
%% {\textbf{Case 1.}} There exists $z \sim u,x_2,c,x_k$.
\begin{case-ppm}
  There exists $z \sim u,x_2,c,x_k$.
\end{case-ppm}
    
Since $z \sim x_2,c$, we have $z \notin \{x_1,x_{k+1}\}$. Since $G$ does not
contain $W_4^-$ and $cx_2x_1x_{k+1}$ is a square, either $z \sim
x_1,x_{k+1}$, or $z \nsim x_{1},x_{k+1}$. In the first case, we have
found a common neighbor $z$ of $x_1,x_2,c,x_{k+1}$, establishing our
assertion for $y=z$.  In the second case, $(c,x_1,x_2,z,x_k,x_{k+1})$
is a $W_5^-$. By Lemma~\ref{lem-prem-W5}, either there exists $y \sim
x_1,x_2,c,x_{k+1}$, or there exists $a,b$ such that $a \sim
x_1,x_2,z,x_k$, $b \sim x_1,z,x_k,x_{k+1},a$, $a \nsim x_{k+1}, c$ and
$b \nsim x_2,c$. In this second case, let $a_1 = a$, $a_2 = z$, $a_3 =
x_k$ and $a_4 = b$, and the lemma holds.
%% \medskip

%% \noindent
%% {\textbf{Case 2.}}
\begin{case-ppm}
There exist $b_1,b_2,b_3,b_4$ such that $b_1 \sim
u,x_2$; $b_2\sim x_2,c,b_1$; $b_3\sim c,x_k,b_1,b_2$; $b_4\sim
u,x_k,b_1,b_2,b_3$; $b_1 \nsim x_k,c$; $b_2 \nsim u,x_k$; $b_3 \nsim
u,x_2$; and $b_4 \nsim x_2,c$.
\end{case-ppm}

Since $c \sim b_2,b_3$ and $c\nsim
x_1$, we have $x_1 \notin \{b_2,b_3\}$; analogously, since $b_2 \sim x_2$ and
$x_2\nsim x_{k+1}$, we have $b_2 \neq x_{k+1}$. If $b_3 = x_{k+1}$ and $b_2
\sim x_1$ then $b_2 \sim x_1,x_2,c,x_{k+1}$ and we are done. If $b_3
= x_{k+1}$ and $b_2 \nsim x_1$ then $c,x_1,x_2,b_2,x_{k+1}$ induce a
$W_4^-$, a contradiction. So we assume now that $b_3 \neq x_{k+1}$.
Since $G$ does not contain $W_4^-$ or $K_{2,3}$, if $x_1 \sim b_2$
(respectively: $x_1 \sim b_3$; $b_2 \sim x_{k+1}$), then $b_2 \sim
x_1,x_2,c,x_{k+1}$ (respectively: $b_3 \sim x_1,x_2,c,x_{k+1}$; $b_2
\sim x_1,x_2,c,x_{k+1}$), and we are done. So, suppose now that $x_1
\nsim b_2,b_3$ and $b_2\nsim x_{k+1}$. If $b_3 \sim x_{k+1}$, then
$(c,x_1,x_2,b_2,b_3,x_{k+1})$ is a $W_5^-$ and the lemma holds by
Lemma~\ref{lem-prem-W5}. Finally, if $b_3 \nsim x_{k+1}$, then
$(c,x_1,x_2,b_2,b_3,x_k,x_{k+1})$ is a $W_6^-$ and the lemma holds by
Lemma~\ref{lem-prem-W6}.
\end{proof}

Let $H$ be an induced
subgraph of a graph $G$. A 2-path $P=(a,v,b)$ of $G$ is
\emph{$H$--fanned} \cite{BCC+} if $a,v,b \in V(H)$ and if there exists an
$(a,b)$--path $P'$ in $H$ avoiding $v$ and such that $v$ is
adjacent to all vertices of $P'$, i.e., $v\sim P'$.  Notice that
$P'$ can be chosen to be an induced path of $G$. A path $P=(x_0,x_1,
\ldots,x_{k-1},x_k)$ of $G$ with $k>2$ is \emph{$H$--fanned} if every
three consecutive vertices $(x_i,x_{i+1},x_{i+2})$ of $P$ form an
$H$--fanned 2-path. When $H$ is clear from the context (typically
when $H = G$), we say that $P$ is \emph{fanned}.  If the end-vertices
of a $2$--path $P=(a,v,b)$ coincide or are adjacent, then $P$ is fanned.
Here is a simple generalization of this remark:

\begin{lemma} \cite{BCC+}
\label{l:chord}
If $P=(x_0,x_1,\ldots,x_k)$ is a fanned path and
$x_{i-1}$ and $x_{i+1}$ coincide or are adjacent, then the paths
$P'=(x_0,\ldots,x_{i-2},x_{i+1},x_{i+2}, \ldots,x_k)$ in the first
case and $P''=(x_0,\ldots,x_{i-1},x_{i+1}, \ldots,x_k)$ in the second
case are also fanned.
\end{lemma}

\begin{lemma}\label{fanned-4cycle} If $C=(v_1,v_2,v_3,v_4)$ is
an induced $4$--cycle of a pre-median graph $G$ such that the
$2$--path $(v_1,v_2,v_3)$ if fanned, then $C$ is included in an
induced $W_4$ or $M_4$. In particular, $C$ is null-homotopic and all
2-paths of $C$ are fanned.
\end{lemma}

\begin{proof} Let $P=(x_0=v_1,x_1,\ldots,x_{k-1},x_k=v_3)$ be a $(v_1,v_3)$--path of $G$ avoiding $v_2$ such that $v_2$ is
adjacent to all vertices of $P$ (such a path exists, because the 2-path $(v_1,v_2,v_3)$ is fanned). In view of Lemma \ref{l:chord},
we can suppose that for any $0<i<k$ the vertices $x_{i-1}$ and $x_{i+1}$ do not coincide neither are adjacent.
Since $v_2\nsim v_4,$ we have $v_4\notin P.$ Let $P_0=(y_0=v_1,y_{1},\ldots,y_{p-1},y_p=v_3)$ be an induced $(v_1,v_3)$--path of
$G$ in which all vertices belong to $P$.
Since $v_1\nsim v_3,$ we have $y_{1}\ne v_3$ and $y_{p-1}\ne v_1$.
If $v_4\sim y_{1},$ then necessarily $y_{1}\sim v_3$, otherwise the vertices
$v_1,v_2,v_3,v_4,y_{1}$ induce a forbidden $W_4^-$. Hence, $y_{1}$ is adjacent to all vertices of $C$ and $C$ together with $y_1$ induce a $W_4$.
Analogously, if $v_4\sim y_{p-1},$ then $C$ together with $y_{p-1}$ induce a $W_4$. Now, suppose that $v_4\nsim y_1,y_{p-1}$. If $v_4$
is adjacent to a vertex $y_j$ of $P_0$ with $1<j<p-1,$ then the vertices $v_1,v_2,v_3,v_4,y_j$ induce a forbidden $K_{2,3}.$
As a result, since $v_2$ is adjacent to all vertices of the induced path $P_0$ and $v_4$ is adjacent only to $v_1$ and $v_3,$
the path $P_0$ together with $v_2$ and $v_4$ induce in $G$ an almost-wheel $W_{p+2}^-$. By Lemma \ref{lem-prem-Wk}, $C$ is
included in an induced $W_4$ or $M_4$.
\end{proof}

\begin{lemma} \label{l:covered} Let $a,b,$ and $v$ be vertices of a pre-median graph $G$ such that $a$ and  $b$ can be connected
by a fanned path avoiding $v$. If $v\sim a,b$, then  there exists a fanned $(a,b)$--path $P$ such that $v\sim P$; in particular, the 2-path $(a,v,b)$ is fanned.
If $v\sim a$ and $d(v,b)=2,$ then there exists a fanned $(a,b)$--path $P$ such that $v \sim P\setminus\{b\}$.
\end{lemma}

\begin{proof}  For an $(a,b)$--path $P$ and a vertex $v$ of $G$, let $\Delta_v(P)=\max\{ d(v,x): x\in P\}$ be the
largest distance from $v$ to a vertex of $P$.  Let $\gamma_v(P) = |\{ xy \in P : d(v,x) = d(v,y) = \Delta_v(P)\}|$
denote the number of edges of
$P$ whose both ends are at maximal distance from $v$.  Let
$\kappa_v(P)=|\{ x\in P: d(v,x)=\Delta_v(P)\}|$ denote the number of vertices of $P$ at
maximal distance from $v$.

Among all fanned $(a,b)$--paths avoiding $v$, consider a path $P=(a=x_0, x_1,\ldots,x_m=b)$ lexicographically minimizing
$(\Delta_v(p),\gamma_v(p),\kappa_v(p))$ and let $k=\Delta_v(p)$.  Clearly, if $k=1$, then
$v\sim P,$ and the 2-path $(a,v,b)$ is fanned. Now, suppose that
$k\ge 2$.  Let $j$ be the smallest index such that $d(x_{j},v)=k$. Since $v\sim a,$ we have $j>0.$
If $j=m$ and $d(v,b)=2$, then $v$ is adjacent to all vertices of $P$ except $b$, i.e., $v \sim P \setminus
\{b\}$, and we are done.  So, further we assume that $0<j<m$.

If $x_{j-1}=x_{j+1}$ (respectively, $x_{j-1}\sim x_{j+1}$), then Lemma
\ref{l:chord} implies that the path
$P_1=(a=x_0,\ldots,x_{j-2},x_{j+1}, x_{j+2},\ldots,x_m=b)$
(respectively, $P_1=(a=x_0,\ldots,x_{j-2},x_{j-1}, x_{j+1},
x_{j+2},\ldots,x_m=b)$) is fanned. Note that $\Delta_v(P_1) \leq
\Delta_v(P)$.  Moreover, if $\Delta_v(P_1)=\Delta_v(P)$, then by our
choice of $j$, $\gamma_v(P_1) \leq \gamma_v(P)$ and
$\kappa_v(P_1)<\kappa_v(P)$, contrary to the minimality choice of $P$.
Thus $x_{j-1}\ne x_{j+1}$ and $x_{j-1}\nsim x_{j+1}$.  Note that
$d(x_{j-1},v)=k-1$ and $k-1 \le d(x_{j+1},v)\le k$.

If $d(x_{j+1},v)=k$, by triangle condition TC($v$) applied to
$x_jx_{j+1}$, there exists $y \sim x_j,x_{j+1}$ such that
$d(y,v)=k-1$. We assert that the path $P_2=(a=x_0,\ldots,x_{j}, y,
x_{j+1},\ldots,x_m=b)$ is fanned. Since the path $P$ is fanned, it
suffices to show that the $2$--paths $(x_{j-1},x_j,y)$,
$(x_j,y,x_{j+1})$, and $(y,x_{j+1},x_{j+2})$ are fanned. Since $x_j\sim
x_{j+1}$, the path $(x_j,y,x_{j+1})$ is trivially fanned. Since $P$ is
fanned, $(x_{j-1},x_j,x_{j+1})$ is fanned, i.e., there exists an
$(x_{j-1},x_{j+1})$--path $Q_0$ such that $x_j \sim Q_0$. Since $y \sim
x_j,x_{j+1}$, $(Q_0,y)$ is an $(x_{j-1},y)$--path such that $x_j \sim
Q_0$ and thus $(x_{j-1},x_j,y)$ is fanned. Similarly, one can show
that $(y,x_{j+1},x_{j+2})$ is fanned.  Consequently, $P_2$ is fanned.
Note that $\Delta_v(P_2) = \Delta_v(P)$ and
$\gamma_v(P_2)<\gamma_v(P)$, contrary to the minimality choice of $P$.

Now, suppose that $d(x_{j+1},v) = k-1$. By the quadrangle condition QC($v$)
applied to $x_{j-1},x_{j+1}\in I(x_j,v)$,
there exists a vertex $y\sim x_{j-1},x_{j+1}$ such that $d(y,v)=k-2$.
Since $(x_{j-1},x_j,x_{j+1})$ is fanned (as a $2$--path of a fanned
path $P$), by Lemma~\ref{fanned-4cycle} the induced $4$--cycle $C=(x_j,x_{j+1},y,x_{j-1})$
either is included in an induced $W_4$ or $M_4$.

First suppose that $C$ is included in a $W_4$ and let $z$ be the vertex of $W_4$ adjacent
to all vertices of $C$. We assert that the path $P_3=(x_0=a,\ldots,x_{j-1}, z, x_{j+1},\ldots,x_m=b)$ is fanned.
Since the path $P$ is fanned, for this it suffices
to show that the 2-paths $(x_{j-2},x_{j-1},z),(x_{j-1},z,x_{j+1}),$ and $(z,x_{j+1},x_{j+2})$ of $P_3$ are fanned.
Since $z$ is adjacent to all vertices of the path $(x_{j-1},x_j,x_{j+1}),$ the 2-path $(x_{j-1},z,x_{j+1})$ is fanned.
Since the 2-path $(x_{j-2},x_{j-1},x_j)$ of $P$ is fanned, there exists
a $(x_{j-2},x_{j})$--path $Q_0$ avoiding  $x_{j-1}$ such that $x_{j-1}\sim Q_0$. Since $x_{j-1}\sim z,$ $(Q_0,z)$ is a $(x_{j-2},y)$--path
avoiding $x_{j-1}$ and whose all vertices are adjacent to $x_{j-1}$. This establishes that the 2-path $(x_{j-2},x_{j-1},z)$ is fanned.
Analogously, one can show that $(z,x_{j+1},x_{j+2})$ is fanned as well. Therefore the path $P_3$ is fanned.
Since $d(z,v)=k-1,$ we conclude that either $\Delta_v(P_3)<\Delta_v(P)$ or $\Delta_v(P_3)=\Delta_v(P),$ $\gamma_v(P_3)=\gamma_v(P),$ and $\kappa_v(P_3)<\kappa_v(P)$,
in both cases getting a contradiction with the minimality choice of $P$.

Now suppose that the $4$--cycle $C$ is included in an induced $M_4$. Denote by
$a_1,a_2,a_3,a_4$ the four additional vertices of $M_4$ not included in
$C$ such that $a_1\sim y,x_{j-1},$ $a_2\sim y,x_{j+1},$ $a_3\sim
x_{j+1},x_j,$ and $a_4\sim x_j,x_{j-1}$. If $k = 2$, then $y = v$ and the
vertices $a_1,a_2,a_3,a_4$ are all distinct from $v$. If $k > 2$,
$a_1,a_2,a_3,a_4 \in B_2(x_j,G)$ and then $a_1,a_2,a_3,a_4$ are all
different from $v$.  We assert that the path $P_4=
(x_0=a,\ldots,x_{j-1}, a_1,a_2,x_{j+1},\ldots,x_m=b)$ is fanned. Since the
path $P$ is fanned, for this it suffices to show that the 2-paths
$(x_{j-2},x_{j-1},a_1),(x_{j-1},a_1,a_2),(a_1,a_2,x_{j+1}),$ and
$(a_2,x_{j+1},x_{j+2})$ of $P_4$ are fanned.  The 2-paths
$(x_{j-1},a_1,a_2)$ and $(a_1,a_2,x_{j+1})$ are fanned
because $y\sim x_{j-1},a_1,a_2,x_{j+1}$.
Analogously to the previous case, let $Q_0$ be a
$(x_{j-2},x_{j})$--path certifying that the 2-path
$(x_{j-2},x_{j-1},x_j)$ of $P$ is fanned, i.e., $Q_0$ avoids $x_{j-1}$
and $x_{j-1}\sim Q_0$.  Since $x_{j-1}\sim a_4,a_1$, $(Q_0,a_4,a_1)$
is a $(x_{j-2},a_1)$--path avoiding $x_{j-1}$ and whose all
vertices are adjacent to $x_{j-1}$. This establishes that the 2-path
$(x_{j-2},x_{j-1},a_1)$ is fanned.  Analogously, one can show that the
2-path $(a_2,x_{j+1},x_{j+2})$ is fanned as well. Hence the path $P_4$
is fanned and $v\notin P_4$. Since $k-2\le d(a_1,v),d(a_2,v)\le k-1$,
we conclude that either $\Delta_v(P_4)<\Delta_v(P)$ or
$\Delta_v(P_4)=\Delta_v(P), \gamma_v(P_4)=\gamma_v(P)$, and
$\kappa_v(P_4)<\kappa_v(P)$, in both cases getting a contradiction
with the minimality choice of $P$. This concludes the proof of the
lemma.
\end{proof}

Let $T=a_0b_0c_0$
be a triangle in $G$ and let $H_0, H_1,
H_2$ be the subgraphs respectively induced by $\{a_0\}, \{a_0, b_0\}$
and $ \{a_0, b_0,c_0\}$. Then for any ordinal $\alpha$ we define the subgraphs $H_{<\alpha}$ and $H_{\alpha}$
as in the transfinite version of the algorithm GATED-HULL (see Subsection~\ref{s:wmgra}). Analogously, we define the subgraph $K$ of $G$. By Lemma \ref{p:gate3},
$K$ is the gated hull in $G$ of the triangle $T$.

\begin{lemma}
\label{p:gate1}
For any ordinal $\alpha$, $H_{\alpha}$ is 2-connected and any 2-path
of $H_\alpha$ is $K$--fanned. In particular, $K$ is 2-connected and any 2-path
of $K$ is $K$--fanned.
\end{lemma}

\begin{proof}
We proceed by induction on $\alpha$.  Clearly, $H_0, H_1, H_2=T$
fulfill these properties. Assume by induction hypothesis that for
every $\beta < \alpha$, $H_\beta$ is $2$--connected and that any 2-path
of $H_{\beta}$ is $K$--fanned.

First we show that $H_{< \alpha}$ is 2-connected and that any 2-path
of $H_{< \alpha}$ is $K$--fanned. Consider any three vertices $a, b, u
\in V(H_{<\alpha})$. There exists $\beta<\alpha$ such that $a,b,u
\in V(H_\beta)$. By the induction hypothesis, there exists a path from
$a$ to $b$ in $H_\beta\setminus \{u\}$. Since $H_\beta\setminus \{u\}$
is a subgraph of $H_{< \alpha}\setminus \{u\}$, $a$ is not
disconnected from $b$ in $H_{< \alpha}\setminus \{u\}$, and thus $H_{<
  \alpha}$ is 2-connected.  For every 2-path $(a,b,c)$ in $H_{<
  \alpha}$, there exists $\beta < \alpha$ such that $a,b,c \in
V(H_\beta)$. By the induction hypothesis, the 2-path $(a,b,c)$ is
$K$--fanned.

If $K = H_{< \alpha}$, we are done. Otherwise, let $v$ be the unique
vertex of $V(G)$ such that $V(H_\alpha) = V(H_{<\alpha}) \cup \{v\}$.
By the definition of $H_{\alpha}$, $v$ has at least two neighbors
$x,y$ in $H_{<\alpha}$.

Suppose that $H_\alpha$ is not 2-connected. Consider three distinct
vertices $a,b,u \in V(H_\alpha)$.  If $a, b \in V(H_{<\alpha}) $, we
know that there exists a path from $a$ to $b$ in $H_{<\alpha} \setminus
\{u\}$. Without loss of generality, assume now that $b = v$ and $u
\neq x$. We know that there exists a path from $a$ to $x$ in
$H_{<\alpha} \setminus \{u\}$ and consequently, there exists a path
from $a$ to $b = v$ in $H_{\alpha} \setminus \{u\}$ since $x \sim
v$. Consequently, for every $a, b, u \in V(H_\alpha)$, $u$ does not
disconnect $a$ from $b$, i.e., $H_\alpha$ is 2-connected.

We will prove that any $2$--path of $H_{\alpha}$ is $K$--fanned.  It
suffices to consider the $2$--paths $Q$ of $H_{\alpha}$ that contain $v$,
since all other $2$--paths lie in $H_{<\alpha}$ and are $K$--fanned.

%% \smallskip
%% \noindent{\textbf{Case 1.}}
\begin{case-fan}
  $Q=(a,v,b)$.
\end{case-fan}

Since $H_{< \alpha}$ is connected and $a,b \in V(H_{< \alpha})$, there
exists an $(a,b)$--path $R$ in $H_{<\alpha}$. Since any 2-path of $H_{<
  \alpha}$ is $K$--fanned by induction hypothesis, $R$ itself is
$K$--fanned.  As $H_{< \alpha}$ is a subgraph of $K$, $R$ belongs to
$K$.  By Lemma \ref{p:gate3}, $K$ is a pre-median
graph, thus we can apply Lemma \ref{l:covered} and conclude that the 2-path
$(a,v,b)$ is $K$--fanned.

%% \smallskip
%% \noindent{\textbf{Case 2.}}
\begin{case-fan}
  $Q=(c,b,v)$.
\end{case-fan}

If $c$ and $v$ coincide or are adjacent, then $Q$ is trivially
fanned. Thus we may assume that $c\ne v$ and $c\nsim v$. Since $v$
has at least two neighbors in $H_{< \alpha}$, there exists a vertex $a
\in H_{< \alpha}$ adjacent to $v$ and different from $b$.  Since
$H_{<\alpha}$ is 2-connected and $a,c \in H_{<\alpha}$, there exists
an $(a,c)$--path $P_0$ in $H_{< \alpha}$ that avoids $b$. The path
$P_0$ is $K$--fanned because all its 2-paths are fanned by the
induction hypothesis. Since $b\sim c$ and $d(b,a) \in \{1,2\},$ by
Lemma \ref{l:covered} there exists a fanned $(a,c)$--path $P$ in $K$
avoiding $b$ such that $b\sim P$ if $b\sim a$ and $b\sim
P\setminus \{ a\}$ if $d(b,a)=2.$ In the first case obviously $b$ is
adjacent to all vertices of the $(c,v)$--path $(P,v),$ yielding that
the 2-path $Q=(c,b,v)$ is fanned. Now assume that $d(b,a)=2$ and
$b\sim P\setminus \{ a\}$. Let $a'$ be the neighbor of $a$ in $P$ and
let $P_1=P\setminus \{ a\}$. If $v\sim a',$ then $b$ is adjacent to
all vertices of the $(c,v)$--path $(P_1,v),$ hence $(c,b,v)$ is
fanned. Suppose now that $v\nsim a'$, i.e., $C=(a',a,v,b)$ is an
induced $4$--cycle of $H_{\alpha}$ and $K$. Since the $2$--path  $(a,v,b)$ is
$K$--fanned by Case 1, by Lemma \ref{fanned-4cycle} applied to the
induced $4$--cycle $C$ of the pre-median graph $K$, we conclude that the
$2$--path $(a',b,v)$ of $C$ is $K$--fanned. Thus in $K$ there exists an
$(a',v)$--path $P'$ avoiding $b$ such that $b\sim P'$.  But then $P_1$
followed by $P'$ is a $(c,v)$--path of $K$ avoiding $b$ whose all
vertices are adjacent to $b$.  This shows that the $2$--path $(c,b,v)$ if
$K$--fanned.
\end{proof}

\begin{lemma}
\label{p:null-homotopic}
For any ordinal $\alpha$, each induced $4$--cycle $C$ of $H_{<\alpha}$
and $H_{\alpha}$ is included in $K$ in an induced $W_4$ or $M_4$. In particular, 
each induced 4-cycle $C$ of $K$ is included in an induced $W_4$ or $M_4$. 
\end{lemma}

\begin{proof} Again we proceed by induction on $\alpha$. Suppose by induction hypothesis that
the assertion holds for every $H_{\beta}$ with $\beta < \alpha$. If there exists a
$4$--cycle $(a,b,c,d)$ in $H_{<\alpha}$, then there exists $\beta < \alpha$
such that $a, b, c, d \in V(H_\beta)$. Since $H_\beta$ is an induced
subgraph of $H_{<\alpha}$, $(a,b,c,d)$ is an induced $4$--cycle of
$H_\beta$, and we can apply the induction hypothesis to $H_\beta$.

If $K = H_{<\alpha}$, we are done. Otherwise, let $v$ be the unique
vertex of $V(G)$ such that $V(H_\alpha) = V(H_{<\alpha}) \cup \{v\}$.
Suppose that $H_{\alpha}$ contains an induced $4$--cycle $C$ such that
$v$ belongs to $C$. Let $C=(v,a,b,c,v).$ Since by Lemma \ref{p:gate1}
the $2$--paths of $H_{\alpha}$ are $K$--fanned, the simple $2$--path
$(a,b,c)$ of $C$ is $K$--fanned and, by Lemma~\ref{fanned-4cycle}, $C$
is included in a $W_4$ or in an $M_4$.
\end{proof}

Since $G$ is a prime pre-median graph, by Theorem
\ref{premedian_Chastand}(i), we have $G=K$.  Therefore, by the
previous lemma, each square of $G$ belongs to an induced $W_4$ or
$M_4$.
This concludes the proof of the implication
(i)$\Rightarrow$(ii) and of the Theorem \ref{prime_premedian}.

For a subclass $\mathcal C$ of ppm-graphs, let $\nabla(\mathcal C)$ denote
the class of all pre-median graphs in which all prime subgraphs belong to the class $\mathcal C$.
The class $\nabla(\mathcal C)$
consists of fiber-complemented graphs to which all the results of Chastand \cites{Cha1,Cha2} apply. In particular, each
graph $G\in \nabla({\mathcal C})$ is isometrically embeddable in a weak  Cartesian product of its
prime subgraphs. Moreover, if each graph of ${\mathcal C}$ is moorable, then all graphs $G\in \nabla({\mathcal C})$ are
retracts of a weak Cartesian products of their prime subgraphs (for the definition of moorings, see Chapter~\ref{s:prel}).
We know that some classes of prime pre-median graphs (in particular, weakly bridged graphs --- see the next section)
are moorable. However, we do not know if this holds for all ppm-graphs:

\begin{question} Is it true that all prime pre-median graphs are moorable?
\end{question}

If this is true, then we will obtain that each pre-median graph is a retract of a weak Cartesian product of prime pre-median graphs.

\section{Examples}\label{examples}

In this section, we present examples of thin and thick pm- and ppm-graphs. In particular, we prove
Proposition \ref{premedian-examples}(1).

\subsection{Thin ppm's: weakly bridged graphs} First consider the class of weakly bridged graphs. A graph $G$ is called {\it weakly
bridged} \cites{ChOs, O-cnpc} if $G$
is a weakly modular graph without induced $4$--cycles. It was shown in \cite{BCC+}*{Theorem 2(iii)} and
it also follows from Theorem \ref{prime_premedian} that 2-connected weakly bridged graphs are prime
pre-median graphs. Equivalently, the 2-connected weakly bridged graphs are exactly the thin ppm-graphs.
It was shown in \cites{BCC+,ChOs} that (a) weakly bridged graphs are
dismantlable, thus their clique complexes are contractible. Also weakly bridged graphs have the following
fixed point property \cites{BCC+,ChOs}: (b) any finite group acting on a weakly bridged graphs by automorphisms
fixes a simplex of its clique complex. The dismantling of weakly bridged graphs implies that (c) weakly bridged graphs
are moorable (or equivalently, that they satisfy the 1-fellow traveler property).

Bucolic graphs are the pre-median graphs whose prime graphs are weakly
bridged and bucolic complexes are the prism complexes derived from bucolic graphs by replacing each Cartesian
product of cliques by a prism cell. Namely, if ${\mathcal W}{\mathcal B}$ denote the class of weakly bridged graphs, then
the class of bucolic graphs coincide with $\nabla({\mathcal W}{\mathcal B})$. Using the properties (a)
and (b), it was shown in \cite{BCC+}*{Theorems 3$\&$4} that locally finite bucolic complexes are contractible and that any finite
group acting by automorphisms on a locally finite bucolic complex $X$  fixes a prism of $X$. Finally, from (c) and the results
of \cites{Cha1,Cha2} it follows that bucolic graphs are retracts of a weak  Cartesian product of their prime subgraphs. Notice also,
that the thick convex subgraphs of a bucolic graph $G$ are weak Cartesian products of complete subgraphs. Thus $C(G)$ coincides with the clique complex
of $G$ and $C(G)$ is a bucolic complex.
Therefore weakly bridged graphs
constitute a subclass of ppm-graphs with contractible clique complexes  and which lead to the class of bucolic graphs whose cell (prism) complexes satisfy  main
nonpositive-curvature-like properties.

\subsection{Half-cubes and Gosset graphs}

We start with the definition of several  distance-regular graphs; we follow \cite{DeLa} and \cite{BrCoNeu} (the hypercubes, half-cubes, Johnson graphs, and
hyperoctahedra have been defined in Subsection \ref{s:bagraphs}).
The {\it Schl\"afli graph} $G_{27}$ is a distance-regular graph of
diameter 2 with parameters $(27,16,10,8)$: it has 27 vertices, is 16-regular (each vertex has 16 neighbors), each pair of adjacent
vertices has 10 common neighbors, and each pair of non-adjacent vertices has 8 common neighbors.
Each 2-interval of $G_{27}$ is the 6-hyperoctahedron. The {\it Gosset graph} $G_{56}$ is a 27-regular graph on  56 vertices, has
diameter and radius 3, and the link of each vertex is the Schl\"afli graph $G_{27}$.
As noticed in \cite{KooMoSte},  $\frac{1}{2}H_{n},$ $J(n,k)$, $G_{27}$, and $G_{56}$ are {\it spherical graphs}, i.e., for each pair of non-adjacent
vertices $u,v$ and each vertex $x\in I(u,v)$ there exists a unique vertex $y\in I(u,v)$ such that $I(x,y)=I(u,v)$. The following two propositions
provide a proof of  Proposition \ref{premedian-examples}(1).

\begin{proposition} \label{half-cubes-Gosset} Half-cubes, Johnson graphs, the Schl\"afli graph, and the Gosset graph are thick ppm-graphs.
\end{proposition}

\begin{proof}  As basis graphs of matroids and even $\triangle$--matroids,  $\frac{1}{2}H_{n}$ and $J(n,k)$ satisfy (TC) \cites{Che_bas,Mau}.
Each 2-interval of $J(n,k)$ induces a 3-hyperoctahedron and each
2-interval of $\frac{1}{2}H_{n}$ induces a 4-hyperoctahedron. Thus each square of $\frac{1}{2}H_{n}$ or $J(n,k)$ is included in a $W_4$. This also shows that
$\frac{1}{2}H_{n}$ and $J(n,k)$ do not contain induced $K_{2,3}$ and $W^-_4$. So, to show that $\frac{1}{2}H_{n}$ and
$J(n,k)$ are ppm-graphs it suffices to show that they satisfy (QC). Let $x,y\in I(u,v)$ be two non-adjacent neighbors of $v$. Since the 2-interval
$I(x,y)$ is a 3- or a 4-hyperoctahedron and $v\in I(x,y),$ the vertices $x,y,v$ belong to a square $xvyu'$ of this hyperoctahedron. Since $d(u,v)>d(u,x)=d(u,y),$
by positioning condition (PC),
we conclude that $d(u,u')<d(u,x),$ thus establishing (QC). This shows that $\frac{1}{2}H_{n}$ and $J(n,k)$ are thick ppm-graphs.

Now, we will establish the same assertion for $G_{27}$ and
$G_{56}$. Since these graphs are spherical \cite{KooMoSte}, each
2-interval contains a square and is a hyperoctahedron. Thus
$G_{27}$ and $G_{56}$ are thick graphs not containing $K_{2,3}$ and
$W^-_4$, and each square of $G_{27}$ or $G_{56}$ is included in a $W_4$. Since $G_{27}$ has diameter 2, it obviously satisfies (QC).
To show that $G_{27}$ satisfies (TC), pick three vertices $x,y,z$ such that $x\sim y$ and
$d(z,x)=d(z,y)=2.$ Since $x$ and $y$ have 10 common neighbors, each
pair $x,z$ and $y,z$ have 8 common neighbors, and $G_{27}$ has 27
vertices, necessarily two of these three sets of common neighbors will
intersect, showing that $x,y,$ and $z$ have a common neighbor.

Finally, consider the Gosset graph $G_{56}$. As noticed above, $G_{56}$ does not contain induced $K_{2,3}$ and $W^-_4$.
We will also use the following property of $G_{56}$: if $d(x,y)=3,$ then the closed neighborhoods of $x$ and $y$ cover all vertices
of $G_{56}$. Indeed, this easily follows by noticing that  $G_{56}$ contains 56 vertices and is 27-regular and that those two
neighborhoods are disjoint.

To establish (QC), let $d(u,v)=3$ and $x,y\in I(u,v),$ $x,y\sim u$, and $x\nsim y$. Then the closed neighborhoods of $u$ and $v$ are disjoint and
cover all vertices of $G_{56}$. Since $G_{56}$ is spherical, the vertices $x,u,y$ belong to a square. Let $w$ be the fourth vertex of this square. Since $w\nsim u,$ necessarily $w\sim v,$ thus $w$ is a common neighbor of $x,y,$ and $v$.

To establish (TC), pick three vertices $x,y,z$ of $G_{56}$ such that $x\sim y$ and $1<d(z,x)=d(z,y):=k\le 3.$ First suppose that $k=3$. Pick any neighbor $u$ of $y$ in $I(y,z)$. If $u\sim x$, then we are done. Otherwise, since the closed neighborhoods of $x$ and $z$ covers $G_{56}$, necessarily $u\sim z$, which is impossible because $d(y,z)=3$. So, suppose that $k=2$. Let $u$ be a vertex of $G_{56}$ such that $d(x,u)=3$ (such a vertex exists because the radius of $G_{56}$ is 3). Since the closed neighborhoods of $x$ and $u$ covers all vertices of $G_{56},$ necessarily, $z$ is adjacent to $u$. The vertices $z$ and $y$ are at distance 2. Thus they belong to the link of some vertex $w$. Since the link $S$ of $w$ is the Schl\"afli graph, in $S$, $z$ and $y$ have 8 common neighbors inducing a 4-hyperoctahedron. If one of these vertices is adjacent to $x$, then we are done. Otherwise, since the closed neighborhoods of $x$ and $u$ covers the vertices of $G_{56},$ all common neighbors of $z$ and $y$ are neighbors of $u$. Pick two such non-adjacent neighbors $s$ and $t$. Since $d(u,x) = 3$, we have $u \nsim y$ and thus the subgraph induced by $u,s,t,z,$ and $y$ is a $W^-_4$, which is forbidden in $G_{56}$.
\end{proof}

\begin{proposition} \label{half-cube} \

  \begin{itemize}
    \item[(1)] If $G$  is a retract of $H$, where $H$ is one of the graphs  $K_{n\times 2}$, $\frac{1}{2}H_{n}$, $J(n,k)$, $G_{27}$, or $G_{56}$, then $G$ is a ppm-graph. Additionally, the retracts of  $K_{n\times 2}$, $\frac{1}{2}H_{n}$,  and $J(n,k)$ do not contain propellers.
    \item[(2)] If $G$  is a retract of a weak Cartesian product of graphs from the  family $K_{n\times 2}$, $\frac{1}{2}H_{n}$, $J(n,k)$, $G_{27}$, $G_{56}$ (the same factor can be taken several times), then $G$ is a pm-graph.
  \end{itemize}
\end{proposition}

\begin{proof} Let $\varphi: V(H)\rightarrow V(G)$ be a retraction map. That $G$ is a pre-median graph follows from Propositions \ref{half-cubes-Gosset} and  \ref{prod-retracts}.
To complete the proof of assertion (1) it remains to show that if $G$ is a retract of $H$ and $H$ is one of the graphs  $K_{n\times 2}$, $\frac{1}{2}H_{n}$, $J(n,k)$, $G_{27}$, $G_{56}$, then $G$ is prime. In view of Theorem \ref{prime_premedian} it suffices to show that each square $C=v_1v_2v_3v_4$ of $G$ is included
in $G$ in a $W_4$. In $H$ the 2-interval between $v_1$ and $v_3$ is an $n$--hyperoctahedron with $n\ge 3$. Hence in $H$ there exist two non-adjacent vertices $x_1,x_2$ which are adjacent to all vertices of $C$. If $x_1$ (or $x_2$) belongs to $G$, then we are done. Otherwise, $\varphi(x_1)$ (and $\varphi(x_2)$) cannot coincide with a vertex of $C$ and must be a common neighbor
of all vertices of $C$. Hence, $C$ is included in a $W_4$. This establishes the first assertion. The second assertion of (1) is immediate because  $K_{n\times 2},\frac{1}{2}H_{n},$ and $J(n,k)$ do not contain propellers.
The assertion (2) is an immediate consequence of Proposition \ref{prod-retracts}.
\end{proof}

The retracts of hypercubes are exactly the median graphs \cite{Ba_retracts} and  the retracts of weak Cartesian products of complete graphs are the quasi-median graphs \cite{BaMuWi}.
We conclude this subsection with the following open question:

\begin{question} \ 
  \begin{itemize}
  \item[(1)] Characterize the retracts of Johnson graphs and half-cubes. In particular, we conjecture that the retracts of Johnson graphs are the  ppm-graphs satisfying (IC3) and not containing propellers.
  \item[(2)] More generally,  characterize the retracts of Cartesian products of graphs from $K_{n\times 2},\frac{1}{2}H_{n},J(n,k),G_{27},G_{56}$.
  \end{itemize}
\end{question}

\subsection{Matroidal pre-median graphs} Now, we consider the finite thick pre-median graphs that are basis graphs of matroids or of even $\triangle$--matroids (for a formal definition, see Subsection \ref{s:bamat});
as such, they are isometrically embeddable into Johnson graphs and half-cubes, respectively.
The \emph{house} is the graph obtained by gluing a triangle and a
square along an edge (see Figure~\ref{fig-houses}, left).  The {\it
  double-house} $H+C_4=2C_4+C_3$ is the graph consisting of two
squares and a triangle, in which each two of them share one edge and
they all share one vertex.  The double-house can be viewed as a house
plus a square sharing with the house two incident edges, one from the
square and another from the triangle (see Figure~\ref{fig-houses}, left).

\begin{figure}[ht]
\begin{center}
\includegraphics[scale=0.5]{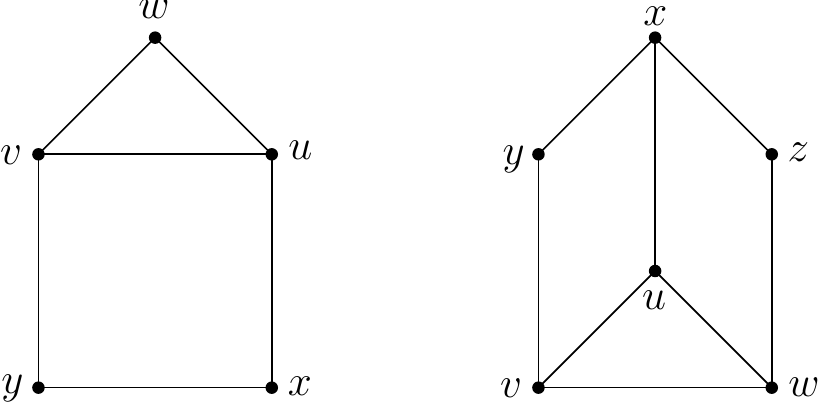}
\end{center}
\caption{A house (left) and a double-house (right)}\label{fig-houses}
\end{figure}

\begin{proposition} \label{positioning-condition} If $G$ is a pre-median graph not containing  propellers, then $G$ satisfies
the positioning condition (PC). In particular, $G$ does not contain induced double-houses.
\end{proposition}

\begin{proof}   Suppose by way of contradiction
that $G$ contains a square $C=v_1v_2v_3v_4$ and a vertex $u$ such that $d(u,v_1)+d(u,v_3)<d(u,v_2)+d(u,v_4).$ Suppose without
loss of generality that $d(u,v_1)\le d(u,v_3)=:k$ and $d(u,v_2)\le d(u,v_4).$ Notice that $d(u,v_1) = k-2$ is impossible.

%% \medskip\noindent
%% {\bf Case 1.}

\begin{case-mpm}
$d(u,v_1)=k$.
\end{case-mpm}

%\noindent
Then $d(u,v_4)=k+1$ and $d(u,v_2)\ge k$. Hence $v_1,v_3\in I(u,v_4).$ By (QC) there exists a vertex $u'\sim v_1,v_3$
such that $d(u,u')=k-1.$ Then $u'\nsim v_4$. If $u'\nsim v_2,$ then the vertices $v_1,v_2,v_3,v_4,u'$ induce a $K_{2,3}$. If
$u'\sim v_2,$ then the same vertices induce a $W^-_4$.

%% \medskip\noindent
%% {\bf Case 2.}
\begin{case-mpm}
$d(u,v_1)=k-1$.
\end{case-mpm}

%% \noindent
Since $v_1$ is adjacent to $v_2$ and $v_4$, we conclude that
$d(u,v_2),d(u,v_4)\le k.$ Since $d(u,v_1)+d(u,v_3)<d(u,v_2)+d(u,v_4),$
necessarily $d(u,v_2)=d(u,v_4)=k=d(u,v_3).$ By (TC) applied to $u$ and
the edge $v_2v_3$, there exists a vertex $w\sim v_2,v_3$ such that
$d(u,w)=k-1$. Since $G$ does not contain induced $W_4^-$, if $w \sim
v_1$ (respectively, $w \sim v_4$), then $w \sim v_4$ (respectively, $w \sim v_1$).

First, suppose that there exists a vertex $w\sim v_2,v_3$ with
$d(u,w)=k-1$ such that $w \sim v_1,v_4$.  By (TC) applied to $u$ and
the edge $v_1w$, there exists $y\sim w,v_1$ such that $d(u,y)=k-2$. But
then the vertices $v_2,v_4,y,w,v_1$ induce a forbidden propeller.

Suppose now that for any common neighbor $w$ of $v_2,v_3$ with
$d(u,w)=k-1$ we have $w\nsim v_1,v_4$. By (TC) applied to $w$ and
the edge $v_1 v_4$
there exists a
vertex $x\sim w,v_1,v_4$.  If $x\nsim v_2$, then the vertices
$v_1,x,w,v_2$ induce a $4$--cycle satisfying the conditions of Case 1,
which is impossible.  Thus $x\sim v_2$. If $x \nsim v_3$, the vertices
$v_2,x,w,v_3,v_4$ induce a $W_4^-$, which is impossible. Consequently,
$x \sim v_1,v_2,v_3,v_4$. Since $x$ is adjacent to $v_1,v_4$
necessarily $k-1 \leq d(u,x)\leq k$. Since $v_1,v_2,v_3,v_4$ do not
have a common neighbor $w$ such that $d(u,w) = k-1$, necessarily
$d(u,x) = k$.  By (QC) for $u$ and $w,v_1,x$, there exists $u'\sim w,v_1$
with $d(u,u')=k-2$. Since $d(u,v_2)=k,$ $u'\nsim v_2$ and the vertices
$v_2,w,x,v_1,u'$ induce a forbidden $W^-_4$. This shows that $G$
satisfies (PC).

Finally, suppose by contradiction that $G$ contains an induced double-house having $x,y,u,v,w,z$ as the set of vertices,  where $uvw$ is a triangle and $xyvu$ and $xuwz$ are the
two squares of  this house. If $y\nsim z$, then the square $xyvu$ and the vertex $z$ violate the positioning condition (PC).
\end{proof}

Let ${\mathcal M}_3$ denote the class of all thick  pre-median graphs satisfying (IC3) and not containing  propellers
($W_4$ and $M_4$ are two examples of such graphs). Let also ${\mathcal M}_4$ denote the class of all thick pre-median graphs satisfying (IC4), not containing propellers,
and in which the links of vertices do not contain induced $W_5$.   Proposition \ref{positioning-condition} implies
that the graphs of  ${\mathcal M}_3$ and  ${\mathcal M}_4$ satisfy the positioning condition (PC). Since the graphs of ${\mathcal M}_4$ do not contains propellers, the links of their vertices also do not
contain induced $W_6$. From the characterizations of basis graphs of matroids and  even $\triangle$--matroids from Theorem \ref{Th_Mau}(1), we obtain
the following result:

\begin{proposition} \label{prime-to-matroids}  All graphs of ${\mathcal M}_3$ are basis graphs of matroids and all graphs of ${\mathcal M}_4$ are basis graphs of even $\triangle$--matroids. In particular,  ${\mathcal M}_3\subsetneq {\mathcal M}_4$.
\end{proposition}

The following example shows that the converse of Proposition \ref{prime-to-matroids} is not true, namely, there exist many basis graphs of matroids or even $\triangle$--matroids which are not weakly modular:

\begin{example} The simplest example is the graph $H'$ obtained from the half-cube $\frac{1}{2}H_n$ for $n\ge 6$ by removing a single vertex; for example, let $H'$ be obtained from $\frac{1}{2}H_6$ by removing the vertex corresponding to $\emptyset$. It can be easily seen that $H'$ is a basis graph of an even $\triangle$--matroid. However, $H'$ violates the quadrangle condition (QC)  for vertices $u=\{ 1,2,3,4\}, x=\{ 1,2\}, y=\{ 3,4\},$ and $w=\{ 5,6\}$. Indeed, in $H'$ and in $\frac{1}{2}H_6$ we have $x,y\in I(u,w), u\sim x,y$ and $d(x,w)=d(y,w)=2$. However, the unique common neighbor of $x,y,$ and $w$ in $\frac{1}{2}H_6$ is the vertex encoded by $\emptyset$.
\end{example}

The following example shows that there exist thick pre-median graphs
containing propellers (other such examples are the graphs $G_{27}$
and $G_{56}$). This graph also satisfies (PC). This shows that for
pre-median graphs, the converse of
Proposition~\ref{positioning-condition} does not hold.

\begin{example} \label{ex:propellers} Let $H'=K_3\times K_3$ be the Cartesian product of two $K_3$.  Now, add to $H'$ two new adjacent vertices $x_1,x_2$ adjacent to all vertices of $H'$. Denote this graph by $H$. Since $H$ contains two universal vertices, it satisfies (TC), has diameter 2, and thus $H$ is weakly modular. The interval between any two nonadjacent vertices $u,v$  consists of a square (which is the interval between $u$ and $v$ in $H'$) and $x_1,x_2$. Thus $H$ is thick and does not contain induced $K_{2,3}$ and $W^-_4$, hence $H$ is a thick pm-graph. It can be also checked that $H$ satisfies (PC). On the other hand, $H$ contain propellers induced by $x_1,x_2$ and any the triplet $y_1,y_2,y_3$ of pairwise non-adjacent vertices of $H'$.
\end{example}

We continue with an interesting example of a ppm-graph $G$ from ${\mathcal M}_3$.

\begin{example} \label{ex:M3}  The vertex set of $G$ consists of $12$ vertices $x_1,x_2,x_3,x_4,a_1,a_2,a_3,a_4,b_1,b_2,b_3,b_4$ such that $x_1,x_2,x_3,x_4$ induces a $4$--cycle $C$ which is included in two (subgraphs isomorphic to)  $M_4$, one generated by $C$ and the vertices
$A=\{ a_1,a_2,a_3,a_4\}$ and another one generated by $C$ and $B=\{
  b_1,b_2,b_3,b_4\}$ (the labeling of the vertices of each $M_4$
  follows that from Fig. \ref{W_4M_4}). Suppose also that $a_i$ is
  adjacent to $b_j$ if and only if $i = j$.  It can be easily checked
  that $G$ is a pre-median graph (of diameter 2). Moreover, $G$ is
  prime because each square extends either to a single $W_4$ or to two
  $M_4$. The 2-intervals of $G$ are either squares or pyramids, thus
  $G$ satisfies (IC3), whence $G$ belongs to ${\mathcal M}_3$. The two
  $M_4$ induced by $A\cup C$ and $B\cup C$ are convex thick
  ppm-subgraphs of $G$ but their intersection -- the $4$--cycle $C$ --
  is thick and convex but no longer a ppm-graph. Notice also that $G$
  contains several other convex subgraphs isomorphic to $M_4$.
\end{example}

\section{$L_1$--Weakly modular graphs}\label{l1wmg}

In this section, we investigate the $L_1$--weakly modular graphs admitting a scale embedding $\varphi$ into a hypercube $H(X)$.
Since $K_{2,3}$ and $W^-_4$ are not $L_1$--graphs (see Lemma \ref{lemma:linksl1}),
all $L_1$--weakly modular graphs are pre-median. It was shown in \cite{BaCh_weak} that all weakly
median graphs are $L_1$--graphs. Next lemma recalls some simple and  well-known properties of $L_1$--graphs
(see for example, Theorem 17.1.8 and Chapter 21  of \cite{DeLa}). Recall that the {\it line-graph} $L(H)$ of a graph $H$
has the edges of $H$ as the set of vertices and
the pairs of incident edges of $H$ as edges.
It is well-known that line-graphs can be characterized as graphs not containing 9 forbidden
subgraphs; for the list, see for example \cite{DeLa}*{Fig. 17.1.3}.

\begin{lemma} \label{lemma:linksl1} If $G$ is an $L_1$--graph, then $G$ does not contain induced $K_{2,3},W^-_4,$ and propellers; in particular,
all $L_1$--weakly modular graphs are pre-median graphs not containing propellers.
If $G$ is an isometric subgraph of a half-cube, then all links of vertices of $G$ are line-graphs. In particular,
the links of vertices of an isometric subgraph $G$ of a half-cube do not contain $K_{1,3}$ (i.e., $G$ does not
contain propellers), $W_5$, and $K_5-e$ (i.e., $G$ does not contain $K_6-e$).
\end{lemma}

\begin{lemma} \label{thick-pm} If $G$ is a finite thick pre-median graph, then $G$ is a Cartesian product of its primes, which are all thick ppm-graphs.
\end{lemma}

\begin{proof} Since  the gated subgraphs of $G$ are thick, all primes (i.e., prime gated subgraphs) of $G$ are thick ppm-graphs.
Since $G$ is a finite pre-median graph, by Theorem \ref{premedian_Chastand}(v),  $G$ can be obtained by gated amalgams from Cartesian products of its
primes. Clearly, Cartesian products of thick ppm-graphs are thick pm-graphs. Therefore, to conclude the
proof, it suffices to show a general fact that a thick graph cannot be represented as a gated amalgam of thick graphs. Suppose that a
graph $G$ is the gated amalgam of two thick graphs $G_1$ and $G_2$ along a common gated subgraph $G_0$. That the gated amalgam of two
thick graphs is not thick
is an immediate consequence of the following assertion:

\begin{claim} There exist  three vertices $u\in V(G_1)\setminus V(G_0),$ $v\in V(G_2)\setminus V(G_0),$ and $z\in V(G_0),$ such that $u\sim z\sim v,$ $d(u,v)=2$, and $I(u,v)=\{ u,z,v\}$.
\end{claim}

\begin{proof}[Proof of the Claim]
%% \begin{proof}
%% \medskip\noindent
%% {\emph{Proof of the Claim.}}
Pick two closest vertices  $u\in V(G_1)\setminus V(G_0)$ and $v\in V(G_2)\setminus V(G_0).$ Let $P=(u,u',u'',\ldots,v',v)$ be a shortest path between $u$ and $v$ in $G$. From the choice of $u,v$, the vertices $u',v'$ as well as the vertices of the subpath of $P$  between $u',v'$ are all contained in $G_0$. First suppose that $u'\ne v'$. Then $u''$ exists but can coincide with $v'$. Since $u,u''\in V(G_1)$ and $G_1$ is thick, $u,u''$ belong to a square $uxu''y$ of $G_1$. Since $G_0$ is a gated and thus is a convex subgraph of $G_1,$ $x$ and $y$ cannot both belong to $G_0$. If say $x\in V(G_1)\setminus V(G_0),$ then replacing $u$ by $x$ we obtain a contradiction with the minimality choice of the pair $u,v$. Hence, suppose that $u'=v':=z$, i.e., $d(u,v)=2$. If $I(u,v)\ne \{ u,z,v\}$, then we can find a common neighbor $x\ne z$ of $u$ and $v$. Then necessarily $x\in V(G_0)$. Since both $x$ and $z$ are adjacent to $u$ and $v$, from the choice of the vertices $u,v$ we obtain a contradiction with the fact that  $G_0$ is a gated subgraph of $G$.
\end{proof}
\end{proof}

\begin{lemma} \label{l_1-pre-med} Let $G$ be an $L_1$--weakly modular graph admitting a scale embedding into a hypercube. Then all primes  of $G$ are either finite subhyperoctahedra or isometric subgraphs of  half-cubes.
\end{lemma}

\begin{proof}
By Theorem \ref{th:shpectorov}, $G$ admits an
isometric embedding into a weak Cartesian product of finite hyperoctahedra and
half-cubes. On the other hand, Graham and Winkler \cite{GrWi} proved
that any graph admits a canonical isometric embedding into a weak Cartesian
product of irreducible graphs; for definitions, see Subsection \ref{prel:other} and
\cite{DeLa}*{Chapter 20}.
Chastand \cite{Cha1}*{Corollary 6.2} proved that
the primes of any pre-median graph (and $G$ is pre-median in view of Lemma \ref{lemma:linksl1})
coincide with the irreducible factors in the Graham and Winkler's  representation.

Let $H$ be a prime of $G$. Then, on the one hand, $H$ is irreducible and, on the other hand, since $H$ is a gated (and thus isometric)
subgraph of $G$, $H$ can be isometrically  embedded into a weak  Cartesian product of hyperoctahedra and
half-cubes (into which $G$ is embedded). Therefore the irreducibility of $H$ implies that  $H$ belongs as an isometric subgraph either to
a hyperoctahedron or  to a half-cube.
\end{proof}

\begin{lemma} \label{thick-prime-pre-med} $G$ is a finite thick prime  $L_1$--weakly modular  graph if and only if either $G$ is a thick subgraph of a finite hyperoctahedron or $G$ belongs to ${\mathcal M}_4$.
\end{lemma}

\begin{proof}  The ``if''  part  directly follows from the definition of ${\mathcal M}_4$ and Proposition \ref{prime-to-matroids}. Now, let $G$ be a finite thick prime $L_1$--weakly modular graph. Then $G$ is pre-median. By
Lemma \ref{l_1-pre-med} either $G$ is a subgraph of a hyperoctahedron or $G$ is an isometric subgraph of a half-cube. In the first case we are done. So, suppose that  $G$ is an
isometric subgraph of a half-cube. By Lemma \ref{lemma:linksl1}, $G$ does not contain propellers and  the links of vertices of $G$ do not contain induced $W_5$.  As an isometric subgraph of a half-cube,  the 2-intervals
of $G$ are subgraphs of the 4-hyperoctahedron. Since $G$ is thick, this implies that $G$ satisfies (IC4). By \cite{Che_bas}*{Theorem 1(v)},  $G$ belongs to ${\mathcal M}_4$.
\end{proof}

Proposition \ref{prime-to-matroids} and Lemmata \ref{thick-pm}, ~\ref{l_1-pre-med}, and \ref{thick-prime-pre-med} complete the proof of Proposition \ref{premedian-examples}(2).

We continue with a characterization of $L_1$--weakly modular graph admitting a scale embedding into hypercubes. As we noticed already, each 2-interval $I(u,v)$ of such a graph $G$ must be
an induced subgraph
of some finite  hyperoctahedron $K_{n\times 2}$. Call such an embedding of $I(u,v)$ into the smallest dimensional hyperoctahedron the {\it completion} of $I(u,v)$.
Then the {\it octahedron-dimension} of an $L_1$--weakly modular graph $G$ is the largest dimension of a completion over all 2-intervals of $G$.
It is well-known \cite{DeLa}  that a finite graph $G$ is an $l_1$--graph if and only if $G$ admits a scale isometric embedding $\varphi$ into a
hypercube $H(X)$. From Theorem \ref{th:shpectorov} and Lemma \ref{l_1-pre-med} we immediately obtain an extension of this result to
arbitrary $L_1$--weakly modular graphs:

\begin{proposition} \label{finite-scale} An $L_1$--weakly modular graph $G$ has a (finite) scale embedding into a hypercube if and only if $G$ has finite octahedron-dimension.
\end{proposition}

\begin{proof} From Lemma \ref{l_1-pre-med} we know that all primes $G_i, i\in \Lambda,$ of $G$ are subhyperoctahedra or isometric subgraphs of  half-cubes. Moreover, by Theorem \ref{premedian_Chastand}(iii), $G$ isometrically embeds into a weak Cartesian product of its primes. If $G$ has finite octahedron-dimension, then for some finite (even) $\lambda$ each prime $G_i$ of $G$ admits a scale $\lambda$ isometric embedding into a hypercube $H(X_i)$. Then $G$ has a scale $\lambda$ embedding into the hypercube $H(X)$, where $X$ is the disjoint union of $X_i, i\in \Lambda$. Conversely, if $G$ has a finite scale embedding into a hypercube, then its octahedron-dimension must be bounded, otherwise, if $G$ contains arbitrarily large suboctahedral 2-intervals, then $G$ will contains arbitrarily large isometric $K_{n}-e$ (i.e. a complete graph $K_n$ minus an edge). Similarly to the hyperoctahedron $K_{n\times 2}$ (which is the completion of $K_{n}-e$), the scale of the scale embedding of $K_n-e$ is an increasing function of $n$, see \cite{DeLa}*{Subsection 7.4, p.101}.
\end{proof}

Let $G$ be an $L_1$--weakly modular graph of finite octahedron-dimension and let $\varphi$ be   a scale isometric embedding of $G$ into a hypercube $H(X)$.
For an element $a\in X,$ let $H_a$ denote the set of all  finite subsets of $X$ (vertices of $H(X)$) which contain the element $a$, and let  $H_{a*}$ denote the set of all finite subsets of $X$ which do not contain $a$. $H_a$ and $H_{a^*}$ induce complementary subhypercubes of $H(X)$. Let $W_a=V(G)\cap H_a$ and $W_{a*}=V(G)\cap H_{a*}$; equivalently,  $W_a=\{ v\in V(G): a\in \varphi(v)\}$ and $W_{a^*}=\{ v\in V(G): a\notin \varphi(v)\}$.   We will call the sets of the form
$W_a, W_{a*},$ $a\in X$, {\it half-spaces} of $G$. Let
$$\partial W_a:=\{ v\in W_a: v\sim u \mbox{ for some } u \in W_{a*}\}$$
denote the {\it boundary} of $W_a$ (the boundary $\partial W_{a*}$ of $W_{a*}$ is defined in a similar way). Let $S_a:=\partial W_a\cup \partial W_{a*}$ and call $S_a$ the {\it carrier} of $a$.
The following result can be viewed as a generalization of a well-known property of CAT(0) cube complexes (or, equivalently, of median graphs) that half-spaces, hyperplanes, and carriers are also
CAT(0) cube complexes.

\begin{theorem} \label{boundary} Let $G$ be an  $L_1$--weakly modular graph  which is scale  embeddable in a hypercube  $H(X)$. Then for any element $a\in X$, the sets $W_a,W_{a*}$
induce convex subgraphs of $G$ and  the boundaries $\partial W_a, \partial W_{a*},$ the carrier $S_a,$ and the sets $W_a\cup \partial W_{a*},$ $W_{a*}\cup \partial W_a$ induce isometric weakly modular subgraphs of  $G$. If $G$ does not contain induced $W_5$, then all those sets induce convex subgraphs of $G$. Additionally, if $G$ is thick, then $W_a=\partial W_a, W_{a^*}=\partial W_{a*}$  and $V(G)$ coincides with the carrier $S_a$.
\end{theorem}

\begin{proof} $H_{a}$ and $H_{a*}$ are two complementary subcubes of the hypercube $H(X)$ so that the edges  between $H_{a}$ and $H_{a*}$ define a perfect matching of $H(X)$. Therefore $H_{a}$ and $H_{a*}$ are connected locally convex subgraphs of $H(X)$. Since $H(X)$ is a median graph, by Lemma \ref{lem-weakly-modular_gated_hull},  $H_{a}$ and $H_{a*}$ are  convex subgraphs of $H(X)$. Hence  $W_a$ and $W_{a*}$ are convex as the intersection of a  scale-isometric subgraph $G$ of $H(X)$ with two complementary convex subsets of $H(X).$  Next we will show that the set $\partial W_a$ induces an isometric pre-median subgraph of $G$ (which is convex if $G$ does not contain a $W_5$); the same result for $W_{a*}$ can be proved in the same way.

\begin{lemma} \label{isometric_neighbors} If $x,y\in\partial W_a,$ $x\sim y,$ $x'\sim x,$ and $x'\in \partial W_{a^*},$  then there exists a vertex  $y'\in \partial W_{a*}$
such that $y\sim y'$ and $x'=y'$ or $x'\sim y'$.
\end{lemma}

\begin{proof} Let $y'$ be a neighbor of $y$ in $\partial W_{a*}$. Suppose that $x'$ and $y'$ are different and non-adjacent. Since $W_{a^*}$ is convex, the path $(x',x,y,y')$ (of length 3) cannot be  shortest, and hence $d(x',y')=2.$ By (TC) applied to the vertices $x',y,y'$ there exists a vertex $y''\sim x',y,y'$. Again, the convexity of $W_{a*}$ implies that $y''\in W_{a*}$. But then the pair of adjacent vertices $x',y''$ satisfies the requirement of the lemma.
\end{proof}

\begin{lemma} \label{isometric_aux} If $x,y\in \partial W_a$, $x\nsim y$, and $v\in I(x,y), v\sim x,$ then either $v\in \partial W_a$, or $v\sim y$, or $y$ has a neighbor $w\in I(y,v)$ such that $w\in \partial W_a$.
\end{lemma}

\begin{proof} Suppose $v\notin \partial W_a$ and $v\nsim y,$ otherwise is nothing to prove. Let $x\sim x'$ and $y\sim y'$ for $x',y'\in \partial W_{a*}$ and suppose that such a  pair $x',y'$ is selected to minimize the distance $k=d(x',y')$.

Since the set $W_{a^*}$ is convex, we have $d(y,x'),d(y',x)\ge k$. If $d(y,x')=d(y',x')$, then by (TC) there exists $y''\sim y,y'$ and having distance $k-1$ to $x'$. Since $y''\in I(x',y'),$ we have  $y''\in W_{a^*}$ by convexity of $W_{a^*}$. This contradicts the choice of the pair $x',y'$. Thus $d(y,x')=k+1$. Analogously, if $d(y',x)=d(y',x'),$ by (TC) we can find $x''\in \partial W_{a^*}$ with  $x''\sim x,x'$ and $d(x'',y')=k-1,$ contrary to the choice of $x',y'$. So $d(y',x)=k+1$. Finally, suppose $d(y,x)<d(y,x')$. Let $w$ be a neighbor of $y$ in $I(y,v)$. Then $w\in I(y,v)\subset I(y,x)\subset I(y,x').$ Since $y'\in I(y,x'),$ by (QC) there exists $w'\in W_{a^*}$ adjacent to $w,y'$ and having distance $k-1$ to $x'$. Hence $w\in I(y,v)\cap \partial W_a,$ and we are done.

So further we can assume (by convexity of $W_a$) that
$d(y,x)=d(y,x'),$ i.e., $d(x,y)=k+1.$ By (TC) there exists $z\sim
x,x'$ with $d(y,z)=k.$ The convexity of $W_a$ implies that $z\in
\partial W_a$. Since $v,z$ are distinct neighbors of $x$ in $I(x,y),$
by (TC) if $v\sim z$ or by (QC) if $v\nsim z$, there exists $u\sim
v,z$ with $d(u,y)=k-1.$ If $u=y,$ then we are done. Otherwise, let $w$
be a neighbor of $y$ in $I(y,u)$. Note that $w\in I(y,u)\subset
I(y,v)$ and thus $w\in W_a$. If $w \sim y'$, then $w \in \partial W_a$
and we are done. If $w \nsim y'$, since $w \in I(y,u) \subset I(y,z)
\subset I(y,x')$, by (QC), there exists $w'$ adjacent to
$w,y'$ and having distance $k-1$ to $x'$. Since $W_{a^*}$ is convex
and $w' \in I(y',x')$, we have $w' \in W_{a^*}$. Consequently $w\in
\partial W_a$ and we are done.
\end{proof}

\begin{lemma} \label{isometric} $G(\partial W_a)$ is an  isometric subgraph of $G$.
\end{lemma}

\begin{proof} Let $x,y\in \partial W_a$. We proceed by induction on $k=d(x,y).$ In view of induction, it suffices to show that either $x$ has a neighbor in $I(x,y)\cap \partial W_a$ or $y$ has a neighbor in $I(x,y)\cap \partial W_a$. Let $v$ be an arbitrary neighbor of $x$ in $I(x,y)$. If $v\in \partial W_a,$ then we are done. Otherwise, if $v\notin \partial W_a$ and $k\ge 3,$ then Lemma \ref{isometric_aux} implies that $y$ has a neighbor $w\in I(y,v)\subset I(y,x)$ belonging to $\partial W_a$, and we are done again. So, let $d(x,y)=2,$ i.e., $v\sim x,y$.
Pick $x',y' \in \partial W_a$ with $x'\sim x,$ $y'\sim x$ so that $d(x',y')$ is minimum. Then as in the proof of Lemma \ref{isometric_aux} one can show that $d(x',y')=1$ and $d(y,x')=2$. By (TC) there exists $w\sim y,x,x'$. The convexity of $W_a$ implies that $w\in W_a$. Since $w\sim x',$   we have $w \in \partial W_a$, as required.
\end{proof}

\begin{lemma} \label{isometric_tc} $G(\partial W_a)$ satisfies the triangle condition (TC).
\end{lemma}

\begin{proof} Let $x,y,z\in \partial W_a$ such that $d(x,y)=d(z,y)=k$ and $x\sim z$. Again, we proceed by induction on $k$. By (TC) for $G$, there exists $v\sim x,z$ with $d(y,v)=k-1$. The convexity of $W_a$ implies that $v\in W_a$. If $v\in \partial W_a,$ then we are done. So, let $v\notin \partial W_a$. If $k\ge 3,$ by Lemma \ref{isometric_aux} there exists a neighbor $w$ of $y$ such that
$w\in I(y,v)\cap \partial W_a$. Since $d(w,x)=d(w,z)=k-1$ by induction assumption there exists $u\in \partial W_a$ such that $u\sim x,z$ and $d(u,w)=k-2.$ Since $u\in I(y,x)\cap I(y,z),$
we are done. Now assume that $k=2,$ i.e., $v\sim x,y,z.$ Let $x',z',$ and $y'$ be neighbors in $\partial W_{a^*}$ of $x,z,$ and $y$, respectively. By Lemma \ref{isometric_neighbors},
$x'$ and $z'$  can be selected to coincide or to be adjacent.

\begin{case-Wa}
%% \medskip\noindent
%% {\bf Case 1.}
$x'=z'$.
\end{case-Wa}

We can suppose that $y'$ is as close as possible to $x'$. As in the proof of Lemma \ref{isometric_aux} one can show that $d(x',y')=1$ and $d(y,x')=2$. By (TC) there exists $w\sim x,x',y$. The convexity of $W_a$ implies that $w\in W_a$, hence $w\in \partial W_a$. Moreover, since $W_a$ is convex, the 2-path $(z,x',w)$ cannot be shortest and hence $z\sim w.$ Therefore (TC) holds in $G(\partial W_a)$ for the triple $x,z,y$.

\begin{case-Wa}
%% \medskip\noindent
%% {\bf Case 2.}
  $x'\ne z'$.
\end{case-Wa}

Then we can assume that $x'\nsim z$ and $z'\nsim x,$ otherwise we are
in the conditions of Case 1. Suppose first that $d(y,x') = 3$ and
assume that $y'$ is selected to minimize $d(y',x')$. By convexity of
$W_{a^*}$ and since $d(y,x') = 3$, we have $2 \leq d(y',x') \leq 3$.
If $d(y',x') = 3$, by TC($x'$), there exists $y'' \sim y,y'$ such that
$d(y'',x') = 2$. By convexity of $W_{a^*}$, we have $y'' \in \partial
W_{a^*}$, contradicting our choice of $y'$. If $d(y',x') = 2$, then
$v,y'\in I(y,x')$, and  by QC($x'$), we can find $v'\in
W_{a^*}$ adjacent to $v,y'$, yielding $v\in \partial W_a$; this is a
contradiction.

Consequently, $d(y,x') \leq 2$, and similarly, one can prove that
$d(y,z') \leq 2$. By convexity of $W_a$, we get $d(y,x') = d(y,z') =
2$. Assume that $y'$ is selected to minimize $d(y',x')$. Since
$W_{a^*}$ is convex, $d(y',x') \leq 2$. If $d(y',x') = d(y,x') = 2$,
by TC($x'$), there exists $y'' \sim y, y', x'$, contradicting our
choice of $y'$. Consequently, $y' \sim x'$.

Suppose $d(y',z')=2.$ By Proposition \ref{positioning-condition}, $G$
satisfies (PC). Applying (PC) to $y'$ and the square
$xx'z'z$ (with $d(x,y') = 2$), we conclude that $d(y',z)=3.$ Hence $z',v\in I(z,y')$, and by
(QC) there exists $v'\sim v,z',y'.$ Since $v'\in W_{a^*}$ by convexity
of $W_{a^*},$ we deduce that $v\in \partial W_a$, and we are done.

Thus, $y'\sim x',z'$. Since $d(y',x)=d(y',z)=2,$ by TC($y'$), there
exists $w\sim y',x,z$. If $w\in W_a,$ then the convexity of $W_a$
implies that $w\sim y$, and we are done because $w\sim x,z,y$ and
$w\in \partial W_a$. Thus $w\in W_{a^*}$ and we are in conditions of
Case 1 because we can set $x'=z'=w$.  This establishes (TC) for
$G(\partial W_a)$.
\end{proof}

\begin{lemma} \label{isometric_qc} $G(\partial W_a)$ satisfies the quadrangle condition (QC).
\end{lemma}

\begin{proof} Let $x,y,z_1,z_2\in \partial W_a$ such that $x\sim z_1,z_2,$ $z_1\nsim z_2,$ and $z_1,z_2\in I(x,y).$ Let $k=d(x,y).$ We proceed by induction on $k$.
By (QC) in $G$ there exists a vertex $v\sim z_1,z_2$ having distance $k-2$ to $y$. If $v\in \partial W_a$, then we are done. Thus, suppose that $v\notin \partial W_a$. If $v\nsim y,$ then by Lemma \ref{isometric_aux} there exists a neighbor $w$ of $y$ in $\partial W_a\cap I(y,v).$ Since $d(w,x)=k-1$ and $z_1,z_2\in I(x,w)$ by induction assumption there exists $u\sim z_1,z_2$ such that $u\in I(w,z_1)\cap I(w,z_2)\cap \partial W_a$. Since $I(w,z_i)\subset I(y,z_i), i=1,2,$ we are done.

So, suppose $v\sim y,$ i.e., $k=3$. Let $x',z'_1,z'_2,y'$ be neighbors in $W_{a^*}$ of $x,z_1,z_2,y,$ respectively. Suppose that first $x'$ and $y'$ are selected as close as
possible, and then the vertices $z'_1$ and $z'_2$ are chosen to satisfy the conditions of Lemma  \ref{isometric_neighbors}. Hence $z'_1$ and $z'_2$ either coincide or are adjacent to
$x'$.  Since $W_a$ is convex, necessarily $z'_1\ne z'_2.$

Similarly to the proof of Lemma \ref{isometric_aux} one can show that the choice of
$x',y'$ and $v\notin \partial W_a$ implies that $d(x',y')=2$ and
$d(y,x') = d(y',x)=3$. If $x'=z'_1$, then $v,y'\in I(y,x')$ and by
(QC) we will find in $W_{a^*}$ a vertex $v'\sim v,y',x',$ yielding
$v\in \partial W_a$. Thus we can suppose that $x'\ne z'_1,z'_2,$ i.e.,
that $x'\nsim z_1,z_2$ and $x\nsim z'_1,z'_2$. By $z_1\nsim z_2$ and
the second assertion of Proposition \ref{positioning-condition}, we have $z'_1\nsim
z'_2$. Since $v\nsim z'_1,z'_2,$ applying (PC) to the vertex $v$ and
to the square $xx'z'_1z_1$, we deduce that $d(v,x')=3$. Hence
$z'_1,z'_2\in I(x',v)$ and by (QC) for $G$ we can find a vertex
$v'\sim v,z'_1,z'_2$. The convexity of $W_{a^*}$ implies that $v'\in
W_{a*},$ yielding $v\in \partial W_a$. This establishes (QC) for
$G(\partial W_a)$.
\end{proof}

\begin{lemma} \label{isometric_convex} If $G$ does not contain any induced $W_5$, then $G(\partial W_a)$ is a convex subgraph of $G$.
\end{lemma}

\begin{proof} Let $x,y\in \partial W_a$ and $k=d(x,y)$. We proceed by induction on $k$. In view of induction assumption, to establish the assertion it suffices to show that any neighbor $v$ of $x$ in $I(x,y)$ belongs to $\partial W_a$. Suppose that there is $v$ in $I(x,y)$ with $v\sim x$ and $v \not \in\partial W_a$. By Lemma \ref{isometric_aux}, either $v\sim y$ or $k\ge 3$ and there exists $w\in I(y,v)\cap \partial W_a$ adjacent to $y$. In the second case,  since $d(w,x)=k-1$, by induction hypothesis $I(w,x)\subset \partial W_a$. Since $v\in I(w,x),$ we obtain a contradiction with the assumption that $v\notin \partial W_a$. Thus $k=2$. Pick $x',y'\in W_{a^*}$ such that $x\sim x', y\sim y'$ and $x',y'$ are as close as possible. Then similarly to the proof of Lemma \ref{isometric_aux} we can deduce that $x'\sim y'$ and $d(y,x')=d(y',x)=2.$  By (TC) for $v,x',y'$ there exists $z\sim v,x',y'$. Since $v\notin \partial W_a,$ necessarily $z\in W_{a}.$ The convexity of $W_a$ implies that $z\sim x,y$ and we
obtain a $W_5$ induced by $x,x',y',y,v,z$.
\end{proof}

\begin{lemma} \label{tc_sa} $G(S_a)$ satisfies (TC).
\end{lemma}

\begin{proof} Let $x,y,z\in S_a$ such that $x\sim z$ and
  $d(x,y)=d(y,z)=k.$ In view of Lemma \ref{isometric_tc} we can
  suppose that $x,y,z$ do not belong to the same set $\partial W_a$ or
  $\partial W_{a^*}$. First suppose that $x\in W_a$ and $z\in
  W_{a^*}$. Let $y\in W_a$. By (TC) in $G$ there exists $v\sim x,z$
  with $d(v,y)=k-1$. Then necessarily $v\in \partial W_a,$ and we are
  done. So, suppose $x,z\in \partial W_a$ and $y\in \partial
  W_{a^*}$. Let $v\sim x,z$ with $d(v,y)=k-1$ provided by (TC) in
  $G$. If $v\in \partial W_a \cup \partial W_{a^*}$, then we are done. Otherwise,
  $v\in W_a\setminus \partial W_a$. Let $P$ be any shortest path between $y$ and $v$.
  Since $\partial W_a$ separates $W_a\setminus \partial W_a$ from $W_{a^*}$, $P$ necessarily contains a vertex $w$ belonging to $\partial W_a$. Then $w\in I(y,v)\subset I(y,x)\cap I(y,z)$ and $d(w,x)=d(w,z).$ By (TC) for $G(\partial W_a)$ (Lemma \ref{isometric_tc}), there exists $u\in \partial W_a,$ $u\sim x,z$ and $d(w,u)=d(w,x)-1$. Since $u\in I(y,x)\cap I(y,z),$ we are done.
\end{proof}

\begin{lemma} \label{qc_sa} $G(S_a)$ satisfies (QC).
\end{lemma}

\begin{proof} Let $x,y,z_1,z_2\in S_a$ such that $z_1,z_2\in I(x,y), z_1\nsim z_2,$ and $x\sim z_1,z_2$. If $x,y\in \partial W_a$ (respectively, $x,y \in \partial W_a^*$),
then $z_1,z_2 \in \partial W_a$ (respectively, $z_1,z_2 \in W_a*$)
by convexity of $W_a$, and (QC) follows from Lemma \ref{isometric_qc}.
Let $x\in \partial W_a$ and $y\in \partial W_{a^*}$.  Let $v$ be a
vertex provided by (QC) for $G$: $v$ is adjacent to $z_1,z_2$ and is
one step closer to $y$ than $z_1$ and $z_2$. Since $W_{a^*}$ is
convex, the vertices $z_1,z_2$ cannot both belong to $W_{a^*}$. Let
$z_1\in \partial W_a$. If $z_2\in \partial W_{a^*},$ then since $v\sim
z_1,z_2$ we immediately conclude that $v\in S_a$. Thus suppose that
$z_1,z_2\in \partial W_a$ and that $v\in W_a\setminus \partial
W_a$. Let $P$ be any shortest path between $y$ and $v$. Then
necessarily $P$ contains a vertex $w$ belonging to $\partial
W_a$. Then $w\in I(y,v)\subset I(y,z_1)\cap I(y,z_2) \subset I(y,x)$
and $d(w,x) - 1 =d(w,z_1) =d(w,z_2)$. By QC($w$) for vertices
$x,z_1,z_2$ in $G(\partial W_a),$ there exists a vertex $u\in \partial
W_a$, such that $u\sim z_1,z_2$ and $d(w,u)=d(w,z_1)-1$. Since $u\in
I(y,z_1)\cap I(y,z_2),$ we are done.
\end{proof}

\begin{lemma} \label{thick_boundary} If $G$ is a thick $L_1$--weakly modular graph, then  $W_a=\partial W_a$ and $W_{a^*}=\partial W_{a*}$.
In particular, the boundaries $\partial W_a$ and $\partial W_{a*}$ are thick convex subgraphs of $G$.
\end{lemma}

\begin{proof} Suppose by way of contradiction that $\partial W_a$ is a proper subset of $W_a$. Then, since $W_a$ is convex  and thus induces a connected subgraph of $G,$ we can find three vertices $x,y,z$ of $G$ such that $y\in W_a\setminus \partial W_a,$ $z\in \partial W_a,$ $x\in \partial W_{a^*},$ and $z\sim x,y$. Since $G$ is a thick graph, the vertices $x$ and $y$ belong to a square $C=xv'yv''$ of $G$ (one of the vertices $v',v''$ may coincide with $z$). Since the sets  $W_a$ and $W_{a^*}$ are convex, necessarily one of the vertices $v',v''$ belongs to $W_{a^*},$ whence $y\in \partial W_a,$ a contradiction.
\end{proof}

From Lemmata \ref{isometric}, ~\ref{isometric_tc}, and \ref{isometric_qc}
we obtain that $G(\partial W_a)$ is an isometric weakly modular
subgraph of $G$. Since $G$ is pre-median, $G(\partial W_a)$ is
pre-median as well. Moreover, if $G$ does not contain induced $W_5,$
then Lemma \ref{isometric_convex} implies that $G(\partial W_a)$ is a
convex subgraph of $G$. The fact that $G(\partial W_a)$ and
$G(\partial W_{a^*})$ are isometric subgraphs of $G,$ immediately
establishes that $G(S_a),G(W_a\cup \partial W_{a*}),$ and
$G(W_{a*}\cup \partial W_a)$ are also isometric subgraphs of $G$. From
Lemmata \ref{tc_sa} and \ref{qc_sa} we obtain that $G(S_a)$ is a
pre-median subgraph of $G$.

Now, we will show that $G(W_a\cup \partial W_{a*})$ and $G(W_{a*}\cup
\partial W_a)$ are pre-median graphs. To prove (TC) for $G(W_a\cup \partial W_{a*})$, pick $x,y,z\in W_a\cup \partial W_{a*}$ such that $d(x,y)=d(z,y)=k$ and $x\sim z$.
Suppose by way of contradiction that all common neighbors $y_0$  of $x$ and $z$ with $d(y,y_0)=k-1$  belongs to $W_{a*}\setminus \partial W_{a*}$
(that such vertices $y_0$ exist follows from (TC) for $G$). Then necessarily $x,z$ belong to $\partial W_{a*}$.
Since $G(S_a)$ is pre-median, we have $y\in W_a\setminus \partial W_a$. Since  $G(\partial W_{a^*})$  separates $y$ from $y_0$, any shortest path between
$y$ and $y_0$ contains a vertex $y'\in \partial W_{a^*}$. Since $d(y',x)=d(y',z):=k'$, by (TC) in $G(W_{a^*})$ we will find a common neighbor $y'_0$ of $x,z$ belonging to
 $W_{a^*}$. Then one can easily see that $d(y,y'_0)=k-1,$ contrary to our assumption. This establishes (TC) for  $G(W_a\cup \partial W_{a*})$. The quadrangle condition
(QC) for $G(W_a\cup \partial W_{a*})$
can be proved in a similar way. This shows that $G(W_a\cup \partial W_{a*})$ and $G(W_{a*}\cup
\partial W_a)$ are pre-median graphs. Finally, the last assertion of the theorem follows
from Lemma \ref{thick_boundary}.
\end{proof}

As we noticed already in Lemma \ref{lemma:linksl1}, the links of vertices of isometric subgraphs $G$ of
half-cubes are line-graphs.  In particular,
the links of vertices  do not contain $K_{1,3}$ (i.e., $G$ does not
contain propellers) and $W_5$. This link condition does not
characterize pre-median isometric subgraphs of half-cubes. Let $S_7$
be the graph consisting of 7 vertices $a,b,c_0,c_1,c_2,c_3,c_4$ where
$ac_1c_2, c_1c_2c_3, c_2c_3c_4,$ and $c_3c_4b$ are 4 triangles and
$c_0$ is adjacent to the 4 vertices  $c_1,c_2,c_3,c_4$. This graph
$S_7$ is pre-median, the links of vertices are line-graphs, but $S_7$ does
not admit an isometric embedding into a half-cube because the interval
$I(a,b)$ is not convex: $c_1,c_4\in I(a,b),$ $c_0\in I(c_1,c_4),$ but
$c_0\notin I(a,b)$ (intervals in all $L_1$--graphs are convex).

Still, we conjecture that isometric subgraphs of half-cubes can be characterized in the following local-to-global manner (and possibly via a compact
list of forbidden isometric subgraphs):

\begin{conjecture}  A weakly modular graph $G$ is isometrically embeddable into a half-cube if and only if all subgraphs induced by balls of radius 2 of $G$
are isometrically embeddable into a half-cube.
\end{conjecture}

\section{$C(G)$ is contractible}\label{top-space}
Let $G$ be an $L_1$--weakly modular graph which admits a scale embedding into a hypercube $H(X)$ for a countable set $X$.
In this section, we present a proof of Theorem \ref{contractible} that the
topological space $C(G)$ is contractible. First we formally define $C(G)$.

\subsection{Definition of ${\mathcal C}(\Gamma)$}  Let $\Gamma_i, i\in \Lambda,$ be a set of weakly modular graphs (indexed by a set $\Lambda$) and let
$\Gamma$ be a weak Cartesian product, i.e., a connected component of $\prod_{i \in \Lambda} \Gamma_i$. Then $\Gamma$ is also a weakly
modular graph. For each $\Gamma_i, i\in \Lambda$,
let $\mathcal{C}(\Gamma_i)$ denote the set of all {\it finite thick isometric weakly modular subgraphs} of $\Gamma_i$. Analogously,
define $\mathcal{C}(\Gamma)$ by taking all finite subgraphs $H$ of $\Gamma$ which are weak Cartesian products of
$\prod_{i \in \Lambda} H_i$, where each $H_i$ belongs
to $\mathcal{C}(\Gamma_i)$. The subgraphs of $\mathcal{C}(\Gamma)$ can be characterized in the following way:

\begin{proposition} \label{thick_Gamma} $H$ belongs to ${\mathcal C}(\Gamma)$ if and only if $H$ is a finite thick isometric weakly modular subgraph of $\Gamma$.
\end{proposition}

\begin{proof} If $H\in \mathcal{C}(\Gamma)$, then $H$ is a weak Cartesian product of $\prod_{i \in \Lambda} H_i$. Since each $H_i$ is a thick weakly modular graph, $H$ is
also a thick weakly modular graph.

Conversely, let $H$ be a finite thick weakly modular isometric subgraph of $\Gamma$. For a vertex $v_i\in V(\Gamma_i)$, let $\Gamma(v_i)$
denote the subgraph of $\Gamma$ induced by all vertices of $\Gamma$ having $v_i$ as their $i$th coordinate and call $\Gamma(v_i)$  the $v_i$--{\it fiber} of $\Gamma$.
$\Gamma(v_i)$  can be viewed as a weak Cartesian product of $\{v_i\}\times \prod_{j\in \Lambda\setminus \{ i\}} \Gamma_j$. Let $H(v_i)$ denote the intersection of $\Gamma(v_i)$
with $V(H)$. Let $H_i$ be the projection of $H$ on $\Gamma_i,$ $i\in \Lambda$: namely, $H_i$ is the subgraph of $\Gamma_i$ induced by all vertices $v_i$ of $\Gamma_i$
such that $H(v_i)$ is nonempty.

Notice that each fiber $\Gamma(v_i)$ is a convex subgraph of $\Gamma$ (as a weak Cartesian product of $\Gamma_j$ $(j\ne i)$ and $\{ v_i\}$, i.e., of convex subgraphs of factors) and $H(v_i)$ is a convex subgraph of $H$ (as an intersection of a convex subgraph of $\Gamma$ with an isometric subgraph of $\Gamma$). For the same reason, if $u_iv_i$ is an edge of $\Gamma_i$, then the subgraph of $\Gamma$ induced by $\Gamma(u_i)\cup \Gamma(v_i)$ is a convex subgraph of $\Gamma$ and if $u_iv_i$ is an edge of $H_i$, then $H(u_i)\cup H(v_i)$ induces a convex subgraph of $H$. This implies that at least one edge $uv$ of $H$ is projected to $u_iv_i$. Indeed, if $u'$ is a vertex of $H(u_i)$ and $v'$ is a vertex of $H(v_i)$, then any shortest $(u',v')$--path $P$ of $H$ is contained in $H(u_i)\cup H(v_i).$
Necessarily $P$ contains two adjacent vertices $u$ and $v$, $u$ from $H(u_i)$ and $v$ from $H(v_i)$.

For simplicity of notation, for a vertex $v$ of $\Gamma$ and index $i\in \Lambda$,
by $v_i$ we denote the projection of $v$ in $\Gamma_i$, i.e., $v_i$ is the $i$th coordinate of $v$. For any edge $uv$ of $\Gamma$ there exists a single $\Gamma_i$ such that $uv$ is projected to an edge $u_iv_i$ of $\Gamma_i$; for all other $\Gamma_j$, $uv$ is projected to a single vertex. In this case we will say that the label of $uv$ is $i$, and we will denote it by $\lambda(uv)=i$. From the elementary properties of Cartesian products it follows that the edges of any triangle  of $\Gamma$ have the same label. The edges of any square of $\Gamma$ either all have the same label (we will call such a square {\it genuine}) or they have two labels and each pair of opposite edges have the same label (we will call such a square {\it composite}). Moreover, if $xyzu$ is a composite square of $\Gamma$, then $I(x,z)=I(y,u)=\{ x,y,z,u\}$.

Next we will prove that each $H_i$ belongs to ${\mathcal C}(\Gamma_i)$.

\begin{lemma} \label{H_ithick} $H_i$ is thick for each $i\in \Lambda$.
\end{lemma}

\begin{proof}
Since $H$ is finite, each $H_i$ is finite as well, moreover only a finite number of $H_i$ are nontrivial. Now, we will show that each $H_i$ is thick. Pick two arbitrary vertices $u_i,v_i$  at distance 2 in $H_i$. Let $u$ and $v$ be two closest in $H$ vertices, $u$ from $H(u_i)$ and $v$ from $H(v_i)$. We assert that $d_H(u,v)=2$. Let $P=(u,w,w'\ldots,w'',v)$ be any shortest $(u,v)$--path in $H$. Since $H$ is an isometric subgraph of $\Gamma$, all vertices of $P$ between $w$ and $w''$ are all projected to the same common neighbor of $u_i$ and $v_i$ in $H_i$. Hence $\lambda(uw)=i$ and $\lambda(ww')=j$ for some $j\ne i$.  Since $H$ is thick, the vertices $u$ and $w'$ belong to a square $C=upw'q$. If $w$ does not belong to $C$, then the square $C$ cannot be composite, i.e., all edges of $C$ have the same label. But then independently of whether $w$ is adjacent to $p$ and/or to $q$, the edges of $C$ as well as $uw$ and $ww'$ will have the same label, which is impossible. Thus $w$ is a vertex of $C$, say $w=q$. Then $C=upw'w$ is a composite square, thus $p$ belongs to $H(u_i)$. Since $d_H(p,v)<d_H(u,v)$, we obtain a contradiction with the minimality choice of the pair $u,v$. Hence $d_H(u,v)=2$. Since $H$ is thick, the vertices $u$ and $v$ belong to a square $uu'vv'$. Since $d_H(u,v)=d_{\Gamma_i}(u_i,v_i)=2$, this square is genuine and its projection is a square $u_iu'_iv_iv'_i$ of $H_i$. This shows that $H_i$ is thick.
\end{proof}

We will prove that each $H_i$ is an isometric weakly modular subgraph of $\Gamma_i$ under a weaker condition. We will say that an isometric subgraph $H$ of $\Gamma$ is {\it almost-thick}
if any two incident edges $uv$ and $vw$ of $H$ with $\lambda(uv)\ne \lambda(vw)$ belong to a composite square of $H$.

\begin{lemma} \label{H_iweaklymodular} Let $H$ be an isometric almost-thick weakly modular subgraph of $\Gamma$.  Then for each $i\in \Lambda$, $H_i$ is an isometric weakly modular subgraph of $\Gamma_i$.
\end{lemma}

\begin{proof} That $H_i$ is an isometric subgraph of $\Gamma_i$ can be easily shown by noticing that each shortest path of $H$ is projected to a shortest path of $H_i$ since $H$ is an isometric subgraph of $\Gamma$. To verify (TC), let $u_i,v_i,w_i$ be three vertices of $H_i$ such that $d_{H_i}(w_i,u_i)=d_{H_i}(w_i,v_i)=k_i$ and $u_i\sim v_i$. Then in $H$ we can select two adjacent vertices $u$ and $v$, $u$ from $H(u_i)$ and $v$ from $H(v_i)$. Let $w$ be an arbitrary vertex of $H$ from $H(w_i)$. Since $d_{H_j}(u_j,w_j)=d_{H_j}(v_j,w_j)$ for any $H_j$ with  $j\ne i$ (because  $u$ and $v$ project to the same vertex of $\Gamma_j$) and $d_{H_i}(w_i,u_i)=d_{H_i}(w_i,v_i)$, we conclude that $d_H(u,w)=d_H(v,w):=k$. By (TC) for $H$, there exists a common neighbor $x$ of $u$ and $v$ at distance $k-1$ from $w$. Since $uvx$ is a triangle of $H$, all its edges are labeled $i$. Hence for all $j\ne i$ we have $x_j=u_j=v_j$ and $x_i$ is adjacent to $u_i$ and $v_i$. Since $d_H(x,w)=k-1$ and all factors $\Gamma_j$ with $j\ne i$ contribute with the same amount to each $d_H(u,w),d_H(v,w),$ and $d_H(x,w)$, we deduce that $d_{H_i}(x_i,w_i)=k_i-1$. This establishes (TC) for $H_i$.

Next we  establish (QC) for $H_i, i\in \Lambda$. Let $u_i,y_i,x_i,$ and $w_i$ be four vertices of $H_i$ such that $u_i\sim x_i,y_i,$ $d_{H_i}(u_i,w_i)=k_i+1,$ and $d_{H_i}(x_i,w_i)=d_{H_i}(y_i,w_i)=k_i$. In $H$ we can select vertices $x,y,u',$ and $u''$ such that $x\sim u'$, $y\sim u''$, and $x$ projects to $x_i$, $y$ projects to $y_i$, and $u',u''$ both project to $u_i$. Suppose that such a quadruplet is selected to minimize $d_H(u',u'')$. We assert that $u'=u''$. Suppose not, and let $u$ be a neighbor of $u'$ in $H$ on a shortest $(u',u'')$--path. Since $u',u''$ belong to $H(u_i)$ and $H(u_i)$ is a convex subgraph of $H$, the vertex $u$ also belongs to $H(u_i)$. Notice that $\lambda(xu')=\lambda(u''y)=i$ and $\lambda(u'u)=j$ for some $j\ne i$. Since $H$ is almost-thick, the incident edges $xu'$ and $u'u$ belong to a composite square $C$ of $H$.  Let $x'$ be the fourth vertex of $C$. Hence $\lambda(ux')=i$ and $\lambda(xx')=j$. This implies that the projection of $x'$ on $\Gamma_i$ is $x_i$, and we can replace $x$ by $x'$. Since $d_H(u,u'')<d_H(u',u''),$ we get a contradiction with the choice of the quadruplet $x,u',u'',y$. This shows that $u'=u''$. Denote this common neighbor of $x$ and $y$ by $u$. Then $\lambda(xu)=\lambda(yu)=i$. This shows that for any $j\ne i$, the vertices $x,u,y$ project to the same vertex of $H_j$. Therefore if $w$ is any vertex of $H(w_i)$ and $w_j$ is the projection of $w$ in $H_j$, then $d_{H_j}(w_j,x_j)=d_{H_j}(w_j,y_j)=d_{H_j}(w_j,u_j)$. Since $x_i,y_i\in I(u_i,w_i)$ in $H_i$, all this shows that $x,y\in I(u,w)$ in $H$. By (QC) in $H$ there exists a common neighbor $z$ of $x,y$ in $H$ which is one step closer to $w$. Since $\lambda(ux)=\lambda(uy)=i$, the square $xuyz$ is genuine and its projection in $H_i$ is the square $x_iu_iy_iz_i$. Since for all $H_j, j\ne i,$  $z$ has the same projection as $x,u,$ and $y,$  we conclude that $d_{H_i}(z_i,w_i)<d_{H_i}(x_i,w_i)$, establishing (QC) in $H_i$.
\end{proof}

\begin{lemma} \label{Hproduct} Let $H$ be an isometric almost-thick weakly modular subgraph of $\Gamma$. Then $H$ is a weak Cartesian product of $\prod_{i\in \Lambda}H_i$.
\end{lemma}

\begin{proof}
Let $H^*$ be the connected component (weak Cartesian product) of $\prod_{i\in \Lambda} H_i$ contained in $\Gamma$. We will show that $H$ coincide with $H^*$. Clearly, $H$ is included in $H^*$. To prove that any vertex  of $H^*$  belongs to $H$ it suffices to show that any vertex $v$ of $H^*$ having a neighbor $u$ in $H$ also belongs to $H$. Since $u$ and $v$ are adjacent in $\Gamma$, there exists an index $i$ such that $\lambda(uv)=i$ and $u_j=v_j$ for any $j\ne i$. Since $v_i$ is a vertex of $H_i,$ $H(v_i)$ is nonempty. Let $w$ be a closest to $v$ vertex of $H(v_i)$. Notice that $d_{\Gamma}(u,w)=d_{\Gamma}(v,w)+1$. To show that $v$ belongs to $H$ we proceed by induction on $d_{\Gamma}(v,w)$.  If $v = w$, we are done.  Now suppose that $d_{\Gamma}(v,w)\ge 1$. Let $x$ be a neighbor of $w$ on a shortest $(w,u)$--path of $H$. From the choice of $w$ we conclude that $x$ does not belong to $H(v_i)$, i.e., the label of the edge $wx$ is $i$. Notice that $d_{\Gamma}(x,v)\ge d_{\Gamma}(w,v)$, otherwise $x$ will belong to $H(v_i)$. If $d_{\Gamma}(x,v)=d_{\Gamma}(w,v)$, by (TC) for $\Gamma$ we will find a common neighbor $y$ of $x$ and $w$ one step closer to $v$. The vertex $y$ belongs to the fiber $\Gamma(v_i)$ because this fiber is convex and $y\in I(w,v)$, whence the label of the edge $yw$ cannot be $i$. On the other hand, the vertices $x,y,w$ constitute a triangle and the label of the edge $wx$ of this triangle is $i$, a contradiction. Finally, suppose that $d_{\Gamma}(x,v)>d_{\Gamma}(w,v)$ and let $z$ be a neighbor of $x$ on a shortest $(x,u)$--path of $H$. Then $w,z\in I(x,v)$ in $\Gamma$ and by (QC) for $\Gamma$ there exists a common neighbor $y$ of $w$ and $z$, one step closer to $v$. Then $y$ belongs to $\Gamma(v_i)$ by the convexity of this subgraph in $\Gamma$, thus the label of the edge $wy$ is not $i=\lambda(wx)$. Hence $C=xzyw$ is a composite square of $\Gamma$.  Since $\lambda(xz)=\lambda(yw)\ne \lambda(wx)$ and $H$ is almost-thick, necessarily $C$ is a composite square of $H$ and therefore $y$ is  a vertex of $H$. This contradicts the choice of $w$ as a closest to $v$ vertex of $H(v_i)$. This establishes that the vertex sets of $H$ and $H^*$ are equal. Since both $H$ and $H^*$ are (induced) subgraphs of $\Gamma$ with the same set of vertices, we obtain the equality between $H$ and $H^*$, concluding the proof of the proposition.
\end{proof}

To conclude the proof of Proposition \ref{thick_Gamma}, it remains to show that any thick isometric weakly modular subgraph $H$ of $\Gamma$ is almost-thick. Indeed, let $uv$ and $vw$ be two incident edges of $H$  with $\lambda(uv)\ne \lambda(vw)$.  Since  $\lambda(uv)\ne \lambda(vw)$, we have $d_H(u,w)=2$. Since $H$ is thick, the vertices $u$ and $w$ belong to a square $C$ of $H$. If $v\notin C$, then $I(u,w)\ne C$ and all edges of the subgraph induced by $I(u,w)$ will necessarily have the same label, leading to a contradiction. Thus $v\in C$ and the square $C$ is composite.
\end{proof}

\begin{Rem} If all $\Gamma_i, i\in \Lambda$ are pre-median graphs, then $\Gamma$ is a pre-median graph as well. If $H$ is a finite thick isometric weakly modular subgraph of $\Gamma$, then $H$ is pre-median and, by Lemma \ref{thick-pm}, $H$ is a weak Cartesian product of its primes $\prod_{j\in \Lambda'} H'_j$.  Each prime $H'_j$ of $H$ has a simply connected clique complex, thus all
its edges will have the same label $i$, and thus $H'_j$ is projected to a single factor $\Gamma_i$ of $\Gamma$. This projection is an isometric subgraph of $H_i$ (the projection of $H$ on $\Gamma_i$). It may happen that several primes of $H$ are projected to the same $\Gamma_i$. Proposition  \ref{thick_Gamma} indicates that $H_i$ is a weak Cartesian product of the primes of $H$ which project to $\Gamma_i$, but we will not prove it here.
\end{Rem}

\subsection{Definition of $C(\Gamma)$} From now on, let $\Gamma_i, i\in \Lambda,$ be a set of graphs  such that each $\Gamma_i$ is either a finite hyperoctahedron or a half-cube.
Suppose also that the dimension of all $\Gamma_i$ which are hyperoctahedra is uniformly bounded. Let $\Gamma$ be a  weak Cartesian product of $\prod_{i \in \Lambda} \Gamma_i$. Then $\Gamma$
has finite octahedron-dimension, thus $\Gamma$ has a  scale $M$ embedding $\varphi$ into a hypercube $H(X)$ for an appropriately chosen even integer $M:=2m$
as described next. The embedding $\varphi$ is obtained as the combination
of the scale $M$ embeddings $\varphi_i$
of $\Gamma_i, i\in \Lambda,$ into the  hypercubes $H(X_i)$. Thus $H(X)=\prod_{i\in \Lambda}H(X_i)$ and  $X$ is the disjoint union of the sets $X_i (i\in \Lambda)$.
The embeddings $\varphi_i$ of the factors $\Gamma_i$ into $H(X_i)$ are defined in the following way.
If $\Gamma_i$ is the half-cube $\frac{1}{2}H(Y_i)$, then let $X_i$ be the disjoint union of $m$ copies of $Y_i$ and $\varphi_i$
is obtained from the trivial scale 2 isometric embedding of $\Gamma_i$ into $H(Y_i)$ by repeating
each coordinate $m$ times. If $\Gamma_i$ is an $n$--octahedron $K_{n\times 2}$, then a scale $\alpha_n$ embedding of $K_{n\times 2}$
into the hypercube $H(Y_i)$ with $|Y_i|=2\alpha_n$ is described in \cite{DeLa}*{Subsection 7.4} (where $\alpha_n$
is $\binom{n-2}{\frac{n}{2}-1}$ %% ${n-2 \choose \frac{n}{2}-1}$
if $n$ is even and $2\binom{n-2}{\frac{n-3}{2}}$ %% $2{n-2 \choose \frac{n-3}{2}}$
if $n$ is odd). Let $M$ be a multiple of 2 and of all distinct
$\alpha_n$ (there is a finite number of them) over all hyperoctahedra from the list $\Gamma_i, i\in \Lambda$. Then a  scale $M$ embedding of $\Gamma_i$ into
the hypercube $H(X_i)$ can be obtained by repeating $\beta_n:=M/\alpha_n$ times the coordinates in the embedding of
$K_{n\times 2}$ into $H(Y_i)$ (thus $X_i$ is the disjoint union of $\beta_n$ copies of $Y_i$).

For each  subgraph $H\in \mathcal{C}(\Gamma)$
we will define a finite-dimensional convex Euclidean polytope  $[H]\subset {\mathbb R}^X$
such that $H$ is the 1-skeleton of $[H]$, showing that the graph $H$ is polyhedral  (recall
that a finite graph $F$ is {\it polyhedral} if $F$ is the 1-skeleton of some Euclidean
convex polyhedron).
For this, we will define first the polytopes $[H_i]\subset {\mathbb R}^{X_i}$ for $H_i\in \mathcal{C}(\Gamma_i)$, $i\in \Lambda$.
By Lemma \ref{thick-prime-pre-med} each $H_i$  is a thick subgraph of a hyperoctahedron or  belongs to ${\mathcal M}_4$.

First suppose that $\Gamma_i$ is a finite hyperoctahedron. Let $[\Gamma_i]$ be the convex hull in  ${\mathbb R}^{X_i}$ of
$\varphi_i(v)$ (viewed as $(0,1)$-vectors) over all vertices $v$ of $\Gamma_i$; recall that $\varphi_i$ is a scale $M$
isometric embedding of $\Gamma_i$ into the hypercube $H(X_i)$ defined above. Analogously, if $H_i$ is a thick subgraph of $\Gamma_i$,
then let $[H_i]$ be the convex hull in ${\mathbb R}^{X_i}$ of the incidence vectors of $\varphi_i(v)$ over all vertices $v$ of $H_i$.

\begin{Lem} \label{octahedron}  If $\Gamma_i$ is a finite hyperoctahedron, then for any thick subgraph $H_i$ of $\Gamma_i$,
the 1-skeleton of $[H_i]$ coincides with $H_i$. In particular, the 1-skeleton of $[\Gamma_i]$ is the hyperoctahedron $\Gamma_i$.
\end{Lem}
\begin{proof} First we prove that each edge $uv$ of $H_i$ is an edge (1-dimensional face) of $[H_i]$. We denote by  $\varphi_i(v)$
the corresponding subset of $X_i$ and, if $\varphi_i(v)=A$,
we denote by $\sigma(A)$ the incidence vector of the set $A$. Suppose without loss of generality that
$\varphi_i(u)=\emptyset$ (otherwise, we can obtain  an isometric embedding with this property by taking the symmetric difference of all
$\varphi_i(x)$ $(x\in V(\Gamma_i))$ with $\varphi_i(u)$). Let $|X_i|=2n_i$. Then $|\varphi_i(w)|=n_i$ for any neighbor of
$w$ of $u$ in $H_i$ (in particular, $|\varphi_i(v)|=n_i$) and $\varphi_i(u')=X_i$ for the possibly unique non-neighbor $u'$ of $u$
in $H_i$. Let $\varphi_i(v)=A$. If the line segment $[\sigma(\emptyset),\sigma(A)]$ is not an edge of the 1-skeleton of $[H_i]$,
then some point $p$ of $[\sigma(\emptyset),\sigma(A)]$ can be expressed as the convex combination of the incidence vectors of the
remaining vertices of $H_i$. Since the coordinates of $p$ are positive on $A$ and are zero on $X_i\setminus A$, this means that
in the convex combination we can have  non-zero coefficients only for the incidence vectors of sets of the form $\varphi_i(w)$ which are
subsets of $A$. Since each $\varphi_i(w)$ has the same or larger size (if $w=u'$) as $A$, this is possible only if
$\varphi_i(w)=A=\varphi_i(v)$, a  contradiction. This shows that indeed $H_i$ is a subgraph of the 1-skeleton of $[H_i]$.

To prove the converse, it suffices to show that if a pair $(u,v)$ is not an edge of $H_i$,
then the line segment $[u,v]$ is not an edge of $[H_i]$.
Since $H_i$ has diameter $2$ and is thick, $d(u,v) = 2$
and there is a nonadjacent pair $(u',v')$ in $H_i$ such that $u',v' \in I(u,v) \setminus \{u,v\}$.
As before, we can suppose without loss of generality that $\varphi_i(u)=\emptyset$ and $\varphi_i(v)=X_i$. Let $\varphi_i(u')=A$ and $\varphi_i(v')=B$.
Then  $|A|=|B|=n_i$. Since $d(u',v')=2,$ necessarily $A\cap B=\emptyset$ and $A\cup B=X_i$. Therefore, the half-integer point $(\frac{1}{2},\ldots,\frac{1}{2})$ of ${\mathbb R}^{X_i}$ belongs to both segments
$[\sigma(\emptyset),\sigma(X_i)]$ and $[\sigma(A),\sigma(B)]$, establishing that the pair $(u,v)$ is not an edge of the 1-skeleton of $[H_i]$.
\end{proof}

Now suppose that $\Gamma_i$ is a half-cube $\frac{1}{2}H(Y_i)$ and let $\varphi_i$ be the scale $M$ isometric embedding of $\Gamma_i$ into the hypercube $H(X_i)$
defined above.   By Lemma \ref{thick-prime-pre-med}, each $H_i\in \mathcal{C}(\Gamma_i)$ is a basis graph of an even $\triangle$--matroid. According to Edmonds \cite{Ed} and Gelfand et al. \cite{GeGoMcPhSe},
a {\it basis matroid polyhedron} is the convex hull of characteristic vectors of bases of a matroid. Basis graphs of matroids and of
even $\triangle$--matroids are polyhedral; in fact,  the following sharper results hold:

\begin{proposition} \label{basis-polyhedron} \cites{Che_bas,GeGoMcPhSe} For a collection ${\mathcal A}$ of subsets of equal size (respectively, of even size) of an $n$--element set,
the convex hull of characteristic vectors of ${\mathcal A}$  is a basis polytope of a
matroid (respectively, of an even $\triangle$--matroid) if and only if its 1-skeleton is isomorphic to
the basis graph of the family $\mathcal A$.
\end{proposition}

From this result and Proposition  \ref{prime-to-matroids} we infer that any $H_i\in {\mathcal C}(\Gamma_i)$ is polyhedral.   We  will denote by $[H_i]$ the copy of the basis  polyhedron of $H_i$ realized as the convex hull in ${\mathbb R}^{X_i}$ of the incidence vectors of $\varphi_i(v)$  over all vertices of $H_i$. Similarly to the proof of Proposition \ref{basis-polyhedron} and  Lemma \ref{octahedron} and taking into account the way $\varphi_i$ was defined, one can
show that $H_i$ is the 1-skeleton of $[H_i]$. Notice that  $[H_i]$ is a finite-dimensional Euclidean polytope of ${\mathbb R}^{X_i}$. Let $C(\Gamma_i)$ be the subspace of ${\mathbb R}^{X_i}$
which is the union of all $[H_i]$ over $H_i\in \mathcal{C}(\Gamma_i)$. (If $X_i$ is finite, then $C(\Gamma_i)$ is just the convex hull of ${\mathbb R}^{X_i}$ defined by the vertices of $\Gamma_i$).

For $H\in \mathcal{C}(\Gamma)$ represented as $H = \prod_{i \in \Lambda} H_i$ with $H_i\in \mathcal{C}(\Gamma_i)$, let $[H]$ be a weak Cartesian product of $[H_i]$ for all $i \in \Lambda$. Note that each $[H]$ is a finite dimensional Euclidean polytope since the number of nontrivial summands $[H_i]$ different from a single point is finite. Let $C(\Gamma)$ be the union of  $[H]$ over all $H \in {\mathcal C}(\Gamma)$:
$$C(\Gamma)=\bigcup \{ [H]: H\in \mathcal{C}(\Gamma)\}.$$
Then $C(\Gamma)$ is a subspace of $\prod_{i \in \Lambda} C(\Gamma_i)$, which itself is a subspace of $\RR^X=\prod_{i \in \Lambda} \RR^{X_i}$.

\subsection{Definition of $C_{\Gamma}(G)$ and $C(G)$}\label{C_Gamma}
Let $G$ be an isometric weakly modular subgraph of $\Gamma$. Then we define $\mathcal{C}(G)$ as the set of all subgraphs $H$ of $G$ which
belong to $\mathcal{C}(\Gamma)$.  By Proposition \ref{thick_Gamma}, $\mathcal{C}(G)$ can be also defined as the set of all finite thick
isometric weakly modular
subgraphs of $G$. Let $C_{\Gamma}(G)$ be the union of all $[H]$ over $H \in {\mathcal C}(G)$:
$$C_{\Gamma}(G)=\bigcup \{ [H]: H\in \mathcal{C}(G)\}.$$
Clearly $\mathcal{C}(G)\subseteq \mathcal{C}(\Gamma)$ and $C_{\Gamma}(G)\subseteq C(\Gamma)$. Moreover, if $G'$ is an isometric weakly modular subgraph of $G$, then
$\mathcal{C}(G')\subseteq \mathcal{C}(G)$ and $C_{\Gamma}(G')\subseteq C_{\Gamma}(G)$.

Let $G$ be an $L_1$--weakly modular graph admitting a scale embedding into a hypercube.  We will define $C(G)$ as $C_{\Gamma}(G)$ for a canonical choice of $\Gamma$.
Let $\{ G_i \}_{i \in \Lambda}$ be the set of all primes of $G$. By Theorem \ref{premedian_Chastand}(iii),
$G$ is an isometric subgraph of a weak Cartesian product of  $\prod_{i \in \Lambda} G_i$. Since $G$ has a scale embedding into a hypercube,
by Lemma \ref{l_1-pre-med} each $G_i$ is either a subgraph of a finite hyperoctahedron $\Gamma^*_i$ or an
isometric subgraph of a half-cube $\Gamma^*_i$. Thus $G$ is also an isometric
subgraph of a connected component $\Gamma^*$ of a weak Cartesian product  $\prod_{i \in \Lambda} \Gamma^*_i$. Since $G$ and all its primes have finite octahedron-dimension,
$\Gamma^*$ and $\Gamma^*_i$ also have finite octahedron-dimension. Hence, $C(\Gamma^*)$ and  $C_{\Gamma^*}(G)$ are well-defined as above. We set $C(G):=C_{\Gamma^*}(G)$ for
this particular choice of $\Gamma^*$.

\begin{Rem}
If $G_i$ is a thin subgraph of a hyperoctahedron $\Gamma^*_i$, then as a suboctahedron of $\Gamma^*_i$, $G_i$ contains only two non-adjacent vertices
$x,y$, all other vertices $S_i$ of $G_i$
are pairwise adjacent and adjacent to $x$  and $y$. Then $G_i$ contains two maximal thick subgraphs $H'_i$ and
$H''_i$, one containing $x$ and $S_i$ and another one containing $y$ and  $S_i$. Then $[H'_i]$ and
$[H''_i]$ are two simplices glued together along the simplex $[S_i]$ defined by $S_i$.
Then $C(G_i)$ is the union of the two simplices $[H'_i]$ and $[H''_i]$ glued along $[S_i]$; $C(G_i)$ can be viewed as
a bipyramid of ${\mathbb R}^{X_i}$.
\end{Rem}

\subsection{Proof of contractibility of $C(G)$}
To prove that $C(G)$ is contractible, we will prove a more general assertion:

\begin{proposition} \label{contractibleC_Gamma}  Let $G$ be an $L_1$--weakly modular graph
which is an isometric subgraph of $\Gamma=\prod_{i\in \Lambda} \Gamma_i$ and suppose that $G$ has a scale embedding into a hypercube $H(Z)$ defined on a countable set $Z$.
Then $C_{\Gamma}(G)$ is contractible.
\end{proposition}

To prove that the topological space $C_{\Gamma}(G)$
is contractible, we will show that it is sufficient to establish this result
for finite graphs. Let $\psi$ be a scale isometric  embedding of $G$ into $H(Z)$ and suppose without loss of generality
that  $Z=\{ 1,2,3,\ldots\}$. Let $Z_i=\{ 1,2,\ldots,i\}$ and let
$H(Z_i)$ be the $i$--cube defined by all finite subsets of $Z_i$. Then $H(Z)$ is a directed union of the cubes $H(Z_i)$, $i=1,2,\ldots$.
Let $R_i$ denote the subgraph of $G$ induced by all vertices $v$ of $G$ such that $\psi(v)$ belongs to $Z_i$. Then $R_i$ is a convex subgraph of
$G$ because $R_i$ is the intersection of $G$ with the convex subgraph $H(Z_i)$ of $H(Z)$. This shows that each $R_i$ is a finite $L_1$--weakly modular graph
and $G$ is the directed union of $R_i, i=1,2,\ldots$.  Consequently, $C_{\Gamma}(G)$  is a directed union of
$C_{\Gamma}(R_i)$  for finite graphs $R_i$ isometrically embedded in $\Gamma$ (and admitting  a scale embedding in $H(Z)$). By the
classical theorem of Whitehead \cite{Hat}*{{Theorem 4.5}},
it suffices to show that each finite complex  $C_{\Gamma}(R_i)$  is contractible.

So, let $G$ be a finite weakly modular graph isometrically embedded in $\Gamma$. Let $\varphi$ be the scale $M$ isometric embedding of $\Gamma$
and $G$ into the hypercube  $H(X)$ which was defined above (and was used in the definition of $C(\Gamma)$ and $C_{\Gamma}(G)$).
We prove the contractibility of $C_{\Gamma}(G)$   by induction on the number of vertices of
$G$ by using the gluing lemma \cite{Bj}*{Lemma 10.3}. If $G$ is thick, then  $G\in {\mathcal C}(G)\subseteq {\mathcal C}(\Gamma)$ by Proposition  \ref{thick_Gamma}, thus
$G$ is a weak Cartesian product of  $G_i\in {\mathcal C}(\Gamma_i), i\in \Lambda$. Consequently, $C_{\Gamma}(G)=[G]=\prod_{i \in \Lambda} C_{\Gamma}(G_i)=\prod_{i \in \Lambda} [G_i]$
and $C_{\Gamma}(G)$ is contractible  as a convex polyhedron $[G]$.  So, suppose that $G$ contains a thin
2-interval $I(x,z)$. We distinguish two cases.

%% \medskip\noindent
%% {\bf Case 1.}
\begin{case-CG}
$G$ contains two vertices $x,z$ with $d(x,z)=2$ such that  $I(x,z)$ is thin and $x$ and
$z$ have at most three common neighbors $y,y',y''$.
\end{case-CG}

%% \medskip\noindent
Since $I(x,z)$ is thin, the vertices $y,y',y''$ are
pairwise adjacent. Since $\varphi$ is a scale $M$ isometric embedding
of $G$ into $H(X)$, the Hamming distance between $\varphi(x)$ and
$\varphi(z)$ is $2M$ and all other Hamming distances between the
vertices of $I(x,z)$ is $M$. We can suppose without loss of generality
that $\varphi(x)=\emptyset$.  Let $Z:=\varphi(z)$ and $Y:=\varphi(y),
Y':=\varphi(y'), Y'':=\varphi(y'')$. Then $|Z|=2M,$
$|Y|=|Y'|=|Y''|=M$, $Y\cup Y \cup Y'' \subseteq Z$, and $|Y\cap
Y'|=|Y\cap Y''|=|Y'\cap Y''|=\frac{M}{2}$.

We first show that there exists $a \in Z$ such that either $a \in Y
\cap Y' \cap Y''$, or $a \notin Y \cup Y' \cup Y''$. If $x$ and $z$
have two common neighbors $y,y'$ (respectively, one common neighbor $y$), then
any $a \in Y \cap Y'$ (respectively, any $a \in Y$) shows that we are in the
first case. Suppose now that $I(x,z) = \{y,y',y''\}$ and assume $Y
\cap Y' \cap Y''=\emptyset$. Since $|Y|=|Y'|=|Y''|=M$ and $|Y\cap
Y'|=|Y\cap Y''|=|Y'\cap Y''|=\frac{M}{2}$, we have that $|Y \cup Y'
\cup Y''| = \frac{3M}{2}$. Since $|Z| = 2M$, there exists $a \in Z$
such that $a \notin Y \cup Y' \cup Y''$. Consequently, there exists
$a$ such that $x \in W_{a^*}$ and $z,y,y',y''\in W_{a}$, or $z \in
W_{a^*}$ and $x,y,y',y''\in W_{a}$. Without loss of generality, assume
we are in the first case.

We claim that $z\notin \partial W_{a}$. Indeed, suppose by way of
contradiction that $z\sim u$ with $u\in W_{a^*}$. Since $W_{a^*}$ is
convex and $u\notin I(x,z)$ (because $I(x,z) \cap W_{a^*} = \{x\}$), we have
$d(x,u)=2$. Since $u,x\in W_{a^*}$ and $W_{a^*}$ is convex, $u\nsim
y,y',y''$. By (TC), there exists a vertex $v\ne y,y',y''$ such that
$v\sim x,u,z$, contrary to our assumption about $I(x,z)$. So $z\notin
\partial W_{a},$ i.e., $\partial W_{a}$ is a proper subset of $W_{a}$.

By Theorem \ref{boundary} the subgraphs $G_1,G_0,G_2$ of $G$ induced
by respectively the halfspace $W_{a}$, the boundary set $\partial
W_{a}$, and the union $W_{a^*}\cup \partial W_{a}$ are isometric
pre-median subgraphs of $G$, thus $G_1,G_0,G_2$ are $L_1$--weakly
modular graphs. Since $z\in W_{a}\setminus \partial
W_{a}$ and $x\in W_{a^*},$ $G_1$ and $G_2$ are proper
subgraphs of $G$ and $G_0$ is a proper subgraph of $G_1$ and $G_2$. Since
$G_1,G_2$ and $G_0$ are isometric subgraphs of $G$ and thus of $\Gamma$,
$C_{\Gamma}(G_1),C_{\Gamma}(G_2),$ and $C_{\Gamma}(G_0)$ are contractible
by induction assumption.

\begin{lemma} \label{gluing} $C_{\Gamma}(G)=C_{\Gamma}(G_1)\cup C_{\Gamma}(G_2)$ and $C_{\Gamma}(G_0)=C_{\Gamma}(G_1)\cap C_{\Gamma}(G_2).$
\end{lemma}

\begin{proof} By the last assertion of Theorem \ref{boundary} each
isometric thick pre-median subgraph $H$ of $G$ is contained in $G_1$,
in $G_2$, or in their intersection. Hence $[H]$ is contained in
$C_{\Gamma}(G_1)$ or in $C_{\Gamma}(G_2)$, therefore $C_{\Gamma}(G)\subseteq C_{\Gamma}(G_1)\cup
C_{\Gamma}(G_2)$. The converse inclusion $C_{\Gamma}(G_1)\cup
C_{\Gamma}(G_2)\subseteq C_{\Gamma}(G)$ is immediate. Moreover, any $H \in \mathcal{C}(G_0)$ is included in
$\mathcal{C}(G_1)$ and in $\mathcal{C}(G_2)$; consequently, $C_{\Gamma}(G_0)
\subseteq C_{\Gamma}(G_1) \cap C_{\Gamma}(G_2)$.

To prove the converse inclusion $C_{\Gamma}(G_1)\cap C_{\Gamma}(G_2)\subseteq C_{\Gamma}(G_0),$
pick any point $p\in C_{\Gamma}(G_1)\cap C_{\Gamma}(G_2)$. Then there exist two cells $[H_1]$ and $[H_2]$, first from $C_{\Gamma}(G_1)$ and
the second from $C_{\Gamma}(G_2)$, such that $p\in [H_1]\cap [H_2]$. Since $H_1$ is contained in $W_a$, the polytope $[H_1]$ is completely
included in the (coordinate) hyperplane $\mathcal H$  of ${\mathbb R}^X$ defined
by the equation $a=1$.   Thus the intersection of $[H_1]$ and $[H_2]$ is included in $\mathcal H$. Now, the hyperplane $\mathcal H$ is a
support hyperplane of the polyhedron $[H_2]$. Therefore the intersection $I$ of $\mathcal H$  with $[H_2]$ coincides with the convex
hull in ${\mathbb R}^X$ of all vertices of $H_2$ lying in  $\mathcal H$. Denote by $H'_2$ the subgraph of $H_2$ (and of $G$) induced by
those vertices. Then $H'_2$ is a thick isometric subgraph belonging to $G_0$. Indeed, $H'_2$ is the intersection of $H_2$
with the convex (in the sense of $G$) set $W_a$ of $G$. Since $H_2$ is thick and isometric, $H'_2$ is also thick and isometric. As a conclusion,
we obtain that $p\in [H_1]\cap [H_2]=[H_1]\cap [H'_2].$ Since $[H'_2]$ belongs to $C_{\Gamma}(G_0)$, we will obtain that $p\in C_{\Gamma}(G_0)$.
This establishes
the inclusion $C_{\Gamma}\cap C_{\Gamma}(G_2)\subseteq C_{\Gamma}(G_0)$.
\end{proof}

From Lemma \ref{gluing}  we conclude that $C_{\Gamma}(G)$ is
obtained from $C_{\Gamma}(G_1)$ and $C_{\Gamma}(G_2)$ by gluing along $C_{\Gamma}(G_0)$. By the gluing
lemma \cite{Bj}*{Lemma 10.3}, we will obtain that $C_{\Gamma}(G)$  is contractible. This
concludes the proof of Case 1.

\begin{case-CG}
%% \medskip\noindent
%% {\bf Case 2.}
For any thin 2-interval $I(x,z)$ of $G$, $I(x,z)\setminus \{ x,z\}$ contains at least
four pairwise adjacent vertices.
%%\medskip\noindent
\end{case-CG}

Recall that $G$ is a finite isometric subgraph of $\Gamma$ and that $\Gamma$ is a weak Cartesian product of finite hyperoctahedra or/and half-cubes $\Gamma_i, i\in \Lambda$.
As in the proof of Proposition \ref{thick_Gamma}, let $G_i$ denote the subgraph of $\Gamma_i$ induced by the projection of $G$ on $\Gamma_i$.
By the condition of Case 2, any two edges $uv$ and $vw$ with $\lambda(uv)\ne \lambda(vw)$ are necessarily contained in a composite square of $G$:
otherwise, we will obtain that $I(u,w)=\{ u,v,w\}$, which is impossible. By Lemma \ref{H_iweaklymodular} each $G_i$ is an isometric weakly modular subgraph of $\Gamma_i$ and
by Lemma \ref{Hproduct} $G$ is a weak Cartesian product of $\prod_{i\in \Lambda} G_i$. From the definition of cells of $C_{\Gamma}(G)$ we conclude that
$C_{\Gamma}(G)$ is the weak Cartesian product of $\prod_{i\in \Lambda} C_{\Gamma_i}(G_i)$. If each $G_i$ contains less vertices than $G$, then by the induction
hypothesis each $C_{\Gamma_i}(G_i)$ is contractible, and thus $C_{\Gamma}(G)$ is contractible as a weak Cartesian product of the contractible spaces
$C_{\Gamma_i}(G_i)$. Now suppose that some $G_i$ contains the same number of vertices as $G$. Then all other graphs $G_j,j\ne i,$ are trivial, i.e., consist of a single vertex.
Hence $G_i$ is isomorphic to $G$. By the condition of Case 2, $G$ (and thus $G_i$) contains a thin 2-interval $I(x,z)$ such that $I(x,z)\setminus \{ x,z\}$ contains at least
four pairwise adjacent vertices. By Lemma \ref{lemma:linksl1} the half-cubes cannot contain induced $K_6-e$,  thus necessarily $\Gamma_i$ is a hyperoctahedron.
As a subgraph of $\Gamma_i$, the graph $G_i$ coincides with its subgraph induced by $I(x,z)$.
Therefore $G_i$ is a thin subhyperoctahedron of $\Gamma_i$.  Consequently,  $C_{\Gamma_i}(G_i)$
(and hence $C_{\Gamma}(G)$)  is  a bipyramid consisting of two simplices (see also the remark at the end of Subsection \ref{C_Gamma}), hence $C_{\Gamma_i}(G_i)$ is contractible.
This concludes the proof of Case 2, of Proposition \ref{contractibleC_Gamma}, and of the first part  of Theorem \ref{contractible}. The second assertion of Theorem \ref{contractible}
(that $G$ coincides the union of 1--skeleta   of cells of $C(G)$) follows  from Lemma \ref{octahedron} and Proposition \ref{basis-polyhedron}.

\medskip
Notice that  in general $C_{\Gamma}(G)$ is not a cell complex: it may happen that $H',H''\in {\mathcal C}(G)$ and $H'$ is a subgraph of $H''$ but $[H']$ is not a face of $[H'']$ because
$[H']$ and  $[H'']$ are polyhedra of the same dimension. Let  ${\mathcal C}^*(G)$ denote the subset of ${\mathcal C}(G)$ consisting of (thick) convex subgraphs of $G$.

\begin{question} \label{cellulation} We conjecture that  $C^*(G)=\bigcup\{ [H]: H\in {\mathcal C}^*(G)\}$ defines a  cellulation of $C_{\Gamma}(G)$,
i.e., $C^*(G)$ is a subdivision of $C_{\Gamma}(G)$ into convex cells.
\end{question}

To prove this property, we need to show that (a) each isometric thick pre-median subgraph $H$ of $G$ is contained in a thick convex subgraph of $G$ and that (b) if two thick
convex subgraphs $H',H''$ of $G$ intersect in a (necessarily thick and convex) subgraph $H_0$, then $[H_0]$ is a face of $[H']$ and $[H'']$.

Notice that if $G$ is an isometric subgraph of a half-cube $\Gamma$, then all
2-faces of $C_{\Gamma}(G)$ are equilateral triangles or squares. Note that for
such a graph $G$, $C_{\Gamma}(G)$ is not always CAT(0): if $G$ is the 5-wheel
$W_5$, then $C_{\Gamma}(G)$ consists of five equilateral triangles sharing a
common vertex; this vertex has positive curvature.

\begin{question} \label{CAT0_premedian} We conjecture that if $G$ is an $l_1$--weakly
modular graph not containing $W_5$ (or, more particularly, an isometric
weakly modular subgraph of a half-cube), then $C_{\Gamma}(G)$ is a CAT(0) space.
\end{question}

We conclude this section with the following general open questions:

\begin{question} \label{polyhedral} (i) What thick ppm-graphs are polyhedral? In particular,
are the thick ppm-graphs without propellers polyhedral?
(ii) Define for an arbitrary pre-median graph $G$ a contractible PE cell complex $C(G)$
such that $G$ is the 1-skeleton of $C(G)$. In particular, for which pre-median graphs
the topological space $C(G)$ derived from the set ${\mathcal C}(G)$ of thick isometric
pre-median subgraphs of $G$ is contractible?
\end{question}

Notice that in the case of weakly bridged graphs, ${\mathcal C}(G)$ consists of cliques and
$C(G)$ coincides with the weakly systolic simplicial complex $X(G)$ of $G$. More generally,
for bucolic graphs $G$, ${\mathcal C}^*(G)$ consists of weak Cartesian products of cliques,
thus the cell complex $C^*(G)$ derived from  ${\mathcal C}^*(G)$ coincides with the bucolic
prism complex of $G$. One can also  show that in this case, the properties (a) and (b) easily
hold, thus  Question \ref{cellulation} has a positive answer.
Both weakly systolic and locally finite bucolic complexes are contractible \cite{BCC+}.
On the other hand, it is well known that the Schl\"afli graph
$G_{27}$ and the Gosset graph $G_{56}$ are polyhedral, therefore they will define  complexes
with a unique maximal cell, showing that $C(G)$ must have types of cells different
from hyperoctahedral and matroidal cells.

Finally notice that since the graphs $G$ from ${\mathcal M}_3$ and ${\mathcal  M}_4$ (as well as $G_{27}$ and $G_{56}$) are polyhedral,
they satisfy the fixed cell property in a trivial way: any automorphism of $G$ extends to an automorphism of $C(G)$ and
fixes the unique maximal cell of $C(G)$. For example, consider the automorphism of $J_{n,2n}$ which maps each $n$--set
to its complement. Then this automorphism  extends to an automorphism of the basis matroid
polyhedron of $J_{n,2n}$ and has a single fixed cell--the basis matroid polyhedron itself. We do not know if this fixed cell property
extends to all $l_1$--weakly modular graphs:

\begin{question} \label{fixed} Does any finite group of automorphisms of $C(G)$ of an $l_1$--weakly modular graph $G$
fixes a (maximal) cell of $C(G)$?
\end{question}

\chapter{Dual Polar Graphs}\label{s:dupol}

In this chapter, we investigate dual polar graphs, i.e., the collinearity graphs of dual polar spaces, whose definition was given in
Subsection \ref{prel:dps}. We refine the classical characterization of dual polar graphs given by P. Cameron \cite{Ca} and show that
dual polar  graphs constitute a natural subclass of weakly modular graphs.
We also present a completely different and simplified approach (based on our local-to-global characterization of weakly modular graphs)
to the characterization of locally dual polar spaces provided in a seminal paper by
A.\ Brouwer and A.\ Cohen \cite{BroCo}. Finally, these results  are used in the next chapter: in Chapter~\ref{sec:swm},
the dual polar gated subgraphs of swm-graphs define the cell structure of cell complexes of those graphs.

\section{Main results}

According to Cameron \cite{Ca}, the dual polar spaces can be characterized by the conditions
(A1)-(A5), rephrased  in \cite{BaCh} in the following (more suitable to our context) way:

\begin{theorem} \cite{Ca} \label{th:cameron_dual_polar} A graph $G$ is the collinearity
graph of a dual polar space $\Gamma$ of rank $n$  if and only if the following axioms are satisfied:
\begin{itemize}
\item[(A1)] for any point $p$ and any line $\ell$ of $\Gamma$ (i.e., maximal clique of $G$), there
is a unique point of $\ell$ nearest to $p$ in $G$;
\item[(A2)] $G$ has diameter $n$;
\item[(A3$\&$4)] the gated hull $\lgate u,v\rgate$
of two vertices $u,v$ at distance $2$ has diameter $2$;
\item[(A5)] for every pair of nonadjacent vertices $u,v$ and
every neighbor $x$ of $u$ in $I(u,v)$ there exists a neighbor $y$ of $v$ in $I(u,v)$
such that $d(u,v)=d(x,y)=d(u,y)+1=d(x,v)+1$.
\end{itemize}
\end{theorem}

The point-line geometries with collinearity graphs satisfying axiom (A1) are called {\it near polygons} \cite{ShuYa}. Lines of near polygons are exactly the
maximal cliques of their collinearity graphs. Thus, in graph-theoretical language, axiom (A1) is equivalent to asserting that each maximal clique is gated.
Notice that (A1) implies the triangle condition (TC).
In fact, the graphs of near polygons can be characterized as  the graphs in which (TC) is fulfilled and
the kite $K^-_4$ ($K_4$ minus one edge, see Figure~\ref{fig-K4-K33-}) does not occur as an induced subgraph.
The original formulation of axioms (A3) and (A4) in Cameron's paper \cite{Ca} is as follows (the remaining axioms (A1),(A2), and (A5) are the same):
\begin{itemize}
\item[(A3)] If $x$ and $y$ are points with $d(x,y)=2$ and $\Delta(x,y)$ is the smallest set of points containing $x,y,$ and any point collinear
with two of its points, then $\Delta(x,y)$ has diameter 2;
\item[(A4)] For $d\le n$, let $x,y,z$ be points with $d(x,y)=2,$ $d(x,z)=d(y,z)=d$. Then either there is a point $w$ joined to $x$ and $y$ with $d(z,w)=d-1$
or there is a point $w\in \Delta(x,y)$ with $d(z,w)=d-2$.
\end{itemize}

\begin{figure}[ht]
\begin{center}
\includegraphics[scale=0.7]{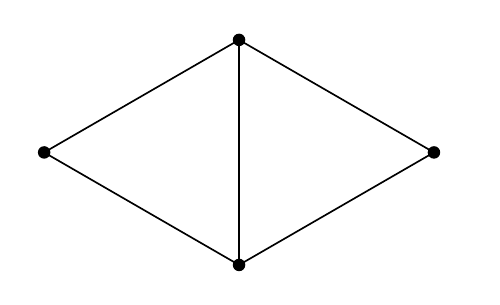} \includegraphics[scale=0.7]{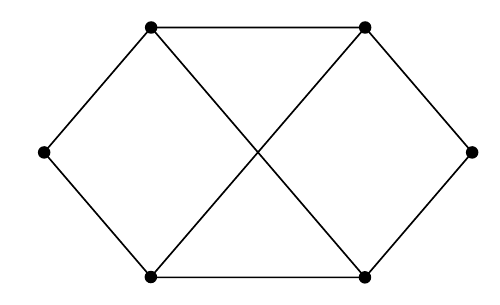}
\end{center}
\caption{A $K_4^-$ (left) and a $K_{3,3}^-$ (right)}\label{fig-K4-K33-}
\end{figure}

Axioms (A3) and (A4) imply that the collinearity graphs of dual polar spaces also satisfy the quadrangle condition (QC). Indeed, let $x,y$ be two nonadjacent
neighbors of $u$ in the interval $I(u,z)$ so that $d(u,z)=d+1$ and $d(x,z)=d(y,z)=d$. By (A4) either there exists a point $w\sim x,y$ with $d(z,w)=d-1$ and we are done or
there exists a point $w\in \Delta(x,y)$ with $d(z,w)=d-2$. In the second case, since $d(x,z)=d(y,z)=d$ and the diameter of $\Delta(x,y)$ is 2, we conclude that
$d(x,w)=d(y,w)=2$. This implies that  $w\in I(x,z)\cap I(y,z)\subset I(u,z),$ whence $d(u,w)=3$. Since $u,w\in \Delta(x,y)$, we obtain a contradiction with (A3).
Therefore the axioms (A1), (A3), and (A4) imply that dually polar graphs are weakly modular (this was first noticed in \cite{BaCh}).  The joint formulation  of Cameron's two axioms (A3) and (A4) from
Theorem \ref{th:cameron_dual_polar} rests
on Lemma \ref{lem-weakly-modular_gated_hull}  that  the gated hull of a connected subset $S$ in a weakly
modular graph $G$ is obtained as the smallest set including $S$
and containing any common neighbor of any two of its vertices. By this lemma, $\Delta(x,y)$ in axiom (A3) is the gated hull of $x$ and $y$, thus axioms in
Theorem \ref{th:cameron_dual_polar} imply (A3). To show that these axioms also imply (A4), let $w$ be the gate of $z$ in $\Delta(x,y)$. Since $\Delta(x,y)$ has diameter 2
and $d(x,z)=d(y,z)$, either $d(x,w)=d(y,w)=1$ or $d(x,w)=d(y,w)=2$ holds, establishing (A4).
Lemma \ref{lem-weakly-modular_gated_hull} also implies that the convex subspaces (i.e., subsets of points which induce convex subgraphs of the collinearity graph and are
subspaces of the incidence geometry) of a dual polar space are exactly the gated subgraphs of its collinearity graph.

\medskip
We call a (non-necessarily finite) graph $G$ a {\it dual polar graph} if it satisfies the axioms (A1),(A3),(A4), and (A5)
of Theorem \ref{th:cameron_dual_polar}, that is, we do not require finiteness of the diameter (axiom (A2)). Our first result of this chapter shows
that dual polar graphs constitute a natural subclass of weakly modular graphs and therefore can be characterized
in a local-to-global way. This also provides a simple characterization of dual polar graphs.
Recall that a graph $G$ is thick if each $2$--interval of $G$ contains an induced square.
Here are the main results of this chapter.

\begin{theorem} \label{th:dual-polar-graph}\label{TH:DUAL-POLAR-GRAPH} For a graph $G=(V,E)$ the following conditions are equivalent:
\label{t:lotoglo_dp}
\begin{itemize}
\item[(i)] $G$ is a dual polar graph;

\item[(ii)] $G$ is a thick weakly modular graph not containing induced  $K^-_4$  and isometric $K^-_{3,3}$;

\item[(iii)] $G$ is a thick locally weakly modular graph not containing induced  $K^-_4$ and isometric $K^-_{3,3}$
and the triangle-square complex $X \trsq(G)$ of $G$ is simply connected.
\end{itemize}
\end{theorem}

We call a graph \emph{locally dual polar} if it is thick, locally weakly modular, and
does not contain induced  $K^-_4$ and isometric $K^-_{3,3}$.

\begin{theorem}
\label{t:lotoglo_dp2}\label{T:LOTOGLO_DP2}
Let $G$ be a locally dual polar graph. Then the $1$--skeleton
$\tG:=\tX \trsq(G)^{(1)}$ of the universal cover $\tX \trsq(G)$ of
the triangle-square complex $X \trsq(G)$ of $G$ is a dual polar graph. If, moreover, $G$ is locally finite, then $\tG$ is a finite dual polar graph.
\end{theorem}

As an analogue of \cite{BroCo}*{Main Theorem (i)} we derive the following form of the local-to-global result.

\begin{theorem}
\label{t:lotoglo_dp3}\label{T:LOTOGLO_DP3}
Every locally finite locally dual polar graph $G$  is a quotient of a dual polar graph  by a group action with the minimal displacement at least $7$.
\end{theorem}

As we will show below, the conditions of Theorems \ref{th:dual-polar-graph} and \ref{t:lotoglo_dp3} are implied by those of
\cite{BroCo}*{Main Theorem (i)}, i.e., Theorems \ref{th:dual-polar-graph} and \ref{t:lotoglo_dp3} can be viewed as a sharpening of
Main Theorem (i) of \cite{BroCo}.

\section{Proof of Theorem \ref{th:dual-polar-graph}}

In view of the local-to-global characterization of triangle-square complexes of weakly modular graphs, conditions (ii) and (iii) are equivalent.
To show that (i)$\Rightarrow$(ii) notice that axioms (A1) and (A3$\&$4) imply that dual polar graphs are weakly modular graphs not containing isometric
 $K^-_4$  and $K^-_{3,3}$. Applying (A5) to $u,v$ at distance 2 and their common neighbor $x$,
we conclude that the induced path $(u,x,v)$ is included in a square, thus $I(u,v)$ contains a square.

Now, we prove that (ii)$\Rightarrow$(i). As noticed before, the axiom (A1) that any maximal clique of $G$ is gated is equivalent to the fact that $G$
satisfies the triangle condition (TC) and does not contain induced $K^-_4$.

\begin{lemma} \label{(A5)} $G$ satisfies the axiom (A5).
\end{lemma}

\begin{proof} We proceed by induction on $k=d(u,v)$. The case $k=2$ holds since
the $2$--intervals of $G$ are thick and since $G$ does not contain any
induced $K_4^-$. So, let $k\ge 3$. Let $x$ be any neighbor of $u$ in
$I(u,v)$ and let $v'$ be a neighbor of $v$ in $I(x,v)$. By induction
hypothesis, there exists a neighbor $y'$ of $v'$ in $I(u,v')$ such
that $k-1 = d(u,v')=d(x,y')=d(u,y')+1=d(x,v')+1$. Since $v'\in I(u,v)$
and $y'\in I(v',u)$, necessarily $y'\nsim v$. By the case $k=2$, the
induced $2$--path $(y',v',v)$ is included in a square. Let $y$ denotes
the fourth corner of this square.  Clearly, $y\in I(v,y')\subseteq
I(v,u)$. Thus $d(y,u)=k-1$. To show that $y$ satisfies (A5) it
suffices to show that $v\in I(y,x)$, i.e., that $d(y,x)=k$. By
contradiction, suppose that $d(y,x)<k$. We distinguish two cases.

\begin{case-A5}
%% \smallskip\noindent
%% {\bf Case 1.}
$d(y,x)=k-2$.
\end{case-A5}

%% \noindent
Then $v',y\in I(v,x)$ and $v\sim v',y$. From the definition of $y$, we have $y\nsim v'$.
By (QC), there
exists $z\sim y,v'$ such that $d(x,z)=k-3$, hence $z \neq y'$ and
$z\nsim y'$ (because $G$ does not contain induced $K^-_4$). Since
$z,y'\in I(v',u)$, $v'\sim y',z$, and $z\nsim y'$,
by (QC) there exists a vertex $s\sim y',z$ such
that $d(u,s)=k-3$ (note that if $k = 3$ then $z = x$ and $s = u$). But
then the vertices $s,z,y',v',y,v$ induce a forbidden isometric
subgraph $K^-_{3,3}$.

%% \smallskip\noindent
%% {\bf Case 2.}
\begin{case-A5}
  $d(y,x)=k-1$.
\end{case-A5}
%%\noindent

Since $d(x,y)=d(x,v)=k-1$ and $v\sim y$, by (TC) there exists a vertex $t\sim y,v$ such that $d(x,t)=k-2$.
Since $t\in I(v,x)$ and $x\in I(v,u)$, we conclude that $t\in I(v,u)$, thus $d(u,t)=k-1=d(u,y)$. By (TC),
there exists a vertex $s\sim y,t$ such that $d(u,s)=k-2$. But then the vertices $s,t,y,v$ induce a forbidden
$K^-_{4}$. This shows that $G$ satisfies the axiom (A5).
\end{proof}

\begin{lemma} \label{(A34)} $G$ satisfies the axioms (A3$\&$4).
\end{lemma}

\begin{proof} We will prove that in $G$ the gated hull $\lgate u,v\rgate$
of two vertices $u,v$ at distance $2$ has diameter $2$. We construct
$\lgate u,v\rgate$ iteratively (using the transfinite version of the procedure GATED-HULL; see Subsection~\ref{s:wmgra})
starting with the pair $\{ u,v\}$ and $H_0:=G(\{ u,v\})$.
Then for any ordinal $\alpha$ we define the subgraphs $H_{<\alpha}$ and $H_{\alpha}$
as in the transfinite version of the algorithm GATED-HULL. Let $v_{\alpha}$ be the vertex
$v$ which is added to $H_{<\alpha}$ to obtain $H_{\alpha}$.
Finally, let $K$ be the subgraph defined  in the procedure GATED-HULL.
Even if the input graph $H_0$ is not connected,
after the first iteration, the subgraph $H_1$ will be a 2-path of $G$ induced by $u$, $v$ and a common neighbor of $u$ and $v$.
Therefore by Lemma \ref{p:gate3}, $K$ is the gated hull $\lgate u,v\rgate$ of $u,v$.

Suppose  by way of contradiction  that the
diameter of $K$ is larger than 2. Let $\alpha$ be the first iteration after
which the subgraph $H_{\alpha}$ contains two vertices $x,y$ at distance 3 in
$G$. Such $\alpha$ necessarily exists because each $H_{\alpha}$ is connected: as noted above, $H_1$ is connected
and at each step $\alpha$ we add to $H_{<\alpha}$
a vertex having two neighbors in $H_{<\alpha}$. Let $d(x,y)=3$, where $x=v_{\alpha}$
and $y=v_{\beta}\in V(H_{<\alpha})$. Additionally, suppose that among all such
$y=v_{\beta}$ with $d(x,y)=3$ we selected the one with the minimal index
$\beta$. Suppose that $x=v_{\alpha}$ was added at step $\alpha$ because $x$ belongs to
a $2$--path $(a,x,b)$ with $a,b\in H_{<\alpha}, a\ne b$. Analogously, either
$y\in \{ u,v\}$ or we can suppose that $y=v_{\beta}$ was added at step $\beta$
because $y$ belongs to a $2$--path $(y',y,y'')$ with $y',y''\in H_{<\beta}$,
$y'\ne y''$.
From the choice of the pair $x,y$ we
conclude that $d(y,a)=d(y,b)=2$ and $1\le
d(a,y'),d(a,y''),d(b,y'),d(b,y'')\le 2$. If $a\sim b$, then by
(TC) there exists a vertex $c\sim y,a,b$ and we will
obtain a forbidden $K_4^-$ induced by $c,a,b,x$. Thus $a\nsim b$.

%% \medskip\noindent
%% {\bf Case 1.}
\begin{case-A34}
$y \notin \{u,v\}$, i.e., $y',y''$ are defined.
\end{case-A34}

%%\noindent
The choice of the pair $x,y$ implies that $d(x,y')=d(x,y'')=2$. If $y'\sim y''$, then by (TC) there exists a vertex $c\sim x,y',y''$ and we obtain a forbidden $K^-_4$  induced by $c,y,y',y''$. Thus $y'\nsim y''$. Now, we assert that each of the vertices $y',y''$ is adjacent to each of the vertices $a,b$. Suppose, by way of contradiction, that $y'\nsim a$. Then $d(y',a)=2$. Since $d(y',x)=2$, by (TC) there exists $z\sim y',a,x$. Since $d(x,y)=3$ and $x\sim z$, necessarily $z\nsim y$. Thus $d(y,z)=2=d(y,a)$. By (TC) there exists $s\sim y,z,a$. Since necessarily $s\nsim x$, the vertices $s,z,a,x$ induce a forbidden $K^-_4$. This contradiction shows that $y'\sim a,b$ and $y''\sim a,b$. But then  $y,y',y'',a,b,x$ induce a forbidden isometric $K^-_{3,3}$.

%% \medskip\noindent
%% {\bf Case 2.}
\begin{case-A34}
$y\in \{ u,v\}$, say $y=v$.
\end{case-A34}
%%\noindent

Since $d(v,x)=3$ and $d(v,a)=d(v,b)=2$, by (QC)
there exists at least one vertex $s\sim a,b,v$. Suppose there exists
another vertex $t\sim a,b,v$. If $s \sim t$, then the
vertices $v,s,t,a$ induce a $K_4^-$. If $s \nsim t$, then the vertices
$v,s,t,a,b,x$ induce an isometric $K^-_{3,3}$, a contradiction. Thus,
the vertices $v,a,b$ have a unique common neighbor $s$. By thickness
of $G$, there exist two (necessarily distinct and distinct from $s$) vertices $a',b'$ such
that $a'\sim a,v$ and $b'\sim b,v$. If $a'\sim b'$,  then by (TC)  there exists a vertex $z\sim a',b',x$ and we obtain a
forbidden $K^-_4$ induced by $v,a',b',z$. Thus $a'\nsim b'$. Notice
also that $s\nsim a',b'$, as otherwise $a,a',v,s$ (or $b,b',s,v$) would induce a forbidden $K_4^-$. We distinguish two subcases:

%% \medskip\noindent
%% {\bf Subcase 2.1.}
\begin{subcase-A34-2}
There exists $w\in I(u,v)\setminus \{ u,v,s\}$ such that $d(w,x)\le 2$.
\end{subcase-A34-2}
%\noindent

Clearly, $d(w,x) \leq 1$ is impossible because $d(x,v) = 3.$ Since $s,w$ are neighbors of $v$ in $I(v,x)$, necessarily $w\nsim s$,
otherwise by (TC)  there exists $t\sim w,s,x$ and we
get a $K^-_4$ induced by $v,w,s,t$. Since $s$ is the
unique common neighbor of $a,b$ and $v$, we can assume that $w \nsim
b$. By the choice of the pair $x,y=v$, we have $d(w,b)=d(v,b)=2$.  Thus, by
(TC)  there exists a vertex $z\sim v,w,b$. Since $d(v,x)=3$ and $v\sim z$, necessarily $d(z,x)=d(w,x)=2$. By
(TC)  there exists $t\sim w,z,x$ and the vertices
$t,w,z,v$ induce a forbidden $K^-_4$.

%% \medskip\noindent
%% {\bf Subcase 2.2.}
\begin{subcase-A34-2}
  For each $w\in I(u,v)\setminus \{ u,v,s\}$  we have $d(w,x)=3$.
\end{subcase-A34-2}
%% \noindent
Let $W=I(u,v)\setminus \{ u,v,s\}$. Thickness implies that
$W\ne\emptyset$. By our assumptions, for any vertex $w\in W$ we have
$d(w,a)=d(w,b)=2$.  Since $d(a,v)=d(b,v)=2$, we can apply (TC)  to the triples $a,v,w$ and $b,v,w$ and deduce that there
exist $p(w)\sim a,v,w$ and $q(w)\sim b,v,w$. If for some $w\in W$ we
have $p(w)\ne s$ (respectively, $q(w) \neq s$), then since $p(w)=a'\in I(a,v)$
(respectively, $q(w)=b'\in I(b,v)$), by what was shown before, $p(w)$ and
$q(w)$ are distinct non-adjacent vertices. Then $p(w),v,w,q(w)$ induce
a forbidden $K^-_4$.  Thus $p(w)=q(w)=s$ for every vertex $w\in W$. If
$s\in I(u,v)$, then $s\sim u$ and the vertices $u,s,w,v$
induce a $K^-_4$. If $s\notin I(u,v)$, then $W=I(u,v)\setminus \{ u,v\}$, and by thickness, the set $W$
contains at least two different and non-adjacent vertices $w',w''$.
Since $s\sim w',w''$ and $s\sim v$, the vertices $w',s,v,w''$ induce a
$K^-_4$.
\end{proof}

From Lemmata \ref{(A5)} and \ref{(A34)} we conclude that $G$ is a dual polar graph, completing the proof of Theorem \ref{th:dual-polar-graph}. \hfill $\square$

%\medskip\noindent
\begin{remark} Recall that a {\it quad} is a convex subspace  (i.e., a gated subgraph) of diameter $2$. Lemma \ref{(A34)} can be compared with the so-called Yanushka's lemma (see \cite{Sh}*{Theorem 8.2.4}) from \cite{ShuYa} about the existence
of quads in near polygons with thick lines. Our result does not impose any requirement on the size of lines (maximal
cliques)  but requires the quadrangle condition. Both proofs are different but both results assert that each pair of vertices at
distance $2$ is contained in a gated subspace of diameter $2$.
\end{remark}

\section{Proof of Theorems \ref{t:lotoglo_dp2} and \ref{t:lotoglo_dp3}}\label{subsec:proof_of_lotoglo_dp23}
To prove Theorem \ref{t:lotoglo_dp3}, note that  by the second assertion of Theorem~\ref{t:lotoglo_dp2}, $\tG$ is a finite dual polar graph
and  the graph $G$ is a quotient
of $\tG$ by the action of $\pi_1(X \trsq)$. By Theorem~\ref{t:lotoglo_G}, the minimal displacement
for this action is at least $7$.

Next, we will prove Theorem \ref{t:lotoglo_dp2}. We start with the proof of the first assertion. Let $G$ be a locally dual polar graph. Let $\tG$ be
the $1$--skeleton of the universal cover $\tX \trsq(G)$ of the triangle-square complex $X \trsq(G)$ of $G$.
By Theorem~\ref{t:lotoglo2}, $\tG$ is weakly modular. By Lemma \ref{covering_map}, $\tG$ is a thick graph. By Lemma~\ref{l:noincov}, $\tG$ does
not contain induced $K_4^-$. By Lemma~\ref{l:noiscov},
it does not contain isometric $K_{3,3}^-$. Therefore, by Theorem~\ref{t:lotoglo_dp},  $\tG$ is a  dual polar graph.

To prove the second assertion of Theorem \ref{t:lotoglo_dp2}, suppose that $G$ is a locally finite locally dual polar graph. Then $\tG$ is locally finite as well.
Therefore,  to show that $\tG$ is finite it suffices to prove that any locally finite dual polar graph  $H$ is finite.

Let $\mathcal{S}$ be the set of all maximal cliques of a dual polar graph $H$. Since $H$ is $K^-_4$--free,
$|X \cap Y| \leq 1$ for all $X,Y \in \mathcal{S}$. For a vertex $x$, let $\mathcal{S}(x)$
denote the set of all maximal cliques of $H$ containing $x$.
For two vertices $x,y$ of $H$, let $\mathcal{S}(x,y)$ denote
the set  of cliques $C$ of $\mathcal{S}(x)$ meeting $I(x,y) \setminus \{x\}$. Note that $\mathcal{S}(x,x)=\emptyset$.
A {\em special subgeometry} $X$ is a set of vertices represented as
\begin{equation}
X = \{  z\in V(H): {\mathcal S}(x,z) \subseteq {\mathcal S}(x,y) \}
\end{equation}
for some vertices $x,y$ of $H$. In the case when the diameter of $H$ is finite,
the original polar space is recovered from $H$ as follows:

\begin{Lem}[{\cite{Ca}*{Lemma 3}}]\label{lem:special_subgeometries}
Suppose that the diameter of $H$ is finite.
Then the set of all special subgeometries, ordered by
reverse inclusion, forms the subspace poset of
a (generalized) polar space.
\end{Lem}

Let $\mathcal{L}(x)$ be the subset of $2^{\mathcal{S}(x)}$
consisting of sets of the form  $\mathcal{S}(x,y)$ for some $y \in V(H)$.
Cameron \cite{Ca} showed that
the poset $\mathcal{L}(x)$ ordered by inclusion forms the subspace poset of a generalized projective
space of possibly infinite rank; his proof does not use (A2).  In particular, we have:

\begin{Lem}\cite{Ca} \label{l:Cameron} The principal ideal of ${\mathcal S}(x,y)$ in $\mathcal{L}(x)$
is a complemented modular lattice of rank $d(x,y)$.
\end{Lem}

The gated hull  $\lgate\bigcup_{C \in {\mathcal S}(x,y)} C \rgate$ will be simply denoted by $\lgate{\mathcal S}(x,y)\rgate$.

\begin{Lem} \label{l:gated_hull_dpg1} For any two vertices $x,y$ of a dual polar graph $H$,
$\lgate x,y \rgate=\lgate \mathcal{S}(x,y)\rgate$.
\end{Lem}

\begin{proof} First observe that all cliques  of $\mathcal{S}(x,y)$ must belong to $\lgate x,y \rgate$, thus
$\lgate \mathcal{S}(x,y)\rgate\subseteq \lgate x,y\rgate$. We prove the converse inclusion by induction
on $k:=d(x,y)$. Since $H$ is a thick graph satisfying the quadrangle condition, $y$ has two non-adjacent neighbors $y',y''\in I(x,y)$
at distance $k-1$ from $x$. Since $y\in I(y',y'')$, we conclude that $y\in \lgate x,y',y''\rgate$. By the induction assumption,
$\lgate x,y' \rgate=\lgate \mathcal{S}(x,y')\rgate$ and  $\lgate x,y'' \rgate=\lgate \mathcal{S}(x,y'')\rgate$. Since
$\mathcal{S}(x,y')\cup \mathcal{S}(x,y'')\subseteq \mathcal{S}(x,y')$, we conclude that  $y\in \lgate \mathcal{S}(x,y)\rgate$
and we are done.
\end{proof}

\begin{Lem} \label{l:gated_hull_dpg2} For any two vertices $x,y$ of a dual polar graph $H$,
$$\lgate x,y \rgate = \{  z \in V(H): \mathcal{S}(x,z) \subseteq  \mathcal{S}(x,y) \}.$$
In particular,  $\lgate x,y \rgate$ induces a dual polar graph of diameter $d(x,y)$.
\end{Lem}

\begin{proof} If $\mathcal{S}(x,z) \subseteq  \mathcal{S}(x,y)$, then from Lemma \ref{l:gated_hull_dpg1} we infer that
$\lgate x,z\rgate=\lgate \mathcal{S}(x,z)\rgate\subseteq \lgate \mathcal{S}(x,y)\rgate=\lgate x,y\rgate$, whence
$z\in \lgate x,y\rgate$.

To establish the converse inclusion, it suffices to show that the set $Z:=\{ z \in V(H): \mathcal{S}(x,z) \subseteq  \mathcal{S}(x,y) \}$
is gated. In view of Lemma \ref{lem-weakly-modular_gated_hull} it suffices to show that if $u,v\in Z$ and $w$ is any common neighbor of $u,v$,  then
$w$ belongs to $Z$.  This is obviously so if $w \in I(x,u)\cup I(x,v)$. Thus we can suppose that $d(x,w)\ge \max \{ d(x,u),d(x,v)\}$.

First, suppose that $d(x,w) = d(x,u)$. We assert that in this case ${\mathcal S} (x,w)\subseteq {\mathcal S}(x,u)$, yielding $w\in Z$.
Indeed, take any neighbor $a$ of $x$ in $I(x,w)$. Then either $a\in I(x,u)$ or $d(u,a)=d(u,x)$. In the first case we conclude that the maximal clique of ${\mathcal S}(x,w)$ containing $a$
also belong to ${\mathcal S}(x,u)$. In the second case, by (TC)  for $x,a,u$
there exists a common neighbor $a'$ of $x,a$ with $d(x,u)=d(a',u)-1$.  Since $H$ is $K^-_4$--free, this means that $a$ and $a'$ belong to the same maximal
clique of  $\mathcal{S}(x)$. Consequently,  $\mathcal{S}(x,w)\subseteq \mathcal{S}(x,u)$.

Now, suppose that $d(x,w)-1=d(x,u)=d(x,v)$.  This implies that
$\mathcal{S}(x,w)$ contains $\mathcal{S}(x,u)$ and $\mathcal{S}(x,v)$.
By Lemma \ref{l:Cameron}, in the lattice $\mathcal{L}(x)$ the element
$\mathcal{S}(x,w)$ covers both elements $\mathcal{S}(x,u)$ and
$\mathcal{S}(x,v)$ and the rank of $\mathcal{S}(x,w)$ is $d(x,w) =
d(x,u)+1 = d(x,v)+1$. Since the rank of $\mathcal{S}(x,u)$ and
$\mathcal{S}(x,v)$ is $d(x,u) = d(x,v)$, necessarily,
$\mathcal{S}(x,w)$ is the join of $\mathcal{S}(x,u)$ and
$\mathcal{S}(x,v)$.  Since $\mathcal{S}(x,y)$ is a common upper bound
of $\mathcal{S}(x,u)$ and $\mathcal{S}(x,v)$, $\mathcal{S}(x,y)$ is
also an upper bound of their join $\mathcal{S}(x,w)$, implying
$\mathcal{S}(x,w) \subseteq \mathcal{S}(x,y)$.  This shows that the
set $Z$ is gated, thus establishing that $\lgate x,y\rgate=Z$.

Since each $2$--interval of $\lgate x,y \rgate$ is thick, by Theorem~\ref{th:dual-polar-graph}, $\lgate x,y
\rgate$ induces a dual polar graph. By Lemma~\ref{l:Cameron}, for any
$z \in \lgate x,y \rgate$, since $\mathcal{S}(x,z) \subseteq
\mathcal{S}(x,y)$, we have $d(x,z) \leq d(x,y)$. We claim that this implies
that the diameter of $\lgate x,y \rgate$ is $d(x,y)$. Indeed, among
all diametral pairs $u,v$ of $\lgate x,y \rgate$, consider a diametral
pair $u,v$ such that $d(u,x)$ is minimum.
If $u \neq x$, consider a quasi-median $x',u',v'$ of $x,u,v$. If $u
\neq u'$, consider a neighbor $u''$ of $u$ in $I(u,u') \subseteq
I(u,x)\cap I(u,v)$. By (A5), there exists $v''\in I(u,v) \subseteq
\lgate x, y \rgate$ such that $d(u'',v'') = d(u,v)$, contradicting our
choice of $u$ since $d(u'',x) < d(u,x)$. Assume now that $u =
u'$. Note that by Lemma~\ref{lem-weakly-modular}, $d(x',v) = d(x',v')
+ d(v',v) = d(u',v')+d(v',v) = d(u',v) = d(u,v)$. Thus, by our choice
of $u$, we get that $u = x'$, i.e., $u \in I(x,v)$, contradicting our
choice of $u$. Consequently, for any diametral pair $u,v$ of $\lgate
x,y \rgate$, $d(u,v) \leq d(x,y)$ and we are done.
\end{proof}

Using  these arguments, one can also show:

\begin{Lem}\label{lem:d(x,y)=D}
If $H$ is a dual polar graph of finite diameter $D$, then
$d(x,y) = D$ if and only if $\lgate x,y \rgate = V(H)$.
\end{Lem}

\begin{proof}
Since the diameter of $\lgate x,y \rgate$ is $d(x,y)$, if $\lgate x,y
\rgate = V(H)$, then $d(x,y) = D$. Conversely, assume that $d(x,y) =
D$ and suppose that $\lgate x,y \rgate \neq V(H)$. Among all vertices
in $V(H) \setminus \lgate x,y\rgate$, consider a vertex $u$ such that
$d(x,u)$ is minimum. If $d(u,x) \geq 2$, since all 2-intervals are
thick, there exist two distinct vertices $u',u'' \sim u$ such that
$d(u',x) = d(u'',x) = d(u,x)-1$. By our choice of $u$, we have $u', u'' \in
\lgate x, y \rgate$ and consequently $u \in \lgate x, y
\rgate$. Assume now that $u \sim x$. Since $d(x,y) = D$, either $
d(u,y) = D-1$, or $d(u,y) = D$. In the first case, $u \in I(x,y)
\subseteq \lgate x,y \rgate$. In the second case, by TC($y$), there
exists $v \sim x,u$ such that $d(v,y) = D-1$; since $v \in I(x,y)$,
$u$ has two distinct neighbors in $I(x,y)$ and thus $u \in \lgate
x,y\rgate$.
\end{proof}

We can now complete the proof of Theorem~\ref{t:lotoglo_dp2}.  Let $H$
be a locally finite dual polar graph and consider a vertex $x$ of
$H$. Then there exists a finite number of different sets of type
${\mathcal S}(x,y)$.  By Lemma \ref{l:gated_hull_dpg1}, $\lgate
x,y\rgate=\lgate {\mathcal S} (x,y)\rgate$, thus the vertex-set of $H$
is covered by a finite collection of gated sets of the form $\lgate
{\mathcal S} (x,y)\rgate$. Since, by Lemma \ref{l:gated_hull_dpg2},
each such set has finite diameter and $H$ is locally finite, each
gated set $\lgate {\mathcal S} (x,y)\rgate$ is finite, yielding that
the graph $H$ is finite as well. This finishes the proof of Theorem
\ref{t:lotoglo_dp2}. \hfill $\square$

\medskip

Finally, we will show that the conditions of our Theorem \ref{t:lotoglo_dp3} are significantly weaker than those of \cite{BroCo}*{Main Theorem (i)}, which we formulate next:

\begin{theorem} \cite{BroCo}*{Main Theorem (i)} Let $\Gamma$ be  a point-line geometry all of whose lines are thick and let $n$ be a natural number, $n\ge 4$. Suppose that $\Gamma$
is a nearly classical $3$--weak near polygon. Suppose, moreover, that $\Gamma$ has finite local rank $n-1$ at some point. Then $\Gamma$ is isomorphic to  a quotient of a dual polar space
by a group action with the minimal displacement at least $8$.  In particular, any finite nearly classical near polygon is a dual polar space.
\end{theorem}

According to \cite{BroCo}, a point-line geometry is a {\it 3-weak near polygon}  if for any two points $x,y$ at distance at most $3$ and any line $\ell$ passing via $y$ there is a unique point on $\ell$ nearest to $x$. A ($3$--weak) near polygon is called {\it thick} if each line has at least $3$ points and any two points at distance $2$ have at least $3$ common neighbors. A {\it hex} is a convex subspace of diameter $3$. Finally, a ($3$--weak) near polygon is called {\it nearly classical} \cite{BroCo} if it is thick and for each point $x$ and each hex $H$ containing $x$, the incidence geometry of the lines and quads on $x$ contained in $H$ is a projective plane.

\begin{proposition}  \label{BroCo} If $G$ is the collinearity graph of a nearly classical $3$--weak near polygon $\Gamma$, then $G$ is a locally dual polar graph.
\end{proposition}

\begin{proof}

By the definition of a locally dual polar graph, we have to show that the $2$--intervals of $G$ are thick, $G$ does not contain isometric $K^-_4$ and $K^-_{3,3}$, and satisfies the local triangle  and local quadrangle conditions. The fact that the $2$--intervals of $G$ are thick (i.e., $G$ is a thick graph in our sense) directly follows from the fact that $\Gamma$ is a thick $3$--weak near polygon.

Let $G$ contain a $K^-_4$ induced by vertices $x,y,z_1,z_2$ with $d(x,y)=2$. Let $\ell$ be the line of $\Gamma$ containing the points $x,z_1,z_2$. Then $y$ does not belong to $\ell$ and has $z_1$ and $z_2$ as nearest points on $\ell$, contrary to the fact that $\Gamma$ is a $3$--weak near polygon. Hence, $G$ does not contain induced $K^-_{4}$.

Now we establish the local triangle condition. Let $d(v,x)=d(v,y)=2$ and $x\sim y$. Let $\ell$ be the line of $\Gamma$ passing via $x$ and $y$. Since $\Gamma$ is a $3$--weak near polygon, $v$ has a unique nearest point $w$ on $\ell$. This implies that $w\sim x,y,v$, establishing LTC($v$).

Next we prove that $G$ does not contain isometric $K^-_{3,3}$. Suppose, by way of contradiction, that $G$ contains  an isometric $K^-_{3,3}$, induced by the vertices $x,y,u',v',u'',v''$, where $d(x,y)=3$ and $x\sim u',v'$; $y\sim u'',v''$; $u' \sim u'',v''$; $v' \sim u'',v''$. Let $\ell$ be the line of $\Gamma$ passing via $y$ and $v''$. Then $v''$ is the nearest to $x$ point of $\ell$. Since $\ell$ is thick, there exists a third point $z$ on $\ell$. Since $\Gamma$ is a 3-weak near polygon, $d(x,z)=3$, thus $z$ cannot be adjacent to $u'$. If $z$ is adjacent to $u''$, then we will obtain a forbidden $K^-_4$ induced by $y,z,u'',v'$. Hence $d(z,u')=d(z,u'')=2$. By LTC($z$), there exists a new vertex $z'\sim z,u',u''$. Since $d(x,z)=3$ and $z\sim z'$, $z'$ cannot be adjacent to $x$. If $z'$ is adjacent to $v'$, then we will obtain a forbidden $K^-_4$ induced by $u',u'',z',v'$. Hence $d(z',x)=d(z',v')=2$. By LTC($z'$) there exists a new vertex $w\sim x,v',z'$.
Since $d(z,x)=3$, necessarily $d(z,w)=2=d(z,v')$. By LTC($z$), there
exists $s\sim w,v',z$. But then the vertices $x,w,v',s$ induce a
forbidden $K^-_4$. This contradiction establishes that the subgraph of
$G$ induced by $x,y,u',v',u'',v''$ is not an isometric $K^-_{3,3}$.

Finally we prove that $G$ satisfies the local quadrangle condition. Let $d(x,y)=3$ and $v',v''\in I(x,y)$ be two non-adjacent neighbors of $y$. Then as is noticed in \cite{BroCo}, a result of \cite{BroWi} implies that $x$ and $y$ are contained in a hex $H$ and Yanushka's lemma from \cite{ShuYa} implies that the pairs $v',x$ and $v'',x$ are contained in quads $Q'$ and $Q''$. Both $Q'$ and $Q''$ belong to $H$. Since the incidence geometry of lines and quads at $x$ from $H$ is a projective plane, the intersection of (two lines) $Q'$ and $Q''$
is a point in this geometry. This point in turn corresponds to a line (maximal clique) $\ell\in Q'\cap Q''$ containing $x$.
Then each of $v'$ and $v''$ has a unique nearest point in $\ell$. Denote these points by $u'$ and $u''$, respectively. We assert that $u'\ne x$ and $u''\ne x$. Indeed, if say $u'=x$, then
$d(v',u')=d(v',x)=2$ must hold (because $d(x,y) = 3$ and $v'\sim y$), and all remaining points of $\ell\subset Q'$ must be located at distance $3$ from $v'$, contradicting the fact that $Q'$ has diameter $2$.  Thus $x\notin \{u', u''\}$, showing that $u'\sim v'$ and $u''\sim v''$.
Suppose that $u'\ne u''$. Then $u'\sim u''$ (because $u'$ and $u''$ belong to the clique $\ell$) and $d(y,u')=d(y,u'')=2$. By LTC($y$), there exists a vertex $z\sim y,u',u''$. But then the vertices $z,u',u'',x$ induce a forbidden $K^-_4$.
Thus $u'=u''$ and, since $u'\sim x$, we obtain the required common neighbor of $x,v'$, and $v''$.
\end{proof}

\begin{remark} In the proof of Proposition~\ref{BroCo} we do not use all requirements on $\Gamma$: actually, we used that all lines of $\Gamma$ are thick (contain at least three points), all $2$--intervals are thick (contain a square), and that $\Gamma$ is a $3$--weak near polygon. Instead of using the fact that $\Gamma$ is nearly classical, to establish the local quadrangle condition it suffices   to use the {\it $3$--cube condition} asserting that each isometric $6$--cycle is included in a $3$--cube. We do not know if this last condition (as well as near classicalness of $\Gamma$)  is necessary to establish that $\Gamma$  is a locally dual polar space. Namely, is it true that if $\Gamma$ is a $3$--weak near polygon with thick lines and thick $2$--intervals, then $\Gamma$ is a locally dual polar space?
\end{remark}

\chapter{Sweakly Modular Graphs}\label{sec:swm}
A {\it sweakly modular graph} (an {\it swm-graph} for short) is
a weakly modular graph having no induced $K_4^-$ and no isometric $K_{3,3}^-$.
The class of  swm-graphs contains a number of nice and important subclasses of
weakly modular graphs:
\begin{itemize}
\item median graphs ($=$ the 1-skeletons of CAT$(0)$ cube complexes),
\item frames,  semiframes\footnote{Frames and semiframes are defined in
Example \ref{ex:frames} below and are particular
strongly modular graphs} ($=$ the 1-skeletons of  CAT$(0)$ $B_2$--complexes),
\item dual polar graphs ($=$ thick swm-graphs),
\item orientable modular graphs,
\item strongly modular graphs ($=$ modular graphs without induced $K_{3,3}^-$).
\end{itemize}
The aim of this chapter is to investigate the swm-graphs
and to show that many fascinating properties,
previously known only for the above mentioned subclasses,
hold for all swm-graphs.

%\smallskip
%{\bf Overview of main results.}
\section{Main results}\label{subsec:swm_main}
In Section~\ref{subsec:latt_char_swm},
we present a lattice-theoretical characterization of swm-graphs;
Theorem~\ref{thm:chara_swm} shows
that the swm-graphs are exactly the near polygons
in which all metric intervals induce  modular lattices.
This generalizes a similar characterization of strongly modular graphs
(Theorem~\ref{thm:BVV2}).
We also provide a characterization of  posets obtained
from  orientable modular graphs that extends Theorem~\ref{thm:BVV1}.

In Section~\ref{subsec:Boolean}, we introduce Boolean-gated sets of swm-graphs,
which we show to be exactly the gated sets inducing  thick subgraphs. By Theorem~\ref{th:dual-polar-graph},
each Boolean-gated set induces a dual polar graph.  In the case of median graphs, the Boolean-gated
sets are exactly  the vertex sets of cubes. Extending the fact that each median graph is a gluing of
its cubes (which are gated subgraphs), this shows that each swm-graph is a gluing of
its gated dual polar subgraphs.  We will see that  these dual polar subgraphs
play an analogous role of what cube subgraphs play in median graphs.

In Section~\ref{subsec:barycentric}, we define the barycentric
graph $G^*$ of an swm-graph $G$ as the covering graph of the poset of
all Boolean-gated sets with respect to the reverse inclusion. This construction
of $G^*$ was previously
considered for semiframes~\cite{Kar98b} and
orientable modular graphs~\cite{HH12}, and
naturally extends the construction of
the subspace poset of a dual polar graph.
We show (extending the known fact that the covering graph of the
subspace poset of a polar space is orientable modular) that the barycentric graph $G^*$ of
an swm-graph $G$ is an orientable modular graph
(Theorem~\ref{thm:G*_is_om}).  Moreover the original graph
$G$ is isometrically embeddable to $G^*$ (the edges of $G^*$ have length $1/2$).
In Chapter~\ref{sec:complex} we
show that the sequence
$G, G^*, (G^*)^*, ((G^*)^*)^*, \ldots$ of swm-graphs converges to
a certain metric simplicial complex associated with $G$.

In Section~\ref{subsec:G^Delta}, we consider the
thickening $G^\Delta$ of an swm-graph $G$, which is obtained from $G$ by making
adjacent all pairs of vertices of $G$ belonging to a common Boolean-gated set.
We show that $G^{\Delta}$ is a finitely Helly graph (Theorem~\ref{thm:Helly}) and thus its clique complex
$X(G^{\Delta})$ is contractible.

In Section~\ref{subsec:normalbgpaths}, extending similar results of
\cite{NiRe} for CAT(0) cube complexes, we prove the existence and
uniqueness of so-called normal Boolean-gated paths in thickenings
$G^{\Delta}$ of all swm-graphs $G$.

In Section~\ref{subsec:building},
we establish a link between swm-graphs and Euclidean buildings of type C.
It is well-known that polar spaces and spherical buildings of type C constitute the same objects.
As we mentioned before,
the dual polar graph $G$ of a polar space $\Pi$ is an swm-graph,
and the covering graph of the subspace poset of $\Pi$,
equals to the barycentric graph $G^*$, and thus is orientable modular.
We show that similar properties hold for Euclidean building of type C.

In Section~\ref{subsec:biauto}, we prove that the groups
acting geometrically on swm-graphs are biautomatic.  The proof is
based on the fact that the normal Boolean-gated paths defined in
Section \ref{subsec:normalbgpaths} can be locally
recognized~\cite{Swiat} and satisfy the $2$--sided fellow traveler
property~\cite{ECHLPT}. This generalizes analogous results of
\cite{NiRe} about groups acting on CAT(0) cube complexes, and of \cite{Nos} about automorphism groups of some buildings.

Finally, in Section~\ref{subsec:0ext}, we present an immediate
algorithmic consequence of previous results to swm-graphs. By using
the construction of $G^*$, we present a simple $2$--approximation
algorithm for the 0-extension problem on swm-graphs. The minimum
0-extension problem is a version of the facility location problem and
arises from clustering or pixel-labeling in computer vision and
machine learning. It was shown recently in \cite{HH12} that the
orientable modular graphs are precisely the polynomially-solvable
instances of the minimum 0-extension problem.

\section{A lattice-theoretical characterization of swm-graphs}\label{subsec:latt_char_swm}

Let $G=(V,E)$ be a graph. For a pair $p,q$ of vertices of $G$ consider the base-point order $\preceq_p$
on the interval $I(p,q)$ with minimum element $p$ and maximum element $q$
(see Subsection~\ref{classes}).
The following characterization of swm-graphs
can be viewed as an extension of Theorem~\ref{thm:BVV2}.

\begin{Thm}\label{thm:chara_swm}
%% Let $G$ be a graph satisfying the triangle condition (TC).
%% Then $G$ is an swm-graph if and only if every interval $I(p,q)$
%% endowed with the base-point order $\preceq_p$ is a modular lattice,
%% and induces its covering graph isometrically.
Let $G$ be a graph satisfying the triangle condition (TC).
Then $G$ is an swm-graph if and only if every interval $I(p,q)$
endowed with the base-point order $\preceq_p$ is a modular lattice,
and the subgraph induced by $I(p,q)$ is isometric in $G$.
\end{Thm}
In other words, swm-graphs are the
near polygons in which  every interval induces (the covering graph of) a modular lattice.

\begin{proof} We start with the ``if'' part.
Pick any $z,v,w,u$ with $d(z,u) = d(v,u) + 1 = d(w,u) + 1$ and $z\sim v,w$. The interval  $I(z,u)$  is a modular lattice
having $v,w$ as atoms. Thus the join $v \vee w$ is a common neighbor of $v,w$ with $d(v \vee w, u) = d(z,u) -2$,
implying the quadrangle condition (QC). Hence $G$ is a weakly modular graph.

If $G$ has an induced $K_4^-$ and $p,q$ are its nonadjacent vertices, then  the
subgraph of $G$ induced by $I(p,q)$ is not isomorphic to the covering graph of $(I(p,q),\preceq_p)$. If $G$ has an
isometric $K_{3,3}^-$,
then take as $p,q$ the vertices of $K_{3,3}^-$ at distance 3.  Then again the subgraph of $G$ induced by $I(p,q)$
is not a lattice. Therefore $G$ is an swm-graph.

Next we establish the ``only if'' part.  For a vertex $x \in I(p,q)$, let $d_x:=d(x, p)$.

%% \medskip\noindent
%%     {\bf Claim 1.} There is no edge $xy$ of $I(p,q)$  with $|d_x - d_y| \neq 1$.

\begin{claim-swm}\label{claim-swm-1}
  There is no edge $xy$ of $I(p,q)$  with $|d_x - d_y| \neq 1$.
\end{claim-swm}

\begin{proof}[Proof of Claim~\ref{claim-swm-1}]
Indeed, if such an edge $xy$ exists, then $d_x = d_y$,  and
applying the triangle condition (TC) to the triplets  $x,y,p$ and  $x,y,q$, we can find
two distinct nonadjacent common neighbors $p',q'$ of $x,y$. Then $x,y,p',q'$ induce a
forbidden $K_4^-$, a contradiction.
\end{proof}
%

%% \medskip\noindent
%% {\bf Claim 2.} For any pair $x,y \in I(p,q)$ there exists a unique median of the triplet $x,y,q$
%% (respectively, of $x,y,p$).

\begin{claim-swm}\label{claim-swm-2}
  For any pair $x,y \in I(p,q)$ there exists a unique median of the triplet $x,y,q$ (respectively, of $x,y,p$). 
\end{claim-swm}

\begin{proof}[Proof of Claim~\ref{claim-swm-2}]
First we show that $x,y,q$ has a median.
Take a quasi-median $x'y'q'$ of $x,y,q$. Suppose to the contrary that $x',y', q'$ are different vertices.
Take a neighbor $w$ of $q'$ belonging to $I(q',x') \subseteq I(p,q)$.
By weak-modularity (Lemma~\ref{lem-weakly-modular}), we have $d(w,y')=d(q',y')$.
By (TC) for $w,q',y'$, there is
a common neighbor $w'$ of $w, q'$
belonging to  $I(q',y') \subseteq I(p,q)$. Then $d_w = d_{w'}$,  a contradiction to Claim~\ref{claim-swm-1}.
This implies that $x'=y'=q'$, i.e., each quasi-median of $x,y,q$ and of $x,y,p$ is a median.

Suppose now that there exist $x,y \in I(p,q)$ such that the triplet $x,y,q$ has two medians $m,m'$.
Among such pairs, take $x,y$ with $d(x,y)$ as small as possible. Then it is impossible that
$d(x,y) = 1$. Suppose that $d(x,y) = 2$; necessarily $d_x = d_y$. Then $m,m'$ are
two distinct common neighbors of $x,y$ such that
$d_m = d_{m'} = d_x +1$ and $d(m,m') = 2$.
Take a median $u$ of $x,y,p$,
which is also a common neighbor of $x,y$ with $d_u = d_x - 1$.
So there is no edge from $u$ to $m$ or $m'$.
Take a median $u'$ of $m,m',q$, which is a common neighbor of $m,m'$
with $d_{u'} = d_x + 2$. Then $d(u,u')=3$ and we conclude that the subgraph
induced by $x,y,m,m',u,u'$ is an isometric $K_{3,3}^-$, a contradiction.

Therefore we can suppose that $d(x,y) > 2$.  Take a median $l$ of
$x,y,p$ and take a neighbor $l'$ of $l$ in $I(l,y)$.  Then
$d(x,l')=d(x,l) + 1$.  By the first part of the proof, we can take a
median $z$ of $m,x,l'$ in $I(m,l')$ and a median $z'$ of $m',x,l'$ in
$I(m',l')$. Note that $x \sim z, z'$. If $z=z'$, then the triplet
$z,y,q$ has two medians $m,m'$ and $d(z,y)<d(x,y),$ contrary to the
minimality assumption. Thus $z$ and $z'$ are distinct with $d(z,z') =
2$.  By (QC) for $x, z,z',l'$, we can find a
median $w$ of $z,z',l'$.  Then $w$ and $x$ are distinct common
neighbors of $z$ and $z'$. Hence $z$ and $z'$ are distinct medians of
$x,w,q$ with $d(x,w)\le 2$, which was proved before to be impossible.
This completes the proof of Claim~\ref{claim-swm-2}.
\end{proof}

\medskip
By Claim~\ref{claim-swm-2}, $x \wedge y$ and $x \vee y$ exist and are equal to the uniquely determined medians of the triplets $x,y,p$ and  $x,y,q$, respectively.
The rank of $x$ is equal to $d_x$. By the definition of a median, we have
\[
d_x + d_y = d_{x \wedge y} + d_{x \vee y}.
\]
Hence the rank function satisfies the modular equality.
This means that $I(p,q)$ is a modular lattice. Also $I(p,q)$ induces an isometric subgraph:
for $x,y \in I(p,q)$,
a path passing through $x$, $x \wedge y$, and $y$ is a shortest path of $G$ and belongs to $I(p,q)$.
\end{proof}

We continue with a  poset characterization of orientable modular graphs.
Let $G=(V,E)$ be an orientable modular graph with an admissible orientation $o$.
It is not difficult to see that $o$ is acyclic (see \cite{HH12}),
and thus $o$ induces a partial order on $V$, where
we interpret $x \rightarrow y$ as $x \succ y$.
The resulting poset is denoted by $\mathcal{P}(G,o)$.
Then we can retrieve $G$ as the covering graph of $\mathcal{P}(G,o)$ and $o$ as
the Hasse orientation. We present a characterization of the posets obtained in this way.
A {\em crown} in a poset is a 6-tuple $(x_1,x_2,\ldots,x_6)$
such that for $i=2,4,6$, the elements
$x_{i-1}$ and $x_{i+1}$ cover $x_i$, and
$x_i$ and $x_{i+3}$ are non comparable, where the indices
are taken modulo 6.

\begin{Thm}\label{thm:chara_om}
The covering graph $G$ of a poset $\mathcal{P}$ is modular
and the Hasse orientation of $G$ is admissible
if and only if $\mathcal{P}$ satisfies the following conditions:
\begin{itemize}
\item[{\rm (a)}] $\mathcal{P}$ is graded;
\item[{\rm (b)}] every interval  is a modular lattice;
\item[{\rm (c)}] every upper-bounded pair has the join and every lower-bounded pair has the meet;
\item[{\rm (d)}] every pairwise upper-bounded triplet has the join and
every pairwise lower-bounded triplet has the meet;
\item[{\rm (e)}] the $2$--skeleton of the order complex $\Delta(\mathcal{P})$ of $\mathcal{P}$ is simply connected.
\end{itemize}
Moreover, the condition  {\rm (d)}
can be replaced by:
\begin{itemize}
\item[{\rm (d$'$)}] for every crown $(x_1,x_2,\ldots,x_6)$ of $\mathcal{P}$
there are $i,j \in \{2,4,6\}$
such that $(x_{i}, x_{i+2})$ is lower-bounded
and $(x_{j-1}, x_{j+1})$ is upper-bounded.
\end{itemize}
\end{Thm}

\begin{proof} We start with the proof of ``only if'' part. Suppose that the covering graph $G$ of $\mathcal{P}$ is
an orientable modular graph.

To (a):  For a shortest path $P = (x = x_0,x_1,\ldots,x_k = y)$ of $G$,
let $d^+(P)$ denote the number of indices $i$ with $x_i \prec x_{i+1}$.
For another shortest path $P' = (x = x'_0,x'_1,\ldots,x'_k = y)$, we will show that
$d^+(P) = d^+(P')$, whence $d^+(P)$ does not depend of the choice
of the shortest path $P$  and thus it can be  denoted by $d^+(x,y)$.
We proceed by induction on $d(x,y)$. We can assume that $x_1 \neq x'_1$, otherwise
we are done by induction hypothesis.
By (QC) for $x,x_1,x'_1,y$, there is a common neighbor
$z$ of $x_1,x'_1$ with $d(z,y) = d(x,y) - 2$.
By the induction hypothesis applied to $\{x_1, y\}, \{x_1', y\}, \{ z, y\}$,
we obtain $d^+(P)=d^+(x,x_1)+d^+(x_1,z)+d^+(z,y)$ and
$d^+(P')=d^+(x,x'_1)+d^+(x'_1,z)+d^+(z,y)$. From the admissibility of the partial order on the $4$--cycle $xx_1zx'_1$ we conclude that
$d^+(x,x_1)+d^+(x_1,z)=d^+(x,x'_1)+d^+(x'_1,z),$ thus establishing the required equality $d^+(P)=d^+(P')$. Now, pick
an arbitrary vertex $x_0$. Define $r: V \to \ZZ$ by setting $r(x) := d(x,x_0) - 2d^{+}(x,x_0)$.
Then  $r$ is a grade function. Hence the poset $\mathcal{P}$ is graded.

To (b): By (a), if $p \preceq q$,
then $d(p,q) = r(p) - r(q)$, and consequently
$I(p,q) = [p,q]$ and $\preceq$ coincides with the base-point order $\preceq_{p}$ on $[p,q]$.
Therefore, by Theorem~\ref{thm:chara_swm}, the interval $[p,q]$ is a modular lattice.

To (c): Suppose by way of contradiction that $x$ and $y$ have two
minimal upper bounds $z,z'$. Take such $x,y$ with minimal $d(x,y)$.
Consider a median $m$ of $x,y,z$.
By (a), $I(x,z) = [x,z]$ and $I(y,z) = [y,z]$ and thus $m$ is a common
upper bound of $x,y$ with $m \preceq z$.  Hence $z = m \in I(x,y)$;
similarly, $z' \in I(x,y)$.
By the minimality, $x,y$ are maximal common lower bounds of $z,z'$,
and belong to $I(z,z')$ (shown as above).
Take a neighbor $x'$ of $x$ in $I(x,z')
\subseteq I(x,y)$. Then $x \prec x'$.  Take a maximal chain $(x =
x_0,x_1,x_2,\ldots,x_k = z)$ from $x$ to $z$.
We can assume that $x_1 \neq x'$ (otherwise we can replace $x$ by $x'$).
We define iteratively
$x'_0, x'_1, \ldots, x'_k$ such that $x'_{i}$ is a neighbor of $x_i$
in $I(x_i,y)$ and such that $x_i \prec x'_i$. Let $x'_0 := x'$; by
definition, $x'_0 = x'$ is a neighbor of $x_0 = x$ in $I(x,y)$ such
that $x_0 \prec x'_0$.  Assuming that $x'_i (\neq x_{i+1})$ has already been defined,
by (QC) for $x_i, x_{i+1}, x'_i, y$, there exists
$x'_{i+1} \sim x_{i+1},x'_i$ such that $x'_{i+1} \in I(x_{i+1}, y)$.
Necessarily $x'_{i+1} \neq x_{i+2}$;
otherwise $d(z,z') < d(z,x) + d(x,z')$, contradicting $x \in I(z,z')$.
Since $x_i \prec x_i'$ and since the partial order is admissible,
considering the square $x_ix'_ix'_{i+1}x_{i+1}$, we have that $x_{i+1}
\prec x_{i+1}'$. Consequently, $z = x_{k} \prec x'_{k}$ and $x'_k \in
I(z,y)$. However, this contradicts the fact that by (a) we have
$I(z,y) = [y,z]$.

To (d):  Let $x,y,z$ be a pairwise upper-bounded triplet of vertices.
By (c), $x \vee y$, $y \vee z$, and $x \vee z$ exist.
Consider a median $m$ of $x \vee y$, $y \vee z$, and $x \vee z$.
By subsequent Lemma~\ref{lem:filter_is_convex}, each principal filter (ideal) is convex.
Thus $m$ belongs to $(x)^{\uparrow}$, $(y)^{\uparrow}$, and $(z)^{\uparrow}$.
This means that $m$ is a common upper-bound of $x,y,z$.
By (c), the join of $x,y,z$ exists.

To (e): Since $G$ is a modular graph, the square complex of $G$ is simply connected.
Every $4$--cycle of $G$ must belong to an interval of rank $2$ in $\mathcal{P}$.
From this, one can see that the homotopy in the square complex induces a homotopy
in the order complex $\Delta(\mathcal{P})$.

Now we prove the ``if'' part.
Suppose that $\mathcal{P}$ satisfies the conditions (a),(b),(c),(d$'$), and (e).
Let $G$ be the covering graph of $\mathcal{P}$.
By (c), every $4$--cycle is
of the form $v \prec  x  \prec u \succ y  \succ v$.
Therefore the Hasse orientation is admissible.
Hence it suffices to show the local quadrangle condition (LQC)
and the simple connectivity of the square complex of $G$.
To establish (LQC),
take an isometric $6$--cycle $C=(x_1,x_2,\ldots,x_6)$.
We will show that $x_1, x_3, x_5$ have a common neighbor
and that $x_2, x_4, x_6$ have a common neighbor.
We can assume that the grade $r(x_1)$ of $x_1$ is maximal
in $C$. Then $x_2 \prec x_1 \succ x_6$. Consider the case
$x_2 \succ x_3$. If $x_3 \succ x_4$,
then by (a) the vertices $x_1,x_2,\ldots,x_6$ all belong to $[x_4,x_1]$,
and by (b) we can take the join $x_3 \vee x_5$ and the meet $x_2 \wedge x_6$
as the required by (LQC) common neighbors.
So, suppose that $x_3 \prec x_4$.
By (a), either $x_1 \succ x_6 \prec x_5 \succ x_4$ or $x_1 \succ x_6 \succ x_5 \prec x_4$.
For the former case, $x_3$ and $x_6$
have two minimal common upper-bounds $x_1,x_5$.
For the latter case, $x_3$ and $x_5$ have
two minimal common upper-bounds $x_1,x_4$.
Both cases contradicts (c).
Hence $(x_1, x_2, x_3, x_4, x_5, x_6)$ is a crown.
By (d$'$), (b), (c),  we can assume that
$x_{1} \vee x_{3}$ exists and covers both $x_{1}, x_{3}$.
Then $x_{1} \vee x_{3}$ also covers $x_{5}$.
Otherwise the bounded pair $x_{6}, x_4$
cannot have the join, contradicting (c).
So we obtain a common neighbor of $x_1,x_3, x_5$.
By the same argument, we obtain a common neighbor of $x_2,x_4,x_6$.

Finally, we show that the square complex of $G$ is simply connected.
For a (possibly nonsimple) cycle $C = (p_1,p_2,\ldots, p_k)$
in the $1$--skeleton of $\Delta(\mathcal{P})$,
a {\em refinement} of $C$ is a (possibly nonsimple) cycle $\widehat{C}$ of $G$ obtained from $C$ by replacing
each $(p_i, p_{i+1})$ by a shortest path between $p_i$ and $p_{i+1}$ in $G$.
Since $p_i \preceq p_{i+1}$ or $p_{i+1} \preceq p_{i}$,
this shortest path forms a maximal chain in $[p_i,p_{i+1}]$ or $[p_{i+1}, p_i]$.
It suffices to show that a refinement $\widehat{C}$ of  any cycle $C$ in the $1$--skeleton of $\Delta(\mathcal{P})$
is $0$--homotopic in the square-complex of $G$.
Suppose that $C$ is homotopic to another cycle $C'$ via an elementary homotopy in the simplicial complex $\Delta (\mathcal{P})$. %: a change of one triangle or of one edge.
Any triangle of $\Delta(\mathcal{P})$ must belong to an interval in ${\mathcal P}$, which is a modular lattice.
Any cycle in the covering graph of a modular lattice $H$ is $0$--homotopic in the square-complex of $H$
(since its covering graph is an orientable modular graph).
Therefore  $\widehat{C}$ is homotopic in the square-complex of $G$ to any refinement $\widehat{C}'$ of $C'$.
Since $C$ is 0-homotopic in $\Delta(\mathcal{P})$, the refinement $\widehat{C}$ of $C$ is 0-homotopic in the square-complex of $G$.
\end{proof}

\begin{Lem}[\cite{HH12}]\label{lem:filter_is_convex}
Let $G$ be an orientable modular graph with an admissible orientation $o$.
Then all principal filters and ideals of ${\mathcal P}(G,o)$ are modular
semilattices and
are convex subgraphs of $G$.
\end{Lem}
\begin{proof}
We first show the convexity.
Let $p$ be an arbitrary vertex.
By (c), the filter $(p)^{\uparrow}$ is  a meet-semilattice.
By Lemma~\ref{lem-weakly-modular_gated_hull}, pick two vertices $x,y$ at
distance 2 in $(p)^{\uparrow}$ and let $z$ be a common neighbor of $x,y$. We
assert that $z \in (p)^{\uparrow}$.
We can assume $y \not \prec x$.
Indeed, if $x \prec y$, then $z \in [x,y] \subseteq (p)^{\uparrow}$.
If $x \prec z \succ y$, then it is obvious that $z \in (p)^{\uparrow}$.
Suppose that $x \succ z \prec y$.
By (c), $z$ must coincide with $x \wedge y$, whence $z= x \wedge y \in  (p)^{\uparrow}$.
This shows that the filter $(p)^{\uparrow}$ is convex.
The fact that $(p)^{\uparrow}$ is a modular semilattice
follows from (b), (c) and (the same argument of the proof of) (d).
\end{proof}

\begin{remark} The topological condition (e) cannot be removed from Theorem \ref{thm:chara_om}:
for example, consider a $2n$--cycle $(n \geq 4)$ with the  zigzag orientation.
Note also that if $\mathcal{P}$ has no infinite chains, then
(e) can be replaced by the condition that the order complex of $\mathcal{P}$ is simply connected.
\end{remark}

\section{Boolean pairs and Boolean-gated sets}\label{subsec:Boolean}

A pair $(p,q)$ of vertices of $G$ is called {\em Boolean} if the subgraph induced by $I(p,q)$ contains
an isometrically embedded $k$--cube with $k=d(p,q)$.
\begin{Prop}\label{prop:char_Boolean}
Let $G$ be an swm-graph. For any two vertices $p,q$ of $G$, the following conditions are equivalent:
\begin{itemize}
\item[{\rm (i)}] $(p,q)$ is a Boolean pair;
\item[{\rm (ii)}] there exists an isometric $2d(p,q)$--cycle containing $p,q$;
\item[{\rm (iii)}] $I(p,q)$ is a complemented modular lattice;
\item[{\rm (iv)}] every pair of vertices of $\lgate p,q \rgate$ is Boolean;
\item[{\rm (v)}] $\lgate p,q \rgate$ induces a thick subgraph of $G$;
\item[{\rm (vi)}] $\lgate p,q \rgate$ induces a dual polar graph of diameter $d(p,q)$.
\end{itemize}
\end{Prop}

\begin{proof}

(i)$\Rightarrow$(ii): This immediately follows because any two vertices $p,q$ of a cube belong to
an isometric cycle of length $2d(p,q)$.

(ii) $\Rightarrow$ (iii): Let $k := d(p,q)$, and
let $(p = x_0,x_1,\ldots,x_k = q = y_0, y_1,\ldots, y_k = p)$ be
an isometric $2k$--cycle, which belongs to the modular lattice $I(p,q)$.
Then $x_i$ is a complement of $y_i$; namely $x_i \wedge y_i = p$ and $x_i \vee y_i = q$.
By the modular equality, for $i=1,2,\ldots,k$, the
meet $a_i := x_i \wedge y_{i-1}$ has rank $1$ and is an atom.
Then $x_i = x_{i-1} \vee a_i$ holds
(otherwise $a_i \preceq x_{i-1}$ giving a contradiction $a_i \preceq x_{i-1} \wedge y_{i-1} = p$).
Consequently $q = a_1 \vee a_2 \vee \cdots \vee a_k$.
Namely, $q$ is a join of atoms. Hence, $I(p,q)$ is a complemented modular lattice.

(iv)$\Rightarrow$(i) and (iv)$\Rightarrow$(v) are obvious.

(v) $\Leftrightarrow$ (vi): $\lgate p,q \rgate$ induces an swm-subgraph (thus, a weakly modular subgraph) of $G$. If $\lgate p,q \rgate$ satisfies condition (v), then
$\lgate p,q \rgate$ induces a dual polar graph by Theorem~\ref{th:dual-polar-graph}. This graph has diameter $d(p,q)$ by Lemma~\ref{l:gated_hull_dpg2}.
Conversely, if $\lgate p,q \rgate$ satisfies condition (vi),  by Theorem~\ref{th:dual-polar-graph} $\lgate p,q \rgate$ induces a thick subgraph.

%The implication  follows from
%Theorem~\ref{th:dual-polar-graph} and Lemma~\ref{l:gated_hull_dpg2}
%of the previous section.

(vi)$\Rightarrow$(iii):
By Lemma~\ref{l:Cameron},
the map $\rho$ defined by $x \mapsto \mathcal{S}(p,x)$ is
an order- and rank-preserving map from modular lattice $I(p,q)$ to
the lower ideal of $\mathcal{S}(p,q)$ in $\mathcal{L}(p)$,
which is a complemented modular lattice.
In particular, $\rho(x \vee y) \supseteq \rho(x) \vee \rho(y)$ holds.
Consider the set $A$ of all atoms of $I(p,q)$.
By definition, $\bigcup_{a \in A} \rho(a)$ is equal to $\mathcal{S}(p,q)$, and is the set
of all atoms of $(\mathcal{S}(p,q))^{\downarrow}$.
Thus $\rho( \bigvee A) =  \mathcal{S}(p,q)$.
Since $\rho$ is rank-preserving and the rank of $\mathcal{S}(p,q)$ is $d(p,q)$ (Lemma~\ref{l:Cameron}),
the rank of $\bigvee A$ in $I(p,q)$ is $d(p,q)$.
This means that $q$ is equal to the join $\bigvee A$ of atoms, and $I(p,q)$ is complemented.

(iii) $\Rightarrow$ (iv):
Apply the algorithm GATED-HULL (see Subsection~\ref{s:wmgra}) with the initial set $I(p,q)$.
Initially, every pair is Boolean; this follows from
the well-known property of a complemented modular lattice that
any two elements (chains) belongs to
a Boolean sublattice (Lemma~\ref{lem:Booleansublattice}) and that the
covering graph of a Boolean lattice is a cube.
Suppose that every pair of the current set $X$ is Boolean.
Suppose that a common neighbor $w$ of $u,v \in X$ is added to $X$.
We assert that $(x,w)$ is Boolean for every $x \in X$.
Since $(x,u)$ and $(x,v)$ are Boolean, $(x,w)$ is also Boolean
by the following Lemma~\ref{lem:Boolean123}.
\end{proof}

\begin{Lem}\label{lem:Boolean123}
Let $(p,q)$ be a Boolean pair of an swm-graph $G$. Then
\begin{itemize}
\item[{\rm (1)}] for  a neighbor $q'$ of $q$ with $d(p,q) = d(p,q')$,
the pair $(p,q')$ is Boolean.
\item[{\rm (2)}] for another Boolean pair $(p,u)$
and a common neighbor $w$ of $q,u$, the pair $(p,w)$ is Boolean.
\end{itemize}
\end{Lem}
\begin{proof}
To (1): Let $k:= d(p,q) = d(p,q')$. By (TC) for $q,q',p$,
there is a common neighbor $h$ of $q,q'$ with $d(p,h) = k - 1$.
Then $h$ is a coatom of $I(p,q)$ and of $I(p,q')$.
By complementarity, we can take an atom $a$ of $I(p,q)$
with $h \vee a = q$ and $h \wedge a = p$.
Then $d(a,q) = k-1$ and $d(a, h) = k$.
Also $d(a,q') = k$.
Otherwise $d(a,q') = k-1$. By (TC) for $q,q',a$
we obtain a common neighbor $h'$ of $q,q'$ with $d(h',a) = k-2$.
Hence $h'$ is a coatom of $I(p,q)$ different from and nonadjacent to $h$,
and thus $q,q',h,h'$ induce a $K_4^-$, a contradiction.

By (TC) for $p,a,q'$ there is a common neighbor $a'$ of $a,p$
with $d(a',q') = k-1$.
Then $a'$ does not belong to $I(p,q)$, and is an atom of $I(p,q')$.
Hence the join of $a'$ and $h$ in $I(p,q')$ is equal to $q'$ (here $h$ is a join of atoms
in $I(p,q)$ and in $I(p,q')$).
Thus $q'$ is a join of atoms in $I(p,q')$.
Thus $I(p,q')$ is complemented and $(p,q')$ is Boolean.

To (2): If $w \in I(p,q)$ or $w \in I(p,u)$, then $I(p,w)$
is also a complemented modular lattice, and the pair $(p,w)$ is Boolean.
By (1) it suffices to consider only the case $d(p,u) = d(p,q) = d(p,w) -1$.
Then $u$ and $q$ are distinct coatoms of $I(p,w)$ and their join is $w$.
Both $u$ and $q$ are joins of atoms in $I(p,w)$.
This means that $I(p,w)$ is complemented, and thus $(p,w)$ is Boolean.
\end{proof}
A set $X$ of vertices of a graph $G$ is called {\em Boolean-gated}
if $X$ induces a gated and thick subgraph of $G$,
or equivalently, if $X$ is a gated set inducing
a dual polar graph $G(X)$ (with possibly infinite diameter).
Since
a dual polar graph with finite diameter
is the gated hull of any maximum distant pair (Lemma~\ref{lem:d(x,y)=D}), we obtain the
following result:
\begin{Lem}\label{lem:X=<p,q>}
A set $X$ of vertices with finite diameter of an swm-graph $G$ is Boolean-gated
if and only if $X = \lgate p,q \rgate$ for some Boolean pair $(p,q)$.
\end{Lem}
Consider now the special case when $G$ is an orientable modular
graph with an admissible orientation $o$.
As above, $o$ induces on $V(G)$ the poset $\mathcal{P}(G,o)$.
Following \cite{HH12} a pair $(x,y)$ of vertices is called {\em $o$--Boolean}
if $x \preceq y$ and
the interval $[x,y]$ of $\mathcal{P}(G,o)$  is a complemented modular lattice,
or equivalently,  if $x$ is the sink and $y$ is the source
of a cube-subgraph oriented by $o$;
the sink and the source are uniquely determined by the admissibility of $o$.

\begin{Lem}\label{lem:o-Boolean}
Let $G$ be an orientable modular graph with an admissible orientation $o$.
A set $X$ of vertices with finite diameter is Boolean-gated if and only
if $X = [p,q]$ for some $o$--Boolean pair $(p,q)$.
\end{Lem}
\begin{proof} We start with the ``if'' part. Since $\mathcal{P}(G,o)$ is graded,
the poset  $I(p,q)$ for $p \preceq q$ is isomorphic to $[p,q]$.
Thus $I(p,q)$ is complemented, and hence $(p,q)$ is Boolean.
By Lemma~\ref{lem:filter_is_convex}, $[p,q] = (p)^{\uparrow} \cap (q)^{\downarrow}$
is convex, whence $[p, q] = \lgate p,q \rgate$.

We continue now with the ``only if'' part. Let $(u,v)$ be a pair  with $\lgate u,v \rgate = X$.
Take an isometric cube-subgraph $H$ containing $u,v$.
Necessarily $H$ belongs to $X$ and its diameter is the same as that of $X$.
Restrict the orientation $o$ to $H$.
By admissibility, the resulting digraph is
the same as the Hasse diagram of a Boolean lattice.
Therefore there exist the minimum element $p$ and the maximum element $q$.
Then $(p,q)$ is $o$--Boolean.
Since $d(p,q) = d(u,v)$, it holds $\lgate p,q \rgate = \lgate u,v \rgate = X$.
As above, $[p, q]$ is convex, and hence $X = [p,q]$, as required.
\end{proof}
We conclude this section with the following intersection property of Boolean-gated sets:
%which immediately follows from Proposition~\ref{prop:char_Boolean}(iv).

\begin{Lem}\label{lem:Boolean_intersection}
For a Boolean-gated set $X$ and a gated set $Y$ of an swm-graph $G$,
the nonempty intersection $X \cap Y$ is Boolean-gated.
\end{Lem}
\begin{proof} Since $X$ and $Y$ are gated, $X\cap Y$ is also gated. Since $X$ is thick and $Y$ is convex (because $Y$ is gated),
$X\cap Y$ is thick. Consequently, $X\cap Y$ is Boolean-gated.
\end{proof}

\section{Barycentric graph $G^*$}\label{subsec:barycentric}

\begin{figure}[ht]
\begin{center}
\includegraphics[scale=0.75]{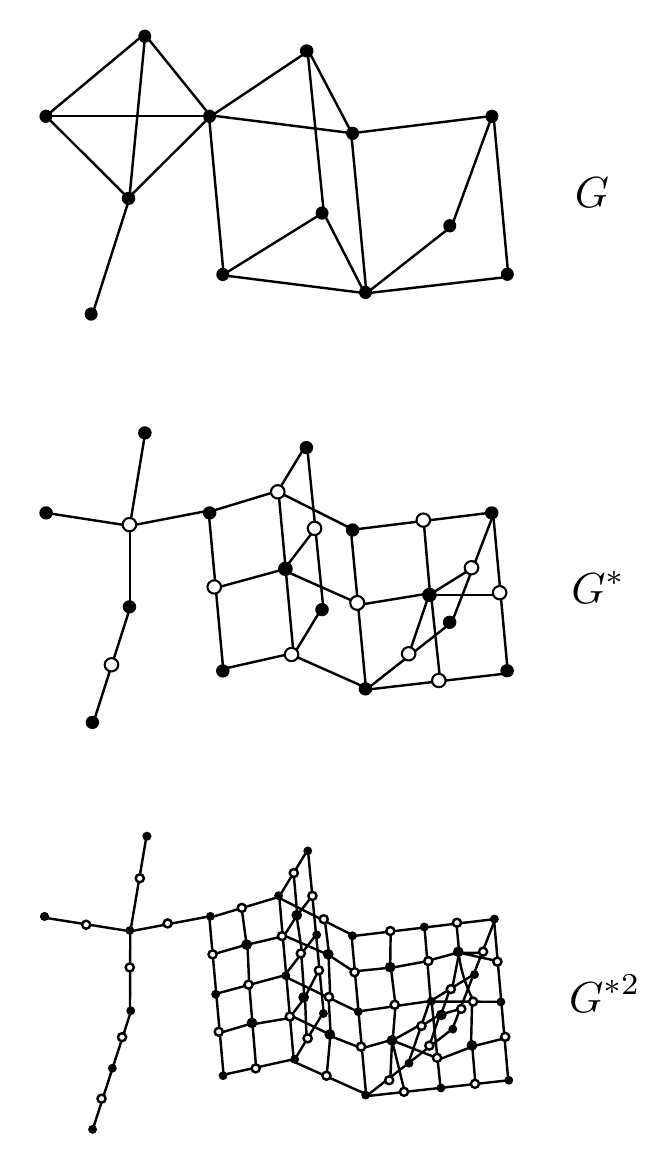}
\end{center}
\caption{An swm-graph $G$ and its barycentric graphs $G^*$ and $G^{*2}$}\label{fig:Gstar}
\end{figure}

Let $\mathcal{B}(G)$ denote the set of all Boolean-gated sets of finite diameter of an swm-graph $G$.
Regard $\mathcal{B}(G)$ as a poset with respect to the reverse inclusion.  Let $G^*$ be the covering
graph of the poset $\mathcal{B}(G)$. We call $G^*$ the {\it barycentric graph} of $G$; see Figure~\ref{fig:Gstar}.
Define the
length of edges of $G^*$ to be  one half of the length of edges of $G$. The Hasse orientation of
the edges of $G^*$
is denoted by $o^*$. If $G$ is an orientable modular graph,
then $G^*$ is equivalent to the $2$--subdivision of $G$ in the sense of~\cite{HH12} and is an orientable modular
graph. We extend this
result to barycentric graphs of all swm-graphs:

\begin{Thm}\label{thm:G*_is_om} The barycentric graph $G^*$ of any swm-graph $G$ is an orientable
modular graph and  $o^*$ is an admissible orientation of $G^*$. Moreover, the graph $G$ is isometrically
embedded  in $G^*$ via the map $p \mapsto \{p\}$.
\end{Thm}

\begin{proof} We start with some properties of the poset $\mathcal{B}(G)$.
\begin{Prop}\label{prop:ideal}
Let $G$ be an swm-graph.
\begin{itemize}
\item[(1)] The poset $\mathcal{B}(G)$ is graded.
\item[(2)] Every principal filter of $\mathcal{B}(G)$ is the subspace poset of a polar space.
\item[(3)] Every principal ideal of $\mathcal{B}(G)$ is a complemented modular join-semilattice.
\end{itemize}
\end{Prop}
\begin{proof}
To (1): For $X\in \mathcal{B}(G)$,  let
$r(X)$ be the diameter of $X$. We show that $- r$ is a grade function.
We use a general property of a dual polar graph $H$ of finite diameter $D$;
see Section~\ref{subsec:proof_of_lotoglo_dp23}.
By Lemma~\ref{l:Cameron},
the rank of $\mathcal{L}(x)$ for $x \in V(H)$
is equal to $\max_{y \in V(H)} d(x,y)$.
Cameron~\cite{Ca}*{p. 80} showed that
the rank of $\mathcal{L}(x)$ is independent of the choice of $x \in V(H)$.
Therefore, for every vertex $x$ there exists a vertex $y$
such that $d(x,y)$ equals $D$. By Lemma~\ref{lem:d(x,y)=D},
$\lgate x,y \rgate = V(H)$. %, where
%the property $\lgate x,y \rgate = V(H)$ follows from Lemma~\ref{lem:d(x,y)=D}.

Consider Boolean-gated sets $X, Y \in \mathcal{B}(G)$
of finite diameter such that $X$ covers $Y$ in the poset $\mathcal{B}(G)$.
By $X \subset Y$, there are adjacent vertices $x,y$
with $x \in X$ and $y \in Y \setminus X$.
Since $X$ induces a dual polar graph of finite diameter,
we can take $z \in X$ with $d(x,z) = r(X)$ and
$\lgate x,z \rgate = X$.
Then $d(z,y) \leq d(z,x)$ is impossible.
If $d(z,y)  = d(z,x) - 1$, then we have $y \in I(x,z) \subseteq X$,
contradicting $y \not \in X$.
If $d(z,y) = d(z,x)$, then (TC) for $x,y,z$ implies that
$y$ has two neighbors in $I(x,z)$, and by Lemma~\ref{lem-weakly-modular_gated_hull}
we have
$y \in \lgate x,z \rgate = X$, contradicting $y \not \in X$.
Thus $d(z,y) = d(z,x) + 1$, and
$X = \lgate z,x \rgate \subset \lgate z,y \rgate \subseteq Y$.
Since $X$ covers $Y$, we have $\lgate z,y \rgate = Y$,
and $r(Y) = r(X) + 1$, as required.

To (2): The filter of $X  \in \mathcal{B}(G)$ is
the set of all Boolean-gated sets contained in $X$,
and thus is the subspace poset of the polar space
defined by the dual polar graph $G(X)$ (Lemma~\ref{lem:special_subgeometries}).

To (3):
By (2) and Theorems~\ref{thm:Birkhoff} and \ref{thm:Tits},
every interval of $\mathcal{B}(G)$ is a complemented modular lattice.
Hence it suffices to show that for every vertex $p$
the principal ideal $(p)^{\downarrow}$ is a complemented modular join-semilattice.
Observe (from no-induced $K_4^-$ condition)
that maximal cliques are precisely Boolean-gated sets with diameter $1$.
Consider maximal cliques $S_1,S_2,\ldots, S_k$ of $G$ containing $p$ (i.e., coatoms of $(p)^{\downarrow}$) such that
for any two indices $i$ and $j$, the union $S_{i} \cup S_{j}$ is contained in a
Boolean-gated set (i.e., each pair of these coatoms has a common lower-bound).
By induction on $k$ we will show that the union  $\bigcup_{i=1}^k S_i$ is contained in a Boolean-gated set (i.e., there is a common lower-bound).
By induction hypothesis,
there is a Boolean-gated set containing $\bigcup_{i=1}^{k-1} S_i$, and hence
there is a vertex $q$ such that $(p,q)$ is Boolean
and $S_i$ contains an atom $a_i$ of $I(p,q)$ for $i=1,2,\ldots,k-1$.
We can assume that $d(p,q) = k-1$, and in particular $\{a_1,a_2,\ldots,a_{k-1}\}$ is a base of $I(p,q)$.
Also we can assume that $S_k \cap I(p,q) = \emptyset$;
otherwise $S_k \subseteq \lgate p,q \rgate$, as required.
Consider the join $z$ of $a_2,\ldots, a_{k-1}$ in $I(p,q)$.
Take any $a_k \in S_k$.
Then $d(a_k, z) = k-1$ must hold.
Indeed, $d(a_k,z) = k-3$ is impossible: otherwise, $a_k \in I(p,q)$.
If $d(a_k,z) = k-2$, then by (TC)
there exists a vertex $a\sim a_k,p$ with  $a \in I(p,q)$.
Then $a$ belongs to $S_k$ and
this contradicts $S_k \cap I(p,q) = \emptyset$.

By induction hypothesis, there is a Boolean-gated set $X$ containing $\bigcup_{i=2}^k S_i$.
By (A5) of Cameron's Theorem~\ref{th:cameron_dual_polar}, in
the dual polar graph induced by $X$
there is $q' \in X$ with $z \sim q'$ and $d(p,q') = k -1 = d(a_k,q') -1$.
Namely, $\{a_2,a_3,\ldots,a_k\}$ is a base of $I(p,q')$.
Again $S_1 \cap I(p,q') = \emptyset$.
Also by (A5) (applied to a Boolean-gated set containing $S_1 \cup S_k$) we can take a common neighbor $w$ of $a_1$ and $a_k$ different from $p$;
so $\{a_1,a_k\}$ is a base of $I(p,w)$.
We are going to show that there is $y$ with $S_i \cap I(p,y) = \{a_i\}$ for $i=1,2,\ldots,k$.
\begin{claim} The following equalities hold:

\begin{itemize}
\item[(1)] $d(a_1,q') = d(a_k, q) = k$.
\item[(2)] $d(w,q) = d(w,q') = k-1$.
\item[(3)] $d(z,w) = k$.
\end{itemize}
\end{claim}
\begin{proof}
(1): Then $d(a_1,q')\in \{ k-2,k-1,k\}$ (by $d(p,q') = k-1$).
If $d(a_1,q') = k-2$, then $a_1 \in I(p,q')$,
contradicting $S_1 \cap I(p,q') = \emptyset$.
If $d(a_1,q') = k-1$, by (TC)
there is $a \in I(p,q')$ with $a_1 \sim a$,
which yields the same contradiction.

(2): Then $d(w,q)\in \{ k-3,k-2,k-1\}$ (by $d(q,a_1) = k-2$).
Since $d(a_k,q) = k$ by (1),  $d(w,q) \leq k-2$ is impossible.

(3): Then  $d(z,w)\in \{ k-2,k-1,k\}$.
Suppose that $d(z,w) = k-2$.
By (QC)  for $a_1,p,w,z$,
there is $a \in I(p,z) \subseteq I(p,q')$ with $a\sim w,p$.
Consider the join $b$ of $a$ and $a_k$ in $I(p,q')$. Then $d(q', b) = k-3$.
Since $d(a_1,q') = k$, we have  $d(a_1,b) = 3$.
But then  $\{a_1,a_k, p, w,a,b\}$ induces an isometric $K_{3,3}^-$, a contradiction.

Suppose that $d(z,w) = k-1$.
By (TC) for $q,z,w$, there is $u\sim q,z$ and $d(u,w) = k-2$.
Since $d(q,a_k) = k$, it must hold $d(u,a_k) = k-1$.
Applying (TC) to the triplet $z,u,a_k$,
we will find $u'$ so that $\{q,z,u,u'\}$ induces a forbidden $K_4^-$.
\end{proof}

By (QC) for $z,q,q',w$, there is $y\sim q,q'$
with $d(w,y) = k-2$. We assert  that $d(p,y) = k$.
Indeed, $k-2\le d(p,y)\le k$.
Necessarily $z \neq y$ and $z \nsim y$.
First suppose that  $d(p,y) = k-2$.
By (QC) for $q,z,y,p$, there is $u \sim y,z$ with $d(u,p) = k-3$; so $q' \neq u$.
Since $d(q,a_k) =k$, we have $y,z, u, q'\in I(q,a_k)$,
where $k - 1 = d(y,a_k) = d(z,a_k)$ and $k - 2 = d(q',a_k) = d(u,a_k)$.
Thus $q, z, y, q', u$ and the join of $q',u$ in $I(q,a_k)$ induce
an isometric $K_{3,3}^-$, a contradiction.
Suppose now that $d(p,y) = k-1$.
Since $d(a_1,q') = k$ and $d(a_1,q)=k-2$, we have $d(a_1,y) = k-1$.
By (TC) for $a_1,p,y$, there is $a\sim a_1,p$ with $d(a,y) = k-2$.
Again, since $d(a_1,q') = k$, we have  $d(a,q') = k-1$.
By (TC) for $a,p,q'$,
there is an atom $a' \in I(p,q')$ with $a' \sim a$.
Since there is no induced $K_{4}^-$,
the atoms $a_1$ and $a'$ are adjacent,
contradicting that $S_1 \cap I(p,q') = \emptyset$.
Hence $d(p,y)=k$. Then $\{a_1,a_2,\ldots,a_k\} \subseteq I(p,y)$ is a base of $I(p,y)$ and $(p,y)$ is Boolean pair with
$\bigcup_{i=1}^k S_i \subseteq \lgate p,y \rgate$, as required.
\end{proof}

Now, we return to the proof of Theorem \ref{thm:G*_is_om}. To prove that $G^*$ is an orientable
modular graph, since $G^*$ is the covering graph of the poset $\mathcal{B}(G)$, it suffices to
show that $\mathcal{B}(G)$ satisfies
the conditions (a)--(e) of Theorem~\ref{thm:chara_om}.
By Proposition~\ref{prop:ideal},
$\mathcal{B}(G)$ satisfies (a) and (b).
By Lemma~\ref{lem:Boolean_intersection}, the nonempty intersection of Boolean-gated sets
is Boolean-gated.
Consequently,  $\mathcal{B}(G)$ satisfies (c).
Next consider the condition (d).
Suppose that a triplet $X,Y,Z$ of Boolean-gated sets
has a pairwise intersection, or equivalently, is pairwise upper-bounded.
By the Helly property for gated sets (Lemma~\ref{lem:gated_Helly}) and Lemma~\ref{lem:Boolean_intersection},
the intersection $X \cap Y \cap Z$ is a nonempty Boolean-gated set.
This means that the poset $\mathcal{B}(G)$ satisfies one half of (d).
To prove the second half of (d), suppose
that for a triplet $X,Y,Z \in \mathcal{B}(G)$
there exist $U,V,W \in \mathcal{B}(G)$ such that
$X \cup Y \subseteq U$, $Y \cup Z \subseteq V$, and $Z \cup X \subseteq W$.
We can assume that $X = U \cap W$, $Y = U \cap V$, and $Z = V \cap W$.
Then $U \cap V \cap W = X \cap Y \cap Z$ is a nonempty Boolean-gated set contained in $X, Y, Z$.
This shows that the triplet $X,Y,Z$ belongs to a common principal ideal and admits pairwise meets.
By Proposition~\ref{prop:ideal}, the principal ideal is a modular join-semilattice, and therefore
there exists the meet of $X,Y,Z$ in $\mathcal{B}(G)$, which is a Boolean-gated set containing $X,Y,Z$.

It remains to establish the condition (e)  that  the order complex $\Delta:=\Delta(\mathcal{B}(G))$ is simply connected.
Let $C = (X_1,X_2,\ldots,X_k)$ be a cycle in
the 1-skeleton of $\Delta$,
i.e., $X_{i} \prec X_{i+1}$ or $X_{i+1} \prec X_{i}$ (where the indices are taken modulo $k$).
We will show that $C$ can be contracted
to a single vertex by using triangles of $\Delta$.
First consider the case where
for each odd index $i$,
$X_{i}$ is a singleton $\{x_i\}$,
and for each even index $i$,
$X_{i}$ is covered by $X_{i-1}$ and $X_{i+1}$.
Such a cycle $C$ is called a {\it ridge cycle}; obviously,
any ridge cycle is identified with a cycle of $G$.
Since any cycle of $G$ is $0$--homotopic
in the triangle-square complex $X\trsq(G)$ of $G$, to show that a ridge cycle is $0$--homotopic
it suffices to show that any ridge cycle corresponding to a triangle or to a $4$--cycle of $G$ is $0$--homotopic
in the square-complex $X_{\square}(G^*)$ of $G^*$. For a triangle $xyz$ in $G$, there is a (unique)
Boolean-gated set (or a unique maximal clique) $X$
containing $x,y,z$ with diameter $1$.
The ridge cycle in $G^*$ corresponding to $(x,y,z)$ is $(x ,X, y, X, z, X)$,
which is obviously contractible in $X_{\square}(G^*)$.
Similarly, if a $4$--cycle belongs to a clique,
then the corresponding ridge cycle is contractible in $X_{\square}(G^*)$.
For a square $xyzw$ in $G$, there is a (unique) Boolean-gated set $X$
containing $x,y,z,w$ with diameter at most $2$.
Then $X$ is adjacent to four Boolean-gated sets $x \wedge y, y \wedge z, z \wedge w, w \wedge x$.
Therefore the ridge cycle $(x, x \wedge y, y, y \wedge z, z, z \wedge x)$ corresponding to cycle $(x,y,z,w)$
is contractible in $X_{\square}(G^*)$ via $X$.

Next we consider an arbitrary cycle $C$ in the $1$--skeleton of $\Delta$.
We will show that $C$ is homotopic to a ridge cycle.
First, for each $i$,
replace $X_{i}, X_{i+1}$ by a maximal chain between $X_i$ and $X_{i+1}$.
Obviously, the resulting cycle is homotopic to the original one.
Next pick an element in $C$ with a minimum grade (maximum diameter),
which we assume without loss of generality to be $X_2$.
If $X_1, X_3$ have a join (nonempty intersection) $Y$,  then
replace $X_2$ by $Y$ in $C$; the resulting cycle $C'$ is homotopic to $C$
via the two triangles $X_2X_1Y$ and $X_2X_3Y$ of
the order complex.
Suppose now that $X_1$ and $X_3$ have an empty intersection.
Take  $z \in V$ with $X_1 \preceq z$ in $\mathcal{B}(G)$.
Then $X_3 \not \preceq z$.
By Theorem~\ref{thm:Tits}~(P3),
there is a unique $w \in V$ such that $X_3 \preceq w$
and $z,w$ are adjacent in $G$.
Replace $X_2$ by $z, z \wedge w, w$ in $C$.
The resulting cycle $C'$ is homotopic to the original one.
Replace $C$ by $C'$.
Then either the minimum grade of vertices in $C$ decreases
or the number of vertices in $C$ having the minimum grade decreases.
Repeating this procedure, after a finite number of steps $C$ becomes a ridge cycle.

Finally we prove that $G$ is isometrically embeddable in $G^*$.
It is easy to see that $d_{G^*}(p,q) \leq d_{G}(p,q)$ for $p,q \in V(G)$.
We establish the converse inequality  by induction on $k := d_{G}(p,q)$.
The case $k=1$ is obvious.
Consider a shortest path $(p= x_0, x_1,x_2,\ldots, x_m = q)$ in $G^*$
and consider an index $j$ with
$x_0 \succ x_1 \succ \cdots \succ x_{j-1} \succ x_j \prec x_{j+1}$.
Then necessarily $x_{j+1} \not \preceq p$.
Since the filter of $x_j$ is the subspace poset of a polar space,
by Theorem~\ref{thm:Tits}~(P3),
there exists a vertex $p' \in V(G)$ adjacent to $p$ in $G$ with
$p' \succeq x_{j+1}$.
Since there is a path $(p' = y_1,y_2,\ldots, y_{j} = x_{j+1})$ from $p'$ to $x_{j+1}$
of length $j-1$, we have $d_{G^*}(p,q)\ge 1 + d_{G^*}(p',q)$.
Applying the induction hypothesis to $p',q$,  we obtain $d_{G^*}(p,q)\ge d_G(p,q)$ as required.
\end{proof}

Since $G^*$ is also an swm-graph, we can consider its barycentric graph $(G^{*})^{*}$, which is
again an swm-graph.
Let $G^{*i}$ denote the barycentric graph of $G^{*(i-1)}$,
where $G^{*0} := G$ and $G^{*1} := G^*$.
Each $G^{*i}$ is isometrically embedded in $G^{*(i+1)}$ (recall that the edge-lengths of $G^{*(i+1)}$ are one-half of the edge-lengths of
$G^{*i}$); denote this  by $G^{*i} \hookrightarrow G^{*(i+1)}$.
See Figure~\ref{fig:Gstar} for example.
As a result, we obtain the following sequence of swm-graphs:
\[
G \hookrightarrow G^{*} \hookrightarrow G^{*2} \hookrightarrow G^{*3} \hookrightarrow \cdots.
\]
In Chapter~\ref{sec:complex},
we will show that this graph sequence {\em converges} to
a certain metric simplicial complex associated with $G$;
see Figure~\ref{fig:KG} for the limit of the swm-graph in Figure~\ref{fig:Gstar}.

By Proposition~\ref{prop:ideal}(2)
every interval of ${\mathcal B}(G)$ is a complemented modular lattice,
and by Lemma~\ref{lem:o-Boolean} a vertex set in $G^*$ is
Boolean if and only if it is an interval in the poset $\mathcal{B}(G)$.
Hence we have:
\begin{Lem}\label{lem:G*2}
$G^{*2}$ is the covering graph of the poset of all intervals
of $\mathcal{B}(G)$ with respect to the reverse inclusion order.
\end{Lem}

The {\em cube-dimension} of an swm-graph $G$ is
the maximum diameter of a Boolean-gated set of $G$,
or equivalently, the maximum diameter of an isometric cube subgraph of $G$.
It is also equal to the maximum length of a chain in $\mathcal{B}(G)$,
which is equal to the cube-dimension of $G^*$ by Lemma~\ref{lem:o-Boolean}.
\begin{Lem}
The cube-dimensions of $G$ and  $G^*$ are the same.
\end{Lem}

\begin{example}[Median graphs]
Consider the case where $G$ is a median graph.
Then every cube subgraph is an isometric (in fact, gated) subgraph of $G$. Moreover,
a vertex set is Boolean-gated if and only if
it is the vertex set of a cube subgraph.
Hence the poset $\mathcal{B}(G)$ is
isomorphic to the face poset of the cube complex (median complex) $X_{cube}(G)$ of $G$.
As a consequence, the barycentric graph $G^*$ of $G$ is obtained
by adding a vertex on the midpoint of
each face and joining them by an edge if and only if one of the faces
is a facet of the other.
The limit of the sequence  $G^{*i}$ is the median cube complex $X_{cube}(G)$ of $G$
endowed with the intrinsic $l_1$--metric;
see Chapter~\ref{sec:complex}.
\end{example}

\begin{example}[Block graphs]
Swm-graphs of cube-dimension $1$ are the $K^-_4$ and $C_4$--free graphs. Those are exactly the {\em block graphs}, i.e.,
the graphs $G$ whose 2-connected components are cliques.
In this case, the barycentric graph $G^*$ of $G$ is
a tree obtained by replacing each maximal clique of $G$ by a star.
The graphs $G^{*2}, G^{*3}, \ldots$ are subdivisions of the tree $G^*$ in the graph-theoretical sense.
The limit of the sequence  $G^{*i}$ is
the $1$--dimensional contractible complex corresponding to the tree $G^*$.
\end{example}

\begin{example}\label{ex:frames}[Frames and semiframes]
A {\em frame} is an orientable bipartite graph
without isometric $k$--cycles with $k \geq 6$.
A {\em semiframe} is a bipartite graph
without induced $K_{3,3}^-$ and isometric $k$--cycles with $k \geq 6$.
It is known that bipartite graphs without  isometric $k$--cycles, $k \geq 6$,
are exactly the {\em hereditary modular graphs}, i.e., modular graphs in which
all isometric subgraphs are modular~\cite{Ba_hereditary}.
Hence frames and semiframes
are nothing but orientable modular graphs and strongly modular graphs
with cube-dimension $\le 2$, respectively.
Frames and semiframes were introduced by Karzanov~\cites{Kar98a, Kar98b}
in the study of metric extensions and multicommodity flows.
In semiframes, Boolean-gated sets of diameter $2$ are
exactly the vertex sets of maximal complete bipartite subgraphs, which we call  {\em bicliques}.
Since frames are orientable, any biclique in a frame must be a $K_{2,n}$ for $n \geq 2$.

The barycentric graph $G^*$ of a semiframe $G$ can be
obtained by
adding a vertex to each maximal biclique, subdividing each edge,
and joining by an edge each vertex corresponding to a biclique to
the subdividing vertex of any edge of that biclique.
The resulting graph $G^*$ is a frame;
notice that the same holds for any swm-graph with cube-dimension $2$.
In \cite{Kar98b},
this procedure of deriving frames from semiframes was called an {\em  orbit-splitting}.
Hence the construction of the barycentric graphs of swm-graphs can be
viewed as a higher-dimensional generalization of the orbit-splitting.

Chepoi~\cite{Ch_CAT} considered
hereditary modular graphs without $K_{3,3}^-$ and $K_{3,3}$, and called them
{\em $F$--graphs}.
$F$--graphs generalize frames and are particular semiframes.
He showed that the B$_2$--complex obtained from an $F$--graph
by filling a {\it folder} to each biclique is CAT$(0)$,
and any CAT$(0)$ B$_2$--complex, called a {\em folder complex},
arises in this way.
We will see again the folder complexes in Chapter~\ref{sec:complex}.
\end{example}

\section{Thickening $G^{\Delta}$}\label{subsec:G^Delta}
The {\it thickening} of an swm-graph $G$ is the graph  $G^\Delta$ obtained
from $G$ by adding edges to all Boolean pairs $(x,y)$ with $d(x,y) \geq 2$.

\begin{Thm}\label{thm:Helly}
Let $G$ be an swm-graph.
Then $G^{\Delta}$ is a finitely Helly graph.
If all Boolean-gated sets of $G$ are finite (in particular $G$ is locally finite),
then $G^{\Delta}$ is a Helly graph and the clique complex
$X(G^{\Delta})$ is contractible.
\end{Thm}

\begin{proof} To prove that  $G^{\Delta}$ is a finitely Helly graph, by Theorem \ref{t:lotogloHell} it
suffices to show that the clique complex of $G^{\Delta}$
is simply connected and that the collection of maximal cliques of $G^{\Delta}$  has finite Helly property.
To prove the simple connectivity of $X(G^{\Delta})$ first we show  that any cycle $C$ of $G^{\Delta}$ is homotopic to a cycle $C_0$ of $G$ and
then we prove that $C_0$ is 0-homotopic. For any edge $xy$ of $C$, either $x$ and $y$ are adjacent in $G$ or they constitute a Boolean pair in $G$. In the second case,
choose any vertex $z \in I(x,y) \setminus \{x,y\}$. Then $(x,z)$ and $(z,y)$ are also Boolean pairs, thus $xz$ and $zy$ are edges of $G^{\Delta}$.
Replace in $C$ the edge $xy$
by the edges $xz,zy$ and denote the resulting cycle by $C'$.
Observe that $xy,xz,zy$ form a triangle in $G^{\Delta}$.
Hence the resulting cycle $C'$ is homotopic to the original cycle $C$. Continuing this process with the new cycle, after a finite number of steps we will obtain
a cycle $C_0$ of $G$ which is homotopic to $C$. Since $G$ is weakly modular, $C_0$ is null-homotopic in the triangle-square complex $X\trsq(G)$ of $G$. Each triangle of $G$
is a triangle of $G^{\Delta}$. On the other hand, each square of $G$ gives rise to a 4-clique of $G^{\Delta}$ because the vertices of that square
are pairwise Boolean.  Thus we can construct a homotopy in $X(G^{\Delta})$ from a homotopy in $X\trsq(G)$, whence $C_0$ is also null-homotopic in $X(G^{\Delta})$.

Next we will prove that each 1-ball $B_1(p)$ of $G^{\Delta}$ defines a gated set of the initial graph $G$. By Proposition \ref{prop:char_Boolean}(iv), if $(p,q)$ is a Boolean pair
and $w\in \lgate p,q \rgate,$ then $(p,w)$ is also a Boolean pair, hence the subgraph of $G$ induced by $B_1(p)$ is connected. By Lemma~\ref{lem-weakly-modular_gated_hull} it suffices to show that
if $u,v\in B_1(p), u\ne v,$ and a vertex $w$ is adjacent to $u,v$ in $G$, then $w$ belongs to $B_1(p)$. Indeed, since $u,v\in B_1(p),$ $(p,u)$ and $(p,v)$ are Boolean pairs of $G$. By
Lemma~\ref{lem:Boolean123}(2), $(p,w)$ is also a Boolean pair of $G$, thus $p$ and $w$ are adjacent in  $G^{\Delta}$. Hence the 1-balls of  $G^{\Delta}$ are gated sets of $G$.
If $C$ is a maximal clique of $G^{\Delta}$, then $C$ is the intersection of the 1-balls $B_1(p)$ of $G^{\Delta}$ centered at the vertices $p\in C$. Since each 1-ball of $G^{\Delta}$
is gated in $G$, $C$ is the intersection of gated sets of $G$, and therefore, $C$ is a Boolean-gated subset of $G$. By Lemma~\ref{lem:gated_Helly}, the collection of maximal cliques of  $G^{\Delta}$ has the finite Helly property, whence $G^{\Delta}$ is  finitely Helly.

Now suppose that all Boolean-sets of $G$ are finite. Since any maximal clique of $G^{\Delta}$ is Boolean-gated, the clique complex of $G^{\Delta}$  is finite-dimensional. Therefore the finite clique-Helly property implies the clique-Helly property. By Theorem \ref{t:lotogloHell}, $G^{\Delta}$  is a Helly graph. Moreover, by Theorem \ref{t:lotogloHell}(v)
the clique complex of $G^{\Delta}$ is contractible. If $G$ is locally finite, then from Proposition
\ref{prop:char_Boolean}(v) and analogously to the proof of Theorem \ref{t:lotoglo_dp3} one can deduce that each Boolean-gated set of $G$ is finite.
\end{proof}

%This assertion
%also follows from Proposition~\ref{Helly_Polat} \cite{Po_helly} that a
%finitely Helly graph without infinite cliques is Helly. As a locally
%finite Helly graph, the graph $G^{\Delta}$ is dismantlable
%(Theorem~\ref{t:lotogloHell}), and thus its
%clique complex $X(G^{\Delta})$ is contractible.
%\end{proof}

We continue with further properties of thickenings  of  swm-graphs.
The edge-length of $(G^*)^{\Delta}$ is defined
as one half of the edge-length of $G^{\Delta}$.
\begin{Prop}
$G^{\Delta}$ is isometrically embeddable in the graph $(G^*)^{\Delta}$
via the map $p \mapsto \{p\}$.
\end{Prop}
\begin{proof}
Since every interval of $\mathcal{B}(G)$ is a complemented modular lattice,
a pair $X,Y \in \mathcal{B}(G)$ is Boolean in $G^*$
if and only if $X \cap Y$ is nonempty ($X \vee Y$ exists)
and there is $Z \in \mathcal{B}(G)$ with $X \cup Y \subseteq Z$ ($X \wedge Y$ exists).
Then it is easy to see that $d_{(G^*)^{\Delta}}(p,q) \leq d_{G^{\Delta}}(p,q)$.
We show now the reverse inequality.
Take a shortest path $(\{p\} = X_0,X_1,\ldots, X_m = \{q\})$ in $(G^{*})^{\Delta}$.
Then $\{p\} = X_0 \subset X_1$,
both $X_{i} \wedge X_{i+1}$ and $X_{i} \vee X_{i+1}$ exist, and $X_{m-1} \supset X_m = \{q\}$.
If $m$ is odd, then the path $(\{p\} = X_0, X_1 \wedge X_2, X_2 \vee
X_3, X_3 \wedge X_4, \ldots,X_{m-2} \wedge X_{m-1}, X_{m} = \{q\})$
has a shorter length, a contradiction.  Thus $m$ is even.  Take a
vertex $p_i$ from $X_{2i} \cap X_{2i+1} (= X_{2i} \vee X_{2i+1})$ for
$i=0,1, 2,\ldots,(m-2)/2$.  Note that $p_0 = p$ and let $p_{m/2} := q$.
Then $p_i$ and $p_{i+1}$ are adjacent in $G^{\Delta}$ since they
belong to Boolean-gated set $X_{2i+1} \wedge X_{2i+2}$.  Hence
$d_{(G^{*})^{\Delta}}(p,q) = m/2 \geq d_{G^{\Delta}}(p,q)$.
\end{proof}

The following property extends the gatedness in $G$ of 1-balls of $G^{\Delta}$ used
in the proof of Theorem \ref{thm:Helly}.

\begin{Prop}\label{gated_ball}
Any ball of $G^{\Delta}$ is gated in $G$.
\end{Prop}
\begin{proof}
Let $B_k(p)$ be the $k$--ball of $G^{\Delta}$ centered at $p$. Again we use
Lemma~\ref{lem-weakly-modular_gated_hull}.
For $x,y \in B_k(p)$, let $u$ be a common neighbor of $x,y$ in $G$.
We assert that $d_{G^{\Delta}}(p,u)\le k$.
Let $B_1(x)$ and $B_1(y)$ denote
the $1$--balls of  $G^{\Delta}$ with centers $x$ and $y$, respectively.
Let $B_{k-1}(p)$ be the $(k-1)$--ball centered at $p$.
By their definition, the balls $B_1(x)$, $B_1(y)$, and $B_{k-1}(p)$ pairwise intersect.
By the finite Helly property, these three balls have a common vertex $q$,
which is a common neighbor of $x,y$ in $G^{\Delta}$
with $d_{G^{\Delta}}(p,q)\le k-1$. Hence $(q,x)$ and $(q,y)$ are Boolean pairs of $G$.
By Lemma~\ref{lem:Boolean123}~(2), the pair $(q,u)$ is also Boolean.
This means that $d_{G^{\Delta}}(p,u)\leq k$ and thus $u\in B_k(p)$.
\end{proof}

By definition, a Boolean-gated set forms a clique in $G^{\Delta}$.
A partial converse also holds:
\begin{Lem}\label{lem:pairwiseBoolean}
Every finite clique of $G^{\Delta}$
belongs to a Boolean-gated set of $G$.
If $G$ has  finite cube-dimension,
then  every clique of $G^{\Delta}$ belongs
to a Boolean-gated set.
\end{Lem}
\begin{proof}
We first prove the claim for finite clique $X$ by induction on $k := |X|$.
Suppose that  $X = \{x_1,x_2,\ldots,x_k\}$.
By induction, there are Boolean-gated sets
$Y$ and $Z$ such that $Y$ contains $x_1,x_2,\ldots,x_{k-1}$
and $Z$ contains $x_2,\ldots,x_k$.
Now $\lgate x_1,x_k \rgate$ is also Boolean-gated.
Let $A := Y \cap \lgate x_1,x_k \rgate$,
$B := Z \cap \lgate x_1,x_k \rgate$, and $C:= Y \cap Z$.
Then $x_1 \in A$, $x_k \in B$, and $x_2,\ldots,x_{k-1} \in C$, and
$A, B, C$ are Boolean-gated by Lemma~\ref{lem:Boolean_intersection}.
Also $(A,B,C)$ is a pairwise lower-bounded triplet in $\mathcal{B}(G)$
by $A \cup B \subseteq \lgate x_1,x_k \rgate$,
$B \cup C \subseteq Z$, and $C \cup A \subseteq Y$.
By Theorem~\ref{thm:G*_is_om} and Theorem~\ref{thm:chara_om}~ (d),
there is a Boolean-gated set containing $A \cup B \cup C \supseteq X$, as required.

Next suppose that $G$ has finite cube-dimension.
Let $X$ be an arbitrary (infinite) clique of $G^{\Delta}$.
By the above finite case,
every finite subset $X'$ of $X$ has meet in $\mathcal{B}(G)$,
which is a unique minimal Boolean-gated set containing $X'$.
Since the cube-dimension of $G$ is finite,
there is no infinite chain in ${\mathcal B}(G)$.
Thus we can take a finite subset $X'$ of $X$
such that the meet $Z$ over $X'$ has minimal grade.
Then $Z$ contains $X$.
Otherwise, take $y \in Z \setminus X$.
Since $\{y\} \cup X'$ is finite,
we can consider the meet $Z'$ over $\{y\} \cup X'$.
Then $Z'$ contains $Z$ properly.
Namely the grade of $Z'$ is less of that of $Z$, a contradiction.
\end{proof}

\section{$\Delta$--gates and geodesic extension property}
\label{subsub:delta-gate}

For two vertices $p,q$, the {\em $\Delta$--gate} of $q$ at $p$ is the
gate of $p$ in $G$ relative to the $(d_{G^{\Delta}}(p,q)-1)$--ball of
$G^\Delta$ centered at $q$ (which is a gated set of $G$ by Proposition
\ref{gated_ball}).
\begin{Lem}\label{lem:delta-gate}
The $\Delta$--gate of $q$ at $p$ is a unique vertex $u$ having the following two properties:
\begin{itemize}
\item[{\rm (1)}] $d_{G^{\Delta}}(p,q) = d_{G^{\Delta}}(u,q) + 1$.
\item[{\rm (2)}] $\lgate p,u \rgate \subseteq \lgate p,v \rgate$ for
every vertex $v$ with $d_{G^{\Delta}}(p,q) = d_{G^{\Delta}}(v,q) + 1$.
\end{itemize}
\end{Lem}
\begin{proof}
Take any vertex $v$ with $d_{G^{\Delta}}(p,q) = d_{G^{\Delta}}(v,q) + 1$.
Then $v$ belongs to the $(d_{G^{\Delta}}(p,q)-1)$--ball centered at $q$.
Therefore $u \in I(p,v)$ (since $u$ is the gate of $p$ in this ball),
implying $I(p,u) \subseteq I(p,v)$ and $\lgate p,u \rgate \subseteq \lgate p,v \rgate$.
In particular, $(p,u)$ is Boolean, and we have (1) and (2).
Suppose that $u \neq v$.
Then $d(p,u) < d(p,v)$,
and hence $\lgate p,u \rgate \subset \lgate p,v \rgate$ (by Lemma~\ref{l:gated_hull_dpg2}). This concludes the uniqueness.
\end{proof}

\begin{Prop}[{\em Geodesic extension property}]\label{prop:gep2}
For any three vertices $x,y,p$, the following two conditions are equivalent:
\begin{itemize}
\item[{\rm (i)}] $d_{G^{\Delta}}(x,y) = d_{G^{\Delta}}(x,p) + d_{G^{\Delta}}(p,y)$;
\item[{\rm (ii)}] the $\Delta$--gates of $x$ and $y$ at $p$ are distinct and not adjacent in $G^{\Delta}$.
\end{itemize}
\end{Prop}
\begin{proof} The implication (i) $\Rightarrow$ (ii) is obvious.
Next we show that (ii) $\Rightarrow$ (i). Let $g$ and $h$ be the
$\Delta$--gates of $x$ and $y$ at $p$, respectively.  Suppose that (i)
fails, i.e., $d_{G^{\Delta}}(x,y) < d_{G^{\Delta}}(x,p) +
d_{G^{\Delta}}(p,y)$.  Take a quasi-median $x'p'y'$ of $x,p,y$ in
$G^{\Delta}$, which by Lemma~\ref{lem-weakly-modular} is an
equilateral metric triangle, since as a Helly graph, $G^{\Delta}$ is weakly
modular.  First suppose that $p' \neq p$.  Then there is a neighbor
$z$ of $p$ with $d_{G^{\Delta}}(z,x) = d_{G^{\Delta}}(p,x) - 1$ and
$d_{G^{\Delta}}(z,y) = d_{G^{\Delta}}(p,y) - 1$.  By
Lemma~\ref{lem:delta-gate}, both $\lgate p, g \rgate$ and $\lgate p, h
\rgate$ belong to $\lgate p,z \rgate$.  Hence, $(g,h)$ is Boolean and
$g,h$ are adjacent in $G^{\Delta}$.

Now suppose that $p' = p$.
Take a neighbor $u$ of $p$ with $d_{G^{\Delta}}(u,x') = d_{G^{\Delta}}(p,x') - 1$.
Since the metric triangles are equilateral,
by (TC) we can take a common neighbor $v$ of $p,u$ with
$d_{G^{\Delta}}(v,y') = d_{G^{\Delta}}(p,y') - 1$.
By Lemma~\ref{lem:pairwiseBoolean},
there is a Boolean-gated set $X$ containing $p,u,v$.
By Lemma~\ref{lem:delta-gate}, $\lgate p,g \rgate \subseteq \lgate p,u \rgate$
and $\lgate p,h \rgate \subseteq \lgate p,v \rgate$.
Hence both $g$ and $h$ belong to $X$, and therefore they are adjacent
in $G^{\Delta}$.
\end{proof}

\begin{lemma}\label{lem-BP-prop}
Let $p,x$ be two arbitrary vertices of an swm-graph $G$ and let $x'$
be the $\Delta$-gate of $p$ at $x$. If $d_{G^\Delta}(p,x) = k$, then
$x' \in B_1(y,G^\Delta)$ for any $y \in B_{k}(p,G^\Delta) \cap
B_1(x,G^\Delta)$.
\end{lemma}

\begin{proof} We distinguish two cases, depending on the value of
$d_{G^\Delta}(p,y)$.  If $d_{G^\Delta}(p,y)=d_{G^\Delta}(p,x')=d_{G^\Delta}(p,x)-1=k-1$,
then $x'\in I(x,y)$ in $G$ since $x'$ is the $\Delta$-gate of $p$ at $x$. Since $(x,y)$
is a Boolean pair, by Proposition~\ref{prop:char_Boolean}, $(x',y)$ is also a Boolean pair
and thus $x' \sim y$ in $G^\Delta$.

If $d_{G^\Delta}(p,y)=d_{G^\Delta}(p,x)=k$, by (TC) applied to the triplet $x,y,p$ in $G^\Delta$,
there exists a vertex $z \in B_{k-1}(p,G^\Delta)$ such that $z \sim x,y$ in $G^\Delta$. Applying the previous
case with $y = z$, we have that $z \sim x'$ in $G^\Delta$. By Lemma~\ref{lem:pairwiseBoolean}, there exists a
Boolean-gated set $B$ containing $x,y,z$. Since $x' \in I(x,z)$, necessarily $x' \in B$. Consequently,
$(x',y)$ is a Boolean pair, i.e., $x' \sim y$ in $G^\Delta$.
\end{proof}

\section{Normal Boolean-gated paths}\label{subsec:normalbgpaths}
%\subsubsection{Normal Boolean-gated paths}\label{normal}

Let $G$ be an swm-graph. A path $\gamma=(p = x_0,x_1,\ldots,x_k = q)$ of $G^{\Delta}$
is called a {\em normal Boolean-gated path} (a {\em normal bg-path} for short) if
for any $i=1,2,\ldots,k-1$ and any Boolean-gated set
$B$ containing $\lgate x_{i-1}, x_{i} \rgate$, the equality
$B \cap \lgate x_{i}, x_{i+1} \rgate = \{x_i\}$ holds. If $G$ is
a median graph, then any  Boolean-gated set is a cube and the notion
of normal bg-path coincides with the notion of normal cube-path
introduced by \cite{NiRe}. The following theorem generalizes a similar result
of \cite{NiRe}  for normal cube-paths in median graphs (see also \cite{Ch_CAT}
for the case of folder complexes):

\begin{Thm}[{Normal bg-paths}]\label{th-nbgp}
For any pair $p,q$ of vertices of an swm-graph $G$, there is a unique
normal bg-path
$\gamma_{pq} = (p = x_0,x_1,x_2,\ldots,x_k = q)$
between $p$ and $q$, whose vertices are  given by
\begin{equation}\label{eqn:given}
 x_i := \mbox{the $\Delta$--gate of $p$ at $x_{i+1}$} \quad (i =k-1,k-2,\ldots,1,0).
\end{equation}
In particular, any normal bg-path is a shortest $(p,q)$--path in $G^{\Delta}$.
\end{Thm}
\begin{proof}
Observe that $\gamma_{pq}$ (defined by (\ref{eqn:given})) is a shortest path
(by Lemma~\ref{lem:delta-gate}).
First we verify that $\gamma_{pq}$ is a normal bg-path.
Pick a vertex $u$ from $B \cap \lgate x_i, x_{i+1} \rgate$
for a Boolean-gated set $B$ including $\lgate x_{i-1}, x_{i} \rgate$.
Then the pair $(u, x_{i-1})$ is Boolean, and
thus $u$ is a common neighbor of $x_{i-1}$ and $x_{i+1}$ in $G^{\Delta}$.
Since $x_{i}$ is the $\Delta$--gate of $p$ at $x_{i+1}$,
we have $\lgate x_i,x_{i+1} \rgate \subseteq \lgate u, x_{i+1} \rgate$.
Since $u$ belongs to $\lgate x_i,x_{i+1} \rgate$,
we have $\lgate x_i,x_{i+1} \rgate = \lgate u, x_{i+1} \rgate$.
By Lemma~\ref{lem:delta-gate} we conclude that $x_i = u$.
Hence $\gamma_{pq}$ is normal Boolean-gated.

Let $\gamma'_{pq} = (p = y_0,y_1,y_2,\ldots,y_l = q)$ be an arbitrary
normal bg-path.  We show that $\gamma'_{pq}$ is a
shortest path in $G^{\Delta}$. Suppose not: then there is an index $i$
such that $(y_i,y_{i+1},\ldots,y_l = q)$ is shortest but $(y_{i-1},
y_i, y_{i+1},\ldots,y_l = q)$ is no longer shortest, i.e.,
$d_{G^\Delta}(q,y_{i-1}) \leq d_{G^\Delta}(q,y_i)$. Let $u$ be the
$\Delta$--gate of $q$ at $y_i$.  By Lemma~\ref{lem-BP-prop}, either
$y_{i-1} = u$ or $y_{i-1}$ is adjacent to $u$.  By
Lemma~\ref{lem:pairwiseBoolean} there is a Boolean-gated set $B$
containing the vertices $y_{i-1}, u, y_{i}$.  Also $u$ is contained in
$\lgate y_i,y_{i+1} \rgate$ by Lemma~\ref{lem:delta-gate}.  This means
that $B \cap \lgate y_i,y_{i+1} \rgate \supseteq \{u,y_i \}$,
contradicting the normality of $\gamma'_{pq}$.

Next we show that $y_i = x_i$ by induction on $k$ and reverse induction on $i$.
Suppose that $y_{i} = x_{i}$ and $x_{i-1} \neq y_{i-1}$.
By Lemma~\ref{lem:delta-gate},
$\lgate y_{i-1},x_{i} \rgate \supset \lgate x_{i-1},x_{i} \rgate$,
and in particular $y_{i-1}$ and $x_{i-1}$ are adjacent.
By (TC) in $G^{\Delta}$  there is a common neighbor $w$ of $y_{i-1}$ and $x_{i-1}$
with $d_{G^{\Delta}}(p,w) = i-1$.
By induction hypothesis on $k$, $y_{i-2}$ must be the $\Delta$--gate of $p$ at $y_{i-1}$.
This means that $\lgate y_{i-1},y_{i-2} \rgate \subseteq \lgate y_{i-1}, w \rgate$.
By Lemma~\ref{lem:pairwiseBoolean},
there is a Boolean-gated set $B$ containing $y_{i-1},w,x_{i-1}$.
Then $B$ contains $\lgate y_{i-1}, y_{i-2} \rgate$.
The intersection $B \cap \lgate y_{i-1}, x_i \rgate$ contains $y_{i-1}$ and $x_{i-1}$,
a contradiction. Thus $x_{i-1} = y_{i-1}$.
\end{proof}

The following hereditary property of normal bg-paths will be used in
Section~\ref{subsec:biauto}.

\begin{lemma}\label{lem-magic}
Let   $\gamma_{p,q}=(p =x_0,x_1,\ldots,x_{k-1},x_k = q)$ be the normal bg-path between  $p$ and $q$  in  an swm-graph $G$.
Then for any vertex $z \in \lgate x_{k-1},x_{k} \rgate$, if $d_{G^\Delta}(p,z) =
d_{G^\Delta}(p,q) = k$, then the path $\gamma'=(p =x_0,x_1,x_2,\ldots,x_{k-1},x_k' = z)$ is the normal
bg-path $\gamma_{pz}$.
\end{lemma}

\begin{proof} By Theorem~\ref{th-nbgp}, it is enough to show that $\gamma'$ is a normal
bg-path. Note that if $k = 1$, then $\gamma'=(p,z)$ is trivially a normal
bg-path, and we are done. Assume now that $k \geq 2$.
Since $\gamma_{pq}$ is the normal bg-path between $p$ and $q$, for any $i \leq
k-2$ and for any Boolean-gated set $B$ containing $\lgate
x_{i-1},x_i\rgate$, we have $B \cap \lgate x_i, x_{i+1} \rgate = \{x_i\}$.
Consider now a Boolean-gated set $B$ containing $\lgate
x_{k-2},x_{k-1} \rgate$. Note that $\lgate x_{k-1},z \rgate \subseteq
\lgate x_{k-1},x_k \rgate$, and consequently, $\lgate x_{k-1},z \rgate
\cap B \subseteq \lgate x_{k-1},x_k \rgate \cap B= \{x_{k-1}\}$ because
$\gamma_{pq}$ is a normal bg-path. This concludes  the proof of
the lemma.
\end{proof}

\section{Euclidean buildings of type C$_n$}\label{subsec:building}
In this section,
we explain how particular swm-graphs arise from buildings of type C$_n$.
Our references on buildings are the books \cites{Ti, BuildingBook}.
Let us briefly review the basic notions of building theory.

\subsection*{Chamber complexes}
A {\em chamber complex} $\Sigma$ is an abstract simplicial complex such that
any maximal simplex, called a {\em chamber}, has the same number $n$ of vertices
and for any two chambers $C,D$
there is a sequence of chambers $C = C_0,C_1,\ldots,C_k = D$, called a {\em gallery},
with $|C_i \cap C_{i+1}| = n-1$. The number $n$ is called the {\em rank} of
$\Sigma$.

A {\em coloring} of a chamber complex $\Sigma$ of rank $n$
is a map $\varphi$
from the vertex set of $\Sigma$ to an $n$--element set $I$
such that $\varphi$ is bijective on the vertex set of each chamber.
A coloring, if it exists, is essentially unique,
and is uniquely determined by a coloring of an arbitrary chamber of $\Sigma$.
For a vertex $v$,  $\varphi(v)$ is called the {\em type of $v$.}

%% \smallskip\noindent
%% {\bf Coxeter complexes of type C$_n$.}

\subsection*{Coxeter complexes of type C$_n$}
Instead of defining general  Coxeter complexes,
we directly introduce spherical and Euclidean Coxeter complexes of type C$_n$.
For $\alpha \in \RR^n \setminus \{0\}$ and $\beta \in \RR$,
let $H_{\alpha,\beta}$ denote the hyperplane
$\{ x \in \RR^n:   \langle \alpha, x \rangle = \beta \}$.
Consider the set $\mathcal{H}$ of all hyperplanes of types
$H_{e_i + e_j, b}$ $(1 \leq i \leq j \leq n, b \in \ZZ)$
and $H_{e_{i} - e_{j}, b}$ $(1 \leq i < j \leq n, b \in \ZZ)$.
The closure of each connected component of
$\RR^n \setminus \bigcup_{H \in \mathcal{H}} H$
is an $n$--dimensional simplex which is the convex hull of the $n+1$ points
\[
 x + \frac{1}{2}(1,1,\ldots,1) + \frac{1}{2} \left( \sigma(i_1) e_{i_1} + \sigma(i_2) e_{i_2} 
+ \cdots + \sigma(i_k) e_{i_k} \right) \quad (k=0,1, 2,\ldots, n)
\]
for an integer vector $x$,
a permutation $(i_1,i_2,\ldots,i_n)$ on $\{1,2,\ldots,n\}$,
and a sign map $\sigma:\{1,2,\ldots,n\} \to \{-1,1\}$.
The set $\Delta$ of all such simplices and their faces is a simplicial complex
on the set $(\ZZ/2)^n$ of all half-integral vectors.
A {\em Euclidean Coxeter complex of type C$_n$}
is an abstract simplicial complex isomorphic to $\Delta$.
A {\em spherical Coxeter complex of type C$_n$} is
an abstract simplicial complex isomorphic to
the subcomplex of $\Delta$ consisting of
simplices lying on the $l_{\infty}$--sphere
$\{ x \in \RR^n: \|x - (1/2,1/2, \ldots,1/2)\|_{\infty} = 1/2\}$ (which is the boundary of $[0,1]^n$).

%% \smallskip\noindent
%% {\bf Buildings of type C$_n$.}

\subsection*{Buildings of type C$_n$}
A {\em spherical/Euclidean building of type C$_n$} is a simplicial complex $\Delta$
that is the union of subcomplexes, called {\em apartements}, satisfying the following axioms:
\begin{itemize}
\item[B0:] Each apartment is a spherical/Euclidean Coxeter complex of type C$_n$.
\item[B1:] For any two simplices $A,B \in \Delta$,
there is an apartment $\Sigma$ containing them.
\item[B2:] If $\Sigma$ and $\Sigma'$ are two apartments containing $A$ and $B$,
then there is an isomorphism $\Sigma \to \Sigma'$ fixing $A$ and $B$ pointwise.
 \end{itemize}
It is known \cite{BuildingBook}*{Proposition 4.6} that a building is a colorable chamber complex, and
the isomorphism in B2 can be taken to
be type-preserving. For an apartment $\Sigma$ and a chamber $C$,
the {\em canonical retraction} $\rho_{\Sigma, C}: {\Delta} \to \Sigma$ is defined as:
for a simplex $A$ in $\Delta$, take
an apartment $\Sigma'$ containing $A$ and $C$, and define
$\rho_{\Sigma, C}(A)$  to be the image of $A$ by the isomorphism
$\Sigma' \to \Sigma$ ensured by B2.
In fact, $\rho_{\Sigma, C}(A)$ is independent of
the choice of $\Sigma'$, hence the map $\rho_{\Sigma, C}$ is well-defined
and is also a retraction (take $\Sigma'$ as $\Sigma$);
see \cite{BuildingBook}*{Section 4.4} for details.

The vertices in an apartment of a spherical or Euclidean building of type C$_n$
are identified with the half-integral vectors of $(\ZZ/2)^n$ as defined above.
For a half-integral vector $x$, consider the map
\[
x = (x_1,x_2,\ldots,x_n) \mapsto  \mbox{the number of $i\in \{0,1,2,\ldots,n\}$
with $x_i \not \in \ZZ$}.
\]
It is easy to see that this map is a coloring of the apartment and that
this coloring can be uniquely extended to a coloring of the building. In the spherical case,
this numbering of vertices is the same as
the {\em natural ordering} in the sense of Tits~\cite{Ti}*{7.4}.
Define the partial order on the vertices by setting $x \preceq y$
iff $x$ and $y$ belong to a common simplex
and the type of $y$ is less than or equal to the type of $x$;
so the vertices corresponding to integral vectors are maximal.
The resulting poset is denoted by $\mathcal{L}(\Delta)$.
The fundamental theorem on spherical buildings of type C$_n$ is the following:

\begin{Thm}[\cite{Ti}]\label{thm:Tits_polar}
For a spherical building $\Delta$ of type C$_n$,
$\mathcal{L}(\Delta)$ is the subspace poset of a polar space minus $0$,
and  $\Delta$ is the order complex of $\mathcal{L}(\Delta)$.
Conversely,  for a polar space $\Pi$ of rank $n$,
the reduced order complex of the subspace poset of $\Pi$
is a spherical building of type C$_n$.
\end{Thm}
The subspace poset of a polar space
is a (complemented) modular semilattice and its covering graph is orientable modular.
These structures are completely determined by the corresponding dual polar graph, which is a particular swm-graph.
The  goal of this section is to show
that analogous connections with swm-graphs also hold for Euclidean buildings of type C$_n$.
With each Euclidean building $\Delta$ of type C$_n$ we associate two graphs $G(\Delta)$ and $H(\Delta)$.
The graph $G(\Delta)$ is the subgraph of the $1$--skeleton of $\Delta$ such that its
edges are the pairs $xy$ such that the difference of types of $x$ and $y$ is $1$.
Define the orientation $o$ of edges of $G(\Delta)$ so that $x \rightarrow y$ if
the type of $x$ is less than that of $y$.
Namely this orientation is equal to the Hasse orientation
of the partial order $\preceq$.
The resulting poset on $V(G(\Delta))$ is denoted by $\mathcal{P}(G(\Delta),o)$.
The graph  $H(\Delta)$ has the vertices of type $0$ of $\Delta$  as the vertex set
and the pairs $xy$ such that $x$ and $y$ have a common neighbor
(of type 1) as edges.  The edge-lengths of $G(\Delta)$ and $H(\Delta)$ are $1/2$ and $1$, respectively.
\begin{Thm}\label{thm:EuclideanBuilding}
Let $\Delta$ be a Euclidean building of type C$_n$.
Then the following hold:
\begin{itemize}
\item[(1)] $G(\Delta)$ is an orientable modular graph and $o$ is an admissible orientation.
\item[(2)] $\Delta$ is the order complex of $\mathcal{P}(G(\Delta),o)$.
\item[(3)] $H(\Delta)$ is an swm-graph.
\item[(4)] The barycentric graph of $H(\Delta)$ is equal to $G(\Delta)$.
\end{itemize}
\end{Thm}
In particular, $\Delta$ is completely recovered from the swm-graph $H(\Delta)$.

\begin{question} It would be interesting to find a characterization (analogous to Cameron's characterization of dual polar graphs)
of all swm-graphs of the form $H(\Delta)$ for some Euclidean building $\Delta$ of type C.
\end{question}

The remainder of this section is devoted to proving
Theorem~\ref{thm:EuclideanBuilding}.
Any apartment $\Sigma$ of $\Delta$ itself is a building, thus
$G(\Sigma)$ is the subgraph of $G(\Delta)$ induced by the vertices of $\Sigma$.
The graph $G(\Sigma)$ is isomorphic to the grid graph on $(\ZZ/2)^n$
obtained by making adjacent all vertices $x,y$ with $\|x-y\|_1 =1/2$.
The isomorphism in B2 induces the isomorphism
between two grid graphs $G(\Sigma)$ and $G(\Sigma')$.
The canonical retraction $\rho_{\Sigma, C}$
induces the retraction from $G(\Delta)$ to $G(\Sigma)$,
which will be also denoted $\rho_{\Sigma, C}$.
\begin{Lem}\label{lem:canonical_retraction}
For vertices $x,y$ of  $\Sigma$ and the  canonical retraction $\rho = \rho_{\Sigma, C}$,
we have
$d_{G}(\rho(x), \rho(y)) \leq d_G(x,y)$ and
the equality holds if $x$ is a vertex of $C$.
In particular,  $G(\Sigma)$ is an isometric subgraph of $G(\Delta)$.
\end{Lem}
\begin{proof}
For vertices $x,y$ in $\Sigma$,
pick a path $P$ connecting $x,y$ in $G(\Delta)$.
Then the image of $P$ by the canonical retraction is a path in $G(\Sigma)$
of the same length.
\end{proof}
First we  establish assertion (2) of Theorem~\ref{thm:EuclideanBuilding}.
\begin{Lem}\label{lem:Delta=ordercomplex_of_P}
$\Delta$ is the order complex of $\mathcal{P}(G(\Delta),o)$.
\end{Lem}
\begin{proof}
This property easily holds when $\Delta$  is a single apartment.
Consider now the general case.
Pick a chain $p_1 \prec p_2 \prec \cdots \prec p_k$ with $p_i p_{i+1} \in E(G(\Delta))$. By induction
on $k$ we show that there is
a simplex in $\Delta$ containing $p_1,p_2,\ldots,p_k$.
By inductive assumption, there are simplices $C,C'$ such that $C$ contains $p_1,p_2,\ldots,p_{k-1}$ and $C'$ contains $p_{k-1},p_{k}$.
By B1 there is an apartment $\Sigma$ containing $C,C'$. Then necessarily $p_1,p_2,\ldots,p_k$ form a simplex in
$\Sigma$ (because $p_1 \prec p_2 \prec \cdots \prec p_k$ also holds in $\Sigma$) and thus in $\Delta$, and we are done.
\end{proof}
We will prove assertion (1) of Theorem~\ref{thm:EuclideanBuilding} by verifying that $\mathcal{P}(G(\Delta),o)$
satisfies the conditions (a), (b), (c), (d$'$), and (e) of Theorem~\ref{thm:chara_om}.
By definition,
the coloring is a grade function of  $\mathcal{P}(G(\Delta),o)$, and hence we have (a).
The next property establishes (b).
\begin{Lem}\label{lem:filter=polar}
For a vertex $p$,
the principal filter (ideal) of $p$ is the subspace poset of a polar space.
\end{Lem}
\begin{proof}
The intersection $\Sigma_p$ of an apartment $\Sigma$ and
the reduced order complex of
$(p)^{\uparrow}$ is isomorphic to a spherical Coxeter complex
of type C$_k$, where $k$ is the type of $p$.
From this, we can see that $\Delta_p$ is
the union of $\Sigma_p$
over all apartments $\Sigma$ containing $p$, and satisfies the building's axioms.
Hence $\Delta_p$ is a building of type C$_k$.
By Theorem~\ref{thm:Tits_polar}, $\Delta_p$ is
the reduced order complex of the subspace poset of a polar space, which must be isomorphic to
$(p)^{\uparrow}$.
\end{proof}
Next we verify (c) that every upper-bounded pair has the join.
Notice that this property holds for
the poset restricted to any apartment $\Sigma$,
since the grid graph $G(\Sigma)$ is orientable modular,
and $o$ is admissible on $G(\Sigma)$.
Let $(x,y)$ be an upper-bounded pair.
Let $z$ and $w$ be minimal common upper bounds.
Take an apartment $\Sigma$ containing the simplices $\{ z,x\}$ and $\{ z,y\}$.
Consider the image $w' = \rho_{\Sigma,C}(w)$ of
the canonical retraction $\rho_{\Sigma,C}$ for a chamber $C$ containing $z$.
Then $w'$ is a  common upper-bound of  $x,y$ in $\Sigma$,
and hence $z \preceq w'$.
So the type of $z$ is at least the type of $w'$ and of $w$.
By interchanging the roles of $z$ and $w$, the types of $z$ and $w$ are the same.
Consequently $z = w'$ holds.
By Lemma~\ref{lem:canonical_retraction},
we have $d(z,w) = d(z,w') = 0$, and hence $z = w$. Thus the join indeed exists.

We verify now (d$'$).
Take a crown $(x, u, y, v, z, w)$, which forms an isometric $6$--cycle.
Take an apartment $\Sigma$ containing the simplices $\{ x, u\}$ and $\{ x, w\}$.
Regard $\Sigma$ as the grid graph on $(\ZZ/2)^n$.
Then $u - x \neq x - w$ or $u - x = x - w \in \{e_i/2, - e_i/2\}$ for some $i$.
The latter case is impossible.
Indeed, take a chamber $C$ containing $x$, and consider
the image $v' = \rho_{\Sigma, C}(v)$.
By Lemma~\ref{lem:canonical_retraction}, $\|x - v'\|_1 = d(x,v') = d(x,v) = 3/2$,
and necessarily $\|x + e_i/2 - v'\|_1 = d(u,v') = d(u,v) = 1 = d(w,v) = d(w,v') = \|x - e_i/2 - v'\|_1$.
By comparing the $i$--th coordinates of $x$ and of $v'$ we conclude that this is impossible.
Thus the former case occurs, and
$u+ w - 2x$ is the join of $u,w$.
By the same argument,
$x$ and $y$ have the meet. Thus we obtain (d$'$).

Finally, we will verify condition (e). In fact,  the simple connectivity of the order complex of $\mathcal{P}(G(\Delta),o)$ follows
from assertion (2) that this complex coincides with $\Delta$ and the known fact that all Euclidean buildings
are contractible and thus  simply connected. Hence, the graph $G(\Delta)$ is modular and the orientation
$o$ is admissible.

Next we are going to prove the assertions (3) and (4) of Theorem~\ref{thm:EuclideanBuilding}.
Denote the graphs $G(\Delta)$ and $H(\Delta)$ by $G$ and $H$, respectively.
By the admissibility of $o$,
for an edge $xy$ of $H$,
there exists a unique vertex of type 1 adjacent to both $x$ and $y$ in $G$.
We next prove that $G$ is isometrically embeddable in $H$:
\begin{Lem}\label{isometric_G_H}
$d_H(x,y) = d_G(x,y)$.
\end{Lem}
\begin{proof}
It is easy to see  that $d_H(x,y) \geq d_G(x,y)$.
Pick two  vertices $x,y$ in $H$ and consider
an apartment $\Sigma$ containing $x$ and $y$.
Regard $\Sigma$ as the grid  $(\ZZ/2)^n$.
Then $x,y$ are integer vectors.
Obviously there is a sequence $(x= x_0, x_1,x_2,\ldots,x_k = y)$
of integer vectors such that $k = ||x - y||_1 = d_G(x,y)$ and
$x_i - x_{i+1}$ is $e_j$ or $- e_j$ for some unit vector $e_j$.
This gives a path in $H$ of length $k$.
\end{proof}

\begin{Lem}
$H:=H(\Delta)$ is an swm-graph.
\end{Lem}

\begin{proof}
To triangle condition (TC):
Consider $v,w,u$ such that $d_H(v,w) = 1$ and $k := d_H(v,u) = d_H(w,u)$. By
Lemma \ref{isometric_G_H},  $d_G(v,u) = d_G(w,u) = k$.
By definition, $d_{H}(v,w) = 1$ implies the existence of a common neighbor $x$ of $v,w$ in $G$.
Then $x$ has type $1$. Consider $d_G(x,u)$, which is $k +1/2$ or $k - 1/2$.
If $d_G(x,u) = k+1/2$,
then by (QC) for $G$ there is a common neighbor $y$ of $v,w$
with $d_G(y,u) = k - 1/2$. Then  $y$ is of type $1$ and
the  $4$--cycle  of $x,v,y,w$ violates admissibility of $o$.
Thus $d_G(x,u) = k - 1/2$.
Take an apartment $\Sigma$ containing $v,x,u$,
and identify $G(\Sigma)$ with the grid graph on $(\ZZ/2)^n$.
Then we can assume that $x = v + e_i/2$ and
$d_G(v,u) = \|v - u\|_1 = 1/2 + \|x - u\|_1$.
Since $v$ and $u$ are integer vectors,
$y := x + e_i/2 = v + e_i$ is an integer vector (i.e., $y \in V(H)$),
and $\| y - u \|_1 = d_G(v,u) - 1 = d_H(y,u) - 1$.
Notice that $y$ must be adjacent to $w$ in $H$, and hence $y$ is a desired vertex.

To $K_4^-$: By the previous argument, any triangle in $H$ must have
a common neighbor (of type 1) in $G$.
From this and (c) proven above, we see that
it is impossible for $H$ to contain an induced $K_4^-$.

To quadrangle condition (QC):
Consider $z,v,w,u$ with $d_H(z,v) = d_H(z,w) = 1 = d_H(v,w) - 1$
and $d_H(z,u) - 1= d_H(v,u) = d_H(w,u) =: k$.
Take vertices $v',w'$ of type 1 such that
$v'$ is a common neighbor of $z,v$ and
$w'$ is a common neighbor of $z,w$. The vertices
$v'$ and $w'$ are different and $d_G(v',u) = d_G(w',u) = k + 1/2$.
Apply (QC) for $z,v',w',u$ in $G$, to  obtain a common neighbor $x$ of $v',w'$
of type 2 such that $d_H(x,u) = k$.
Then $x,v,w$ are all different (since $d(v,w) = 2$).
By (QC), we obtain a common neighbor
$a$ of $x,v$ and a common neighbor $b$ of $x,w$
such that $d_G(a,u) = d_G(b,u) = k - 1/2$.
Again, $a$ and $b$ are different, and of type 1.
By (QC) for $x,a,b,u$,
we obtain a vertex $y$ of type 0 such that
$d_H(u,y) = k-1$, which is a common neighbor of $v,w$ in $H$.

To $K_{3,3}^-:$ Next consider a $K_{3,3}^-$ subgraph.
Let $x,y,z,w,u,v$ be six vertices inducing $K_{3,3}^-$ in $H$.
The color classes are $\{x,y,z\}$ and $\{w,u,v\}$,
and $x$ and $v$ are not adjacent.
Suppose to the contrary that $d_H(x,v) = 3$.
In $G$ pick  a common neighbor $a$ of $w,x$ and a common neighbor $b$ of $u,x$.
By (QC) for $x,a,b,y$ and for $x,a,b,z$,
we obtain a common neighbor $c$ of $a,b$ belonging to $I(x,y)$
and a common neighbor $c'$ of $a,b$ belonging to $I(x,z)$.
Then $c = c'$ must hold. Otherwise,
by (QC) for $a,c,c',v$
there is a common neighbor $c''$ of $c,c'$
belonging to $I(c'',v)$, and $x, a,b,c,c',c''$ induces a $K_{3,3}^-$ in $G$, which is impossible.
By using (QC), take a common neighbor $f$ of $w,c$ belonging $I(a,y)$
and take a common neighbor $f'$ of $w,c$ belonging $I(a,z)$.
As above $f = f'$ holds. Then $f$ is adjacent to each of $z$ and $y$.
This means that $z$ and $y$ are adjacent in $H$, contradicting the first assumption.
\end{proof}

Finally we show that the barycentric graph
of $H$ is isomorphic to $G$:

\begin{Lem}
The intersection of $V(H)$ and any principal filter of $\mathcal{P}(G,o)$
is Boolean-gated in $H$.
Conversely, every Boolean-gated set of $H$ is obtained  in this way.
Namely  $\mathcal{P}(G,o)$ is isomorphic to  $\mathcal{B}(H)$
by map $p \mapsto (p)^{\uparrow} \cap V(H)$.
\end{Lem}

\begin{proof}
Consider a vertex $p$ of $V(G)$, and the filter $(p)^{\uparrow}$ of $p$.
Since $V(H) \cap (p)^{\uparrow}$ induces a dual polar graph
(Lemma~\ref{lem:filter=polar}),
it suffices to
show that $V(H) \cap (p)^{\uparrow}$ is gated in $H$.
To apply Lemma~\ref{lem-weakly-modular_gated_hull},
pick $x,y \in V(H) \cap (p)^{\uparrow}$
and a common neighbor $z$ of $x,y$ in $H$.

First suppose that $x,y$ are adjacent.
Take a unique common neighbor $u = x \wedge y$ of $x,y$ in $G$.
Then $u$ must be adjacent to $z$.
Necessarily $u$ belongs to $(p)^{\uparrow}$ and has type 1,
hence $z \succ u \succeq p$.

Suppose now that $x,y$ are not adjacent.
By $x,y \in (p)^{\uparrow}$ and Lemma~\ref{lem:Delta=ordercomplex_of_P},
we can take an apartment
containing $\{x, p\}$ and $\{y,p\}$, and
consider the image $w$ of $z$ by a canonical retraction.
Then $w$ is a common neighbor of $x,y$ in $V(H) \cap (p)^{\uparrow}$.
We can assume that $w \neq z$.
Then $x,y,z,w$ form an isometric $4$--cycle in $H$.
Take a common neighbor $a$ of $z,x$
and a common neighbor $b$ of $z,y$.
Applying (QC) to $z,a,b,w$,
we get a common neighbor $c$ of $a,b$ with $d_G(c,w) = 1$.
Here $c$ is equal to $x \wedge y$ and has type 2.
Then $p \preceq x \wedge y \preceq a \preceq z$.
This means that $z$ belongs to $(p)^{\uparrow}$.

Consider an arbitrary Boolean-gated set $X$ of finite diameter in $H$.
Since $X$ induces a dual polar graph,
we can take $x,y \in X$ so that $k:= d_H(x,y)$ is equal to
the diameter of $X = \lgate x,y \rgate$
(Lemmata~\ref{lem:d(x,y)=D} and \ref{lem:X=<p,q>}).
Take an apartment $\Sigma$ containing $x,y$.
Then $X \cap V(\Sigma)$ induces an isometric cube of diameter $k$ (in $H$).

Indeed, we can regard $X \cap V(\Sigma) \subseteq \ZZ^n$.
So we may assume that $x = 0$ and $y = \sum_{i} \lambda_i e_i$ for nonnegative integers $\lambda_i$ with $\sum_{i} \lambda_i = \|x - y\|_1 = d_H (x,y) = k$.
Since $\Sigma$ is isometric and
$X = \lgate x,y \rgate$ is gated, $X \cap V(\Sigma)$ is
a convex set of diameter $k$ in the grid graph, and hence is equal to
$\{ z \in \ZZ^n: 0 \leq z_i \leq \lambda_i \ (i=1,2,\ldots,n)\}$.
Then it holds that $\lambda_i \leq 1$ for each $i$.
Suppose $\lambda_i \geq 2$. Then $x = 0, x' = e_i, x'' = 2 e_i$ are contained in
$X \cap V(\Sigma)$.
Since $X$ is thick, there is $z \in I(x,x'')$ with different from
(and nonadjacent with) $x'$.
Consider the image $z'$ of $z$ by the canonical retraction $\rho_{\Sigma,C}$
for a simplex $C$ in $\Sigma$ containing $x'$.
By the nonexpansive property of $\rho_{\Sigma,C}$
(between $z$ and $x,x''$), we must have $x' = z'$.
However this contradicts $d_H(x',z) = d_H(x',z') = 2$
(Lemma~\ref{lem:canonical_retraction}).

Thus we can take a vertex $p (= (x+y)/2) \in V(\Sigma)$ of type $k$
so that $(p)^{\uparrow}$ contains $X \cap V(\Sigma)$.
As seen above, the intersection $X'$ of $V(H)$ and $(p)^{\uparrow}$ is gated,
and hence $(p)^{\uparrow}$ contains the whole set $X$.
Since the diameter of the subgraph induced by $(p)^{\uparrow}$
is (at most) equal to $k$ in $G$,
the diameter of the Boolean-gated set $X'$ is also $k$, and thus this subgraph is equal to $X$.
\end{proof}

\section{Application I: Biautomaticity of swm-groups}\label{subsec:biauto}

Biautomaticity is a strong property implying numerous algorithmic and
geometric features of a group \cites{ECHLPT,BrHa}.  Sometimes the fact
that a group acting on a space is biautomatic may be established from
geometric and combinatorial properties of the space. For example,
one of the important and nice results about CAT(0) cube complexes is a
theorem by Niblo and Reeves \cite{NiRe} that the groups acting
geometrically, that is properly and discontinuously on CAT(0) cube
complexes are biautomatic. Januszkiewicz and {\'S}wi{\c{a}}tkowski
\cite{JS} established a similar result for groups acting on systolic
complexes. It is also well-known that hyperbolic groups are
biautomatic \cite{ECHLPT}.  {\'S}wi{\c{a}}tkowski \cite{Swiat}
presented a general framework of locally recognized path systems in a
graph $G$ under which proving biautomaticity of a group acting on $G$
is reduced to proving local recognizability and the $2$--sided fellow
traveler property for some paths.

Analogously to cubical and systolic groups, we will call groups acting
geometrically on swm-graphs {\it swm-groups}.

\begin{theorem} \label{th-biautomatic-swm}
The set of normal Boolean-gated paths of an swm-graph $G$ defines a
regular geodesic bicombing of $G^{\Delta}$. Consequently, swm-groups
are biautomatic.
\end{theorem}

The rest of this section is organized as follows: we first recall
the necessary definitions and results about biautomaticity. Then we
show that normal Boolean-gated paths can be locally recognized and
satisfy the 2-sided fellow traveler property.

\subsection{Bicombings and biautomaticity}

We continue by recalling the definitions of (geodesic) bicombing and
biautomatic group \cites{ECHLPT,BrHa}.  Let $G=(V,E)$ be a graph and
suppose that $\Gamma$ is a group acting geometrically by automorphisms
on $G$.  These assumptions imply that the graph $G$ is locally finite,
moreover, the degrees of vertices of $G$ are uniformly bounded.
Denote by ${\mathcal P}(G)$ the set of all paths of $G$.  A {\it path
  system} $\mathcal P$ \cite{Swiat} is any subset of ${\mathcal
  P}(G)$.  The action of $\Gamma$ on $G$ induces the action of
$\Gamma$ on the set ${\mathcal P}(G)$ of all paths of $G$.  A path
system ${\mathcal P}\subseteq {\mathcal P}(G)$ is called
$\Gamma$--{\it invariant} if $g\cdot \gamma \in \mathcal P$, for all
$g\in \Gamma$ and $\gamma \in \mathcal P$.

Let $[0,n]^*$ denote the set of integer points from the segment
$[0,n].$ Given a path $\gamma$ of length $n=|\gamma|$ in $G$, we can
parametrize it and denote by $\gamma:[0,n]^*\rightarrow V(G)$. It will
be convenient to extend $\gamma$ over $[0,\infty]$ by setting
$\gamma(i)=\gamma(n)$ for any $i>n$.  A path system $\mathcal P$ of a
graph $G$ is said to satisfy the {\it 2-sided fellow traveler
  property} if there are constants $C>0$ and $D\ge 0$ such that for
any two paths $\gamma_1,\gamma_2\in \mathcal P$, the following
inequality holds for all natural $i$:
$$d_G(\gamma_1(i),\gamma_2(i))\le C\cdot \max\{
d_G(\gamma_1(0),\gamma_2(0)),d_G(\gamma_1(\infty),\gamma_2(\infty))\}+D.$$
A path system $\mathcal P$ is \emph{complete} if any two vertices are
endpoints of some path in $\mathcal P$.  A {\it bicombing} of a graph
$G$ is a complete path system $\mathcal P$ satisfying the $2$--sided
fellow traveler property.  If all paths in the bicombing ${\mathcal
  P}$ are shortest paths of $G$, then ${\mathcal P}$ is called a {\it
  geodesic bicombing}.

We recall here quickly the definition of a biautomatic structure for a group. Details can be found in  \cites{ECHLPT,BrHa,Swiat}.
Let $\Gamma$ be a group generated by a finite set $S$. A \emph{language} over $S$ is some set of words in $S\cup S^{-1}$ (in the
free monoid $(S\cup S^{-1})^{\ast}$).
A language over $S$ defines a $\Gamma$--invariant path system in the Cayley graph Cay$(\Gamma,S)$.
A language is \emph{regular} if it is accepted by some finite state automaton.
A \emph{biautomatic structure} is a pair $(S,\mathcal L)$, where $S$ is as above, $\mathcal L$ is a regular language over $S$, and
the associated path system in Cay$(\Gamma,S)$ is a bicombing. A group is
\emph{biautomatic} if it admits a biautomatic structure.
In what follows we use specific conditions implying biautomaticity for groups acting geometrically on graphs. The method, relying on the notion of locally recognized path system, was developed by {\'S}wi{\c{a}}tkowski \cite{Swiat}.

Let $G$ be a graph and let $\Gamma$ be a group acting geometrically on $G$.
Two paths $\gamma_1$ and $\gamma_2$ of $G$ are $\Gamma$-{\it congruent} if there is $g\in \Gamma$ such that $g\cdot\gamma_1=\gamma_2$. Denote by ${\mathcal S}_k$ the set of $\Gamma$-congruence classes
of paths of length $k$ of $G$.  Since $\Gamma$ acts cocompactly on $G$, the sets ${\mathcal S}_k$ are finite for any natural $k$. For any path $\gamma$ of $G$, denote by $[\gamma]$ its
$\Gamma$-congruence class.

For a subset $R\subset {\mathcal S}_k$, let ${\mathcal P}_R$ be the path system in $G$ consisting of all paths $\gamma$ satisfying the following two conditions:
\begin{itemize}
\item[(1)] if $|\gamma|\ge k$, then $[\eta]\in R$ for any subpath $\eta$ of length $k$ of $\gamma$;
\item[(2)] if $|\gamma|<k$, then $\gamma$ is a prefix of some path $\eta$ such that $[\eta]\in R$.
\end{itemize}

A path system $\mathcal P$ in $G$ is $k$--{\it locally recognized} if for some $R\subset {\mathcal S}_k$, we have ${\mathcal P}={\mathcal P}_R$, and $\mathcal P$ is {\it locally recognized}
if it is $k$--locally recognized for some $k$. The following result of  {\'S}wi{\c{a}}tkowski  \cite{Swiat} provides sufficient conditions for biautomaticity in terms of local recognition
and bicombing.

\begin{theorem} \label{swiat} \cite{Swiat}*{Corollary 7.2} \label{th:swiat} Let $\Gamma$ be group acting geometrically on a graph $G$ and let $\mathcal P$ be a path system in
$G$ satisfying the following conditions:
\begin{itemize}
\item[(1)] $\mathcal P$ is locally recognized;
\item[(2)] there exists $v_0\in V(G)$ such that any two vertices from the orbit $\Gamma \cdot v_0$ are connected by a path from $\mathcal P$;
\item[(3)] $\mathcal P$ satisfies the $2$--sided fellow traveler property.
\end{itemize}
Then $\Gamma$ is biautomatic.
\end{theorem}

\subsection{Normal Boolean-gated paths can be $2$-locally recognized}

We continue by showing that normal Boolean-gated paths of swm-graph
can be $2$-locally recognized.

\begin{proposition}\label{2recog_swm} Let $G$ be an swm-graph and let $\Gamma$ be a group acting geometrically on $G$.  Then $\Gamma$ acts geometrically on $G^{\Delta}$ and  the set of normal bg-paths of $G^{\Delta}$ is
2-locally recognized.
\end{proposition}

\begin{proof} Since the degrees of vertices of $G$
	are uniformly bounded, from Lemma \ref{l:gated_hull_dpg2} it follows that for any vertex $u$ of $G$ the number of Boolean-gated sets of the form  $\lgate u,v \rgate$
	is uniformly bounded. It follows that the group $\Gamma$ acts properly discontinuously and cocompactly on $G^{\Delta}$.
	
	By definition,  a
	path $(p = x_0,x_1,\ldots,x_k = q)$ of $G^{\Delta}$ is a normal bg-path if and only if  for any $i=1,2,\ldots,k-1$,
	the subpath $(x_{i-1},x_i,x_{i+1})$ is a normal bg-path from $x_{i-1}$ to $x_{i+1}$.
	This proves that the set of normal bg-paths of $G^{\Delta}$ is
	2-locally recognized.
	%The group $\Gamma$ acts properly discontinuously and cocompactly on $G^{\Delta}$.  Since the degrees of vertices of $G$
	%are uniformly bounded, from Lemma \ref{l:gated_hull_dpg2} it follows that for any vertex $u$ of $G$ the number of Boolean-gated sets of the form  $\lgate u,v \rgate$
	%is uniformly bounded. Since $\Gamma$ acts cocompactly on $G$, $\Gamma$ has a finite number of orbits, therefore $G$ has a uniformly
	%bounded number of types of Boolean-gated sets (this also follows from the fact that the degrees of vertices and the  diameters of Boolean-gated sets
	%of $G$ are uniformly bounded).  Therefore there exists an integer $r$ such that for any vertex $u$ all Boolean-gated sets containing $u$ are included
	%in the ball $B_r(u)$. There exists only a finite number of types of balls of radius $r$ of $G$. Consequently, there exists only a finite number of
	%types of balls of radius $1$ of $G^{\Delta}$.
	%
	%By definition,  a
	%path $(p = x_0,x_1,\ldots,x_k = q)$ of $G^{\Delta}$ is a normal bg-path if and only if  for any $i=1,2,\ldots,k-1$,
	%the subpath $(x_{i-1},x_i,x_{i+1})$ is a normal bg-path from $x_{i-1}$ to $x_{i+1}$. Now, a path $(x_{i-1},x_i,x_{i+1})$
	%is a normal bg-path if and only if for any Boolean-gated set $B$ containing $\lgate x_{i-1}, x_{i} \rgate$, the equality
	%$B \cap \lgate x_{i}, x_{i+1} \rgate = \{x_i\}$ holds.  
	%%Therefore to verify if
	%%a triplet $(x_{i-1},x_i,x_{i+1})$ defines a normal bg-path it suffices to check this
	%%condition in  a finite number of balls centered at $x_i$. 
	%This proves that the set of normal bg-paths of $G^{\Delta}$ is
	%2-locally recognized.
\end{proof}

\subsection{Normal Boolean-gated paths are fellow travelers}

The following result together with Proposition \ref{2recog_swm} establishes Theorem \ref{th-biautomatic-swm}. Notice that in order to
prove the fellow traveler property, it suffices to establish it for the normal bg-paths $\gamma_{px}$ and $\gamma_{qy}$
for four vertices $p,q,x,y$ such that $d_{G^\Delta}(p,q) \leq 1$, $d_{G^\Delta}(x,y) \leq 1$. Our proof of bicombing is different from that of \cite{NiRe}, in particular
because in general swm-graphs there is no notion of hyperplane.

\begin{proposition} \label{bicombing_swm} Let  $G$ be an swm-graph.  Consider four vertices $p,q,x,y$ and two integers $k'\ge k$
such that $d_{G^\Delta}(p,q) \leq 1$ and $d_{G^\Delta}(x,y) \leq 1$. If $\gamma_{px}=(p=x_0,x_1,
  \ldots,x_{k'-1},x_{k'}=x)$ and $\gamma_{qy}=(q=y_0,y_1, \ldots, y_{k-1},y_k = y)$
  are the normal bg-paths from $p$ to $x$ and from $q$
  to $y$, then $\gamma_{px}$ and $\gamma_{qy}$ are $1$--fellow travelers, namely,
  $d_{G^\Delta}(x_i,y_i)\leq 1$ for any $0 \leq i \leq k$ and $d_{G^\Delta}(x_i,y_k)\le 1$ for any $k\le i\le k'$.
\end{proposition}

\begin{proof}
  We prove the result by induction on $k'$.  If $ k' \leq 1$, there is
  nothing to prove. Assume now that $k' \geq 2$ and that the lemma
  holds for any vertices $p,q,x,y$ with $d_{G^\Delta}(p,q) \leq 1$,
  $d_{G^\Delta}(x,y) \leq 1$, and $d_{G^\Delta}(q,y) \leq
  d_{G^\Delta}(p,x) \leq k'-1$.

  Suppose first that $k < k'$, i.e., $k+1 \leq k' \leq k+2$. Since
  $d_{G^\Delta}(y,p)\le d_{G^\Delta}(x,p),$ by Lemma \ref{lem-BP-prop}
  we deduce that $d_{G^\Delta}(x_{k'-1},y) \leq 1$. By induction
  hypothesis applied to $p, q, x_{k'-1}$, and $y$, we have that
  $d_{G^\Delta}(x_i,y_i)\leq 1$ for any $0 \leq i \leq k$ and
  $d_{G^\Delta}(x_i,y_k)\le 1$ for any $k\le i\le k'-1$. Since by our
  assumptions, we have $d_{G^\Delta}(x_{k'},y_k)\le 1$, we are done.

  Assume now that $k = k'$. First, suppose that $d_{G^\Delta}(y,p)
  \leq k$.  Since $d_{G^\Delta}(y,p) \leq d_{G^\Delta}(x,p)$, by
  Lemma~\ref{lem-BP-prop}, we have $x_{k-1} \in
  B_1(y,G^\Delta)$. Since $d_{G^\Delta}(x_{k-1},q) \leq
  d_{G^\Delta}(x_{k-1},p) + 1 = k = d_{G^\Delta}(y,q)$, by
  Lemma~\ref{lem-BP-prop}, we have $y_{k-1} \in
  B_1(x_{k-1},G^\Delta)$. By induction hypothesis applied to
  $p,q,x_{k-1}$, and $y_{k-1}$, we conclude that the lemma
  holds. Similarly, the lemma holds when $d_{G^\Delta}(x,q) \leq k$.

  Suppose now that $d_{G^\Delta}(x,q) = d_{G^\Delta}(y,p) = k+1$. Let
  $z$ be the $\Delta$-gate of $p$ at $y$.  Since
  $d_{G^\Delta}(p,y_{k-1})= d_{G^\Delta}(p,x) = d_{G^\Delta}(p,z) =
  k$, we have that $z \in I(x,y)$ and $z \in
  I(y_{k-1},y)$. Consequently, we have $z\in B_1(x,G^\Delta) \cap
  B_1(y_{k-1},G^\Delta)$ and $z \in \lgate y_{k-1}, y\rgate$.  Since
  $d_{G^\Delta}(x,q) = k+1$ and $d_{G^\Delta}(y_{k-1},q) = k-1$,
  necessarily $d_{G^\Delta}(z,q) = k$.  Therefore, by
  Lemma~\ref{lem-magic}, $(q=y_0,y_1, \ldots, y_{k-1},y_k' = z)$ is
  the normal bg-path $\gamma_{qz}$. Since
  $d_{G^\Delta}(q,z)=d_{G^\Delta}(p,z)=k$, we can apply the previous
  case to $p, q, x, z$ and consequently conclude that $d_{G^\Delta}(x_i,y_i) \leq 1$
  for any $0 \leq i \leq k-1$.
\end{proof}

\section[Application II: $2$--approximation to 0-ext. problem on swm-graphs]{Application II: $2$--approximation to 0-extension problem on swm-graphs}\label{subsec:0ext}

The original motivation of studying orientable modular graphs comes
from the tractability classification of a combinatorial optimization
problem called the minimum 0-extension problem.  All graphs in this
section are assumed to be finite.  For a (finite) graph $G =
(V,E)$, the {\em minimum 0-extension problem} on $G$, denoted by {\bf
  0-EXT}$[G]$, is formulated as follows:
\begin{eqnarray*}
\mbox{INPUT:} && n \geq 0, b_{iv} \geq 0\ (1 \leq i \leq n, v \in V), c_{ij} \geq 0\
(1 \leq i < j \leq n) \\[0.2cm]
\mbox{Minimize} && \sum_{v \in V} \sum_{1 \leq i \leq n}
b_{iv}d_{G}(v, x_i) + \sum_{1 \leq i < j \leq n} c_{ij}d_{G}(x_i, x_j) \\
\mbox{subject to} && x = (x_1,x_2,\ldots,x_n) \in V \times V \times \cdots \times V.
\end{eqnarray*}
The present formulation is  equivalent to
the original formulation due to Karzanov~\cite{Kar98a}.
This problem can be interpreted as follows:
we are going to locate $n$ new facilities $1,2,\ldots,n$
to some vertices $x_1,x_2,\ldots, x_n$ of the graph $G$,
where the new facilities can communicate to each other as well as to
each other vertex $v$ of $G$, with the communication cost
proportional to the distances between the respective pairs.
The goal is to find a location $(x_1,x_2,\ldots,x_n)$
with minimum total communication cost.
This is a version of the classical multifacility location problem.
Its recent applications include image segmentation in computer vision and
 clustering related tasks in machine learning; see~\cite{KT02}.

The computational complexity of {\bf 0-EXT}$[G]$ depends on the input graph $G$.
For example,  {\bf 0-EXT}$[K_2]$ is the minimum cut problem,
which can be solved in polynomial time, while {\bf 0-EXT}$[K_n]$ is the multiway cut problem,
which is NP-hard for $n\ge 3$~\cite{DJPSY94}.
Concerning the polynomial time solvability,
the close relations to median and modular graphs were explored by~\cites{Ch_facility, Kar98a, Kar04}.
Finally, the following ultimate dichotomy was established:
\begin{Thm}[\cite{Kar98a}]
If $G$ is not an orientable modular graph, then  {\bf 0-EXT}$[G]$ is NP-hard.
\end{Thm}

\begin{Thm}[\cite{HH12}]\label{thm:HH12}
If $G$ is an orientable modular graph, then  {\bf 0-EXT}$[G]$ is solvable in polynomial time.
\end{Thm}

Concerning the approximability of {\bf 0-EXT}$[G]$,
recent studies~\cites{KKMR09, MNRS08} show
that it seems difficult to approximate the 0-extension problem
within a constant factor.
So, it is interesting to
find classes of graphs $G$ for which {\bf 0-EXT}$[G]$
admits a constant-factor approximation algorithm.
By an {\em  $\alpha$--approximation algorithm} we mean a polynomial time algorithm
to find a solution of {\bf 0-EXT}$[G]$ having the cost
at most $\alpha$ times the optimal cost; see \cite{Vazirani}.
For example, if $G$ is a complete graph,
then {\bf 0-EXT}$[G]$ is a multiway cut problem, and
there is a $1.5$--approximation algorithm; see \cite{Vazirani}*{Section 4}.

The main result in this section (which can be viewed as a direct
consequence of previous  results of this chapter) is that
the 0-extension problem on swm-graphs admits a $2$--approximation:
\begin{Thm}\label{thm:2approx}
For an swm-graph $G$,
there exists a $2$--approximation algorithm for {\bf 0-EXT}$[G]$.
\end{Thm}

The proof idea is to consider {\bf 0-EXT}$[G^*]$ on the barycentric graph $G^*$ of $G$
as a {\em half-integral relaxation} of {\bf 0-EXT}$[G]$.
Our algorithm may be regarded as a far-reaching generalization
of the classical 2-approximation algorithm for the multiway cut problem~\cite{DJPSY94}.

We first verify that the barycentric graph $G^*$
can be constructed in polynomial time.
\begin{Lem}\label{lem:construct_G*}
An swm-graph $G = (V,E)$ has $O(|V|^2)$ Boolean-gated sets
and its barycentric graph $G^*$ can be constructed in polynomial time.
\end{Lem}
\begin{proof}
Any Boolean-gated set $X$ of $G$ has the form $\lgate x, y \rgate$
for some vertices $x,y$. Hence the number of Boolean-gated sets
is bounded by $|V|^2$.
Next we show that
there is a polynomial-time algorithm to decide
if a given pair $(x,y)$ is Boolean.
We can suppose that the distance matrix of $G$ has been computed by  Dijkstra's algorithm.
Using this, we can construct the modular lattice $I(x,y)$ in polynomial time.
Take a maximal chain $(x = x_0, x_1,\ldots, x_k = y)$ in this lattice.
For $i=1,2,\ldots,k-1$, check the existence of
a neighbor $z_i$ of $x$
such that $x_{i+1} = x_{i} \vee z_i$.
If $z_i$ exists for all $i$, then $I(x,y)$ is complemented, and the pair $(x,y)$ is Boolean.
Otherwise, the pair $(x,y)$ is not Boolean.

If the pair $(x,y)$ is Boolean, then by applying the procedure GATED-HULL (see Subsection~\ref{s:wmgra}) to $I(x,y)$,
we can obtain $\lgate x, y \rgate$ in polynomial time.
Therefore we can enumerate all Boolean-gated sets of $G$ and
construct the barycentric graph $G^*$ in polynomial time.
\end{proof}

Next we define a rounding map from $G^*$ to $G$.
For a vertex $u$ in $G$,
let $\phi_u: \mathcal{B}(G) \to V$ be the map defined by
\begin{equation}
 \phi_u(X) := \mbox{the gate of $u$ in $X$} \quad (X \in \mathcal{B}(G)).
\end{equation}
Notice that $\phi_u(X)$ is the unique vertex in $X$ nearest to $u$
and can be determined  in polynomial time.
The map $\phi_u$ expands distances within a factor of $2$:
\begin{Lem}\label{lem:round}
$d_{G}(\phi_u(X), \phi_u(Y)) \leq 2 d_{G^{*}}(X,Y)$ for $X,Y \in \mathcal{B}(G)$.
\end{Lem}
\begin{proof}
It suffices to prove the inequality in the case where $X$ and $Y$
are adjacent in $G^*$, i.e., $d_{G^{*}}(X,Y) = 1/2$.
Indeed, take a shortest path $(X = X_0,X_1,\ldots,X_k = Y)$ in $G^*$.
Then $d_{G}(\phi_u(X),\phi_u(Y)) \leq \sum_{i} d_{G}(\phi_u(X_i),\phi_u(X_{i+1}))
\leq 2 \sum_{i} d_{G^*}(X_i, X_{i+1}) = 2  d_{G^{*}}(X,Y)$.

We may assume that $X$ covers $Y$ in the poset $\mathcal{B}(G)$.
Let $x := \phi_u(X)$ and $y := \phi_u(Y)$.
If $y$ belongs to $X$, then $y$ must be the gate of $u$ in $X$
and $x = y$. Suppose not; then $x \neq y$.
By (P3) in Theorem~\ref{thm:Tits},
there exists (a uniquely determined) $x' \in X$
such that $x'$ and $y$ are adjacent in $G$.
Then $d(u,y) + 1 \geq d(u,x') = d(u,x) + d(x,x') = d(u,y) + d(y,x) + d(x,x')$.
This means that $d(x,x') = 0$ and $d(x,y) = 1$, as required.
\end{proof}
We are now ready to describe the  2-approximation algorithm for {\bf 0-EXT}$[G]$.
Let $\nu$ and $\nu^*$ denote the optimal values of {\bf 0-EXT}$[G]$ and
{\bf 0-EXT}$[G^*]$ (for the same input), respectively.
Then  $\nu^* \leq \nu$
since $G$ is isometrically embeddable in $G^*$.
Let $X = (X_1,X_2,\ldots, X_n)$
be an optimal solution of the relaxation {\bf 0-EXT}$[G^*]$,
which can be obtained in polynomial time by Theorem~\ref{thm:HH12}
and Lemma~\ref{lem:construct_G*}.
Pick an arbitrary vertex $u$ in $G$.
Let $x_i := \phi_u(X_i)$ for $i=1,2,\ldots,n$.
By Lemma~\ref{lem:round} and since $\phi_u(\{v\}) = v$,
we have
%% \begin{eqnarray*}
%% \sum b_{iv} d_G(v,x_i) +  \sum c_{ij} d_G(x_i,x_j) & \leq & 2 \sum b_{iv} d_{G^*}(\{v\},X_i) +  2 \sum c_{ij} d_{G^*}(X_i, X_j)\\
%% & = & 2 \nu^* \leq 2 \nu.
%% \end{eqnarray*}
\begin{align*}
\sum b_{iv} d_G(v,x_i) +  \sum c_{ij} d_G(x_i,x_j) & \leq  2 \sum b_{iv} d_{G^*}(\{v\},X_i) +  2 \sum c_{ij} d_{G^*}(X_i, X_j)\\
& =  2 \nu^* \leq 2 \nu.
\end{align*}
This means that $x := (x_1,x_2,\ldots,x_n)$ is a $2$--approximate solution of {\bf 0-EXT}$[G]$, thus
proving Theorem~\ref{thm:2approx}.

%% \smallskip
%% \noindent
%%     {\bf Relation to the classical 2-approximation algorithm for multiway cuts.}
\subsection*{Relation to the classical 2-approximation algorithm for multiway cuts.}
    It is instructive to compare our 2-approximation algorithm for
    swm-graphs with the classical 2-approximation algorithm for
    multiway cuts. Namely, if $G$ is a complete graph on the vertex
    set $v_1,v_2,\ldots,v_k$, the minimum 0-extension problem on $G$
    is nothing but the multiway cut problem.  Indeed, construct the
    network $N$ on the vertex set $\{x_1,x_2,\ldots,x_n,v_1,v_2,
    \ldots,v_k\}$ by adding an edge between $x_i$ and $x_j$ of
    capacity $c_{ij}$, and an edge between $v_l$ and $v_i$ of capacity
    $b_{il}$.  Then {\bf 0-EXT}$[G]$ is the multiway cut problem on
    $N$ with terminal set $T:= \{v_1,v_2, \ldots,v_k\}$, i.e., the
    problem of finding a partition $\{U_1,U_2,\ldots,U_k\}$ of
    vertices with $v_i \in U_i$ so that the sum of capacities of edges
    joining different parts is minimum.

Consider the barycentric graph $G^*$, which is a star with $k$ leaves $v_1,v_2,\ldots,v_k$.
Let $v_0$ denote the center vertex
(that corresponds to the Boolean-gated set $\{v_1,v_2,\ldots, v_k\}$).
The minimum 0-extension problem on $G^*$ can be easily solved;
see \cite{Kar98a}*{Section 5} for example.
For each $i$, take a $(v_i, T \setminus \{v_i\})$--minimum cut $X_i$
in the above network $N$.
By the standard uncrossing argument, we can assume that
$X_1,X_2,\ldots, X_k$ are disjoint.
Define $x_i := v_l$ if $x_i$ belongs to $X_l$ and $x_i := v_0$ otherwise.
Then the resulting $x$ is an optimal solution of  {\bf 0-EXT}$[G^*]$.
Fix an arbitrary terminal, say $v_1$.
Our algorithm rounds $x_i$ to $v_1$ if $x_i := v_0$.
Then the resulting solution $x$
is a 2-approximation solution of {\bf 0-EXT}$[G]$, and
the corresponding multiway cut $\{X_2, X_3,\ldots, X_k, V(N) \setminus \bigcup_{j=2}^k X_j\}$
is a 2-approximation solution of the multiway cut problem on $N$.
This algorithm is essentially the same as the classical 2-approximation
algorithm for multiway cut~\cite{DJPSY94}; see also~\cite{Vazirani}*{Algorithm 4.3, Section 3}.

\chapter{Orthoscheme Complexes of Modular Lattices and Semilattices}\label{s:ortho}

An $n$--dimensional {\em orthoscheme}
is a simplex $\sigma$ of $\RR^n$ such that
for some ordering $v_0,v_1,\ldots,v_n$ of its vertices,
the vectors $v_{i} - v_{i-1}$ $(i=1,2,\ldots,n)$ form an orthogonal basis of $\RR^n$.
This ordering of vertices is called a {\em regular order}.
If the vectors $v_i- v_{i-1}$ form an orthonormal basis, then the orthoscheme
$\sigma$ is said to be {\em standard}.
If $\| v_{i} - v_{i-1}\|_2$ is a constant  $s > 0$, then
$\sigma$ is said to have {\em uniform size $s$}.
A standard orthoscheme is congruent to the simplex on vertices
\[
0, e_1, e_1+e_2, e_1 + e_2 + e_3, \ldots, e_1+e_2 + \cdots + e_n,
\]
where $e_i$ is the $i$--th unit vector in $\RR^n$.
Figure~\ref{fig:ortho} illustrates a $3$--dimensional standard orthoscheme.
\begin{figure}[t]
\begin{center}
\includegraphics[scale=0.5]{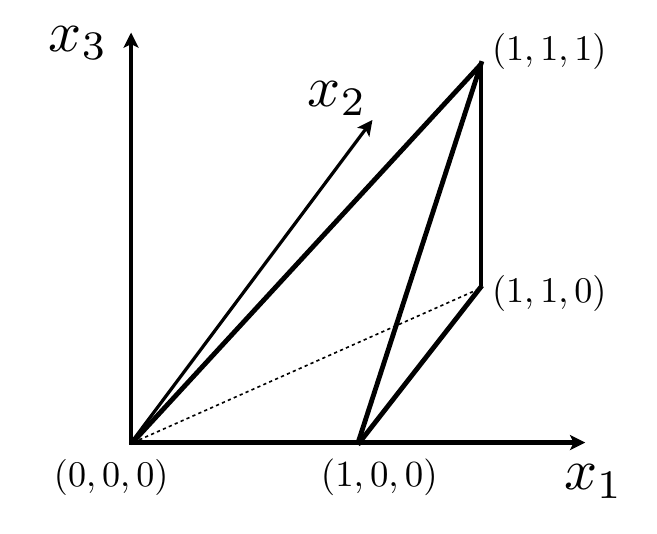}
\end{center}
\caption{A $3$--dimensional standard orthoscheme}
\label{fig:ortho}
\end{figure}
An {\em orthoscheme complex} is a Euclidean simplicial complex
obtained by gluing orthoschemes along common isometric faces.
This notion is due to Brady and McCammond~\cite{BradyMcCammond}.
Moreover, as done in \cite{BradyMcCammond},
the order complex of any graded poset becomes an orthoscheme
complex in a natural way; see below. In this chapter, we investigate the
orthoscheme complexes of modular lattices and semilattices.

Let us now formally define the orthoscheme complex $K({\mathcal P})$
of a graded poset $\mathcal P$. Let $r$ be the grade function of ${\mathcal P}$.
Consider the order complex $\Delta({\mathcal P})$
of  $\mathcal{P}$ (see Subsection~\ref{s:balat}). With each abstract simplex
$\{p_0,p_1,\ldots,p_k\}$ of  $\Delta({\mathcal P})$ (i.e., a chain of $\mathcal P$),
associate a (geometric) simplex $\sigma$,
which is the set of all formal convex combinations $\sum_{i=0}^k \lambda_i p_i$
on $\{p_0,p_1,\ldots,p_k\}$.
The resulting cell complex is denoted by $K(\mathcal{P})$.
The metric on $K(\mathcal{P})$ is defined as follows.
Let $\sigma$ be the simplex corresponding to the
chain $(p_0 \prec p_1 \prec \cdots \prec p_k)$.
Then $\sigma$ is mapped to the
$r[p_0, p_k]$--dimensional orthoscheme of uniform size $s > 0$ via the
map $\varphi_{\sigma}$ defined by
\begin{equation}
\varphi_{\sigma} (x) :=
s \sum_{i=1}^{k} \lambda_i (e_1 + e_2 + \cdots + e_{r[p_0, p_i]})
\quad \mbox{ for }  x = \sum_{i=0}^k \lambda_i p_i \in \sigma.
\end{equation}
Recall that for
an interval $[x,y]$ of $\mathcal P$, $r[x,y]=r(y)-r(x)$ is the length of  a maximal chain
from $x$ to $y$.
Using this map, the $l_p$--metric $D_{p}$ on $\sigma$ is defined by
\begin{equation}
D_p(x,y) := \|  \varphi_{\sigma}(x) - \varphi_{\sigma}(y) \|_{p} \quad \mbox{ for } x,y \in \sigma,
\end{equation}
where $\| \cdot \|_p$ is the $l_p$--metric of $\RR^n$.
This metric is actually well-defined;
for a face $\sigma'$ of $\sigma$
corresponding to a subchain $p_{i_0} \prec p_{i_1} \prec \cdots \prec p_{i_l}$,
if $x = \sum_{i=0}^k \lambda_i p_i$ belongs to $\sigma'$, then it holds
\[
\varphi_{\sigma}(x) = s(e_1 + e_2 + \cdots + e_{r[p_0,p_{i_0}]}) +
s \sum_{j=1}^{l} \lambda_{i_j} (e_{r[p_0,p_{i_0}] +1} + \cdots + e_{r[p_0, p_{i_j}]}).
\]
From this, we see that
$\|  \varphi_{\sigma}(x) - \varphi_{\sigma}(y) \|_{p} =  \|  \varphi_{\sigma'}(x) - \varphi_{\sigma'}(y) \|_{p}$ for $x,y \in \sigma'$.
Thus $D_{p}$ can be extended to the length $D_p$--metric on $K({\mathcal P})$.
Namely, $D_p(x,y)$ is the infimum of the lengths  of all finite paths of $K({\mathcal P})$
connecting $x$ and $y$.
The resulting metric simplicial complex $K({\mathcal P})$
is called the {\em orthoscheme complex} of ${\mathcal P}$
with respect to the $l_p$--metrization $D_p$.
Figure~\ref{fig:folder} illustrates the orthoscheme complex
of a modular lattice with rank $2$, which is isometric to a {\em folder}.
\begin{figure}[ht]
\begin{center}
\includegraphics[scale=0.4]{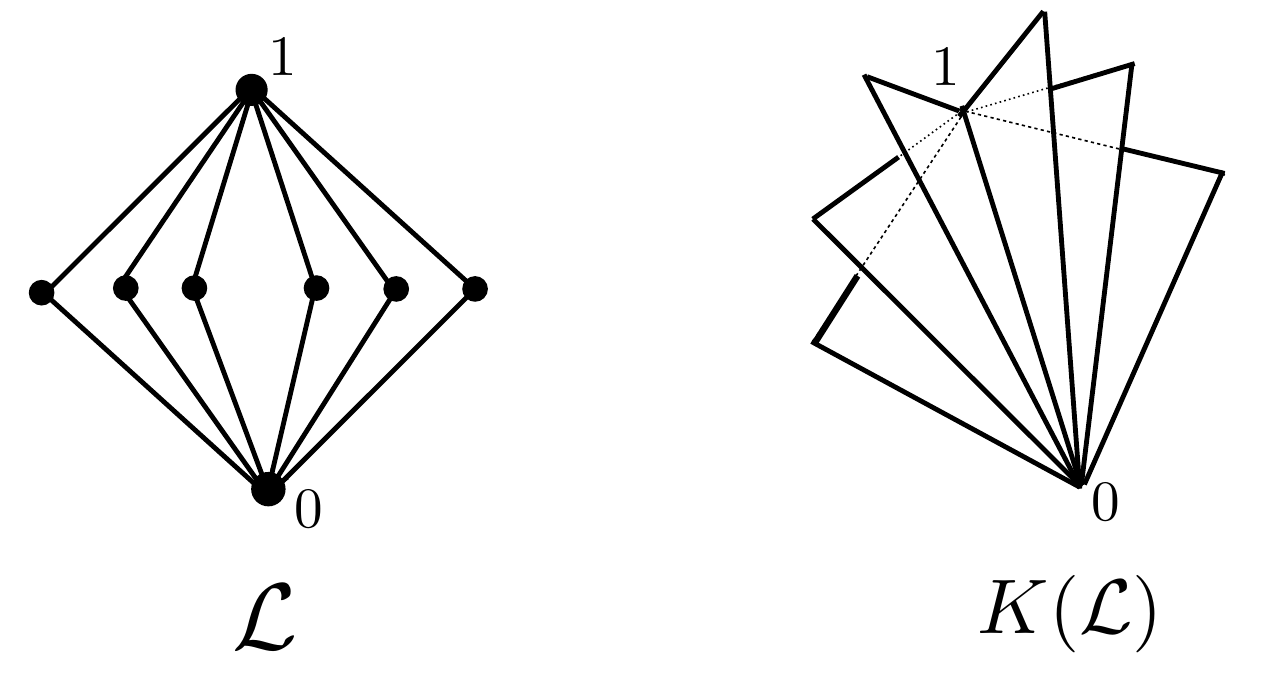}
\end{center}
\caption{The orthoscheme complex of a modular lattice with rank $2$}
\label{fig:folder}
\end{figure}
Throughout this chapter, the size $s$ is supposed to be $1$; namely each orthoscheme
is supposed to be standard.

\section{Main results}\label{subsec:lattice_main}
Brady and McCammond  \cite{BradyMcCammond}*{Conjecture 6.10} conjectured that the orthoscheme complex
of a modular lattice is CAT$(0)$.
Recently, Haettel, Kielak, and Schwer~\cite{HKS}
proved this conjecture for complemented modular lattices,
based on the fact that the diagonal link has the structure of a spherical building,
which is known to be CAT$(1)$; see \cite{BrHa}*{II.10, Appendix}:
\begin{Thm}[\cite{HKS}]\label{thm:HKS}
Let ${\mathcal L}$ be a complemented modular lattice of  finite rank.
Then $(K({\mathcal L}), D_2)$ is {\rm CAT}$(0)$.
\end{Thm}
The main result of this chapter
is to prove Brady-McCammond's conjecture
in the  general case:
\begin{Thm}\label{thm:modularlatticeCAT(0)}
Let ${\mathcal M}$ be a modular lattice of finite rank.
Then $(K({\mathcal M}), D_2)$ is {\rm CAT}$(0)$.
\end{Thm}
We further study the orthoscheme complexes of modular semilattices.
We suggest the following conjecture, generalizing Theorem~\ref{thm:modularlatticeCAT(0)}:
\begin{Conj}\label{conj:modular_CAT(0)}
Let $\mathcal{M}$ be a modular semilattice of  finite rank.
Then $(K(\mathcal{M}), D_2)$ is {\rm CAT}$(0)$.
\end{Conj}
We present several special cases for which this conjecture is true.
A semilattice is called {\em median} if
every principal ideal is a distributive lattice and
every pairwise bounded triple has the join.
It is known that median semilattices are in one-to-one correspondence with pointed median graphs;
see \cite{vdV}*{6.26}.
Namely, for a median graph $G$ and a vertex $b$,
the poset $(V(G), \preceq_b)$ with respect to the base point order $\preceq_b$
is a median semilattice.
Conversely, all median semilattices are obtained in this way.
Also recall that the subspace poset
of a polar space is a canonical example of
a modular semilattice (Lemma~\ref{lem:polar=>modular_semilattice}).

\begin{Prop}\label{prop:specialcases}\label{PROP:SPECIALCASES}
Conjecture~\ref{conj:modular_CAT(0)} is true for the following classes of modular semilattices:
\begin{itemize}
\item[(i)] median semilattices of finite rank.
\item[(ii)] modular semilattices of rank $2$.
\item[(iii)] the subspace posets  of polar spaces.
\end{itemize}
\end{Prop}
To prove (i), we will use an inductive construction of median semilattices
from smaller median semilattices. In general,
if a graded poset $\mathcal{P}$ is the direct product of two graded posets $\mathcal{Q}, \mathcal{Q}'$,
then $(K(\mathcal{P}), D_2)$ is isometric to the product $(K(\mathcal{Q}) \times K(\mathcal{Q}'), D_2)$,
where the metric $D_2$ on $K(\mathcal{Q}) \times K(\mathcal{Q}')$ is
given by $D_2((p,p'),(q,q')) := (D_2(p,q)^2 + D_2(p',q')^2)^{1/2}$; see~\cite{BradyMcCammond}*{Remark 5.3}.
Therefore, by \cite{BrHa}*{1.15}, if both $(K(\mathcal{Q}), D_2)$ and $(K(\mathcal{Q}'), D_2)$ are CAT$(0)$,
then so is $(K(\mathcal{P}), D_2)$.

We will show that a similar property holds for gated amalgams, which can be viewed as
an analogue of Reshetnyak's gluing theorem~\cite{BrHa}*{Theorem 11.1}.
A subsemilattice $\mathcal{M'}$ of a modular semilattice $\mathcal{M}$ is called {\em gated}
if (1) for $p,q,r \in \mathcal{M}$ with $p \preceq r \preceq q$,
$p,q \in \mathcal{M}'$ implies $r \in \mathcal{M'}$, and
(2) for $p,q \in \mathcal{M}$ with $p \vee q \in \mathcal{M}$
$p,q \in \mathcal{M}'$ implies $p \vee q \in \mathcal{M'}$.
From Lemma~\ref{lem-weakly-modular_gated_hull}, this condition rephrases
that the covering graph of $\mathcal{M'}$
is a gated subgraph of the covering graph of $\mathcal{M}$.
By Theorem~\ref{thm:BVV1},  $\mathcal{M'}$ is a modular semilattice.
So the class of modular semilattices is closed under gated amalgams.

\begin{Prop}\label{prop:gluing_semilattices}\label{PROP:GLUING_SEMILATTICES}
Let $\mathcal{L}$, $ \mathcal{M}$, and $\mathcal{M'}$ be modular semilattices
such that $\mathcal{L}$ is a gated amalgam of $\mathcal{M}$ and $\mathcal{M'}$.
If $(K(\mathcal{M}), D_2)$ and $(K(\mathcal{M'}), D_2)$ are {\rm CAT$(0)$},
then $(K(\mathcal{L}), D_2)$ is {\rm CAT}$(0)$.
\end{Prop}

Actually, this result will be derived by a combination of Reshetnyak's gluing theorem
and the following useful fact:

\begin{Prop}\label{prop:gatedsemilattice=>iso}\label{PROP:GATEDSEMILATTICE=>ISO}
Let ${\mathcal{L}}$ be a modular semilattice and $\mathcal{M}$ be a gated semilattice of ${\mathcal{L}}$.
Then for any $p > 0$, $(K(\mathcal{M}), D_p)$ is an isometric subspace of $(K(\mathcal{L}), D_p)$.
\end{Prop}
The remainder of this chapter is organized as follows.
In Section~\ref{subsec:ortho_Boolean}
we summarize basic properties of orthoscheme complexes of Boolean and complemented modular lattices,
which will be also used in Chapter~\ref{sec:complex}.
In Section~\ref{subsec:ortho_distributive}.
we generalize these properties to orthoscheme complexes of distributive and modular lattices.
Finally, in Sections~\ref{subsec:ortho_proof1} and \ref{subsec:ortho_proof2},
we prove Theorem~\ref{thm:modularlatticeCAT(0)} and
Propositions~\ref{prop:specialcases}, \ref{prop:gluing_semilattices}, and \ref{prop:gatedsemilattice=>iso}.

\section{Boolean and complemented modular lattices}\label{subsec:ortho_Boolean}
We continue with the basic properties of orthoscheme complexes of Boolean
and complemented modular lattices. Most of them are well-known and can be found in \cites{BradyMcCammond, HKS, Ti, BuildingBook}.

Let ${\mathcal L}$ be a Boolean lattice of rank $n$. Then ${\mathcal L}$
has exactly $n$ atoms $a_1,a_2,\ldots,a_n$,
and every its element is uniquely represented as the join of atoms. Thus
${\mathcal L}$ is isomorphic to the power set
of $\{a_1,a_2,\ldots,a_n\}$ ordered by inclusion.
In particular, $K({\mathcal L})$ consists of points
$\sum_{k=0}^n \lambda_i (a_{i_1} \vee a_{i_2} \vee \cdots \vee a_{i_k})$
for some permutation  $(i_1,i_2,\ldots,i_n)$ of $\{1,2,\ldots,n\}$
and nonnegative coefficients $\lambda_0,\lambda_1,\ldots, \lambda_n$ with
$\sum_{k=0}^n \lambda_i = 1$, where
the term for $k= 0$ is defined to be $\lambda_0 \cdot 0$.

\begin{Lem}[{see \cite{BradyMcCammond}*{Example 5.2}}]\label{lem:cube}
Let ${\mathcal L}$ be a Boolean lattice of rank $n$
and let $a_1,a_2,\ldots,a_n$ be the set of its atoms.
The orthoscheme complex $K({\mathcal L})$
is isometric to the $n$--dimensional unit cube $[0,1]^n$ of $\RR^n$,
where the isometry is given by
\begin{equation}
K({\mathcal L}) \ni \sum_{k=0}^n \lambda_i (a_{i_1} \vee a_{i_2} \vee \cdots \vee a_{i_k})
 \mapsto \sum_{k=1}^n \lambda_i (e_{i_1} + e_{i_2} + \cdots + e_{i_k}) \in [0,1]^n.
\end{equation}
\end{Lem}

Next we consider the case when
$\mathcal{L}$ is a complemented modular lattice of rank $n$.
A {\em base} of $\mathcal{L}$ is
a set of $n$ atoms $a_1,a_2,\ldots,a_n$ with
$a_1 \vee a_2 \vee \cdots \vee a_n = 1$.
Let $\langle a_1,a_2,\ldots,a_n \rangle$ denote
the sublattice generated by $a_1,a_2,\ldots,a_n$,
which is isomorphic to a Boolean lattice of rank $n$.
It is well-known that
the reduced order complex of $\mathcal{L}$
forms a spherical building of type A,
where apartments
are the subcomplexes of Boolean sublattices generated
by bases; see \cites{Ti, BuildingBook}.
In particular, the following property of bases is nothing but
the axiom (B1) of buildings (see Section~\ref{subsec:building}):

\begin{Lem}\label{lem:Booleansublattice}
For two chains $C,C'$ of $\mathcal L$, there is a base $\{a_1,a_2,\ldots,a_n\}$
such that  the Boolean sublattice $\langle a_1,a_2,\ldots,a_n \rangle$
contains $C$ and $C'$.
\end{Lem}
For a Boolean sublattice $\mathcal{F}$ generated by a base $\{a_1,a_2,\ldots,a_n\}$  and
a maximal chain
$C = (0 = p_0 \prec p_1 \prec p_2 \prec \cdots \prec p_n = 1)$
with $p_i := a_1 \vee a_2 \vee \cdots \vee a_i$,
define a map
$\rho_{\mathcal{F},C}: \mathcal{L} \to \mathcal{F}$
by setting
\begin{eqnarray*}
\rho_{\mathcal{F},C}(p) :=
\bigvee \{  a_i :  p_i \wedge p \succ p_{i-1} \wedge p  \} \quad (p \in \mathcal{L}).
\end{eqnarray*}
The map $\rho_{\mathcal{F},C}$ is order-preserving, and extended to a map from
$K(\mathcal{L})$ to $K(\mathcal{F})$ by
\[
\sum_{i} \lambda_i p_i \mapsto \sum_{i} \lambda_i \rho_{\mathcal{F},C}(p_i).
\]
The map $\rho_{\mathcal{F},C}$
is essentially a canonical retraction; see \cite{BuildingBook}*{Definition 4.38}.
So, the following property of $\rho_{\mathcal{F},C}$ can also be derived
from a standard argument in building theory.
\begin{Lem}\label{lem:nonincrease1}
The map $\rho = \rho_{\mathcal{F}, C}: K({\mathcal L}) \to {K(\mathcal{F})}$ has the following
properties:
\begin{itemize}
\item[(1)] $\rho$ is a retraction from $K({\mathcal L})$ to $K(\mathcal{F})$.
\item[(2)] $\rho$ maps each simplex isometrically onto its image.
\item[(3)]
$D_p(\rho(x), \rho(y)) \leq D_p(x,y)$ for any $x,y \in K(\mathcal L)$, and
the equality holds if $x \in {K({C})}$.
\item[(4)] $(K(\mathcal{F}),D_p)$
is an isometric subspace of $(K(\mathcal{L}),D_p)$ and is
isometric to the cube~$[0,1]^n$ endowed with the $l_p$--metric.
\end{itemize}
\end{Lem}

As in \cite{BrHa}*{II.10, Appendix} and \cite{BuildingBook}*{Section 11.2},
one can prove Theorem~\ref{thm:HKS} as
a simple consequence of
Lemmata~\ref{lem:cube}, \ref{lem:Booleansublattice}, and \ref{lem:nonincrease1}.
To prove Theorem~\ref{thm:modularlatticeCAT(0)}, in the next section
we will establish generalizations of the above three lemmata for modular lattices.

\section{Distributive  and modular lattices}\label{subsec:ortho_distributive}
Let ${\mathcal D}$ be a distributive lattice of rank $n$, i.e.,
a lattice satisfying the distributive law
$u \wedge (v \vee w) = (u \wedge v) \vee (u \wedge w)$.
An element $p$ is said to be {\em join-irreducible} if $p$ cannot be represented
as $q \vee q'$ for some $q,q'$ different from $p$.
Let ${\mathcal D}^{\rm ir}$ denote the set of nonzero join-irreducible elements
of ${\mathcal D}$. Regard ${\mathcal D}^{\rm ir}$ as
a poset by restricting the partial order of ${\mathcal D}$.
A {\em down set} $U$
is a subset of  ${\mathcal D}^{\rm ir}$
such that $y \preceq x \in U$ implies $y \in U$.
Let ${\mathcal I}({\mathcal D}^{\rm ir})$ be
the set of all down sets ordered by inclusion.
Then ${\mathcal I}({\mathcal D}^{\rm ir})$ is a subposet of
the power set $2^{{\mathcal D}^{\rm ir}}$ of ${\mathcal D}^{\rm ir}$.
It is well-known that the poset of down sets of a poset is a distributive lattice,
and that the converse is also true:
\begin{Thm}[{{Birkhoff representation theorem} \cite{Gratzer}*{Theorem 107}}] \label{thm:Birkhoff-representation}
A distributive lattice ${\mathcal D}$ is isomorphic to ${\mathcal I}({\mathcal D}^{\rm ir})$
by the correspondence:
\begin{eqnarray*}\label{eqn:map}
{\mathcal I}({\mathcal D}^{\rm ir}) \ni X & \mapsto & \bigvee_{x \in X} x \in {\mathcal D}, \\
 {\mathcal D} \ni p &\mapsto &\{ r \in {\mathcal D}^{\rm ir}: r \preceq p \} \in {\mathcal I}({\mathcal D}^{\rm ir}).
\end{eqnarray*}
In particular, the cardinality of ${\mathcal D}^{\rm ir}$ is equal to the rank of ${\mathcal D}$.
\end{Thm}
If  ${\mathcal D}$ is a Boolean lattice, then
${\mathcal D}^{\rm ir}$ is exactly the set of its atoms. An extension of Lemma~\ref{lem:cube} is the
following result, which asserts that the orthoscheme complex of a distributive lattice can be realized
as a convex polytope in $\RR^n$:
\begin{Prop}\label{thm:poly}
Let ${\mathcal D}$ be a distributive lattice with ${\mathcal D}^{\rm ir} = \{ r_1,r_2,\ldots, r_n\}$.
Then the orthoscheme complex $K({\mathcal D})$ is isometric to
the following convex polytope in $\RR^n$:
\begin{equation}\label{eqn:poly}
P({\mathcal D}) := \{ y \in [0,1]^n: y_i \geq y_j \ \mbox{ for all } i,j \in \{1,2,\ldots,n\} \mbox{ with } r_i \preceq r_j \},
\end{equation}
where the isometry is given by
\begin{equation}\label{eqn:isometry}
K({\mathcal D}) \ni p = \sum_{i} \lambda_i p_i  \mapsto
\sum_{i} \lambda_i  \sum_{j: r_j \preceq p_i} e_j.
\end{equation}
\end{Prop}
The polytope $P(\mathcal{D})$ is known as the {\em order polytope},  and
this triangulation of the order polytope by orthoschemes
also appears in several contexts in discrete mathematics;
see \cite{Matousek}*{Section 12.3} for example.
Even if Proposition~\ref{thm:poly} can be considered as  a folklore,
we provide a self-contained proof for completeness.
\begin{proof}
Denote the map in (\ref{eqn:isometry}) by $\varrho$.
It is not difficult to see that $\varrho$ is injective.
We show that $\varrho$ is surjective.
For an arbitrary point $x \in P({\mathcal D})$  we will construct
a point $p \in K({\mathcal D})$ such that $\varrho(p) = x$.
Arrange the coordinates of $x = (x_1,x_2,\ldots,x_n)$
so that $x_{i_1} \geq x_{i_2} \geq \cdots \geq x_{i_n}$.
Then $r_{i_{k}} \prec r_{i_{k'}}$ implies $k < k'$.
Define $p_{k} := r_{i_1} \vee r_{i_2} \vee \cdots \vee r_{i_k}$
for $k=1,2,\ldots,n$.
Then $(0 = p_0 \prec p_1 \prec \cdots \prec p_n = 1)$ is a maximal chain of ${\mathcal D}$.
Let $\lambda_k := x_{i_k} - x_{i_{k+1}}$ for $k=0,1,2,\ldots,n$,
and let $p:= \sum_{k=0}^n \lambda_k p_k$,
where we let $x_{i_{0}} := 1$ and $x_{i_{n+1}} := 0$.
Then $\varrho(p) = x$ holds. Indeed,
\begin{eqnarray*}
\varrho(p) & = & \varrho \left( \sum_{k=0}^n \lambda_k p_k \right)  =
\varrho \left(\sum_{k=0}^n \lambda_k (r_{i_1} \vee r_{i_2} \vee \cdots \vee r_{i_k}) \right) \\
& = & \sum_{k=1}^n (x_{i_k} - x_{i_{k+1}}) ( e_{i_1} + e_{i_2} + \cdots + e_{i_k}) = x.
\end{eqnarray*}
Thus $\varrho$ is a bijection.

Next we show that $\varrho$ is an isometry.
Since $P(\mathcal{D})$ is convex and
the length metric on $P(\mathcal{D})$ coincides with the induced metric of $\RR^n$, it suffices to show that $\varrho$
is an isometry on each maximal simplex $\sigma$. Suppose that $\sigma$ corresponds
to a maximal chain
$(0 = p_0 \prec p_1 \prec p_2 \prec \cdots \prec p_n = 1)$.
There exists a unique permutation $(i_1,i_2,\ldots,i_n)$ on $\{1,2,\ldots,n\}$ such that
$p_k = r_{i_1} \vee r_{i_2} \vee \cdots \vee r_{i_k}$ holds for $k=1,2,\ldots,n$.
This permutation determines an isometry $\tau:\RR^n \to \RR^n$ (with respect to any $l_p$--metric) defined
by $x = (x_1,x_2,\ldots,x_n) \mapsto
(x_{i_1},x_{i_2},\ldots,x_{i_n})$.
Then $\varrho = \tau \circ \varphi_{\sigma}$ holds on $\sigma$.
Indeed, for $p = \sum_k \lambda_k p_k \in \sigma$, we have
\begin{equation*}
\tau \circ \varphi_{\sigma} (p) =
\tau \left(\sum_{k=1}^n \lambda_k (e_1+e_2 + \cdots + e_k) \right)
=  \sum_{k=1}^n \lambda_k (e_{i_1} + e_{i_2} + \cdots + e_{i_k}) = \varrho(p).
\end{equation*}
Hence $\| \varrho(x) - \varrho(y) \|_{p} = \| \tau \circ \varphi_{\sigma}(x) - \tau \circ \varphi_{\sigma} (y) \|_{p}
= \| \varphi_{\sigma}(x) - \varphi_{\sigma} (y) \|_{p} = D_{p}(x,y)$ for $x,y \in \sigma$.
\end{proof}

We consider now the case of modular lattices.
An analogue of Lemma~\ref{lem:Booleansublattice} is the following:

\begin{Thm}[{Dedekind-Birkhoff,  \cite{Gratzer}*{Theorem 363}}]\label{thm:Dedekind-Birkoff}
For two chains $C,C'$ of a modular lattice ${\mathcal M}$ of rank $n$,
the sublattice generated by $C \cup C'$ is distributive.
In particular, there is a distributive sublattice containing $C$ and $C'$.
\end{Thm}

Let $C = (0 = p_0 \prec p_1 \prec \cdots \prec p_n = 1)$ be a maximal chain of $\mathcal M$
and let ${\mathcal D}$ be a distributive sublattice containing $C$.
We can suppose that ${\mathcal D}^{\rm ir} = \{r_1,r_2,\ldots,r_n\}$
and $p_{i} = r_1 \vee r_2 \vee \cdots \vee r_{i}$ for $i=1,2,\ldots,n$.
The distributive sublattice ${\mathcal D}$ can be naturally regarded
as a sublattice of the Boolean lattice $2^{{\mathcal D}^{\rm ir}}$:
\[
{\mathcal M} \supseteq {\mathcal D} \simeq {\mathcal I}({\mathcal D}^{\rm ir}) \subseteq 2^{{\mathcal D}^{\rm ir}}.
\]
This implies:
\[
K({\mathcal M}) \supseteq K({\mathcal D}) \simeq P({\mathcal D}) \subseteq [0,1]^n  \simeq K(2^{{\mathcal D}^{\rm ir}}).
\]

We can define a map $\rho_{{\mathcal D}, C}: {\mathcal M} \to 2^{{\mathcal D}^{\rm ir}}$ by
\begin{equation}
\rho_{{\mathcal D}, C}(p) := \{ r_j:  p_{j} \wedge p \succ p_{j-1} \wedge p  \} \quad \mbox{ for any } p \in {\mathcal M}.
\end{equation}
We see in the proof of the next lemma
that $\rho$ is order-preserving, and hence $\rho$ maps a chain in ${\mathcal M}$
to a chain in $2^{{\mathcal D}^{\rm ir}}$.
So we can extend this map $\rho_{{\mathcal D}, C}$ to $K({\mathcal M}) \to K(2^{{\mathcal D}^{\rm ir}})$ by
\begin{equation}
\sum_{i} \lambda_i q_i \mapsto \sum_{i} \lambda_i \rho_{{\mathcal D}, C} (q_i).
\end{equation}
Then a generalization of Lemma~\ref{lem:nonincrease1} is the following:
\begin{Lem}\label{lem:nonincrease}
The map $\rho = \rho_{{\mathcal D}, C}: K({\mathcal M}) \to K(2^{{\mathcal D}^{\rm ir}})$ has the following properties:
\begin{itemize}
\item[$(1)$] $\rho$ is the identity on $K({\mathcal D})$
under the identification ${\mathcal D} \simeq \mathcal{I}(\mathcal{D}^{\rm ir})$.
\item[$(2)$] $\rho$ maps each simplex isometrically onto its image.
\item[$(3)$]
$D_p( \rho(x),  \rho(y)) \leq D_p(x,y)$ for any $x,y \in K({\mathcal M})$ and
the equality holds if $x \in K(C)$.
\item[$(4)$] $K(\mathcal{D})$ is an isometric subspace of $K(\mathcal{M})$ isometric to the convex polytope $P(\mathcal{D})$.
\end{itemize}
\end{Lem}
\begin{proof}
To (1):
In view of  Theorem~\ref{thm:Birkhoff-representation},
for $q \in {\mathcal D}$, we have $q = \bigvee_{r_j \preceq q} r_j$ and
\[
p_i \wedge q = \bigvee_{r_j \preceq p_i \wedge q} r_j = \bigvee_{r_j \preceq q, r_j \preceq p_i} r_j,
\]
and hence $p_i \wedge q \succ p_{i-1} \wedge q$
if and only if $r_i \preceq q$; recall that $p_k = r_1 \vee r_2 \vee \cdots \vee r_k$.
This means
\begin{equation}\label{eqn:rho(q)}
\rho(q)
= \{ r_j : p_{j} \wedge q \succ p_{j-1} \wedge q \} =
\{ r_j : r_j \preceq q\}
\simeq \bigvee_{r_j \preceq q} r_j = q.
\end{equation}

To (2):
We first show that the map
$\rho$ actually preserves the partial order, i.e.,
\begin{equation*}
\rho(p) \subseteq \rho(q)\ {\rm if}\ p \preceq q.
\end{equation*}
Indeed, $p_{j} \wedge p \succ p_{j-1} \wedge p$ implies $p_{j} \wedge q \succ p_{j-1} \wedge q$;
otherwise $p_{j} \wedge q = p_{j-1} \wedge q$ (by $p_j \succ p_{j-1}$), and
$p_j \wedge p = p_{j} \wedge q \wedge p = p_{j-1} \wedge q \wedge p = p_{j-1} \wedge p$, a contradiction. Hence, by definition, we have $\rho(p) \subseteq \rho(q)$.
Therefore $\rho$ maps a chain to a chain;
hence $\rho: K({\mathcal M}) \to K(2^{{\mathcal D}^{\rm ir}})$ is well-defined.

Moreover $\rho$ preserves the rank:
\[
r(p) = r(  \rho(p)) \quad  \mbox{ for } p \in {\mathcal M}.
\]
Indeed, consider a (possibly repeated)
chain $(0 = p_0 \wedge p \preceq  p_1 \wedge p \preceq p_2 \wedge p
\preceq \cdots \preceq p_n \wedge p =  p)$.
Then the rank of $p$ is equal to the number of indices $j$ with $p_{j} \wedge p \succ p_{j-1} \wedge p$, which is equal to the rank
$r(\rho(p)) = |\{ r_j:  p_{j} \wedge p \succ p_{j-1} \wedge p  \}|$ of $\rho(p)$.

We next show that $\rho$ maps each simplex $\sigma$ isometrically.
Suppose that $\sigma$ corresponds to a chain $(q_0 \prec q_1 \prec \cdots \prec q_k)$.
The simplex $\sigma$ is bijectively mapped to a simplex
$\rho(\sigma)$ corresponding to a chain $(\rho(q_0) \prec \rho(q_1) \prec \cdots \prec \rho(q_k))$.
For a point $x = \sum_{i} \lambda_i q_i \in \sigma$, it holds
\begin{align*}
\varphi_{\sigma}(x) &= \sum_{i=1}^k \lambda_i (e_1 + e_2 + \ldots + e_{r[p_0, p_i]})
 \\ &=\sum_{i=1}^k \lambda_i (e_1 + e_2 + \ldots + e_{r[\rho(p_0), \rho(p_i)]}) = \varphi_{\rho(\sigma)}(\rho(x)),
\end{align*}
where we use the rank-preserving property of $\rho$.
Thus we have (2): 
\begin{align*}
D_p(x,y) &= \|\varphi_{\sigma}(x) - \varphi_{\sigma}(y) \|_p  \\ &
=\|\varphi_{\rho(\sigma)}(\rho(x)) - \varphi_{\rho(\sigma)}(\rho(y)) \|_p = D_p(\rho(x), \rho(y)) \quad (x,y \in \sigma).
\end{align*}

To (3)$\&$(4): The first part of (3) follows from (2).
Indeed,
take two points $x,y \in K({\mathcal M})$ and a geodesic $P$
connecting $x$ and $y$. Let  $x = x_0,x_1,\ldots, x_m = y$ be points of $P$
such that for any $i=0,\ldots,m-1$, the points $x_i, x_{i+1}$ belong to a common simplex.
Hence the line-segment between $x_i$ and $x_{i+1}$ in $P$
is isometrically mapped to $K(2^{{\mathcal D}^{\rm ir}})$.
This means that the image $\rho(P)$ of $P$
is a path between $\rho(x)$ and $\rho(y)$ of the same length as $P$.
Hence $D_p( \rho(x),   \rho(y)) \leq D_p( \rho(P))\le  D_p(P) = D_p(x,y)$.

In particular, if $x,y$ belong to $K(\mathcal{D})$,
then $x = \rho(x)$, $y = \rho(y)$, and $\rho(P)$ is
also a geodesic between $x$ and $y$ belonging to $K(2^{\mathcal{D}^{\rm ir}})$.
Identify $K(\mathcal{D})$ with $P(\mathcal{D})$
and $K(2^{\mathcal{D}^{\rm ir}})$ with $[0,1]^n$.
Since $P(\mathcal{D})$  is isometric to a convex polytope in $[0,1]^n$,
the segment between $x$ and $y$ is a geodesic belonging to $P(\mathcal{D})$.
This implies (4): $K(\mathcal{D})$
is an isometric subspace of $K(\mathcal{M})$ (and of $K(2^{{\mathcal{D}}^{\mathrm{ir}}})$).

Finally we show the second part of (3).
Take a maximal chain $C'$ such that $K(C')$ contains $y$.
Consider a distributive sublattice $\mathcal{E}$ containing $C$ and $C'$.
We can suppose that $\mathcal{E}^{\mathrm{ir}} = \{r'_1,r'_2,\ldots, r'_n\}$
and $p_{i} = r'_1 \vee r'_2 \vee \cdots \vee r'_{i}$ for $i=1,2,\ldots,n$.
Define map $\vartheta: 2^{\mathcal{E}^{\mathrm{ir}}} \to 2^{\mathcal{D}^{\mathrm{ir}}}$
by $\{r'_{i_1}, r'_{i_2}, \ldots, r'_{i_k}\} \mapsto
\{ r_{i_1}, r_{i_2}, \ldots, r_{i_k}\}$, and extend it to
$K(2^{\mathcal{E}^{\mathrm{ir}}}) \to K(2^{\mathcal{D}^{\mathrm{ir}}})$ as above.
Then it is easy to see
\begin{eqnarray*}
&& D_p (\vartheta(u), \vartheta(v)) = D_p(u,v) \quad (u,v \in K(2^{\mathcal{E}^{\mathrm{ir}}})),  \\
&& \vartheta(x) = x = \rho(x) \quad (x = \sum_i \lambda_i p_i \in K(C)),
\end{eqnarray*}
where the latter equalities follow from (1) and
$\vartheta (x) = \sum_i \lambda_i \vartheta (\{ r'_1,r'_2,\ldots,r'_i \})
= \sum_i \lambda_i \{ r_1,r_2,\ldots,r_i \} \simeq x$.
We show $\vartheta(y) = \rho(y)$ for $y \in K(\mathcal{E})$,
implying $D_p(x,y) = D_p(\vartheta(x), \vartheta(y)) = D_p(\rho(x), \rho(y))$
for $x \in K(C)$.
It suffices to show that $\vartheta(q) = \rho(q)$ for $q \in \mathcal{E}$.
Since $q = \bigvee_{r'_j \preceq q} r'_j =  \bigvee_{j: p_j \wedge q \succ p_{j-1} \wedge q} r'_j \simeq \{r'_j :  p_j \wedge q \succ p_{j-1} \wedge q \}$
(by (\ref{eqn:rho(q)})),
we have $\vartheta(q) = \{r_j :  p_j \wedge q \succ p_{j-1} \wedge q \}
= \rho(q)
$, as required.
\end{proof}

\section{Proof of Theorem~\ref{thm:modularlatticeCAT(0)}}\label{subsec:ortho_proof1}

Let $D := D_2$.
Pick any  $x,y \in K({\mathcal M})$.
By Theorem~\ref{thm:Dedekind-Birkoff},
there exists a distributive lattice ${\mathcal D}$ (of rank $n$) such that $x, y \in K({\mathcal D})$.
By Lemma~\ref{lem:nonincrease}~(4), we can also take a geodesic $\gamma(x,y) \subseteq K({\mathcal D})$
connecting $x,y$.
By \cite{BuildingBook}*{Proposition 11.4},
it suffices to show that for any $z \in K({\mathcal M})$, $t \in [0,1]$ and $p \in \gamma(x,y)$
with $D(x,p) = t D(x,y)$,  it holds that
\begin{equation}\label{eqn:D^2(z,p)}
D^2(z,p) \leq (1 - t)D^2(z,x) + t D^2(z,y) - t(1- t)D^2(x,y).
\end{equation}
Pick  a maximal chain $C$ in ${\mathcal D}$ with $p \in K(C)$.
Consider the image $z' := \rho_{{\mathcal D}, C}(z)$ of $z$.
Then $x, y, p, z'$ can be viewed as points of $[0,1]^n \simeq K(2^{\mathcal{D}^{\rm ir}})$
so that $p = (1-t) x + t y$.
Hence we have
\[
D^2(z',p) = (1 - t)D^2(z',x) + t D^2(z',y) - t(1- t)D^2(x,y).
\]
See \cite{BuildingBook}*{p.\ 551}.
By Lemma~\ref{lem:nonincrease}~(3),
we have $D^2 (z',p) = D^2(z,p)$, $D^2(z',x) \leq D^2(z,x)$,
and $D^2(z',y) \leq D^2(z,y)$, to obtain (\ref{eqn:D^2(z,p)}).

\begin{remark}
If  the modular lattice $\mathcal{M}$ in question is embedded
into a complemented modular lattice $\mathcal{L}$, then one can show that
$K(\mathcal{M})$ is an isometric subspace of $K(\mathcal{L})$.
Hence the CAT$(0)$--property of $K(\mathcal{M})$ follows from the CAT$(0)$--property of
$K(\mathcal{L})$ (Theorem~\ref{thm:HKS}).
However, there exist modular lattices that cannot be embedded into
complemented modular lattices~\cite{Gratzer}*{Corollary 443}.
\end{remark}

\section{Proof of Propositions~\ref{prop:gluing_semilattices} and \ref{prop:gatedsemilattice=>iso}}
Let $G$ be an orientable modular graph and $o$ be an admissible orientation of $G$.
Consider the poset $\mathcal{P}(G,o)$ on $V(G)$ induced by $o$.
Then $\mathcal{P}(G,o)$ is graded.
Hence we can consider the orthoscheme complex $K(\mathcal{P}(G,o))$. Then
Proposition~\ref{prop:gatedsemilattice=>iso} is a special case of the following:
\begin{Prop}\label{prop:gatesemilattice=>iso'}
Let $G$ be an orientable modular graph with an admissible orientation $o$
and let $H$ be a gated subgraph of $G$.
For any $p > 0$,
there exists a nonexpansive retraction from $(K({\mathcal P}(G,o)), D_p)$ to
$(K(\mathcal{P}(H,o)), D_p)$.
In particular, $(K(\mathcal{P}(H,o)), D_p)$ is an isometric subspace of $(K({\mathcal P}(G,o)), D_p)$.
\end{Prop}
\begin{proof}
Let $X$ be the vertex set of  $H$.
Let $\phi: V \to X$ denote the projection to the gated set $X$, where
$\phi(v)$ is the gate in $X$ of $v\in V$.

\begin{claim}
For $p,q \in V(G)$, if $p \preceq q$, then $\phi(p) \preceq \phi(q)$
\end{claim}
\begin{proof}
It suffices to consider the case where $q$ covers $p$.
Since $\phi$ is nonexpansive, $\phi(p) = \phi(q)$ or
$\phi(p)$ and $\phi(q)$ are adjacent.
Suppose that $\phi(p)$ and $\phi(q)$ are adjacent.
Consequently,
$d(p, \phi(q)) = d(p, \phi(p)) + 1 = d(q, \phi(p)) + 1 = d(q, \phi(q))$ must hold.
Pick  a shortest path $(p = p_0,p_1,\ldots,p_m = \phi(p))$.
By successive application of (QC), there is a shortest path $(q = q_0,q_1,\ldots,q_m = \phi(q))$
such that $p_i$ and $q_i$ are adjacent.
Hence $p \preceq q$ and $p_{i} \preceq q_{i}$, which implies $\phi(p) \preceq \phi(q)$.
\end{proof}

In particular, $\phi$ is order-preserving on ${\mathcal{P}}(G,o)$  and can be extended
to the map $K({\mathcal{P}}(G,o)) \to K({\mathcal{P}}(H,o))$ by setting
\[
\phi(x) := \sum_i \lambda_i \phi(x_i) \quad \mbox{ for } x = \sum_{i} \lambda_i x_i \in K({\mathcal{P}}(G,o)).
\]

We are going to show that $\phi$ is a desired nonexpansive retraction.
Consider a maximal simplex $\sigma$ corresponding to a maximal chain
$(y = y_0 \prec y_1 \prec \cdots \prec y_n = z)$ (with $r[y_{i}, y_{i+1}]=1$).
For a point $x = \sum_{i=0}^n \lambda_i y_i \in \sigma$, we have
\begin{equation}\label{eqn:k=i_to_n}
\varphi_\sigma (x) = \sum_{i=1}^n \lambda_i (e_1+e_2+ \cdots + e_i)
= \sum_{i=1}^n \left( \sum_{k=i}^n \lambda_k \right) e_i.
\end{equation}
Next consider the image $\sigma'$ of $\sigma$ by the map $\phi$.
Then simplex $\sigma'$ corresponds to
a chain
$(y' = y'_0 \prec y'_1 \prec \cdots \prec y'_{m} = z')$,
where $y'_0, y'_1, \ldots,y'_m$ are the images of $y_1,y_2,\ldots, y_n$,
$r[y'_{j},y'_{j+1}] = 1$, and $m \leq n$.
Let $f:\{0, 1,2,\ldots,n\} \to \{0, 1,2,\ldots, m\}$
be the map defined by $\phi(y_i) = y'_{f(i)}$ for $i=0,1,2,\ldots,n$.
Then the inverse images $f^{-1}(j)$ for $j=0,1,2,\ldots,m$
are nonempty disjoint intervals of $\{0, 1,2,\ldots,n\}$.
Then we have
\[
\phi(x) = \sum_{i=0}^n \lambda_i \phi(y_i)
= \sum_{j=0}^m \left( \sum_{i \in f^{-1}(j)} \lambda_i \right) y'_j.
\]
For $j=1,\ldots, m$,
let $k_j$ denote the minimum index in $f^{-1}(j)$.
Then $1 \leq k_1 < k_2 < \cdots < k_m \leq n$, and it holds
\begin{equation}\label{eqn:k=k_j_to_n}
\varphi_{\sigma'} (\phi (x))  =
\sum_{j=1}^{m} \left( \sum_{i \in f^{-1}(j)} \lambda_{i} \right) (e_1 + e_2 + \cdots + e_j)
 =  \sum_{j=1}^m \left( \sum_{i = k_j}^n \lambda_{i} \right) e_j.
 \end{equation}
Pick any two points $x = \sum_i \lambda_i y_i$ and $y = \sum_i \mu_i y_i$ in $\sigma$.
By (\ref{eqn:k=i_to_n}) and (\ref{eqn:k=k_j_to_n}), we have
\begin{align*}
D_p(x,y) &= \left\| \sum_{i=1}^n \left( \sum_{k=i}^n \lambda_k - \mu_k \right) e_i \right\|_p
 \\&\geq\left\| \sum_{j =1}^m \left( \sum_{k=k_j}^n \lambda_k - \mu_k \right) e_j \right\|_p
= D_p(\phi(x),\phi(y)).
\end{align*}
Therefore $\phi$ is nonexpansive on each simplex,
and, consequently,  $\phi$ is a nonexpansive map from
$K({\mathcal{P}}(G,o))$ to $K({\mathcal{P}}(H,o))$.
\end{proof}
To prove Proposition \ref{prop:gluing_semilattices}, suppose that a modular semilattice $\mathcal{L}$ is
the gated amalgam of two modular semilattices $\mathcal{M}$ and $\mathcal{M}'$
with respect to a common gated subsemilattice $\mathcal{L}_0$ of $\mathcal{M}$ and $\mathcal{M}'$.
By Proposition~\ref{prop:gatedsemilattice=>iso},
$(K(\mathcal{L}_0),D_2)$ is an isometric subspace of
both $(K(\mathcal{M}),D_2)$ and $(K(\mathcal{M}'),D_2)$.
Since $K(\mathcal{M})$ and $K(\mathcal{M}')$ are CAT$(0)$,
$K(\mathcal{L}_0)$ is a common convex subspace.
Hence $K(\mathcal{L})$ is obtained by gluing $K(\mathcal{M})$ and $K(\mathcal{M}')$
along a convex subspace $K(\mathcal{L}_0)$.
By Reshetnyak's gluing theorem~\cite{BrHa}*{Theorem 11.1},
$(K(\mathcal{L}),D_2)$ is CAT(0). %% \hfill $\Box$

\section{Proof of Proposition~\ref{prop:specialcases}}\label{subsec:ortho_proof2}

To (ii):
The following argument is essentially the same as the one
in \cite{Ch_CAT}*{p.\ 131}. It suffices to show that the link in $K(\mathcal{M})$
of every vertex has no isometric cycles of length less than $2\pi$.
It suffices to consider the link of $0$
(since the links of other vertices are stars with center $0$, and cannot have cycles).
The respective link is the graph obtained from
the covering graph of $\mathcal{M}$ by deleting $0$ and
with edge-lengths equal to $\pi/4$.
By the second condition of the definition,
this graph cannot have a cycle with less than $8$ edges.
This implies that the link has no isometric cycles of length less than $2\pi$.

To (iii):
Suppose that the rank of ${\mathcal L}$ is $n$.
A {\em polar frame} $F$ is a set of $2n$ points (atoms)
such that every point in $F$ is collinear with all others points except one.
Namely, $F$ is partitioned into $n$ pairs $\{a_i,  \bar a_i \}$ $(i=1,2,\ldots,n)$
so that  $a_i$ and $\bar a_i$ are not collinear.
Let $\langle F \rangle$ be the subsemilattice
of ${\mathcal L}$ generated by $F$.
Every element of $\langle F \rangle$ is uniquely represented as $\bigvee_{b \in X} b$
for some $X \subseteq F$ with $|X \cap \{a_i, \bar a_i\}| \leq 1$ for $i=1,2,\ldots,n$.
As in Lemma~\ref{lem:cube}, one can see that
$K(\langle F \rangle)$ is isometric to the $n$--cube $[-1,1]^n$ in $\RR^n$
by the map
\[
\sum_{k=0}^n \lambda_k (b_{i_1} \vee b_{i_2} \vee \cdots \vee b_{i_k})
\mapsto \sum_{k=1}^n \lambda_k ( s(b_{i_1}) e_{i_1} +  s(b_{i_2}) e_{i_2} + \cdots + s(b_{i_k}) e_{i_k}).
\]
where $b_i \in \{a_i, \bar a_i\}$ and $s(b_{i}) := + 1$ if $b_i = a_i$, and $s(b_{i}) := - 1$ if $b_i = \bar a_i$.

The reduced order complex of $\mathcal{L}$ is
a spherical building of type C, and
the system of apartments is
the set of subcomplexes
corresponding to all polar frames~\cite{Ti}.
Consider a canonical retraction
$\rho = \rho_{{\Sigma}, C}$ to an apartment $\Sigma$.
Extend $\rho$ to the order complex of $\mathcal{L}$ by defining identity on ${\bf 0}$,
and extend $\rho$ to $K({\mathcal L})$.
The resulting map $\rho$ satisfies the properties of
Lemma~\ref{lem:nonincrease1}.
Also the analogue of Lemma~\ref{lem:Booleansublattice} holds
for polar frames (which is nothing but the building's axiom B1).
Therefore precisely the same argument of
the proof of Theorem~\ref{thm:modularlatticeCAT(0)}
works to conclude that $K(\mathcal{L})$ is CAT$(0)$.

To (i): We start with the following basic lemma:

\begin{Lem}\label{gated_semi}
Let $\mathcal{M}$ be a median semilattice and let $\mathcal{D}$
be a median subsemilattice of $\mathcal{M}$.
For an atom $a$ of $\mathcal{M}$,
define $\mathcal{D}^a$ and $\mathcal{D}_a$ by setting
\[
\mathcal{D}^a  :=  \{ p \in \mathcal{D}:  \mbox{$p \vee a$ exists in $\mathcal{M}$}  \}, \quad
\mathcal{D}_a  :=  \{ p \in \mathcal{D}:  p \wedge a = 0 \}.
\]
Then both $\mathcal{D}^a$ and $\mathcal{D}_a$ are gated subsemilattices of $\mathcal{D}$,
and $\mathcal{D}$ is the gated amalgam of $\mathcal{D}^a$ and $\mathcal{D}_a$.
\end{Lem}
\begin{proof}
Pick $p,q \in \mathcal{D}^a$; then $p \vee a$ and $q \vee a$ exist in $\mathcal{M}$.
It is obvious that $p \wedge q$ and $a$ have the join, implying $p \wedge q \in \mathcal{D}^a$.
If $p \preceq q$, then obviously $[p,q] \subseteq \mathcal{D}^a$.
If the join of $p, q$ exists in $\mathcal{D}$, then it must be equal to the join of $p,q$ in $\mathcal{M}$
and $p \vee a, q \vee a$, and $p \vee q$ exist, implying that $p \vee q \vee a$ also exists in $\mathcal{M}$.
Thus $p \vee q \in \mathcal{D}^a$, and $\mathcal{D}^a$ is a gated subsemilattice of $\mathcal{D}$.

Similarly,  pick  $p,q \in \mathcal{D}_a$; then $p \wedge a = q \wedge a = 0$.
It is obvious that $[p,q] \subseteq \mathcal{D}_a$ if $p \preceq q$.
Since $p \wedge q \wedge a = 0$,
we have $p \wedge q \in \mathcal{D}_a$.
If  $p \vee q$ exists, then $(p \vee q) \wedge a = (p \wedge a) \vee (q \wedge a)  =0$,
implying $p \vee q \in \mathcal{D}_a$. Hence $\mathcal{D}^a$ is a gated subsemilattice.
Since $p \wedge a \succ 0$ implies $p \succeq a$ and $p \vee a = p$,
we have $\mathcal{D} = \mathcal{D}^a \cup \mathcal{D}_a$, whence $\mathcal{D}$
is a gated amalgam of $\mathcal{D}^a$ and $\mathcal{D}_a$.
\end{proof}

Let $\mathcal{D}$ be a median semilattice.
Let $\mathcal{D}^{\rm ir}$ be the set of irreducible elements of $\mathcal{D}$.
Since every principal ideal is distributive,
any element $p$ in $\mathcal{D}$ is uniquely represented as
\[
 p =  \bigvee \{  r \in \mathcal{D}^{\rm ir}:  r \preceq p\}.
\]
Let ${\mathcal B}({\mathcal D})$ be the set of subsets $X$
of $\mathcal{D}^{\rm ir}$
such that the join over $X$ exists in $\mathcal{D}$.
Regard ${\mathcal B}({\mathcal D})$ as a poset with respect to the inclusion order.
\begin{Lem}\label{lem:B(D)}
${\mathcal B}({\mathcal D})$ is a median semilattice such that each principal ideal is a Boolean lattice.
\end{Lem}
\begin{proof}
Obviously each principle ideal of ${\mathcal B}({\mathcal D})$ is a Boolean lattice.
Take a pairwise bounded triple $X,Y,Z$ of ${\mathcal B}({\mathcal D})$.
By definition, we can choose elements $s,t,u \in \mathcal{D}$ such that
$s = \bigvee X$, $t = \bigvee Y$, and $u = \bigvee Z$.
Then $s,t,u$ is a pairwise bounded triplet in $\mathcal{D}$.
Hence $s \vee t \vee u$ exists, which implies that $X \cup Y \cup Z$ is a member of ${\mathcal B}({\mathcal D})$.
\end{proof}

Notice that $\mathcal{D}$ is embedded in ${\mathcal B}({\mathcal D})$
by the  map $\tau: \mathcal{D} \to {\mathcal B}({\mathcal D})$ defined by
\[
{\mathcal D} \ni p \mapsto \{  r \in \mathcal{D}^{\rm ir}: r \preceq p  \} \in {\mathcal B}({\mathcal D}).
\]
This map is obviously injective, and preserves the partial order and the meet:
\begin{equation*}
\tau (p \wedge q) =  \{ r \in {\mathcal D}^{\rm ir}: r \preceq p \wedge q  \} = \{ r \in {\mathcal D}^{\rm ir}:  r \preceq p, r \preceq q  \}  =  \tau (p) \cap \tau (q).
\end{equation*}
Therefore we can regard $\mathcal{D}$ as a median subsemilattice
of ${\mathcal B}({\mathcal D})$.

\begin{Lem}\label{lem:ideal_distributive}
The intersection of $\mathcal{D}$ and any principal ideal of ${\mathcal B}({\mathcal D})$
is a distributive sublattice of $\mathcal{D}$.
\end{Lem}
\begin{proof}
Pick $X \in {\mathcal B}({\mathcal D})$.
Then an element $p$ of $\mathcal{D}$ belongs to the principal ideal of $X$
if and only if $p \preceq \bigvee X$ (or equivalently
if every irreducible element $r$ with $r \preceq p$ belongs to $X$).
This means that the intersection of $\mathcal{D}$ and the principal ideal of $X$
is equal to the principal ideal of $\bigvee X$ in $\mathcal{D}$.
Thus this intersection  is a distributive lattice.
\end{proof}

Let us start the proof of (i).
We first consider the case when  ${\mathcal D}$ is a finite median semilattice.
We use the induction on the number of elements of ${\mathcal{D}}$.
Consider the semilattice $\mathcal{B}({\mathcal{D}})$.
Suppose that there are atoms $A =\{a\}, A' = \{a'\}$
in $\mathcal{B}({\mathcal{D}})$ $(a,a' \in \mathcal{D}^{\rm ir})$
such that $A$ and $A'$ do not have the join, i.e.,
$\{a,a'\} \not \in \mathcal{B}({\mathcal{D}})$.
By Lemma \ref{gated_semi}, ${\mathcal{D}}$
is the gated amalgam of ${\mathcal{D}}^A$ and ${\mathcal{D}}_{A}$.
Then $\{  r \in \mathcal{D}^{\rm ir}:  r \preceq a' \}$, which is an element of $\mathcal{D}$, does not have
the join with $A = \{a\}$ in $\mathcal{B}(\mathcal{D})$
(otherwise $\{a,a'\} \in \mathcal{B}({\mathcal{D}})$, a contradiction).
Also $\{r \in \mathcal{D}^{\rm ir}:  r \preceq a \}$ (that is in $\mathcal{D}$)
does not belong to $\mathcal{D}_A$.
Therefore both ${\mathcal D}^A \setminus {\mathcal{D}}_A$ and
${\mathcal D}_A \setminus {\mathcal{D}}^A$ are nonempty.
By induction, both $K({\mathcal D}^A)$ and $K({\mathcal D}_A)$ are CAT$(0)$.
By Proposition~\ref{prop:gluing_semilattices},
$K({\mathcal D})$ is CAT$(0)$.

Suppose now that any pair of atoms in $\mathcal{B}({\mathcal{D}})$ has the join.
By Lemma~\ref{lem:B(D)}, 
$\mathcal{B}({\mathcal{D}})$ is a complemented median semilattice.
Thus the join of all atoms exists, and $\mathcal{B}({\mathcal{D}})$ is a Boolean lattice.
By Lemma~\ref{lem:ideal_distributive},
${\mathcal{D}}$ is a distributive lattice. By Proposition~\ref{thm:poly},
$K({\mathcal{D}})$ is a convex polytope, and thus is CAT$(0)$.

Suppose that $\mathcal{D}$ is a (possibly infinite) median semilattice of finite rank $n$.
Then $K(\mathcal{D})$ is a $M^n_0$--polyhedral complex in the sense of \cite{BrHa}.
Notice that $K(\mathcal{D})$ has a finite number of isometry types of cells.
By Bridson's theorem~\cite{BrHa}*{7.19},
$K(\mathcal{D})$ is a (complete) geodesic space. Suppose by way of contradiction
that $(K(\mathcal{D}), D_2)$ is not CAT$(0)$,
and thus it is not uniquely geodesic~\cite{BrHa}*{5.4}. Then two points $x$ and $y$ of $K(\mathcal{D})$
can be joined by two distinct geodesics.
Let $K_0$ denote the set of all simplices of $K(\mathcal{D})$
traversed by those two geodesics. Since any geodesic meets a finite number of simplices~\cite{BrHa}*{7.29},
$K_0$ consists of a finite number of simplices, and hence contains a finite set of vertices.
Let $S_0 \subseteq \mathcal{D}$ denote the set of all such vertices.
Add $0$ to $S_0$.
Consider the convex (gated) hull $\mathcal{D}'$ of $S_0$ in the covering graph of $\mathcal{D}$,
which is a median graph. Then $\mathcal{D}'$ is a gated median subsemilattice of $\mathcal{D}$.
Since $S_0$ is finite, $\mathcal{D}'$ is also finite by Lemma~\ref{lemma:L1_finete_hull}.
Also the rank of any element in $\mathcal{D}'$ is equal to its rank in $\mathcal{D}$.
Therefore, for any chain  of $\mathcal{D}'$,
the corresponding orthoschemes in $K(\mathcal{D}')$ and in $K(\mathcal{D})$ are the same.
So $K(\mathcal{D}')$ can be regarded as a subcomplex of  $K(\mathcal{D})$.
Since $\mathcal{D}'$ is a finite median semilattice,
$(K(\mathcal{D}'), D_2)$ is a CAT$(0)$ space by the first part of the proof.
By construction, the two geodesics between $x$ and $y$
are also geodesics in $K(\mathcal{D}')$. This contradicts
the fact that a CAT$(0)$ space is uniquely geodesic.

\chapter{Orthoscheme Complexes of Swm-Graphs}\label{sec:complex}

Let $G$ be an swm-graph (with uniform edge-length).
Let $K(G)$ denote the orthoscheme complex of
the graded poset ${\mathcal B}(G)$ of
all Boolean-gated sets of $G$ ordered by reverse inclusion,
where each orthoscheme is an orthoscheme of uniform size
and this size is one half of the edge-length of $G$.
The dimension of $K(G)$ is equal to the cube-dimension of $G.$
Orthoscheme complexes of swm-graphs generalize several known CAT$(0)$ complexes, including
median complexes, folder complexes, and Euclidean buildings of type $C_n$.
The objective of this chapter is to investigate the metric properties
of $K(G)$ under $l_1$, $l_2$, and $l_{\infty}$--metrizations.

Let us recall some metric terminologies.
A metric space $(X,d)$ is called {\em modular} if every triplet $x_1,x_2, x_3$ of points in $X$ has a median, i.e., a point $w$ satisfying
$d(x_i, w) + d(w,x_j) = d(x_i, x_j)$ for $1 \leq i < j \leq 3$, and
is called {\it strongly modular} if $(X,d)$  is  a modular space not containing
scale-embedded copies of $K_{3,3}^{-}$.
A metric space $(X,d)$ is said to be {\em (finitely) hyperconvex}
if for every (finite) collection $\mathcal{R}$ of pairs $(x,r)$ of $x \in X$
and $r \geq 0$ satisfying $d(x,y) \leq r + s$ for all $(x,r), (y,s) \in \mathcal{R}$,
there exists $w \in X$ satisfying $d(x,w) \leq r$ for all $(x,r) \in \mathcal{R}$.
\section{Main results}
The most part of this chapter is devoted to the proof of the following theorem:
\begin{Thm}\label{thm:ortho_main}
Let $G$ be an swm-graph of  finite cube-dimension.
Then the orthoscheme complex $K(G)$ of $G$
has the following properties:
\begin{itemize}
\item [(1)] $(K(G), D_1)$ is a strongly modular space;
\item[(2)] $(K(G),D_\infty)$ is finitely hyperconvex, moreover,
if $G$ is locally finite, then $(K(G),D_\infty)$ is hyperconvex;
\item[(3)] $K(G)$ is contractible.
\end{itemize}
\end{Thm}

These three properties are well-known for median complexes (i.e., CAT(0) cube complexes) and
folder complexes~\cites{BaCh,Ch_CAT,vdV}.  In the locally finite case,
the contractibility of $K(G)$ follows from hyperconvexity  and classical results about
hyperconvex spaces.
Indeed, it is known in~\cite{AP56} that hyperconvex metric spaces and injective metric spaces are the same.
Isbell~\cite{Isbell64}*{1.2} showed that any injective space is contractible.
An alternative approach to contractibility of $K(G)$ in the locally finite case is to use the
contractibility of the clique complex of a Helly graph.
Indeed, by Borsuk's Nerve Lemma, $K(G)$ is homotopy equivalent to the clique complex of $G^{\Delta}$,
which by Theorem \ref{thm:Helly} is a Helly graph  and thus  its clique complex is contractible.
Both approaches do not directly extend  to the general non-locally finite case; in the general case,
we will prove contractibility of $K(G)$ by directly constructing a deformation retraction  to a single point.

As we noticed above and also will be shown below,  various CAT$(0)$--complexes can be viewed as orthoscheme complexes of particular swm-graphs.
Therefore, for the $l_2$--metrization of $K(G)$, it seems reasonable to make the following conjecture.
\begin{Conj}\label{conj:CAT(0)}
Let $G$ be an swm-graph of  finite cube-dimension.
Then $(K(G),D_2)$ is a {\rm CAT}$(0)$ space.
\end{Conj}
We present two sufficient criteria for an swm-graph $G$
to have a CAT$(0)$ orthoscheme complex $K(G)$.
Similarly to the local-to-global characterization of CAT(0)
cell complexes via their links, the first one is a local condition.
It relates the CAT$(0)$--property of $K(G)$ to the CAT(0) property
of the orthoscheme complexes of the link posets of its vertices,
which are all complemented modular  semilattices studied in the previous chapter.
For a vertex $x$ of $G$,
the {\em link poset} $\mathcal{L}_x$ at $x$ is the poset of
all Boolean-gated sets containing $x$ and ordered by inclusion.
\begin{Prop}\label{prop:localsemilattice}
Let $G$ be an swm-graph $G$ of finite cube-dimension.
Then $(K(G), D_2)$ is {\rm CAT}$(0)$ if and only if for every vertex $x$ in $G$,
$(K(\mathcal{L}_x), D_2)$ is {\rm CAT}$(0)$.
\end{Prop}
In particular, this shows that
Conjecture~\ref{conj:modular_CAT(0)}
implies Conjecture~\ref{conj:CAT(0)}.

The second criterion is a global condition analogous to Reshetnyak's gluing theorem for CAT(0) spaces \cite{BrHa}; compare to Proposition~\ref{prop:gluing_semilattices}.
Recall that by Proposition \ref{prod-retracts} the class of swm-graphs is closed under Cartesian products
and gated amalgams. We show that for swm-graphs, the CAT$(0)$--property is also closed
under these two operations:
\begin{Prop}\label{prop:g-amalgam}
Let $H$ and $H'$ be two swm-graphs and let $G$ be the Cartesian product or a gated amalgam of $H$ and $H'$.
If $(K(H), D_2)$ and $(K(H'), D_2)$ are  {\rm CAT}$(0)$ spaces,
then $(K(G), D_2)$ is also a {\rm CAT}$(0)$ space.
\end{Prop}
Again the proof
is a combination of Reshetnyak's gluing theorem
and of the following useful property of gated sets:
\begin{Prop}\label{prop:gset=>isosubsp}
Let $G$ be an swm-graph and $H$ a gated subgraph of $G$. Then
for any $p > 0$, $(K(H), D_p)$ is an isometric subspace of $(K(G), D_p)$.
\end{Prop}
This property is  known for median graphs~\cite{ChMa}*{Proposition 1}.

The rest of this chapter is organized as follows.
After discussing  in Section~\ref{subsec:ortho_examples} examples of
orthoscheme complexes of swm-graphs, in Section~\ref{subsec:ortho_pre}
we prove some preliminary results, in particular, Proposition~\ref{prop:gset=>isosubsp}.
In Sections~\ref{subsec:l_1} and  \ref{subsec:l_inf}
we prove the assertions (1) and (2) of Theorem~\ref{thm:ortho_main}.
Propositions~\ref{prop:localsemilattice} and \ref{prop:g-amalgam} are proved in Section \ref{subsec:l_2}.
In the final Section ~\ref{subsec:contractible},
we establish the contractibility of $K(G)$ (Theorem~\ref{thm:ortho_main}~(3)).

\section{Examples}\label{subsec:ortho_examples}
We show here that our construction of the complex $K(G)$ for swm-graphs $G$ encompasses
various important classes of CAT($0$) complexes, thus providing evidences to
Conjecture~\ref{conj:CAT(0)}.
\subsection{Median complexes}
The {\em median complex} $X_{cube}(G)$ of a median graph $G$ is a cube complex obtained
by replacing each cube in $G$ by a unit cube cell of the same dimension.
CAT$(0)$ cube complexes are exactly the median complexes of their $1$--skeleton graphs,
which are known to be  median graphs~\cite{Ch_CAT}.
The cell structure of the orthoscheme complex $K(G)$ is different from that of $X_{cube}(G)$
but $K(G)$ can be regarded as a simplicial subdivision of $X_{cube}(G)$; so $K(G)$ is isometric to $X_{cube}(G)$.
To see this fact,
for each cube $\kappa$ in $X_{cube}(G)$, add a new vertex $v_{\kappa}$
in the midpoint of $\kappa$, and
consider simplices of pairwise incident vertices, where
two vertices $v_{\kappa}$ and $v_{\kappa'}$ are  {\em incident}
iff $\kappa \subseteq \kappa'$ or $\kappa' \subseteq \kappa$.
Then each cube of $X_{cube}(G)$ is triangulated by those simplices.
Notice that each such simplex is an orthoscheme.
As mentioned in Section~\ref{subsec:barycentric}, $\mathcal{B}(G)$ is the poset
of cube subgraphs of $G$ with respect to reverse inclusion.
Therefore $K(G)$ is a simplicial subdivision of $X_{cube}(G)$ and is isometric to $X_{cube}(G)$.
Since a median complex is CAT$(0)$,  so is $(K(G), D_2)$.
It is known that a finite median graph is constructed
by gated amalgams from products of $K_2$.
Hence the CAT$(0)$ property of a finite median complex
follows from Proposition~\ref{prop:g-amalgam}.

\subsection{Folder complexes}
A {\em folder} is a Euclidean cell complex obtained by gluing right isosceles triangles
along the common hypotenuse (longer side); see Figure~\ref{fig:folder}.
A {\em folder complex} is a CAT$(0)$--complex consisting of folders (and edges).
For an swm-graph $G$ of cube-dimension at most $2$,
the orthoscheme complex $K(G)$ is actually a folder complex.
Indeed, each interval of ${\mathcal B}(G)$
is a complemented modular lattice of rank at most $2$.
Each chain of length $2$ corresponds to an incident triplet
of a vertex, an edge, and a quad.
The corresponding orthoscheme is a right isosceles triangle so that
the vertex and the quad are joined by the hypotenuse.
Thus the orthoscheme complex of the interval is a folder and $K(G)$ is a gluing of those folders.
Notice that an interval of rank $2$
induces a maximal biclique (which must be of the form $K_{2,m}$ $(m \geq 2)$) in the frame $G^*$,
and every maximal biclique is obtained in this way.
So $K(G)$ is obtained by replacing each maximal biclique $K_{2,m}$ of $G$
by a folder of $m$ triangles.
It is known \cite{Ch_CAT} that this complex   is a folder complex.
Every link poset $\mathcal{L}_x$ is a modular semilattice of rank $2$
and $K(\mathcal{L}_x)$ is CAT$(0)$ (Proposition~\ref{prop:specialcases}).
Hence the CAT(0)-property of folder complexes
also follows from Proposition~\ref{prop:localsemilattice}.
Figure~\ref{fig:KG} illustrates the orthoscheme complex $K(G)$
of the graph $G$ in Figure~\ref{fig:Gstar}, which is a folder complex.
\begin{figure}[ht]
\begin{center}
\includegraphics[scale=0.3]{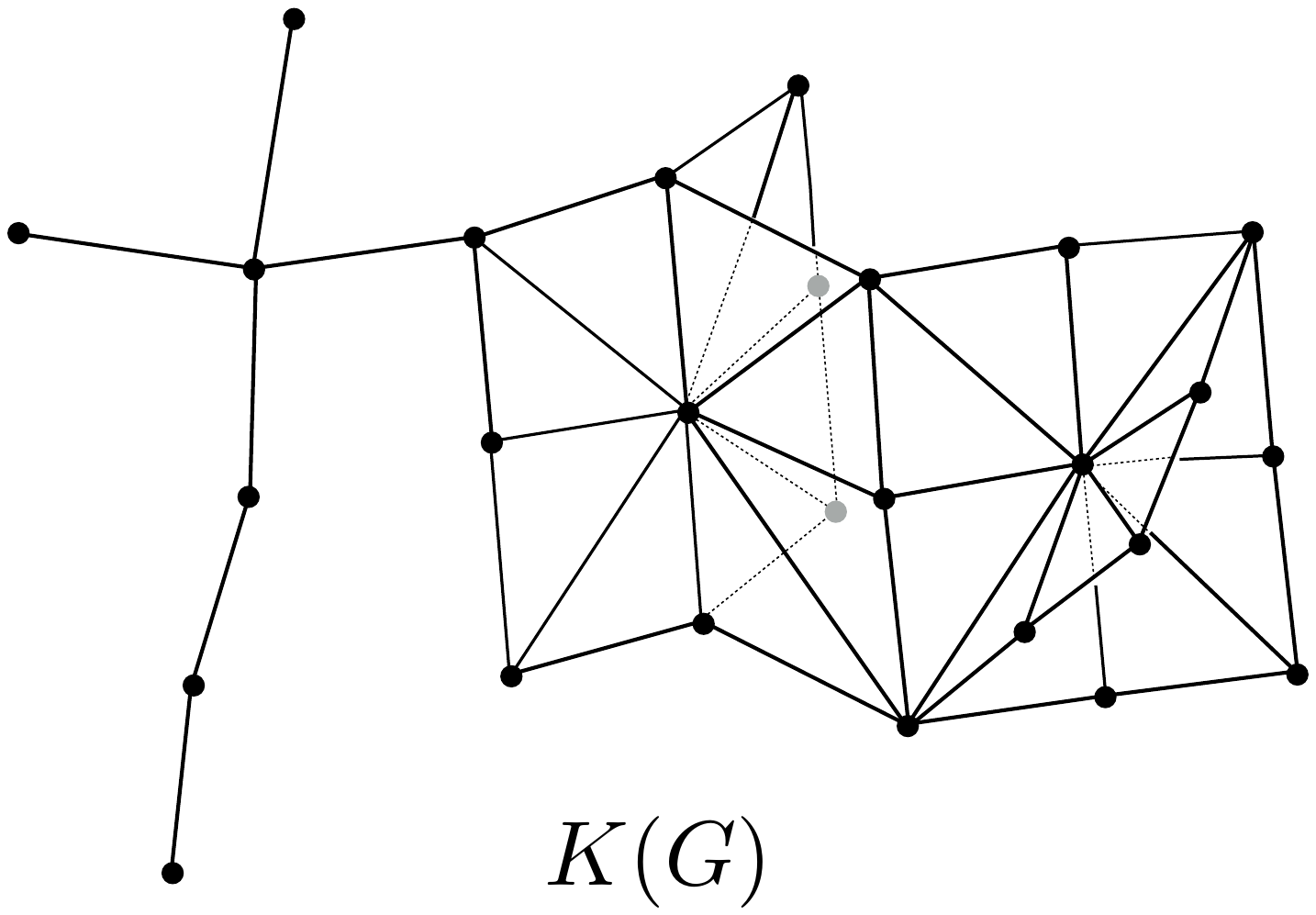}
\end{center}
\caption{The orthoscheme complex of the swm-graph in Figure~\ref{fig:Gstar}}
\label{fig:KG}
\end{figure}

\subsection{Dual polar graphs}
Suppose that $G$ is a dual polar graph.
Then $\mathcal{B}(G)$ is isomorphic to the subspace poset $\mathcal{L}$ of
the polar space corresponding to $G$, hence $K(G)$ is equal to
the orthoscheme complex $K(\mathcal{L})$
of $\mathcal{L}$. We have seen in
Proposition~\ref{prop:specialcases} that this complex is CAT$(0)$.

\subsection{Complemented modular lattices}
Let $\mathcal{L}$ be a complemented modular lattice (of finite rank).
Its covering graph $G$ is an orientable modular graph.
As a corollary of Proposition~\ref{prop:K(G)=K(G,o)},
the orthoscheme complex $K(G)$ of $G$ is a simplicial subdivision of $K({\mathcal{L}})$,
and $K(G)$ and $K(\mathcal{L})$ are isometric.
By Theorem~\ref{thm:HKS}, $K(G)$ is CAT$(0)$.

\subsection{Euclidean buildings of type C$_n$}
Let $\Delta$ be a Euclidean building of type C$_n$; see Section~\ref{subsec:building}.
Then the graph $H(\Delta)$ on the vertices of type 0 is an swm-graph
and $\Delta$ is the order complex
of ${\mathcal B}(H(\Delta))$ (Theorem~\ref{thm:EuclideanBuilding}).
In particular, $K(H(\Delta))$ is a geometric realization of $\Delta$.
For each apartment $\Sigma$,
the subcomplex $K(\Sigma)$ of $K(H(\Delta))$ is isometric
to $(\RR^n, l_{p})$.
Therefore the $l_2$--metrization of $K(H(\Delta))$ is equal to
the standard metrization of the geometric realization of $\Delta$ \cite{BuildingBook}*{Section 11.2},
which is known to be CAT$(0)$~\cite{BuildingBook}*{Theorem 11.16}.
This also follows from
Propositions~\ref{prop:specialcases} and~\ref{prop:localsemilattice} and
the fact that every link poset of the graph $H(\Delta)$ is
the subspace poset of a polar space (Lemma~\ref{lem:filter=polar}).

\section{Preliminary results}
\label{subsec:ortho_pre}
We first consider the case when $G$ is an orientable modular graph.
Let $o$ be an admissible orientation of $G$.
Let $\mathcal{P}(G,o)$ denote the poset on $V(G)$ induced by $o$.
Recall from Proposition~\ref{prop:ideal} that $\mathcal{P}(G,o)$ is graded.
Hence we can consider the orthoscheme complex
$K(\mathcal{P}(G,o))$ of the poset $\mathcal{P}(G,o)$,
where its orthoschemes are standard.
Our interest lies in a  special subcomplex of $K(\mathcal{P}(G,o))$ defined
in the following way. Recall that a pair $(p,q)$ of vertices of $G$ is  $o$--Boolean if
$p \preceq q$ and
$[p,q]$ is a complemented modular lattice (see Lemma~\ref{lem:o-Boolean}).
A chain $p = p_0 \prec p_1 \prec \cdots \prec p_k = q$
is called {\em Boolean} if $(p,q)$ is $o$--Boolean.
Notice that any subchain of a Boolean chain is also Boolean.
Let $K^o(G)$ denote the subcomplex of $K(\mathcal{P}(G,o))$
consisting of all simplices corresponding to Boolean chains.
This complex was already considered in \cite{HH12}.
For an swm-graph $G$, every interval of $\mathcal{B}(G)$ is complemented modular,
and every chain is Boolean. Hence we have:
\begin{Lem}\label{lem:K(G)=K^o*(G^*)}
For an swm-graph $G$, the equality $K(G) = K^{o^{*}}(G^*)$ holds.
\end{Lem}
As expected,
$K(G^*)$ is a subdivision of $K(G)$.
We prove this fact in a slightly more general form
for an orientable modular graph $G$ with admissible orientation $o$.
Recall Lemma~\ref{lem:o-Boolean} that $\mathcal{B}(G)$
consists of intervals $[p,q]$ for all $o$--Boolean pairs $(p,q)$.
Hence $K(G)$ consists of $\sum_{i} \lambda_i [q_i, q'_i]$
for $[q_k, q'_k] \subset  [q_{k-1}, q'_{k-1}] \subset \cdots \subset [q_2,q'_2] \subset [q_1,q'_1]$
with $(q_1, q'_1)$ $o$--Boolean.
\begin{Prop}\label{prop:K(G)=K(G,o)}
Let $G$ be an orientable modular graph with an admissible orientation $o$.
Then $K(G)$ is isometric to $K^o(G)$, where the isometry $\varrho$ is given by
\begin{equation*}
\sum_{i} \lambda_i [q_i, q_i'] \mapsto \sum_{i} \lambda_i (q_i + q'_i)/2.
\end{equation*}
In particular, for an swm-graph $H$,
$K(H^{*i})$ is a simplicial subdivision of $K(H^{*(i-1)})$, and is isometric to $K(H)$.
\end{Prop}
\begin{proof}
Note that the first statement implies the second statement
since $K(H^{*i})$ subdivides $K^{o^*}(H^{*i})$, which is equal to
$K(H^{*(i-1)})$ by the previous lemma.

First we verify that the map $\varrho$ is well-defined.
Indeed, since
$[q_k, q'_k] \subset  [q_{k-1}, q'_{k-1}] \subset \cdots \subset [q_1,q'_1]$,
and $(q_1,q'_1)$ is $o$--Boolean,
the chain $q_1 \preceq q_2 \preceq \cdots \preceq q_k \preceq q'_k
\preceq q'_{k-1} \preceq \cdots \preceq q'_1$ is a Boolean chain.
Now we show that $\varrho$ is bijective.
Pick a point $x = \sum_{i} \mu_i p_i \in K^o(G)$
with Boolean chain $p_1 \prec p_2 \prec \cdots \prec p_m$.
We construct a point $\sum_j \lambda_j [ q_j, q'_j] \in K(G)$
by the following algorithm, where we let $j:=1$ initially.

\medskip
\begin{itemize}
\item[Step 1:]
Let $a := \min \{i: \mu_i  > 0\}$, $b := \max \{i: \mu_i  > 0\}$,
$\alpha := \min (\mu_a, \mu_b)$,  $(q_j, q'_j)  := (p_a, p_b)$, and
\[
\lambda_j :=  \left\{
\begin{array}{ll}
2 \alpha & {\rm if}\ a \neq b, \\
\alpha & {\rm if}\ a = b,
\end{array} \right.  \quad
\mu_i \leftarrow  \left\{
\begin{array}{ll}
\mu_i - \alpha  & {\rm if}\ i = a\ {\rm or}\ b, \\
\mu_i & {\rm otherwise,}
\end{array} \right.  \ (1 \leq i \leq m).
\]
\item[Step 2:]
If $\mu_i = 0$ for all $i$, then stop.
Otherwise $j \leftarrow j+1$,  and go to step 1.
\end{itemize}
\medskip
By construction, it holds that $[q_{k},  q'_{k}] \subset [q_{k-1},q'_{k-1}] \subset  \cdots \subset [q_{2}, q'_{2}] \subset [q_{1}, q'_{1}]$ with $(q_1, q'_1)$ $o$--Boolean
and $\sum_{j=1}^k \lambda_j = \sum_{i=1}^m \mu_i = 1$,
where $k (\leq m)$ is the number of the iterations.
Thus the resulting point $\sum_j \lambda_j [q_j, q'_j]$ belongs to $K(G)$.
Moreover, $\mu_i = \sum_{j:  q_j =  p_i} \lambda_{j}/2 +  \sum_{j:  q'_j =  p_i} \lambda_{j}/2$, and thus
\[
\sum_{j} \lambda_j( q_{j} +  q'_{j})/2 = \sum_{i} \mu_i p_i.
\]
This implies that $\varrho$ is surjective.
Observe that for $x \in K(G)$,
the above algorithm applied to $\varrho(x)$ outputs the original $x$.
This means that the above algorithm
gives the inverse map $\varrho^{-1}$,
i.e., $\varrho^{-1}(\varrho(x)) = x$ for $x \in K(G)$.
Necessarily $\varrho$ is injective, and hence bijective.

To complete the proof, it suffices to show that each maximal simplex $\sigma$ in $K(G)$ is
isometrically embedded in  $K^o(G)$ by $\varrho$.
Suppose that $\sigma$ corresponds to a maximal chain
$[q_0, q'_0] \subset  [q_1, q'_1] \subset \cdots \subset [q_n, q'_n]$ and
the image $\varrho(\sigma)$ belongs to
a maximal simplex $\sigma'$ of $K^o(G)$, where $\sigma'$
corresponds to a maximal Boolean chain $p_0 \prec p_1 \prec \cdots \prec p_n$.
Then $\{q_n,q_{n-1},\ldots,q_0,q'_0,q'_1,\ldots, q'_n\} = \{p_0,p_1,\ldots,p_n\}$ holds.
By maximality, it holds that
 $q_0 = q'_0$,  $(q_n, q'_n) = (p_0,p_n)$,
and either $q_{i-1} = q_i$ and $q'_{i}$ covers $q'_{i-1}$ or
$q'_{i-1} = q'_i$ and $q_{i-1}$ covers $q_{i}$.
Then we can define a map $f: \{1,2,\ldots,n\} \to \{1,2,\ldots,n\}$
so that $q'_i = p_{f(i)}$
if $q_i = q_{i-1}$, and
$q_{i-1} = p_{f(i)}$ if $q'_i = q'_{i-1}$.
Here $f$ is a bijection; $f^{-1}(j)$ is
the minimum index $i$ such that $p_j = q_i'$ if $p_j \succ q_0 = q'_0$
and $p_j \succ q_i$ if $p_j \preceq q_0 = q'_0$.
Then, for each $i=1,2,\ldots,n$ we have
\begin{align*}
& \varphi_{\sigma'} (\varrho([q_i, q'_i])) - \varphi_{\sigma'} (\varrho([q_{i-1}, q'_{i-1}])) = \varphi_{\sigma'} ( (q_i + q'_i)/2 ) - \varphi_{\sigma'} ((q_{i-1} + q'_{i-1})/2) \\
& = \varphi_{\sigma'}(q_i)/2 + \varphi_{\sigma'}(q'_i)/2
- \varphi_{\sigma'} (q_{i-1})/2 - \varphi_{\sigma'} (q'_{i-1})/2 \\
& = \left\{
\begin{array}{ll}
(e_1+ \cdots + e_{f(i)})/2 - (e_1+ \cdots + e_{f(i)-1})/2  =  e_{f(i)}/2  & {\rm if}\  q_i = q_{i-1}, \\
(e_1+ \cdots + e_{f(i)-1})/2 - (e_1+ \cdots + e_{f(i)})/2   = - e_{f(i)}/2 & {\rm if}\ q'_i = q'_{i-1}.
\end{array} \right.
\end{align*}
This means that the image $\varphi_{\sigma'} (\varrho (\sigma))$ is
an orthoscheme of uniform size $1/2$ and
the regularity of the ordering of vertices is preserved.
Therefore the orthoscheme $\varphi_{\sigma'} (\varrho (\sigma))$ is obtained from
$\varphi_{\sigma}(\sigma)$ by the composition $\tau$ of
a translation, transpositions, and sign-changes of coordinates.
Namely it holds
\[
\tau (\varphi_{\sigma} (x)) = \varphi_{\sigma'} (\varrho(x)) \quad (x \in \sigma).
\]
Since $\tau$ is an isometry with respect to any $l_p$--metric, we have
\begin{eqnarray*}
D_p (x,y) &:= & \|  \varphi_{\sigma}(x) - \varphi_{\sigma}(y) \|_p
= \|  \tau (\varphi_{\sigma}(x)) - \tau (\varphi_{\sigma}(y)) \|_p =\\
& = &\|   \varphi_{\sigma'} (\varrho(x)) -   \varphi_{\sigma'} (\varrho(y)) \|_p
 =  D_p(\varrho(x), \varrho(y)) \quad (x,y\in \sigma).
\end{eqnarray*}
\end{proof}

For a gated subgraph $H$ of an swm-graph $G$,
the orthoscheme complex $K(H)$ is a subcomplex of $K(G)$.
Proposition~\ref{prop:gset=>isosubsp} immediately follows
from:
\begin{Prop}\label{prop:gset=>isosubsp'}
Let $G$ be an swm-graph and let $H$ be a gated subgraph of $G$.
For any $p > 0$,
there exists a nonexpansive retraction from $(K(G), D_p)$ to $(K(H), D_p)$.
\end{Prop}
\begin{proof}
Let $X$ be the vertex set of  $H$.
A subset $Y \subseteq X$ is Boolean-gated in $H$ if and only if it is Boolean-gated in $G$.
Let $\mathcal{X}$ denote the set of all Boolean-gated sets contained in $X$.
Then $H^*$ is equal to the subgraph of $G^*$ induced by $\mathcal{X}$.
By Proposition~\ref{prop:gatedsemilattice=>iso} applied to
$K(G) = K(\mathcal{P}(G^*,o^*))$ and $K(H) = K(\mathcal{P}(H^*,o^*))$,
it suffices to show that $\mathcal{X}$ is gated in $G^*$.
We verify the conditions in Lemma~\ref{lem-weakly-modular_gated_hull}.
Take $Y,Z \in \mathcal{X}$.
Let $W$ be a Boolean-gated set adjacent to both $Y$ and $Z$; we assert that  $W \subseteq X$.
This  is obviously so when $W \subseteq Y$ or $W \subseteq Z$.
Suppose that $Y \cup Z \subseteq W$, i.e.,
$W$ is a minimum Boolean-gated set containing $Y \cup Z$.
By Lemma~\ref{lem:Boolean_intersection},
$W \cap X$ is Boolean-gated, and hence $W = W \cap X$ must hold.
\end{proof}

We continue with some  metric properties of $K(G)$.
If $G$ has finite cube-dimension,
then the length of a chain in $\mathcal{B}(G)$ is bounded by
the cube-dimension. Consequently, $K(G)$ has finitely many
isometry types of simplices. Thus, by Bridson's theorem~\cite{BrHa}*{Theorem 7.19}
on metric-polyhedral complexes,
we have:
\begin{Lem}\label{lem:complete}
Let $G$ be an swm-graph of finite cube-dimension.
Then the metric space $(K(G), D_2)$ is a complete geodesic space.
\end{Lem}

Note that the topology of $K(G)$ induced by $D_p$ is independent of $p > 0$.
In particular, $(K(G),D_p)$ is a complete length space.
Since $G^*$ is a subgraph of the $1$--skeleton of $K(G)$,
$G^*$ is embedded into $K(G)$. In particular, the set of $0$--cells of $K(G)$
is equal to $V^* :=V(G^*)$.
More generally, the vertex set $V^{*i}$ of
the $i$--iterated barycentric graph $G^{*i}$ is equal to
the set of $0$--cells of $K(G^{*(i-1)})$, which
can be regarded as a subset of $K(G)$ by Proposition~\ref{prop:K(G)=K(G,o)}.
We next show that $\bigcup_{i\ge 1} V^{*i}$ is a dense subset of $K(G)$.
A sequence $(x^1,x^2, x^3, \ldots)$ of points in $K(G)$ is said to be
{\em digging} if $x^i$ belongs to $V^{*i}$ for all $i$,
and $x^i, x^{i+1}$ belong to a common simplex of $K(G^{*i})$.
\begin{Lem}\label{lem:R/2^i}
Let $G$ be an swm-graph of finite cube-dimension $R$. Then
any digging sequence  $(x^1,x^2, x^3,\ldots )$ is a Cauchy sequence with
\begin{equation}\label{eqn:R/2^i}
D_1(x^{i+1}, x^i) \leq R/2^{i+1} \quad (i=1,2,\ldots),
\end{equation}
and is convergent to some point $x \in K(G)$.
For any point $x \in K(G)$,
there is a digging sequence convergent to $x$.
\end{Lem}
\begin{proof}
Any simplex of $K(G^{*i})$ is an orthoscheme of dimension $\le R$ and size $1/2^{i+1}$.
Hence the $l_1$--diameter of any simplex of $K(G^{*i})$ is bounded by $R/2^{i+1}$,
implying (\ref{eqn:R/2^i}).
Then a digging sequence is Cauchy, and it converges to a point of $K(G)$
by completeness (Lemma~\ref{lem:complete}).

For $x \in K(G)$ and $i=1,2,\ldots$,
define $x^i \in V^{*i} \subseteq K(G)$ in the following way.
Since $K(G^{*(i-1)})$ is a subdivision of $K(G)$ by Proposition~\ref{prop:K(G)=K(G,o)},
we can take a simplex $\sigma$ of $K(G^{*(i-1)})$ containing $x$ in its relative interior.
Define $x^i$ as any vertex of the simplex $\sigma$.
Again, by Proposition~\ref{prop:K(G)=K(G,o)},
the simplex of $K(G^{*i})$ containing $x$ in its relative interior is contained in $\sigma$.
Hence the sequence $(x^1,x^2, \ldots)$ is digging.
Also $D_1(x,x^i) \leq R/2^{i}$ holds, and this sequence converges to $x$.
\end{proof}

Furthermore, the next result shows that the graphs $G$ and $G^{\Delta}$ are isometrically embedded in $(K(G), l_1)$ and  $(K(G), l_{\infty})$, respectively.
\begin{Prop}\label{prop:l_1&1_infty}
Let $G$ be an swm-graph. Then for any vertices $x,y \in V$, we have:
\begin{itemize}
\item[(1)] $d_{G}(x,y) = D_1(x,y)$.
\item[(2)] $d_{G^{\Delta}}(x,y) = D_{\infty}(x,y)$.
\end{itemize}
\end{Prop}
\begin{proof} In view of Lemma~\ref{lem:K(G)=K^o*(G^*)} and taking $G^{*}$ as $G$,
it suffices to prove the statement for $K^o(G)$ of an orientable modular graph $G$
with admissible orientation $o$. An edge of $G$ is an edge of $K^o(G)$ of unit length.
Therefore a path in $G$ can be  regarded as a path in the 1-skeleton of $K^o(G)$ with the same length,
and thus we have $d_{G}(x,y) \geq D_1 (x,y)$.
For a Boolean pair $(x,y)$, we can take an
$o$--Boolean pair $(p,q)$ with $x,y \in [p,q]$ (see Lemma~\ref{lem:o-Boolean}).
Notice that $[p,q]$ is convex (gated) in $G$ (Lemma~\ref{lem:filter_is_convex}), and
$K([p,q])$ is an isometric subcomplex of $K^o(G)$ (Proposition~\ref{prop:gset=>isosubsp}).
Moreover, $[p,q]$ is a complemented modular lattice.
By Lemmata~\ref{lem:Booleansublattice} and \ref{lem:nonincrease1},
we can find an isometric subcomplex $B$ of $K([p,q])$ such that
$B$ is isometric to the cube $([0,1]^n, l_{\infty})$ and contains $x,y$ as its vertices.
Hence, $D_{\infty}(x,y) = 1 = d_{G^{\Delta}}(x,y)$ and consequently, $d_{G^{\Delta}}(x,y) \geq D_{\infty}(x,y)$
holds for arbitrary $x,y \in V$.

Next we will establish the converse inequalities.
Pick a path $P$ in $K^o(G)$ connecting $x$ and $y$.
Consider the subcomplex $K_{x}$ consisting of all $K([p,q])$
over the $o$--Boolean pairs $(p,q)$ with $x \in [p,q]$.
Consider $P \cap K_x$ and let $P'$ be
its connected component  containing $x$.
Since $K_x$ contains all simplices of $K^o(G)$ containing $x$,
$P'$ is a curve of positive length with one end in $x$.
Let $x'$ be the other end of $P'$.
Then $x'$ must belong to the relative boundary of $K_x$.
Take an $o$--Boolean pair $(p,q)$ with $x,x' \in [p,q]$.
Since $[p,q]$ is convex in $G$ (Lemma~\ref{lem:filter_is_convex}),
$K([p,q])$ is an isometric subspace of $K^o(G)$.
We can make $P'$ lie in $K([p,q])$ without increasing its length.
Let $\sigma \subset K([p,q])$ be the minimal simplex containing $x'$.
Suppose that $\sigma$ corresponds to
a Boolean chain from $u$ to $v$.
Then $x \not \in [u,v] \subseteq [p,q]$ must hold.
If $x \in [u,v]$, then
any simplex containing $\sigma$ is also a simplex of $K_x$,
and $x'$ is not located in the relative boundary of $K_x$, a contradiction.

To (1): %Pick a shortest (with respect to the intrinsic  $l_1$--metric) $(x,y)$--path $P$ in $K(G)$.
It suffices to show that $P$ can be modified to a path connecting $x,y$ in $G$ without increasing its $l_1$--length.
As above, we can take
an isometric subcomplex $B$ of $K^o(G)$ containing $x,x'$ so that  $B$ is isometric to an $l_1$--cube $([0,1]^n, l_1)$.
We can assume that the subpath $P'$ of $P$ belongs to $B$.
In the cube $B$,
$x$ is a vertex and $x'$ belongs to a face not containing $x$ (because $x \not \in [u,v]$).
Therefore there is a vertex $z$ in that face
such that $D_1(x,x') = D_1(x,z) + D_1(z,x')$.
In the cube $B$, $x$ and $z$ can be joined by an $l_1$--shortest path $Q$ consisting only of edges of $B$.
Replace the subpath $P'$ (in $P$) by the union of $Q$
and any shortest path between $z$ and $x'$. Then the length of $P$ does not increase.
Consider the subpath of $P$ from $z$ to $y$, and apply the same procedure.
Eventually we will obtain a path between
$x$ and $y$ consisting only of edges of $G$.
%Consequently we have $d_G(x,y) \leq D_1(x,y)$.

To (2):
We show that $P$ can be modified to a path connecting $x,y$ in $G^{\Delta}$ without increasing its $l_\infty$-length.
%Pick a shortest path $P$ between $x$ and $y$ in $K(G)$ (with respect to the $l_{\infty}$--metric) so that
%the dimension of the simplex $\sigma$ defined above is as small as possible (recall that $\sigma$ corresponds to
%a Boolean chain from $u$ to $v$).
Recall the simplex $\sigma$ above.
%We will show that $\sigma$ is a vertex, i.e.,  $x'= u= v$. Since $(x,x')$ is Boolean,  by applying the same
%procedure to $(x',y)$, we will obtain a path of $G^{\Delta}$ having the same length as $P$.
%
Consider the subcomplex $K_{\sigma}$
consisting of all $K([p',q'])$ over the $o$--Boolean pairs $(p',q')$ with $u,v \in [p',q']$.
Let $P''$ be the connected component of $(K_{x} \cup K_{\sigma}) \cap P$ containing $x$.
Then $P''$ is a path  containing $P'$ and thus having one end at $x$.
Let $x''$ denote the other end of $P''$.
Take a minimal interval $[s,t]$
such that $K([s,t])$ contains $x'$ and $x''$.
As above we can find an isometric subcomplex $B$ of $K([s,t])$
containing $x', x''$ (and $s,t,u,v$) and isometric to a cube $([0,1]^n, l_{\infty})$.
Then $x',x'',s,t, u,v$ can be regarded as points in $[0,1]^n$, moreover, $u,v,s,t$ are vertices of this cube.
So $D_{\infty}(x',x'') = \max_{1 \leq i \leq n} |x'_i - x''_i|$.
We can assume without loss of generality that $u$ is the zero vector $(0,0,\ldots,0)$
and $v$ is $1$ in the first $k\ge 0$ coordinates,
i.e., $v = (1,1,\ldots,1,0,\ldots,0)$.
Then $0 < x'_i < 1$ if $1 \leq i \leq k$ and $x'_i = 0$ otherwise.
Let $F$ be the face of cube $B$
consisting of the points $w$ with $w_i = 0$ for $i > k$.
Then $F$ belongs to $K_x$  and every point in $F$ and $x$ are joined by
a path of unit length.
Thus we can modify $P$ so that $x'_i = x''_i$ for $1 \leq i \leq k$ without increasing its $l_{\infty}$-length.
We can further assume that $0 < x'_i = x''_i < 1$ for $1 \leq i \leq k$;
otherwise the dimension of $\sigma$ decreases.
Since $x''$ belongs to the boundary of $B$,
either $x''_j = 0$ or $x''_j=1$ for some $j > k$.
Suppose $x''_j = 0$. Both $x'$ and $x''$ are contained in a proper face $B'$ of $B$.
Then the minimum vertex $s'$ and the maximum vertex $t'$ of $B'$ with respect to $o$
are uniquely determined. Then $[s',t']$ is a proper subinterval of $[s,t]$,
and $K([s',t'])$ contains $B \ni x',x''$. This is a contradiction to the minimality of $[s,t]$.
Thus $x''_j = 1$. Then $D_{\infty}(x',x'') = 1 = D_{\infty}(u, x'')$.
Let $Q$ be the unit length path between $x$ and $u$,
and let $Q'$ be the unit length path between $u$ and $x''$.
Replace the subpath $P''$
in $P$ between $x$ and $x''$ by the union of $Q$ and $Q'$.
Then the $l_{\infty}$--length of $P$ does not increase.
Now $Q$ is an edge of $G^{\Delta}$.
Repeat the same procedure for the subpath of $P$ between $u$ and $y$,
and finally obtain a desired path between $x$ and $y$ consisting only of edges of $G^{\Delta}$.
%Thus we have $d_{G^{\Delta}}(x,y) \leq D_{\infty}(x,y)$.
\end{proof}
Together with Lemma~\ref{lem:R/2^i}, previous result shows that
the metric spaces $(K(G), D_1)$ and $(K(G), D_\infty)$
can be arbitrarily well  approximated by iterated barycentric graphs $G^{*i}$ and
by their thickenings $(G^{*i})^{\Delta}$, respectively.
\section{$l_1$--metrization}\label{subsec:l_1}
We present here the proof of Theorem~\ref{thm:ortho_main}~(1).
By Proposition~\ref{prop:K(G)=K(G,o)}, we can assume that
$G$ is an orientable modular graph.
Let $o$ denote an admissible orientation of $G$. We will show that $(K^o(G), l_1)$ is
a strongly modular space.
Since a median of a triplet of points in general is not unique,
we need a preliminary argument showing how to construct
a sequence convergent to a median of a given triplet of points of $K(G)$.
For a quadruplet $x_1,x_2, x_3, w$ of vertices of $G$, let $M(w; x_1,x_2,x_3)$
be a nonnegative number defined by setting
\[
M(w; x_1,x_2, x_3)  :=  \sum_{1 \leq i < j \leq 3} (d(x_i,w) + d(x_j,w) - d(x_i,x_j))/2.
\]
In particular, $M(w; x_1,x_2,x_3)=0$  if and only if $w$ is a median of $x_1,x_2,x_3$.
\begin{Lem}\label{M(x,y,z)}
For a modular graph $G$ and vertices $x_1,x_2, x_3, w$ of $G$,
there exists a median $m$ of $x_1,x_2,x_3$ with $d(m,w) \leq M(w; x_1,x_2, x_3)$.
In addition, if $w$ is a median of $y_1,y_2,y_3$, then
\[
d(w,m) \leq 2 (d(x_1,y_1) + d(x_2,y_2) + d(x_3,y_3)).
\]
\end{Lem}
\begin{proof} If $w$ is not a median of $x_1,x_2,x_3$, then $d(x_i,w) + d(w,x_j) > d(x_i, x_j)$ for some $x_i,x_j$, $i\ne j$. By
modularity of $G$ we can find a neighbor $w'$ of $w$ with
$d(x_i,w') = d(x_i,w) - 1$ and $d(x_j,w') = d(x_j,w) - 1$. Then $M(w'; x_1,x_2,x_3) \leq M(w; x_1, x_2, x_3) - 1$
and repeating this process for $w',x_1,x_2,x_3$, we will obtain a required median $m$.

Suppose that $w$ is a median of $y_1,y_2,y_3$.
For $i < j$ we have
\begin{eqnarray*}
&& d(x_i, w) + d(w, x_j) - d(x_i,x_j)  \\
&& \leq d(x_i,y_i) + d(y_i,w) + d(w, y_j) + d(y_j, x_j) - d(y_i,y_j) + d(y_i,y_j) - d(x_i,x_j) \\
&& =  d(x_i,y_i)  + d(x_j,y_j) + (d(y_i, y_j) - d(x_i,y_j)) + (d (x_i,y_j) - d(x_i,x_j)) \\
&&  \leq 2 d(x_i,y_i) + 2 d(x_j,y_j),
\end{eqnarray*}
where we use the triangle inequalities and
$d(y_i,w) + d(w,y_j) - d(y_i,y_j) = 0$.
From this we obtain $M(w; x_1,x_2,x_3) \leq 2 (d(x_1,y_1) + d(x_2,y_2) + d(x_3,y_3))$,
yielding the required inequality.
\end{proof}

To show the modularity of $(K(G),D_1)$,
pick an arbitrary triplet $x,y,z$ of points of $K(G)$.
By Lemma~\ref{lem:R/2^i},
we can find three digging sequences $\{x^{i}\}, \{y^i\}, \{z^i\}$ convergent to
$x,y,z$, respectively.
We recursively construct a sequence $m^1,m^2,\ldots,$
that converges to a median of $x,y,z$.
Let $m^1$ be an arbitrary median of $x^1, y^1,z^1$.
Suppose that $m^{i-1}$ has been defined and is a median of $x^{i-1}, y^{i-1},z^{i-1}$ in $G^{*(i-1)}$.
By Lemma \ref{M(x,y,z)} applied to the graph $G^{*i}$,
we can define $m^{i}$ as a median of  $x^{i}, y^{i}, z^{i}$
so that
\[
D_1(m^{i}, m^{i-1}) \leq 2 (D_1(x^{i},x^{i-1}) + D_1(y^{i},y^{i-1}) + D_1(z^{i},z^{i-1}) ) \leq 3 R/{2^{i-1}}.
\]
Then $m^1,m^2,\ldots$ is a Cauchy sequence, and it converges
to a point $m$ of $K(G)$ by completeness (Lemma~\ref{lem:complete}).
Since $M(m^{i}; x^i,y^i,z^i) = 0$ for all $i$ and the distance function $D_1$ is continuous, $M(m;x,y,z) = 0$ must hold.
Hence $m$ is a median of $x,y,z$.

Next we show that $K(G)$ does not contain $K_{3,3}^-$--subspaces.
Suppose by way of contradiction that
there exist a positive $\alpha$ and a six-point subset $U =\{x,y,z,u,v,w\}$ of $K(G)$
such that $(U, D_1)$ is isometric to $(U, \alpha d_{H})$
for a graph $H$ on $U$ isomorphic to $K_{3,3}^-$. Let $\{x,y,z\}$ and $\{u,v,w\}$ be the two color classes of $H$ and
suppose that $z$ and $w$ are not adjacent in $H$.
Take the digging sequences $\{x^i\},\{y^i\},\{z^i\},\{u^i\},\{v^i\},\{w^i\}$ convergent to
$x,y,z,u,v,w$, respectively.
Define $\bar x^i$ as a median of $x^i$, $w^i$, and $z^i$, and
define $\bar y^i$ as a median of $y^i$, $w^i$, and $z^i$.
Then $\bar x^i$ and $\bar y^i$ also
converge to $x$ and $y$, respectively.
Since $G^{*i}$ is strongly modular,
$I(w^{i},z^{i})$ is a modular lattice, and a median of $\bar x^i, \bar y^i, z^i \in I(w^i,z^i)$
is uniquely determined; denote it by $m^i$.
By Lemma \ref{M(x,y,z)},
we have $D_1(m^i,u^i) \leq M(u^i; \bar x^i, \bar y^i, z^i)$ and
$D_1(m^i,v^i) \leq M(v^i; \bar x^i, \bar y^i, z^i)$.
Here $M(u^i; \bar x^i, \bar y^i, z^i)$ converges to $M(u;x,y,z) = 0$;
recall that $u$ is a median of $x,y,z$.
Similarly $M(v^i; \bar x^i, \bar y^i, z^i)$ converges to $M(v;x,y,z) = 0$.
This means that $m^i$ converges to distinct $u$ and $v$, a contradiction.

\section{$l_{\infty}$--metrization}\label{subsec:l_inf}
Here we prove Theorem~\ref{thm:ortho_main}~(2).
Recall that a metric space $(X,d)$ is said to be (finitely) hyperconvex
if for every (finite) set $\mathcal{R} \subseteq X \times \RR_+$ satisfying
\begin{equation}\label{eqn:r+s}
d(x,y) \leq r + s \quad \mbox{ for all } (x,r), (y,s) \in \mathcal{R},
\end{equation}
there exists a point $z$, called a {\em common point},
satisfying $d(x,z) \leq r$ for all $(x,r) \in \mathcal{R}$.
The proof goes along the same lines as in the $l_1$--case.
By taking $G^*$ as $G$, we can assume that $G= (V,E)$ is
an orientable modular graph
with an admissible orientation $o$. We will prove the hyperconvexity of $K^o(G)$.
Let $\mathcal{R}$ be a finite subset of $V \times \ZZ_+$ satisfying (\ref{eqn:r+s}).
Since $G^{\Delta}$ is a finitely Helly graph,
a common point for $\mathcal{R}$ always exists in $V$.
For $z \in V$,
define $M(z; \mathcal{R})$ by setting
\[
M(z; \mathcal{R}) := \sum_{(x,r) \in \mathcal{R}} \max \{ 0, d_{G^{\Delta}}(z,x) - r \}.
\]
Then $z$ is a common point of $\mathcal{R}$
if and only if $M(z; \mathcal{R}) = 0$.
\begin{Lem}
For any $z \in V$,
there is a common point $w$ of $\mathcal{R}$
such that $d_{G^\Delta}(z, w) \leq M(z; \mathcal{R})$.
\end{Lem}
\begin{proof}
Suppose that $d_{G^{\Delta}}(x,z) > r$ for some $(x,r) \in \mathcal{R}$.
Take the $\Delta$--gate $g$ of $x$ at $z$ (see Section~\ref{subsub:delta-gate}).
Then $M(g; \mathcal{R}) \leq M(z; \mathcal{R}) - 1$ holds.
Indeed, for any $(y,s) \in \mathcal{R}$
with $d_{G^{\Delta}}(y,z) \geq s$,
we have $d_{G^{\Delta}}(x,y) \leq r + s < d_{G^{\Delta}}(x,z) + d_{G^{\Delta}}(y,z)$.
Therefore, by the geodesic extension property (Proposition~\ref{prop:gep2}),
the $\Delta$--gate of $y$ at $z$ is adjacent to $g$.
This means that $d_{G^{\Delta}}(y,g) \leq d_{G^{\Delta}}(y,z)$.
Therefore $\max\{0, d_{G^{\Delta}}(z,y) - s\}$ does not increase
for any $(y,s) \in \mathcal{R} \setminus \{(x,r)\}$,
and decreases by one for $(y,s) = (x,r)$.

Replace $z$ by $g$ and repeat this process.
After at most $M(z; \mathcal{R})$ steps,
we obtain a common point $w$, and it  necessarily holds $d_{G^\Delta}(z, w) \leq M(z; \mathcal{R})$.
\end{proof}
Let ${\mathcal R}$ be a finite set of pairs $(x,r) \in K(G) \times \RR_+$
such that $D_{\infty}(x,y) \leq r + s$ for every $(x,r),(y,s) \in \mathcal{R}$.
We are going to construct $z \in K(G)$
satisfying $D_{\infty}(z,x) \leq r$ for $(x,r) \in \mathcal{R}$.
For each $(x,r) \in \mathcal{R}$,
we can take a digging sequence $\{x^{i}\}$ convergent to $x$
with $D_{\infty}(x^i, x) \leq 1/2^{i}$.
Also we define $r^{i}$ as the minimum of $z \in \ZZ/2^{i}$ with $z \geq r+ 1/2^{i}$.
Then $\{r^i\}$ converges to $r$.
For $(x,r), (y,s) \in \mathcal{R}$ we have
\begin{eqnarray*}
d_{(G^{*i})^{\Delta}}(x^{i}, y^{i}) &= & D_{\infty}(x^i,y^i) \leq
D_{\infty}(x^i,x) + D_{\infty}(x,y) + D_{\infty}(y,y^i) \\
& \leq & r + s + 1/2^i + 1/2^{i}  \leq  r^i + s^i.
\end{eqnarray*}
Define $\mathcal{R}^i := \{(x^i,r^i):  (x,r) \in \mathcal{R}\}$.
Since $(G^{*i})^{\Delta}$ is a finitely Helly graph (Theorem~\ref{thm:Helly}),
there is a common point of $\mathcal{R}^i$.
Construct a sequence
$z^1,z^2,\ldots$ as follows.
Define $z^1$ as an arbitrary common point of ${\mathcal R}^1$ in $(G^{*})^\Delta$.
Suppose now that a common point $z^{i-1}$ of ${\mathcal R}^{i-1}$ was already defined.
According to the above lemma,
we can choose a common point $z^{i}$ of ${\mathcal R}^i$ so that
%\begin{eqnarray*}
\begin{align*}
D_{\infty} (z^{i-1}, z^{i}) & \leq  M(z^{i-1}; {\mathcal R}^{i})
= \sum_{(x,r) \in {\mathcal R}} \max \{ 0, D_{\infty}(z^{i-1}, x^{i}) - r^{i} \} \\
& \leq  \sum_{(x,r) \in {\mathcal R}} \max \{ 0, D_{\infty}(z^{i-1}, x^{i-1}) - r^{i-1} + D_{\infty}(x^{i-1},x^i) - r^{i} + r^{i-1} \} \\
& \leq  M(z^{i-1}; {\mathcal R}^{i-1}) +
\sum_{(x,r) \in \mathcal{R}} D_{\infty}(x^{i},  x^{i-1}) + r^{i-1} - r^{i} \\
& \leq  |\mathcal{R}| (1/2^{i} + 1/2^{i-1}) \leq 3|\mathcal{R}|/2^{i},
\end{align*}
%\end{eqnarray*}
where $M(z^{i-1}; {\mathcal R}^{i-1}) = 0$,
$r^{i-1} - r^i \leq 1/2^{i-1}$, and
$D_{\infty}(x^{i},  x^{i-1}) \leq 1/2^i$ (since $x^i,x^{i-1}$ are adjacent in $(G^{*i})^{\Delta}$).
Therefore $z^1,z^2,\ldots$ is a Cauchy sequence, and it converges to $z$ by completeness.
For every $(x,r) \in {\mathcal R}$,
$D_{\infty}(z^i,x^i) - r^i$ is a nonpositive number,
and converges to $D_{\infty}(z,x) - r$, which must be nonpositive.
This means that $z$ is a common point of ${\mathcal R}$.

Suppose now that $G$ is locally finite.
Then every closed ball in $K(G)$ is compact.
For a (possibly infinite)  family $\mathcal{F}$ of
pairwise intersecting balls,
pick any ball $B$ from $\mathcal{F}$.
Then every finite subset of $\mathcal{F}$ containing $B$
has a common point in $B$.
Therefore, by compactness of $B$, $\mathcal{F}$ has a common point,
implying that $K(G)$ is hyperconvex.

\section{$l_2$--metrization}\label{subsec:l_2}
Here we prove Propositions~\ref{prop:localsemilattice} and \ref{prop:g-amalgam}.
First we establish Proposition~\ref{prop:g-amalgam}.
Suppose that an swm-graph $G$ is the Cartesian product of swm-graphs $H$ and  $H'$.
For any Boolean-gated sets $Y$ in $H$ and $Y'$ in $H'$,
the vertex set $Y \times Y' = \{ (x,x') \mid x \in Y, x' \in Y' \}$ in $G$ is Boolean-gated.
Conversely, every Boolean-gated set in $G$ can be represented in this way.
Therefore, the poset $\mathcal{B}(G)$ is isomorphic to the product of $\mathcal{B}(H)$ and $\mathcal{B}(H')$.
As mentioned in Section~\ref{subsec:lattice_main}, $K(G) (= K(\mathcal{B}(G)))$
is isometric to the product $K(H) \times K(H')$ with the product metric.
Hence, by~\cite{BrHa}*{1.15},  $(K(G), D_2)$ is CAT$(0)$. $\Box$

Next suppose that $G$ is
the gated amalgam of swm-graphs $H$ and $H'$
with respect to a common gated subgraph $H_0$ of  $H$ and $H'$.
By Proposition~\ref{prop:gset=>isosubsp},
$(K(H_0),D_2)$ is an isometric subspace of both $(K(H),D_2)$ and $(K(H'),D_2)$.
Since $K(H)$ and $K(H')$ are CAT$(0)$, $K(H_0)$ is a common convex subspace.
Hence $K(G)$ is obtained by gluing $K(H)$ and $K(H')$
along a convex subspace $K(H_0)$.
By Reshetnyak's gluing theorem~\cite{BrHa}*{Theorem 11.1},
$(K(G), D_2)$ is CAT(0).  $\Box$

Next we prove Proposition~\ref{prop:localsemilattice}.
Since $K(G)$ contains the square-complex of the orientable modular graph $G^*$,
$K(G)$ is simply connected.
By Cartan-Hadamard theorem, it suffices to show that $K(G)$
is nonpositively curved around every vertex.
Take  a vertex $X$ of $K(G) = K(\mathcal{B}(G))$,
which corresponds to a Boolean-gated set in $G$.
Here $X$ is regarded as a vertex of $K(G^*)$ since
$K(G^*)$ is a simplicial subdivision of $K(G)$ (Proposition~\ref{prop:K(G)=K(G,o)}).
Consider the star $K:=\mr{St}(X,K(G^*))$ of $X$,
which contains a neighborhood around $X$ in $K(G)$.
$K$ is induced by the ball around $X$ in $(G^{*2})^{\Delta}$ with radius $1/4$.
By Propositions~\ref{gated_ball} and \ref{prop:gset=>isosubsp},
$K$ is an isometric subspace of $K(G)$.
Hence it suffices to show that $K$ is CAT(0).
Since $G^{*2}$ is the covering graph of
the poset of all intervals of $(\mathcal{B}(G), \supseteq)$ ordered by inclusion (Lemma~\ref{lem:G*2}),
$K$ is the orthoscheme complex $K({\mathcal P})$ of
the poset ${\mathcal P}$ of all intervals containing $X$.
So $\mathcal{P}$ is the product of the ideal $(X)^{\downarrow}$
and the filter $(X)^{\uparrow}$.
By Proposition~\ref{prop:gluing_semilattices},
it suffices to verify that both $K((X)^{\downarrow})$ and $K((X)^{\uparrow})$ are CAT$(0)$.
Since every principal filter is isomorphic
to the subspace poset of a polar space (Proposition~\ref{prop:ideal}),
$K((X)^{\uparrow})$ is CAT(0) by Proposition~\ref{prop:specialcases}.
Here $X$ belongs to the link poset ${\mathcal{L}}_y = (\{y\})^{\downarrow}$ for a vertex $y \in V(G)$.
By the assumption, the orthoscheme complex $K(\mathcal{L}_y)$
is CAT$(0)$. Also $(X)^{\downarrow}$ is convex in $\mathcal{L}_y$
(Lemma~\ref{lem:filter_is_convex}).
By Proposition~\ref{prop:gatedsemilattice=>iso},
$K((X)^{\downarrow})$ is an isometric subspace of $K(\mathcal{L}_y)$ and
is CAT$(0)$.  $\Box$

\section{$K(G)$ is contractible}\label{subsec:contractible}
Finally,  we prove Theorem~\ref{thm:ortho_main}~(3).
Let $R$ denote the cube-dimension of $G$.
For $p \in K(G)$ and $r \in \RR_+$,
let $B_r(p)$ denote the $l_{\infty}$--ball with center $p$ and radius $r$:
\begin{equation*}
B_r(p) := \{q \in K(G):  D_{\infty}(p,q) \leq r\}.
\end{equation*}
Recall Proposition~\ref{gated_ball} that $B_r(p)$ induces a gated subgraph in $G$.
\begin{Lem}\label{r-r'}
For $x, p \in V(G)$ and $r,r' \in \ZZ_+$,
let $g$ and $g'$ be the gates of $x$ in $B_r(p)$ and the balls $B_{r'}(p)$ of $G$, respectively.
Then $d(g, g') \leq R |r- r'|$.
\end{Lem}
\begin{proof}
Suppose that $r' = r+1$.
Then $g$ is also the gate of $g'$ at $B_r(p)$,
and hence is the $\Delta$--gate of $p$ at $g'$
(see Section~\ref{subsub:delta-gate}).
In particular, $(g,g')$ is Boolean (by Lemma~\ref{lem:delta-gate}~(1)),
and hence $d(g,g') \leq R$.
Consequently,  for arbitrary $r,r'$
we obtain $d(g,g') \leq R|r - r'|$.
\end{proof}

\begin{Lem}\label{lem:gated_in_K(G)}
For $p \in V(G)$ and $r \in \RR_+$,
the ball $B_r(p)$ is gated in $(K(G), D_1)$.
\end{Lem}
\begin{proof}
Take any $x \in K(G)$.
Take a digging sequence $\{x^i\}$ convergent to $x$.
Define a sequence $\{r^i\}$, where
$r^i$ is defined to be the minimum of $z \in \ZZ_+/2^i$
with $z \geq r$. Then $\{r^i\}$ converges to $r$.
The ball $B_{r^i}(p)$ is gated in $G^{*i}$.
Let $g^i$ be the gate of $x^i$ in $B_{r^i}(p)$.
By Lemma \ref{r-r'} and the fact that
the projection map is nonexpansive, we have
\begin{equation*}\label{eqn:g^ig^j}
D_1 (g^{i}, g^{i+1} ) \leq D_{1} (g^{i}, g') + D_{1} (g', g^{i+1}) \leq D_1(x^{i},x^{i+1}) + R |r^{i} - r^{i+1}| \leq R/2^{i},
\end{equation*}
where $g'$ is the gate of $x^{i+1}$ in $B_{r^i}(p)$.
Therefore $\{g^{i}\}$ is a Cauchy sequence, and thus it converges to a point $g \in K(G)$.
Then $g$ must belong to $B_r(p)$.

We next verify that $g$ is the gate of $x$ in $B_r(p)$.
Pick any $y \in B_r(p)$ and consider a sequence $\{y^i\}$ convergent to $y$
with $y^{i} \in V^{*i} \cap B_{r^i}(p)$. Then we have
\[
D_1(x^i, y^i) = D_1(x^i,g^i) + D_1(g^i, y^i) \quad (i=1,2,\ldots).
\]
Since the distance function $D_1$ is continuous on $K(G) \times K(G)$,
it holds that
\[
D_1(x, y) = D_1(x, g) + D_1(g, y).
\]
This holds for every $y \in B_r(p)$, i.e., $g$ is the gate of $x$,
whence $B_r(p)$ is gated.
\end{proof}
Fix an arbitrary vertex $p \in V(G)$. For $x\in K(G)$, set $r(x):=D_{\infty}(p,x)$.
Let $\phi:K(G) \times [0,1] \to K(G)$ be defined by
\begin{equation}
\phi(x, t) := \mbox{the gate of $x$ in $B_{t  r(x)}(p)$} \quad \mbox{ for } x \in K(G), t \in [0,1].
\end{equation}
Clearly $\phi(x,1) = x$ and $\phi(x,0) = p$ for $x \in K(G)$.
Also we have
\[
 D_1 ( \phi(x, t) ,  \phi(x, t')  ) \leq R |t - t' | \quad \mbox{ for } x \in K(G), t,t' \in [0,1]
\]
by Lemma~\ref{lem:gated_in_K(G)} with
convergent sequences $\{x^i\},\{t^i\},\{(t')^i\}$.
Moreover we have
\begin{eqnarray*}
D_1 ( \phi(x, t) ,  \phi(x', t')  )  & \leq & D_1 ( \phi(x, t) ,  \phi(x', t)  ) + D_1 ( \phi(x', t) ,  \phi(x', t')  )  \\
& \leq & D_1(x,x') + R |t - t'| \quad   \mbox{ for } x,x' \in K(G), t,t' \in [0,1].
\end{eqnarray*}
Thus $\phi$ is continuous, and
is a homotopy map between $K(G)$ and a single point $p$. %$\Box$

\chapter{Metric Properties of Weakly Modular Graphs}\label{s:metric}

In this chapter, we investigate some further metric properties of weakly modular
graphs. First, we show that meshed graphs (thus, in particular, weakly modular graphs) satisfy a quadratic
isoperimetric inequality. Then, we characterize weakly modular graphs
that are $\delta$--hyperbolic by forbidding large metric triangles and
large isometric square grids. We show that any Breadth-First-Search
traversal (BFS) of a weakly modular graph provide isometric subgraphs.
Finally, we propose a notion of ``weakly modular" complex, and we provide several particular examples.

\section{Quadratic isoperimetric inequality}
\label{s:Dehn}

Analogously to weakly modular graphs it can be shown that triangle-square complexes of meshed graphs are simply connected (see also Proposition \ref{p:meshsc}).
Thus each cycle of a meshed graph $G$ has a filling disk diagram in which all faces are triangles or squares. For a cycle $C$ of $G$,
its \emph{minimal area} $\ar (C)$ is the minimum number of triangles and squares
in a filling disk diagram for $C$. As usually, the {\it length} $\ell (C)$ of a cycle $C$ of $G$ is the number of edges of $C$.
In this section, we prove that any meshed graph satisfies a
quadratic isoperimetric inequality or, in other words, has the quadratic Dehn function. (Note however that meshed
graphs fail to have an important nonpositive-curvature-like feature, as explained in Section~\ref{s:locmesh}.)

\begin{theorem}
\label{t:ripe}
In a meshed graph $G$, for any cycle $C$, we have $\ar (C) \leq 2\ell(C)^2$.
\end{theorem}

\begin{proof}
We begin the proof with a lemma.

\begin{lemma}\label{lem-construction-geodesic-meshed} For any triplet of vertices
$u,v,w$  of a meshed graph $G$ such that $v \sim w$, for  any shortest $(u,v)$--path $P$
there exists a shortest $(u,w)$--path $Q$ such
that  $\ar (C) \leq 4k$, where $C$ is the cycle formed by the  paths $P, Q$, and the
edge $vw$, and $k=d(u,v)$.
\end{lemma}

\begin{proof} Let $d(u,w) = \ell$. Let $v'$ be the neighbor of $v$ in $P$ and $P'$ denote the subpath of $P$ between $u$ and $v'$. Since $v$ and
$w$ are adjacent, $|k-\ell|\le 1$. We distinguish three cases depending on the value of $k - \ell$.

%% \medskip\noindent
%% {\bf Case 1.} $d(u,w)> d(u,v)$, i.e., $\ell = k+1$.
  \begin{case-ar}
    $d(u,w)> d(u,v)$, i.e., $\ell = k+1$.
  \end{case-ar}

Then, as $Q$ we can take $P$ followed by the edge $vw$.  Then
$\ar(C) = 0 \leq 4k$ and we are done.

%% \medskip\noindent
%%     {\bf Case 2.} $d(u,w)= d(u,v)$, i.e., $\ell = k$. In this case, we prove that $\ar(C) \leq 2k$.

\begin{case-ar}
  $d(u,w)= d(u,v)$, i.e., $\ell = k$.
\end{case-ar}

In this case, we prove that $\ar(C) \leq 2k$.  We proceed by induction
on $k$. If $w$ is adjacent to $v'$ (in particular, if $k=1$), then as
$Q$ we can take the path $P'$ followed by the edge $v'v$; then $\ar(C)
= 1 \leq k$.

Assume now that $v' \nsim w$.  Since $d(u,w) = d(u,v')+1$ and
$d(v',w) = 2$, by meshedness of $G$, there exists a vertex $w'\sim v', w$ such that
$d(u,w') = d(u,v') = k-1$. By induction hypothesis applied to $P'$, there exists a
shortest $(u,w')$--path $Q'$ such that
$\ar(C') \leq 2(k-1)$, where $C'$ is the cycle formed by the paths $P',Q'$ and the edge $v'w'$.
Let $Q$ be the shortest $(u,w)$--path consisting of $Q'$ followed by the edge $w'w$.
Since $\ar(w,v,v',w') \leq 2$, we have
$\ar(C) \leq \ar(C') +2 \leq 2k$.

%% \medskip\noindent
%% {\bf Case 3.} $d(u,w)< d(u,v)$, i.e., $\ell = k-1$.

\begin{case-ar}
 $d(u,w)< d(u,v)$, i.e., $\ell = k-1$.
\end{case-ar}

Again we proceed by induction on $k$ and prove the existence of a  shortest $(u,w)$--path $Q$ such
that $\ar(C) \leq 4k$. If $k = 1$, or, more
generally, if $w = v'$, then we simply pick $P'$ as $Q$; then $\ar(C) = 0 \leq 4k$.
If $v' \sim w$,  then by  Case 2 applied to $v'$ and $w$, there
exists a shortest $(u,w)$--path $Q'$
such that $\ar(C') \leq
2\ell = 2(k-1)$, where $C'$ is the cycle formed by the paths $P',Q'$ and the edge $v'w$. Let $Q'$ be $Q$.
Since $\ar(w,v,v') = 1$, we have
$\ar(C) \leq \ar(C') +1 < 2k$.

Assume now that $w \nsim v'$ and $w \neq v'$. Since $d(w,v') = 2$
and $d(u,v') = d(u,w) = k-1$, by meshedness, there exists $z'
\sim v',w$ such that $k-2 \leq d(u,z') \leq k-1$.

Suppose first that $d(u,z') = k-2$. By induction hypothesis for $P'$, there
exists a shortest $(u,z')$--path $Q'$ such
that $\ar(C') \leq
4(k-1)$, where $C'$ is the cycle formed by the paths $P',Q'$ and the edge $v'z'$.
Then let $Q$ be the shortest $(u,w)$--path obtained from $Q'$ by adding the edge $z'w$.
Since $\ar(w,v,v_1,z_1) =1$, we have $\ar(C) \leq
\ar(C') +1 < 4k$.

Suppose now that $d(u,z') = k-1$. Note that $z' \sim v'$ and that
$d(u,z') = d(u,v') = k-1$. By Case 2, there exists a shortest $(u,z')$--path
$Z$ such that $\ar(C') \leq 2(k-1)$, where $C'$ is the cycle formed by the paths $P',Z$ and the edge $v'z'$.  Since $z' \sim w$ and
$d(u,z') = d(u,w) = k-1$, by Case 2 there exists a shortest $(u,w)$--path
$Q'$ such that $\ar(C'') \leq 2(k-1)$, where $C''$ is the cycle formed by the paths $Z,Q'$ and the edge $z'w$. Letting $Q$ be the shortest $(u,w)$--path
$Q'$, since $\ar(v,v_1,z_1,w_1) \leq 2$, we conclude
$\ar(C) \leq\ar(C')+\ar(C'') + 2 \leq 4(k-1) +2<4k$.
\end{proof}
\medskip

Consider a cycle $C =(v_0,v_1,v_2,\ldots,v_{n-1},v_0)$ of length $n$ of a meshed graph $G$. For each $0
\leq i \leq n-1$, we will define a shortest path $P_i$ from $v_0$
to $v_i$ such that $\ar(C_i)\leq
4d(v_0,v_i)$, where the cycle $C_i$ is the concatenation of the paths  $P_i, P_{i+1}$ and the edge $v_iv_{i+1}$.
Let $P_0$ be the one-vertex path $(v_0)$.  Assume we have
constructed $P_{i-1}$. By
Lemma~\ref{lem-construction-geodesic-meshed}, there exists a shortest
path $P_{i+1}$ from $v_0$ to $v_i$ such that
$\ar(C_i) \leq 4d(v_0,v_i)$. Since  $d(v_0,v_i)\le \frac{n}{2}$ for each $i$, we obtain that $\ar(C_i)\le 2n$.
Since one can fill $C$ using the collection of $n$ cycles $C_0,\ldots,C_{n-1}$, each satisfying the inequality $\ar(C_i)\le 2n$, we conclude that
$\ar(C)\le 2n^2=2\ell(C)$.  This finishes the proof of the theorem.
\end{proof}

The proof of the previous theorem also shows  the following:

\begin{proposition}
\label{p:meshsc}
The triangle-square complex of a meshed graph is simply connected.
\end{proposition}

\section{Hyperbolicity}
Gromov  showed that if the balls of an appropriate fixed radius of a geodesic metric space
are hyperbolic for a given $\delta$, then the whole space is $\delta '$--hyperbolic for some
$\delta '>\delta $ (see \cite{Gr}, 6.6.F). More precisely, we have the following local-to-global
characterization of hyperbolicity:

\begin{theorem}\cite{Gr}\label{Gromov}  Given $\delta >0$, let  $(X,d)$ be a simply connected geodesic metric space in which each loop of length
$<100\delta$ is null-homotopic inside a ball of diameter $<200\delta$. If every ball $B_{10^5\delta}(x_0)$ of $X$
is $\delta$--hyperbolic, then $X$ is $200\delta$--hyperbolic.
\end{theorem}

For weakly modular graphs, this result have been significantly sharpened as follows:
\begin{proposition}\cite{CCPP-hyperbolic}\label{prop-CCPP-hyp-wm}
If $G$ is a weakly-modular graph such that every ball
$B_{10\delta+5}(v,G)$ of $G$ is $\delta$--hyperbolic, then $G$ is
$(736\delta+368)$--hyperbolic.
\end{proposition}

In this section, we refine this result in the following way: we show
that a weakly modular graph $G$ is hyperbolic if and only if $G$ does
not contain arbitrarily large metric triangles or arbitrarily large
isometric square grids:

\begin{theorem}
\label{t:hyp01}
For a weakly modular graph $G$ the following are equivalent:
\begin{itemize}
\item[(i)] there exists $\delta$ such that $G$ is $\delta$--hyperbolic;
\item[(ii)]  there exist $\mu, \kappa$ such that the metric triangles of $G$
  have sides of length at most $\mu$ and $G$ does not contain
  isometric square grids of side $\kappa$.
\end{itemize}
\end{theorem}

\begin{proof}
We first show that if $G$ is $\delta$--hyperbolic, then one can bound
the size of the isometric square grids appearing in $G$.

\begin{lemma}
\label{l:hyp02}
If $G$ is a $\delta$--hyperbolic graph $G$, then every isometric square
grid of $G$ is of side at most $\delta$.
\end{lemma}

\begin{proof}
Consider an isometric square grid $H$ of $G$ of side $k$ and let $u,v,x$,
and $y$ be the corners of $H$ such that $d(u,v) = d(v,x) = d(x,y) =
d(y,u) = k$.  Observe that $d(u,v)+d(x,y) = d(u,y)+d(v,x) = 2k$ and
that $d(u,x)+d(v,y) = 4k$. Since $G$ is $\delta$--hyperbolic, $2k =
d(u,x)+d(v,y) - \max(d(u,v)+d(x,y),d(u,y)+d(v,x)) \leq 2\delta$ and
thus $k \leq \delta$.
\end{proof}

Using Lemma~\ref{lem-weakly-modular}, we now show that if $G$ is
$\delta$--hyperbolic, then one can bound the size of the sides of its metric
triangles. While the previous result holds for every
$\delta$--hyperbolic graph, the following lemma needs
weak modularity.

\begin{lemma}
\label{l:hyp04}
If $G$ is a  $\delta$--hyperbolic weakly modular graph, then
every metric triangle of $G$ is of side at most $4\delta$.
\end{lemma}

\begin{proof}
Consider a metric triangle $uvw$ of $G$ of side $k$. Recall that by
Lemma~\ref{lem-weakly-modular}, we have $d(u,v) = d(u,w) = d(v,w)=k$. Let $x
\in I(u,v)$ such that $d(x,u) = \lfloor\frac{k}{2}\rfloor$ and $d(x,v)
= \lceil\frac{k}{2}\rceil$.  Since $x \in I(u,v)$, by
Lemma~\ref{lem-weakly-modular}, $d(w,x) = k$. Consequently, if we
consider the four points $u,v,w,x$, we get $d(u,v)+d(w,x) = 2k$,
$d(u,w) + d(v,x) = \lceil\frac{3k}{2}\rceil$ and $d(v,w) + d(u,x) =
\lfloor\frac{3k}{2}\rfloor$. Since $G$ is $\delta$--hyperbolic,
$2k-\lceil\frac{3k}{2}\rceil\le 2k - \lfloor\frac{3k}{2}\rfloor \leq
2\delta$ and consequently, $k \leq 4\delta$.
\end{proof}

Lemma~\ref{l:hyp02} and Lemma~\ref{l:hyp04} establish the implication $(1) \Rightarrow (2)$ in
Theorem~\ref{t:hyp01}. The reverse implication follows from Proposition~\ref{p:hyp06} below.
In order to prove it, we use the following result. (Recall that the intervals of a graph $G$ are $\nu$--{\it thin} \cite{Papa}
if for any two vertices $u,v$ of $G$ and any two vertices $x,y\in I(u,v)$ such that $d(u,x)=d(u,y)$
and $d(v,x)=d(v,y)$, we have  $d(x,y)\le \nu$.)

\begin{proposition} \cite{ChDrEsHaVa} \label{mu-nu} If $G$ is a graph
  in which all intervals are $\nu$--thin and the metric triangles of
  $G$ have sides of length at most $\mu$, then $G$ is
  $(16\nu+4\mu)$--hyperbolic.
\end{proposition}

\begin{proposition} \label{prop-no-big-interval}
\label{p:hyp06} Let $G$ be  a weakly modular graph  in which any metric triangle is of side
at most $\mu$ and such that any isometric
 square grid contained in $G$ is of side at most
 $\kappa$. Then all intervals of $G$ are $(2\kappa+\mu)$--thin, and $G$ is
 $(32\kappa+20\mu)$--hyperbolic.
\end{proposition}

\begin{proof}
 Let $I(u,v)$ be an interval of $G$ and $x,y\in I(u,v)$ such that
 $d(u,x)=d(u,y)=k, d(v,x)=d(v,y)=l$, that is, $l+k=d(u,v).$ Let $u'x'y'$ be
 a quasi-median of the triplet $u,x,y$ and let $v''x''y''$ be a
 quasi-median of the triplet $v,x,y$ (see
 Figure~\ref{fig-grid-interval}). Since $G$ is weakly modular, from
 Lemma~\ref{lem-weakly-modular}, $d(u',x') = d(u',y')=d(x',y')$. Let
 $k':=d(u',x')$ and let $a: = k-k'-d(u,u') = d(x',x) = d(y',y)$.
 Analogously, let $l':=d(v'',x'')=d(v'',y'')=d(x'',y'')$ and
 $b:=l-l'-d(v,v'') = d(x,x'') = d(y,y'')$. Since the metric triangles of
 $G$ are $\mu$--bounded, $k' \leq \mu$ and $l' \leq \mu$.

 Without loss of generality, assume that $a \leq b$. If $a \leq
 \kappa$, then $d(x,y) = 2a+k' \leq 2\kappa+\mu$ and we are
 done. Assume now that $b \geq a > \kappa$. Let $P'=(x = x_{0,0},
 x_{1,0}, \ldots, x_{a,0} = x')$ be a shortest path from $x$ to $x'$,
 and let $P''=(x = x_{0,0}, x_{0,1}, \ldots, x_{0,b} = x'')$ be a shortest
 path from $x$ to $x''$. We want to show that there exists a set of
 vertices $S=\{ x_{i,j}: 0\leq i,j \leq a\}$ such that $S$ induces an
 isometric subgrid of $G$ of size $a >
 \kappa$, contradicting the hypothesis of the proposition.

 \begin{figure}[ht]
\begin{center}
\includegraphics[scale=0.6]{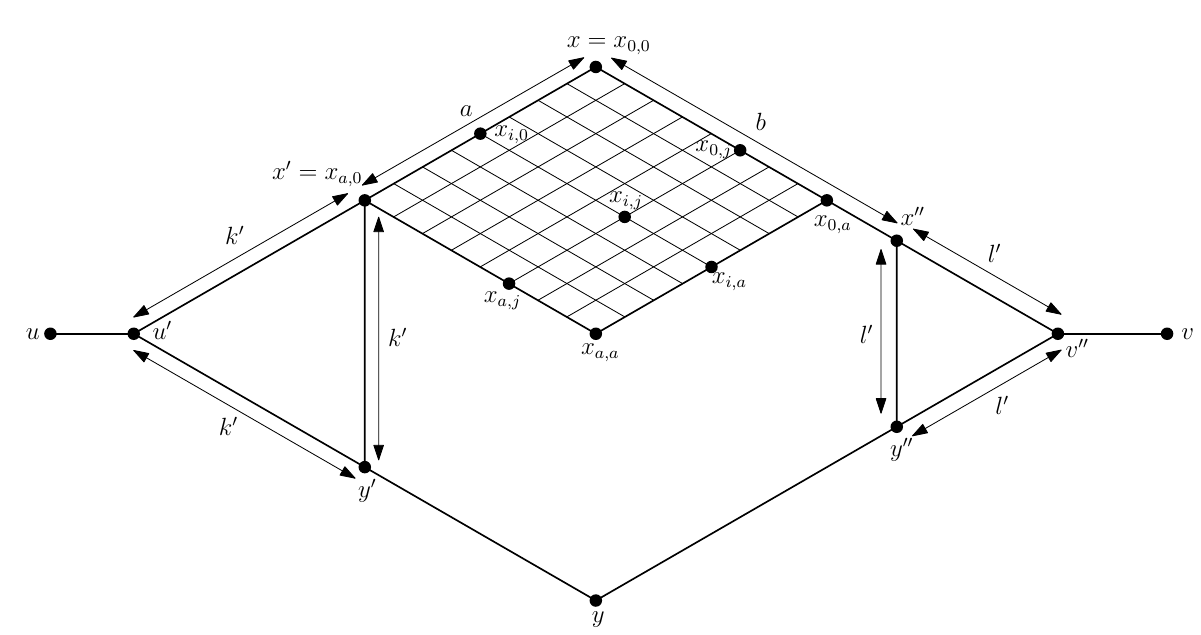}
\end{center}
\caption{The construction from the proof of
  Proposition~\ref{prop-no-big-interval}.}\label{fig-grid-interval}
\end{figure}

\begin{claim}
There exists a set of vertices $S=\{ x_{i,j}: 0\leq i,j \leq a\}$ such
that for every $i,j$, $x_{i,j} \sim x_{i-1,j}$ (if $i \geq 1$),
$x_{i,j} \sim x_{i,j-1}$ (if $j \geq 1$), $d(x_{i,j},u) = k - i +j$,
$d(x_{i,j},v) = l - j +i$, $d(x_{i,j},x) = i+j$ and $d(x_{i,j},y) =
d(x,y) - i-j$.
\end{claim}

\begin{proof}[Proof of the Claim]
In order to prove the claim, we iteratively define the vertices
$x_{i,j}$ $1 \leq
i,j\leq a$, in such a way that, when the vertex $x_{i,j}$ is defined, all vertices
$x_{i',j'}$ with $i'+j' < i+j$ have been already defined.  Note that
for $i = 0$ or $j=0$, we already
defined vertices $x_{i,j}$ that satisfy the conditions (these are the vertices
on the shortest paths $P'$ and $P''$, respectively).
Suppose now
that $x_{i-1,j-1}, x_{i,j-1}$ and $x_{i-1,j}$ have already been
defined and satisfy the claim. Note that $d(x,y)-i-j+1 =
d(x_{i,j-1},y) = d(x_{i-1,j},y) = d(x_{i-1,j-1},y) -1$ and that
$x_{i-1,j-1} \sim x_{i,j-1}, x_{i-1,j}$. By QC($y$) applied to
$x_{i-1,j-1}, x_{i,j-1}, x_{i-1,j}$, there exists $x_{i,j} \sim
x_{i,j-1}, x_{i-1,j}$ such that $d(x_{i,j},y) = d(x_{i-1,j},y) -1 =
d(x,y) -i -j$. Since $x_{i,j}\sim x_{i-1,j}$, $d(x_{i,j},x) \leq i+j$,
and since $d(x,y) \leq d(x,x_{i,j}) +d(x_{i,j},y) = d(x_{i,j},x) +
d(x,y) -i -j$, necessarily $d(x_{i,j},x) = i+j$. Since $d(x_{i-1,j},u)
= k-i+1+j$, $d(x_{i,j-1},u) = k-i+j-1$, and since $x_{i,j} \sim
x_{i-1,j}, x_{i,j-1}$, necessarily $d(x_{i,j},u) = k-i+j$. Using the
same arguments, one can show that $d(x_{i,j},v) = k-j+i$ .
\end{proof}

We now show that $S$ induces an
isometric subgrid of $G$.  Consider two arbitrary vertices $x_{i,j},
x_{i',j'}$ of $S$. By construction, $d(x_{i,j}, x_{i',j'}) \leq
|i-i'|+|j-j'|$. We will show that the equality holds. Suppose
that $d(x_{i,j}, x_{i',j'}) <
|i-i'|+|j-j'|$. Without loss of generality, let $i \leq
i'$. If $j \leq j'$, then $d(x,y) - i - j = d(x_{i,j},y) \leq
d(x_{i,j},x_{i',j'}) + d(x_{i',j'},y) < i'-i +j'-j+ d(x,y) - i' - j' =
d(x,y) - i - j$, a contradiction. If $j > j'$, then $k - i + j =
d(x_{i,j},u) \leq d(x_{i,j},x_{i',j'}) + d(x_{i',j'},u) < i'-i +j-j'+
k - i' +j' = k -i +j$, a contradiction. Consequently, $S$ induces
an isometric square grid of $G$ of size $a > \kappa$. This ends the proof
of the first assertion of the proposition. The second assertion
is an immediate consequence of
Proposition~\ref{mu-nu}.
\end{proof}

This finishes the proof of Theorem \ref{t:hyp01}.
\end{proof}

%===================================================================

%===================================================================

\section{BFS gives a distance-preserving ordering}

For several subclasses of weakly modular graphs, BFS
(Breadth-First-Search) and its specifications turn out to provide
orderings with interesting and strong properties, which can be used, for instance,
to prove contractibility of associated clique complexes. First, it was shown
in~\cite{Ch_bridged} that for locally finite  bridged graphs, any BFS ordering is a
dismantling ordering, showing in particular that the clique complexes of bridged
graphs are contractible.  Polat \cite{Po-infbrid} proved that arbitrary
connected graphs (even if they are not locally finite) admit a BFS ordering and,
extending the result of~\cite{Ch_bridged}, he showed that  BFS provides a dismantling
order for non-locally-finite bridged
graphs.  For weakly bridged graphs the same kind of results has been obtained
for specific BFS orderings. Namely, any LexBFS ordering of a locally
finite weakly bridged graph provides a dismantling
ordering~\cite{ChOs}. In the case of non-locally-finite graphs, it is
not always possible to define a LexBFS ordering. However, for graphs
without infinite cliques, it was shown in~\cite{BCC+} that it is always
possible to define an ordering, intermediate between BFS and LexBFS, and called SimpLexBFS,
and it was shown that for weakly bridged graphs, any SimpLexBFS ordering is a
dismantling ordering. Notice also that the contractibility of Kakimizu complexes was
established by defining a BFS-like orderings of their vertices; for details,
see Section 5 of \cite{PrzSch}.

Weakly modular graphs are in general not dismantlable. However, we will show that BFS orderings $\prec$  of arbitrary
weakly modular graphs $G$ are distance-preserving. Recall that a well-ordering $\prec$ of the vertices of a
graph $G$ is \emph{distance-preserving} \cite{Ch_dism} if for each
$v$, the subgraph $G_v$ of $G$ induced by the level set  $L_v:=\{u: u \preceq v\}$ is an isometric subgraph of $G$.
The following result was conjectured in \cite{Ch_dism} (and was proven in
that paper for house-free weakly modular graphs):

\begin{theorem}
\label{t:dpBFS}
For a  weakly modular graph $G$, any BFS ordering of
its vertices  is a distance-preserving ordering.
\end{theorem}

\begin{proof}  Following Polat \cite{Po-infbrid}, a well-order $\preceq$ on the vertex set $V(G)$ of a graph $G$ is called a {\it BFS order} if there exists a family
$\{ A_x: x\in V(G)\}$ of subsets of $V(G)$ such that, for
every vertex $x\in V(G)$,
\begin{enumerate}[(S1)]
\item $x\in A_x$;
\item if $x\preceq y$, then $A_x$ is an initial segment of $A_y$;
\item $A_x=A_{(x)}\cup N(x)$, where $A_{(x)}:=\{ x\}$ if $x$ is the least element of $(V(G),\preceq)$ and $A_{(x)}:=\bigcup_{y\prec x} A_y$ otherwise.
\end{enumerate}

\begin{lemma} \label{l:Polat_BFS} \cite{Po-infbrid}*{Lemma 3.6} There exists a BFS order on the vertex set of any connected graph.
\end{lemma}

The vertex $x$ will be called the {\it parent} of each vertex of $A_x\setminus A_{(x)}$. We will denote by $f$ the map from $V(G)$ to $V(G)$
such that $f(v)$ is the parent of $v$, for every $v\in V(G)$. The least element of $(V(G),\preceq)$ will be called the {\it base-point}
and will be denoted by $v_0$ (by convention, we set $f(v_0)=v_0)$). Notice that like in case of finite graphs, for every vertices $x$ and $y$ of $G$,
$x\preceq y$ implies $d(v_0,x)\le d(v_0,y)$, and $d(v_0,x)<d(x_0,y)$ implies $x\prec y$. In particular, $d(v_0,x)=d(v_0,f(x))+1$.
For two distinct vertices $x$ and $y$ of $G$, we set $\max\{ x,y\}=x$ if $y\prec x$ and $\max\{ x,y\}=y$ if $x\prec y$.

\begin{lemma} \label{local-distance-pres} Let $\preceq$ be a well-order on the vertex set  of a graph $G$ such that
for any two vertices $u,w$ with $d(u,w)=2$ there exists a vertex $v \in I(u,w)\setminus \{ u,w\}$
such that $v\prec \max\{ u,w\}$. Then $\prec$ is a distance-preserving ordering of the vertices of $G$.
\end{lemma}

\begin{proof} To prove that $\prec$ is  a distance-preserving ordering,
pick any two vertices $u, w$ with $u \prec w$ and let $k:=d(u,w)$. It suffices to show
that the vertices $u$ and $w$ can be connected in the subgraph  $G_w$ by a
shortest $(u,w)$--path. For  a shortest $(u,w)$--path $P$, denote by $m(P)$ the
maximum vertex of $P$ in the well-order $\prec$.  Let $P^*$ be a shortest $(u,w)$--path such that
 $m(P^*)$ is minimum, i.e., for any other shortest $(u,w)$--path
 $P$, we have $m(P^*) \preceq m(P)$. Such  $P^*$ exists because $\prec$ is a
 well-order. Note that $u \prec w \prec
 m(P^*)$, otherwise $P^*$ is the required shortest path in $G_w$. Set $x:=m(P^*)$ and
 denote by $y$ and $z$ the neighbors of $x$ in $P^*$. Suppose without loss of generality that $y\prec z$.
 Since $d(y,z)=2$ and $y,z\prec x$,   by our assumption there exists $x'\sim y,z$, $x'\ne x$,  with $x'\prec z\prec x$.
 Consider the path $P'$ obtained by replacing in $P^*$ the vertex $x$ by $x'$; $P'$ is a shortest $(u,w)$--path and
 $m(P')\prec m(P^*)$, contradicting our choice of $P^*$. This shows that $P^*$ is indeed a path of $G_w$.
\end{proof}

Assume that a BFS ordering $\prec$ with base-point $v_0$ of a weakly modular graph $G$
is not a distance-preserving ordering. By Lemma \ref{local-distance-pres} we can find two vertices
$u,w$ with $d(u,w)=2$ such that for any $v\sim u,w$ we have $\max\{ u, w\} \prec v$.  Suppose
without loss of generality
that $u \prec w$.

First suppose that $d(u,v_0) = d(w,v_0)-1 = k-1$. Let $v \sim
u,w$; due to our choice of $u,w$, we have $d(v,v_0) = k$. By TC($v_0$),
there exists $s \sim v,w$ such that $d(s,v_0) = k-1$; consequently, $s
\prec w$ and thus $s \nsim u$. By QC($v_0$), there exists $t \sim u,s$
such that $d(t,v_0) = k-2$. Observe that $d(u,w) = d(t,w) = 2$. By
TC($w$), there exists $x \sim t,u,w$. Since $d(x,v_0) = k-1$, $x \prec
w$ and we get a contradiction with our choice of the pair $u,w$.

Assume now that $d(u,v_0) = d(w,v_0) = k$. For any vertex $v$, we
denote by $f^i(v)$ the vertex obtained by applying iteratively $i$
times the function $f$ on $v$ (in particular, $f^1(v)=f(v)$ is the parent of $v$).
The proof of this case is based on the following
lemma.

\begin{lemma} \label{iteration}
For every $0 \leq i \leq k$, the following properties hold:
\begin{itemize}
\item[(i)] $d(f^i(u),w) = i+2$;
\item[(ii)] there is no  $x \sim f^i(u)$ such that $d(x,w)=i+1$ and $x \prec f^i(w)$.
\end{itemize}
\end{lemma}

\begin{proof}
Since $u \prec w$, by the rules of BFS, for every $i$, if $f^i(u) \neq
f^i(w)$, then $f^i(u) \prec f^i(w)$.  We establish the assertions of the lemma
by induction on $i$. For $i = 0$, we have $f^0(u) = u, f^0(w) = w$ and the lemma states
that $d(u,w) = 2$ and that there is no vertex $x \sim u,w$ such that
$x \prec w$, which is true by our assumptions.

Suppose the property holds for $i<k$ and consider $f^{i+1}(u)$. Since
$d(f^i(u),w) = i+2$, we have $i+1 \leq d(f^{i+1}(u),w) \leq i+3$.  If
$d(f^{i+1}(u),w) = i+1$, then $x=f^{i+1}(u)$ satisfies
$d(x,w) = i+1$ and $x \prec f^i(u) \prec f^i(w)$, contradicting the
induction hypothesis. Consequently, $d(f^{i+1}(u),w) > i+1 =
d(f^{i+1}(w),w)$, and thus $f^{i+1}(u)\ne f^{i+1}(w)$. Therefore,
$f^{i+1}(u) \prec f^{i+1}(w)$.  If $d(f^{i+1}(u),w) = i+2$ , then by TC($w$), there exists $x \sim
f^i(u),f^{i+1}(u)$ such that $d(x,w) = i+1$. Since $f(x) \preceq
f^{i+1}(u) \prec f^{i+1}(w)$, we obtain $x \prec f^i(w)$, contradicting the
induction hypothesis.  Consequently, $d(f^{i+1}(u),w)=i+3$, as required.

Suppose now that there exists $x_{i+1} \sim f^{i+1}(u)$ such that
$d(x_{i+1},w)=i+2$ and $x_{i+1} \prec f^{i+1}(w)$. Then $f^i(u),x_{i+1}\in I(f^{i+1}(u),w)$. By QC($w$) if $f^i(u)\nsim x_{i+1}$ or by TC($w$) if
$f^i(u)\sim x_{i+1}$, there exists a vertex $x_{i} \sim f^i(u), x_{i+1}$ such that $d(x_{i},w)=i+1$.
Since $f(x_i) \preceq x_{i+1} \prec f^{i+1}(w)$, we obtain
$x_i \prec f^i(w)$, contradicting the induction hypothesis.
\end{proof}

Applying %the assertions of 
Lemma \ref{iteration} with $i=k$, we obtain
that $k+2 = d(f^k(u),w) = d(v_0,w) = k$, a contradiction. This ends
the proof of Theorem \ref{t:dpBFS}.
\end{proof}

%==========================================================

\section{Weakly modular complex}
\label{s:WMcpl}
Although in several previous chapters we studied complexes associated with various classes of weakly modular graphs, there is no  general notion of
a ``weakly modular complex''. One would like this complex to possess some natural features, analogous to the ones of e.g.\ CAT(0) cubical complexes, systolic complexes, bucolic complexes, etc.
In particular, we would like this complex to be locally finite (as a complex) when the
underlying graph is so, invariant under automorphisms of the underlying graph (functoriality), and contractible. In this section we propose
a very particular construction, that might lead to a complex with those properties --- we call such an object ``WM--complex'' (from ``weakly modular''). In several  particular
cases this complex behaves as required. However, at the moment we have no proof that it behaves well in general.

\begin{definition}[Reduced diagonal extension]
The \emph{reduced $k$--diagonal extension} of a graph $G$, denoted $D^{(k)}(G)$, is defined as follows: $D^{(0)}(G)=G$, and
$D^{(k+1)}(G)$ is obtained from $D^{(k)}(G)$ by adding the diagonals of all squares of $D^{(k)}(G)$ containing at least
one edge of the initial graph $G$.
\end{definition}

\begin{definition}[WM--complex]
For a graph $G$, its \emph{diagonal rank}, ${\rm rk}_D(G)$ is a minimal number $k$ (or $\infty$) such that $D^{(k+1)}(G)=D^{(k)}(G)$. If ${\rm rk}_D(G)=m<\infty$, then set
$D^*(G):=D^{(m)}(G)$.  For a weakly modular graph $G$ of finite diagonal rank, by a \emph{WM--complex} (``weakly modular complex") associated with $G$,
denoted $X_{\bowtie}(G)$, we mean the flag simplicial complex with $1$--skeleton $D^*(G)$ (i.e., $X_{\bowtie}(G)$ is the clique complex
of $D^*(G)$).
\end{definition}

\begin{example}
\begin{enumerate}
\item For a bridged graph $G$, the WM--complex $X_{\bowtie}(G)$ is  the corresponding systolic complex (${\rm rk}_D=0$, i.e.\ no edges are added);
\item
For a median graph $G$, its WM--complex $X_{\bowtie}(G)$ is the thickening  of the CAT(0) cubical complex with $1$--skeleton $G$ (cf.\ \cite{O-cnpc});
\item
For a bucolic graph $G$, its WM--complex $X_{\bowtie}(G)$ is the thickening of the corresponding bucolic complex;
\item
The WM--complex $X_{\bowtie}(G)$ of an octahedron is a simplex spanned by all vertices of the octahedron.
\end{enumerate}
\end{example}

We show here that
the diagonal extension of an swm-graph is equal to its thickening.
For an swm-graph $G$,
let $G^{\Delta,k}$ denote the graph obtained
by joining all Boolean pairs $(x,y)$ with $d(x,y) \leq k$.
By definition, $G^{\Delta} = \lim_{k \rightarrow \infty} G^{\Delta,k}$.
\begin{Prop}
For an swm-graph $G$,
its reduced $k$--diagonal extension $D^{(k)}(G)$ is equal to
$G^{\Delta,k+1}$. In particular, $X_{\bowtie}(G)$ is equal to the clique complex of $G^{\Delta} = D^{*}(G)$, and
the diagonal rank ${\rm rk}_D(G)$ is equal to
the cube-dimension of $G$ minus one.
\end{Prop}
\begin{proof}
We show $D^{(k)}(G) = G^{\Delta, k+1}$
by the induction on $k$; obviously $D^{(0)}(G) = G = G^{\Delta,1}$.
Suppose $D^{(k)}(G) = G^{\Delta, k+1}$.
First, we show that any edge $xy$ of $G^{\Delta, k+2}$
is an edge of $D^{(k+1)}(G)$.
Then $(x,y)$ is a Boolean pair with $d(x,y) \leq k+2$.
If $d(x,y) \leq k+1$, then $xy$ belongs to
$G^{\Delta,k+1} = D^{(k)}(G)$, and hence is an edge of  $D^{(k+1)}(G)$.
Suppose that $d(x,y) = k+2$.
Recall that $I(x,y)$ is a complemented modular lattice
(Section~\ref{subsec:latt_char_swm}).
Pick an atom $a$ in $I(x,y)$ and let $b$ be a complement  of $a$ in $I(x,y)$.
Then $(a,y)$ and $(x,b)$ are Boolean pairs of distance $k+1$, while  $(x,y)$ and $(a,b)$ are Boolean pairs of distance $k+2$.
Thus $aybx$ is a square in $D^{(k)}(G) = G^{\Delta,k+1}$
containing the edges $xa$ and $yb$ of $G$.
By definition, $xy$ and $ab$ are edges of  $D^{(k+1)}(G)$.

Next, we show the converse: any edge of $D^{(k+1)}(G)$ is an edge of $G^{\Delta,k+2}$.
We will use the following lemma for squares and 4-cycles.
\begin{lemma} \label{squaresD}
  \

\begin{enumerate}[(1)]
\item For  a square $xyzw$ in $D^{(k)}(G) = G^{\Delta, k+1}$ with $yz
  \in E(G)$, we have $d(x,z) = d(x,y) + 1$ and $d(w,y) = d(w,z) + 1$.
%%   it holds
%% \begin{equation}\label{eqn:wemusthave}
%% d(x,z) = d(x,y) + 1, \ d(w,y) = d(w,z) + 1.
%% \end{equation}
\item For any 4-cycle $(x,y,z,w)$ in $G^\Delta$ such that  $yz \in E(G)$,
$d(x,z) = d(x,y) + 1$ and $d(w,y) = d(w,z) + 1$, $(x,z)$ is
Boolean if and only if $(y,w)$ is Boolean.
\end{enumerate}
\end{lemma}
\begin{proof} To (1): Suppose this is not true, and assume that
  $d(x,z) \leq d(x,y)$.  If $d(x,z) = d(x,y)-1$, i.e., $z \in I(x,y)
  (\subseteq \lgate x,y \rgate)$, then $(x,z)$ is a Boolean pair with
  $d(x,z) \leq k$.  If $d(x,z) = d(x,y)$, then by
  Lemma~\ref{lem:Boolean123}(1), $(x,z)$ is a Boolean pair with
  $d(x,z) \leq k+1$.  In both cases, $x$ and $z$ define an edge in
  $G^{\Delta, k+1}$, contradicting the fact that $xyzw$ is a square in
  $G^{\Delta, k+1}$. Hence, $d(x,z) = d(x,y) + 1$ and for similar
  reasons, we have $d(w,y) = d(w,z) + 1$.
		
To (2): If $(x,z)$ is Boolean, then by
Lemma~\ref{lem:pairwiseBoolean}, there exists a Boolean gated set $B$
containing $x,z,w$.  Then $B$ contains $I(x,z)$, and $I(x,z)$ contains
$y$.  Thus the pair $(y,w)$ is also Boolean.
\end{proof}
	
Consider a square $xyzw$ in $D^{(k)}(G) = G^{\Delta, k+1}$ with $yz
\in E(G)$.  Then $(x,y)$, $(z,w)$, and $(w,x)$ are Boolean pairs of
distances at most $k+1$. Since $xz,yw \notin E(D^{(k)}(G))$, we know
that by induction hypothesis, at least one of the distances is
$k+1$. In the following, we show that $(x,z)$ and $(y,w)$ are Boolean
pairs. Note that this implies that $d(x,z) = d(y,w) = k+2$ as $xz,yw
\notin E(G^{\Delta, k+1})$ and that both $xz$ and $yw$ are edges of
$G^{\Delta,k+2}$.

Note that by Lemma~\ref{squaresD}(1), we have $d(x,z) = d(x,y)+1$ and
$d(w,y) = d(w,z) +1$.  By Lemma \ref{squaresD}(2), it is thus
sufficient to show that one of the pairs $(x,z)$ and $(y,w)$ is
Boolean.  We prove this by induction on $\min \{d(x,y),d(w,z)\}$ and
without loss of generality, we assume that $d(x,y) \leq d(w,z)$.

%% By Lemma \ref{squaresD}(2),  it suffices to show that one of the pairs $(x,z)$ and $(y,w)$
%% is Boolean. We proceed by induction on the distance $d(x,y)$.
%% %and we distinguish two cases.

If $z\in I(x,w)$, then $(x,z)$ is Boolean since $(x,w)$ is
Boolean. Assume now that $z \notin I(x,w)$. We assert that there
exists a vertex $z'\in I(x,z)$ such that $z'\sim z$, $d(z',w)\le
d(z,w)$, and $(z',w)$ is Boolean.  If $I(z,x)\cap I(z,w)\ne \{z\}$,
let $z'$ be a neighbor of $z$ in $I(z,x)\cap I(z,w)$. Since $z'\in
I(z,w)$ and $(z,w)$ is Boolean, $(z',w)$ is also Boolean and
$d(z',w)<d(z,w)$. Now suppose that $I(z,x)\cap I(z,w)=\{z\}$, and
consider a quasi-median $zx'w'$ of $z,x,w$ with $x' \in I(x,z)$ and
$w' \in I(w,z)$. Let $z'$ be a neighbor of $z$ in $I(z,x')$, and note
that by Lemma~\ref{lem-weakly-modular}, we have $d(z',w') =
d(z,w')$. Consequently, $d(z',w) = d(z,w)$  and $(z',w)$ is Boolean
by Lemma~\ref{lem:Boolean123}(1).

Note that since $d(w,y) = d(w,z)+1$ and since $d(w,z') \leq d(w,z)$,
we have $z' \neq y$.  Since $d(x,y)=d(x,z')$ and $z\sim y,z'$, the
vertices $y$ and $z'$ are not adjacent: otherwise, by (TC) we can find
$y'\sim y,z'$ with $d(x,y')=d(x,y)-1$ and the vertices $y,z,z',y'$
induce a forbidden $K_4^-$. Hence $y \nsim z'$. By (QC), we can find
$y' \sim y,z'$ with $d(x,y') = d(x,y) - 1$. If $y' = x$, i.e., if
$d(x,y) = 1$, then $x,z$ belongs to a square $xyzz'$ and $(x,z)$ is
Boolean. Suppose now that $y' \neq x$ and note that $d(z',w) \geq
d(z,w) - 1 \geq d(y',x) \geq 1$.  Since $y'\in I(x,y)$ and $(x,y)$ is
Boolean, $(x,y')$ is also Boolean.

Note that $d(x,z') = d(x,y') + 1$ and that $d(w,z) - 1 \leq d(w,z') \leq
d(w,z) = d(w,y) - 1 \leq d(w,y') \leq d(w,z')+1
\leq d(w,z)+1$, implying that $d(w,y') = d(w,z')$ or $d(w,y') = d(w,z') +1$.
If $d(w,y') = d(w,z')$, then by Lemma~\ref{lem:Boolean123}(1),
$(w,y')$ is a Boolean pair since $(w,z')$ is a Boolean pair.
If $d(w,y') = d(w,z') + 1$, the 4-cycle $(x,y',z',w)$ satisfies the
conditions of Lemma~\ref{squaresD}(2). Therefore, if this 4-cycle is
not induced in $G^{\Delta,k+1}$, then by Lemma \ref{squaresD}(2) both
pairs $(x,z')$ and $(y',w)$ are Boolean.  On the other hand, if the
4-cycle $(x,y',z',w)$ is a square in $G^{\Delta,k+1}$, since $d(x,y')<
d(x,y)$ we can apply the induction hypothesis and deduce that the
pairs $(x,z')$ and $(y',w)$ are Boolean.

Consequently, in both cases $(w,y')$ is Boolean and by
Lemma~\ref{lem:Boolean123}(2) applied to the Boolean pairs
$(w,z),(w,y')$ and a common neighbor $y$ of $y',z$, we conclude that
$(w,y)$ is Boolean.
\end{proof}

Clearly, the definition of $X_{\bowtie}(G)$ is functorial, and the WM--complex in the above examples satisfies all the required properties.
However, in general the following questions are open.
\medskip

\begin{question}
Is the diagonal rank ${\rm rk}_D(G)$ of a uniformly locally finite weakly modular graph $G$ finite? Is the resulting WM--complex $X_{\bowtie}(G)$ contractible?
\end{question}

%----------------------- end of Sections ---------------------------------------

\backmatter

%=======================================================================

\printindex

\end{document}